\documentclass[reqno]{amsart}
\usepackage[foot]{amsaddr}
\usepackage{geometry}
\usepackage{hyperref}
\usepackage{xcolor}
\usepackage{tikz}
\usetikzlibrary{decorations.pathreplacing}
\usepackage[utf8]{inputenc}
\usepackage[english]{babel}
\usepackage{upgreek}
\usepackage{tipa}
\usepackage{lastpage}
\usepackage{float}
\usepackage[shortlabels]{enumitem}
\usepackage{subcaption}
\usepackage{multirow}
\usepackage{graphicx} 
\graphicspath{ {c:/Users/Owner/Desktop} }
\usepackage{amsmath}
\usepackage{amssymb}
\usepackage{dsfont}
\usepackage{amsthm}
\theoremstyle{plain}
\newtheorem{thm}{Theorem}[section]
\newtheorem{lem}[thm]{Lemma}

\newtheorem{prop}[thm]{Proposition}
\newtheorem{conj}[thm]{Conjecture}

\theoremstyle{definition}
\newtheorem{define}[thm]{Definition}

\theoremstyle{remark}
\newtheorem*{remark}{Remark}
\usepackage{enumitem}
\newcommand{\conditions}[2]{$#1 #2$}
\usepackage{thmtools}
\usepackage{thm-restate}

\title{Interacting Particle Systems and Jacobi style identities}
\author{M\'arton Bal\'azs\(^1\)}
\address{\(^1\)University of Bristol}
\author{Dan Fretwell\(^2\)}
\author{Jessica Jay\(^2\)}
\address{\(^2\)Heilbronn Institute for Mathematical Research, University of Bristol}
\date{}

\begin{document}

\begin{abstract}
We consider the family of nearest neighbour interacting particle systems on $\mathbb{Z}$ allowing $0$, $1$ or $2$ particles at a site. We parametrize a wide subfamily of processes exhibiting product blocking measure and show how this family can be ``stood up'' in the sense of Bal\'azs and Bowen \cite{blocking}. By comparing measures we prove new three variable Jacobi style identities, related to counting certain generalised Frobenius partitions with a $2$-repetition condition. By specialising to specific processes we produce two variable identities that are shown to relate to Jacobi triple product and various other identities of combinatorial significance. The family of $k$-exclusion processes for arbitrary $k$ are also considered and are shown to give similar Jacobi style identities relating to counting generalised Frobenius partitions with a $k$-repetition condition.
\end{abstract}

\maketitle

\section{Introduction}
The Jacobi triple product is a two variable identity relating an infinite sum with an infinite product: \[\sum_{k\in\mathbb{Z}}q^{k^2}z^{k} = \prod_{i=1}^{\infty}(1-q^{2i})(1+q^{2i-1}z)(1+q^{2i-1}z^{-1}).\] It is a foundational identity that has many applications in combinatorics and number theory, e.g.\ the theory of partitions and the theory of modular forms (Jacobi forms).

While the Jacobi triple product is a classical identity and can be proved using elementary ad hoc methods, one finds that most modern proofs exhibit it as an equality of traces of operators acting on certain infinite dimensional vector spaces. For example, algebraically it is well known that the identity manifests itself in the characters of certain infinite dimensional representations of affine Lie algebras (e.g.\ it is the denominator identity for $\mathfrak{sl}_2^{(1)}$). In Physics it is well known as a consequence of the Boson-Fermion correspondence; both sides of the identity represent the same partition function, calculated in two ways by cataloguing fermionic and bosonic states by their energy and charge relative to a ground state (eigenvalues of operators acting on the two infinite dimensional Fock spaces of states).

In \cite{blocking}, Bal\'azs and Bowen found a probabilistic proof of the Jacobi Triple Product by interpreting it as an equality between  the ASEP blocking measure (product Bernoulli) and the corresponding ``stood up'' AZRP measure (product geometric). The infinite sum arises via normalization of the ASEP measure. The $\text{ASEP}$-$\text{AZRP}$ correspondence can be seen as the probabilistic analogue of the Boson-Fermion correspondence (with the conserved quantity of a state relative to the zero site being the analogue of charge relative to a ground state). 

The advantage of the algebraic/physical interpretations of the identity is that they can be vastly generalised, e.g.\ by considering other representations, other Lie algebras, other Fock spaces. Each approach provides more Jacobi style identities, e.g.\ the Macdonald identities and the the Rogers-Ramanujan identities. These lead to other interesting connections with partitions and modular forms. Given this, it is natural to ask which other interacting particle systems provide Jacobi style identities. 

Of course not all processes will work, there are restrictions. For example, in order to get Jacobi style identities it seems necessary to find processes on $\mathbb{Z}$ that can be stood up and have interesting blocking measures (product or near product). The only such example for $0$-$1$ processes is ASEP (as seen in \cite{blocking}). 

In this paper we parametrise a family of (non-degenerate) $0$-$1$-$2$ systems on $\mathbb{Z}$ with product blocking measure, purely in terms of conditions on the rates. We then show that they are a valid family to consider in the above sense, giving precise details of stationary blocking measures and stood up processes. The corresponding measures then depend on three parameters $\tilde{q},t$ and $c$ and viewing these as variable over the family we prove the following three variable Jacobi style identities (letting $z = \Tilde{q}^{-c}$):

\begin{restatable}{thm}{main}\label{main} 
\par For $0<\Tilde{q}<1$, $t\geq 1$ and $z> 0$
\begin{align*}
2\sum_{\ell\in\mathbb{Z}} S_{\text{even}}(\Tilde{q},t) \Tilde{q}^{\ell(\ell+1)} z^{2\ell} &= \prod_{i\geq 1}(1+tz\Tilde{q}^i+z^2\Tilde{q}^{2i})(1+tz^{-1}\Tilde{q}^{i-1}+z^{-2}\Tilde{q}^{2(i-1)}) \\ &+ \prod_{i\geq 1}(1-tz\Tilde{q}^i+z^2\Tilde{q}^{2i})(1-tz^{-1}\Tilde{q}^{i-1}+z^{-2}\Tilde{q}^{2(i-1)}) \\
2t\sum_{\ell\in\mathbb{Z}} S_{\text{odd}}(\Tilde{q},t) \Tilde{q}^{(\ell+1)^2} z^{2\ell+1} &= \prod_{i\geq 1}(1+tz\Tilde{q}^i+z^2\Tilde{q}^{2i})(1+tz^{-1}\Tilde{q}^{i-1}+z^{-2}\Tilde{q}^{2(i-1)}) \\ &- \prod_{i\geq 1}(1-tz\Tilde{q}^i+z^2\Tilde{q}^{2i})(1-tz^{-1}\Tilde{q}^{i-1}+z^{-2}\Tilde{q}^{2(i-1)}),
\end{align*}
\end{restatable}

The functions $S_{\text{even}}(\Tilde{q},t)$ and $S_{\text{odd}}(\Tilde{q},t)$ are normalising factors for measures on the two irreducible components of the stood up process: \begin{align*}
S_{\text{even}}(\Tilde{q},t) &= \sum\limits_{\underline{\omega}\in \mathcal{H}^e}\Tilde{q}^{\sum\limits_{\textrm{ }i \text{ odd}}i\omega_{-i}+\sum\limits_{i \text{ even}}
  i(\omega_{-i}-1)}t^{2\left(\sum\limits_{i \text{ odd}}\mathbb{I}\{\omega_{-i}\geq 1\}-\sum\limits_{i \text{ even}}\mathbb{I}\{\omega _{-i}=0 \}\right)} \\  S_{\text{odd}}(\Tilde{q},t) &= \sum\limits_{\underline{\omega}\in \mathcal{H}^o}\Tilde{q}^{\sum\limits_{\textrm{ }i \text{ even}}i\omega_{-i}+\sum\limits_{i \text{ odd}}
  i(\omega_{-i}-1)}t^{2\left(\sum\limits_{i \text{ even}}\mathbb{I}\{\omega_{-i}\geq 1\}-\sum\limits_{i \text{ odd}}\mathbb{I}\{\omega _{-i}=0 \}\right)}.
  \end{align*} 
 Here the state space $\mathcal{H}^e$ is the set of sequences $(\omega_{-i})_{i\geq 1}$ of non-negative integers with no two consecutive zeroes and agreeing with the sequence $(0,1,0,1,...,)$ far enough to the right. The state space $\mathcal{H}^o$ is similar but agrees with $(1,0,1,0,...)$ for enough to the right. 

The normalising factors $S_{\text{even}}(\tilde{q},t)$ and $S_{\text{odd}}(\tilde{q},t)$ as written above might look unappealing (even though they do count explicit combinatorial objects). The form presented here has direct motivation from the standing up procedure of the 0-1-2 state particle systems. In its proof we demonstrate how this form came to life via such probabilistic ideas. We thank an anonymous referee for pointing out that a purely combinatorial proof is also possible, without the use of probability. We thought it is worth seeing the probability connection while proving Theorem \ref{main}, hence opted to keep our original argument.

We also thank the anonymous referee for the challenge of finding natural alternative forms for $S_{\text{even}}$ and $S_{\text{odd}}$. We spent considerable effort finding a closed form of these expressions. While we were not successful, an alternative probabilistic interpretation, giving rather different but still not closed formulas, is subject of a forthcoming paper. Nevertheless, numerical computation to high accuracy has led us to conjecture the following equalities:

\begin{conj}
\begin{align*}
S_{\text{even}}(\tilde{q},t) &= \frac{1}{\prod_{m\geq 1}(1-\tilde{q}^{2m})} + \frac{\sum_{i\geq 1}\sum_{n\geq i}(-1)^{n-i}\binom{n+i-1}{2i-1}\frac{n}{i}\tilde{q}^{n^2}t^{2i}}{\prod_{m\geq 1}(1-\tilde{q}^m)^2}\\
\\
S_{\text{odd}}(\tilde{q},t) &= \frac{\prod_{m\geq 1}(1-\tilde{q}^{2m})^3}{\prod_{m\geq 1}(1-\tilde{q}^m)^2} + \frac{\sum_{i\geq 1}\sum_{n\geq i}(-1)^{n-i}\binom{n+i}{2i}\frac{2n+1}{2i+1}\tilde{q}^{n(n+1)}t^{2i}}{\prod_{m\geq 1}(1-\tilde{q}^m)^2}.
\end{align*}
\end{conj}

The identities in Theorem \ref{main} are shown to have combinatorial significance. By adapting the ``General Principle" found in Andrews' book (\cite{frobenius}) we find that the RHS of these identities relate to generating functions for generalised Frobenius partitions (GFP's) satisfying a $2$-repetition condition on the rows and a condition on the number of non-repeated entries. The content of these identities is then that the normalising factors $S_{\text{even}}(\Tilde{q},t)$ and $S_{\text{odd}}(\Tilde{q},t)$ are really formal generating functions for such GFP's. This is not clear from their explicit definitions and we explain this in detail in Section \ref{general idenities section}.

In Section \ref{specialising section} we specialise to specific $0$-$1$-$2$ systems and recover various two variable identities, also of combinatorial significance. The ASEP$(q,1)$ process of Redig et al.\ (found in \cite{redig}) is an extension of classical ASEP to allow two particles. It gives identities relating directly to Jacobi triple product. A ($3$-state) asymmetric particle-antiparticle exclusion process is considered and gives identities relating directly to the square of Jacobi triple product (which is also seen to have a combinatorial interpretation in terms of $2$-coloured GFP's). The $2$-exclusion process gives identities relating to GFP's with only a $2$-repetition condition on the rows.

Finally, in Section \ref{k-exc section} we generalise the last example to the entire family of $k$-exclusion processes on $\mathbb{Z}$. These are not $0$-$1$-$2$ systems and so we have to adapt our techniques further. However they are sufficiently nicely behaved to allow product blocking measure and have stood up processes, which we describe in detail. The identities we recover are as follows.

\begin{restatable}{thm}{kexc}\label{k-exc main} For $0<q<1$, $z\neq 0$ and $m\in\{0,1,...,k-1\}$

\[k\sum\limits_{\ell \in \mathbb{Z}}S_{-m}^{(k)}(q) q^{\frac{k\ell(\ell+1)}{2}-m\ell}z^{k\ell-m} = \sum_{r=0}^{k-1}\zeta_k^{-rm}\left(\prod_{i\geq 1}\left(\sum_{\alpha=0}^k \zeta_k^{-\alpha r}q^{\alpha i}z^{\alpha}\right)\left(\sum_{\alpha=0}^k \zeta_k^{\alpha r}q^{\alpha (i-1)}z^{-\alpha}\right)\right).\]
\end{restatable}

The functions $S^{(k)}_{-m}(q)$ for $m\in\{0,1,...,k-1\}$ are normalising factors for measures on the $k$ irreducible components of the stood up process: \[S_{-m}^{(k)}(q)=\sum\limits_{\underline{\omega}\in\mathcal{H}^{-m}}q^{\sum\limits_{i \notin k \mathbb{Z}+m}i\omega_{-i}+\sum\limits_{i \in k \mathbb{Z}+m}i(\omega_{-i}-1)}.\] Here the state space $\mathcal{H}^{-m}$ is the set of sequences $(\omega_{-i})_{i\geq 1}$ of non-negative integers that have no $k$ consecutive zeroes and satisfy $\omega_{-i} = \mathbb{I}\{i\equiv m \bmod k\}$ for large enough $i$ (in analogy with $\mathcal{H}^e$ and $\mathcal{H}^o$). The ``General Principle" once again explains the combinatorial nature of these identities. The content is that the normalising factors $S_{-m}^{(k)}(q)$ are formal generating functions for GFP's with $k$-repetition condition in the rows (functions studied in detail in Andrews' book \cite{frobenius}).

Throughout the paper we discuss all of the above in both probabilistic and combinatorial detail, giving full justification where possible and otherwise outlining the mysteries implied by one approach to the other.

\section*{Acknowledgements}
\par \vspace{1mm} \noindent The authors are grateful for discussions with J\'anos Engl\"ander regarding the parity of the sum of Bernoulli random variables and for suggestions of an anonymous referee regarding the combinatorial background of our identities. M.\ Bal\'azs was partially supported by the EPSRC EP/R021449/1 Standard Grant of the UK. This study did not involve any underlying data.

\section{Interacting particle systems with product blocking measure}\label{blocking meaures}
We recall the family of particle systems with product blocking measure introduced by Bal\'azs and Bowen in \cite{blocking}. For possibly infinite integers, $-\infty\leq \ell \leq 0\leq \mathfrak{r}\le \infty$, we define $\Lambda:=\{i : \ell-1<i<\mathfrak{r}+1\} \subseteq \mathbb{Z}$, and for two other possibly infinite integers $-\infty\leq \omega^\textrm{min} \leq 0 < \omega^\textrm{max}\leq \infty $, we define $I:=\{z:\omega^\textrm{min}-1<z<\omega^\textrm{max}+1\}\subseteq \mathbb{Z}$. We consider interacting particle systems on the state space $\Omega=\{\underline{\eta} \in I^\Lambda: (\ell >-\infty \textrm{ or } N_{p}(\underline{\eta})< \infty)\textrm{ and }(\mathfrak{r}<\infty \textrm{ or } N_h(\underline{\eta})<\infty)\}$, where $N_p,N_h:I^\Lambda \rightarrow \mathbb{Z}_{\geq 0}\cup \{\infty\}$ are defined by,
$$N_p(\underline{\eta})= \sum\limits_{i=\ell}^0 (\eta_i-\omega^\textrm{min})\hspace{5mm} \textrm{ and } \hspace{5mm} N_h(\underline{\eta})=\sum\limits_{i=1}^\mathfrak{r}(\omega^\textrm{max}-\eta_i)$$
(when $\omega^\textrm{min}=0$ and $\omega^{\textrm{max}}<\infty$ this is just the number of particles to the left and holes to right of $\frac{1}{2}$). Given a state $\underline{\eta}\in\Omega$ the state $\underline{\eta}^{(i,j)}\in\Omega$ obtained by a particle jumping from site $i$ to $j$ is given by 
$$\left(\underline{\eta}^{(i,j)}\right)_k=\begin{cases}
\eta_k &\textrm{if } k \neq i,j\\
\eta_i-1 &\textrm{if } k=i \\
\eta_j+1 &\textrm{if }k=j.
\end{cases}$$
\par \noindent We will only consider processes with nearest neighbour interactions, i.e.\ the only states we can reach from $\underline{\eta}$ are $\underline{\eta}^{(i,i+1)}$ or $\underline{\eta}^{(i+1,i)}$. The system then evolves according to Markov generators; in the `bulk', i.e.\ $\ell \leq i\leq \mathfrak{r}-1$ the generator has the form 
$$\left(L^\textrm{bulk}\varphi\right)(\underline{\eta})=\sum\limits_{i=\ell}^{\mathfrak{r}-1}\left\{p(\eta_i,\eta_{i+1})\left(\varphi(\underline{\eta}^{(i,i+1)})-\varphi(\underline{\eta})\right)+q(\eta_i,\eta_{i+1})\left(\varphi(\underline{\eta}^{(i+1,i)})-\varphi(\underline{\eta})\right)\right\}$$
\par \noindent for some cylinder function $\varphi:\Omega \rightarrow \mathbb{R}$ and functions $p,q:I^2\rightarrow [0,\infty)$ (the right and left jump rates respectively). If $\ell>-\infty$ we consider an open left boundary with boundary jump rates $p_\ell,q_\ell :I \rightarrow [0,\infty)$. We introduce the notation $\underline{\eta}^{(\ell-1,\ell)}$ to denote the state reached from $\underline{\eta}$ by a particle entering the system through the left boundary and similarly $\underline{\eta}^{(\ell,\ell-1)}$ to be the state where a particle has left through the boundary. So the left boundary generator is of the form 
$$\left(L^\ell \varphi\right)(\underline{\eta})=p_\ell(\eta_\ell)\left(\varphi(\underline{\eta}^{(\ell-1,\ell)})-\varphi(\underline{\eta})\right)+q_\ell(\eta_\ell)\left(\varphi(\underline{\eta}^{(\ell,\ell-1)})-\varphi(\underline{\eta})\right).$$
\par \noindent Similarly if $\mathfrak{r}<\infty$ we consider an open right boundary with boundary jump rates $p_\mathfrak{r}, q_\mathfrak{r}:I \rightarrow [0,\infty)$. We let $\underline{\eta}^{(\mathfrak{r}+1,\mathfrak{r})}$  denote the state reached from $\underline{\eta}$ when a particle enters the system through the right boundary and similarly $\underline{\eta}^{(\mathfrak{r},\mathfrak{r}+1)}$ the state where a particle has left through the boundary. So the right boundary generator is of the form 
$$\left(L^\mathfrak{r} \varphi\right)(\underline{\eta})=p_\mathfrak{r}(\eta_\mathfrak{r})\left(\varphi(\underline{\eta}^{(\mathfrak{r},\mathfrak{r}+1)})-\varphi(\underline{\eta})\right)+q_\mathfrak{r}(\eta_\mathfrak{r})\left(\varphi(\underline{\eta}^{(\mathfrak{r}+1,\mathfrak{r})})-\varphi(\underline{\eta})\right).$$
\par In order to be a member of the blocking family the jump rates of the system  must obey the following conditions:
\begin{enumerate}[label=(\conditions{B}{{\arabic*}})]
    \item For $-\infty<\omega^{\textrm{min}}< \omega^{\textrm{max}}<\infty$ we have that 
    $$p(\omega^{\textrm{min}}, \cdot)=p(\cdot, \omega^{\textrm{max}})=q(\omega^{\textrm{max}}, \cdot)=q(\cdot, \omega^{\textrm{min}})=0.$$
    \par \noindent If $\ell>-\infty$, 
    $$q_\ell(\omega^\textrm{min})=p_\ell(\omega^\textrm{max})=0$$
    \par \noindent and if $\mathfrak{r}<\infty$,
    $$q_\mathfrak{r}(\omega^\textrm{max})=p_\mathfrak{r}(\omega^\textrm{min})=0.$$
    \item The system is attractive; that is $p(\cdot,\cdot)$ is non-decreasing in the first variable and non-increasing in the second whilst $q(\cdot,\cdot)$ is non-increasing in the first variable and non-decreasing in the second. If $\ell>-\infty$, $p_\ell$ is non-increasing and $q_\ell$ non-decreasing. If $\mathfrak{r}<\infty$, $p_\mathfrak{r}$ is non-decreasing and $q_\mathfrak{r}$ non-increasing.
    \item There exist $p_{\text{asym}},q_{\text{asym}}\in\mathbb{R}$ satisfying $\frac{1}{2}<p_{\textrm{asym}}=1-q_{\textrm{asym}} \leq 1$, and functions $f: I \rightarrow [0,\infty)$ and $s: I \times I \rightarrow [0,\infty)$ such that,
    $$p(y,z)=p_{\textrm{asym}}\cdot s(y,z+1)\cdot f(y) \hspace{5mm}\textrm{ and } \hspace{5mm}q(y,z)=q_{\textrm{asym}}\cdot s(y+1,z)\cdot f(z).$$
    If $\omega^\textrm{min}$ is finite then $f(\omega^\textrm{min})=0$, and if $\omega^\textrm{max}$ is finite  we extend the domain of $s$ and require that $s(\omega^\textrm{max}+1,\cdot)=s(\cdot,\omega^\textrm{max}+1)=0$. (Note that attractivity implies that $s$ is non-increasing in both of its variables and $f$ is non-decreasing). 
\end{enumerate}

A priori the above definition allows many of the rates to be zero. However in this paper we will assume that all rates, except those in $(B1)$, are non-zero. Roughly speaking this allows us to ignore degenerate processes and assume that $f(z)>0$ if $z>\omega^{\text{min}}$ (which we will do from now on).

\par In general there are many stationary distributions for systems satisfying the above. In this paper we will be interested in the following one-parameter family of product stationary blocking measures, written explicitly in terms of $p_\textrm{asym}$, $q_\textrm{asym}$ and $f$.
\begin{thm}\label{stationary dist}\emph{(\cite{blocking}, Theorem $3.1$)}
\par \noindent  For each $c\in\mathbb{R}$ there is a product stationary blocking measure $\underline{\mu}^c$ on $\Omega$, given by the marginals 
$$\mu_i^c(z)=\frac{1}{Z_i^c}\frac{\left(\frac{p_\textrm{asym}}{q_\textrm{asym}}\right)^{(i-c)z}}{f(z)!}\hspace{5mm} \textrm{ for } i\in \Lambda \text{ and } z \in I,$$
\par \noindent where $Z_i^c$ is the normalising factor and $f(z)!:=\begin{cases} \prod\limits_{y=1}^zf(y) &\textrm{for } z>0 \\
1 &\textrm{for } z=0 \\
\frac{1}{\prod\limits_{y=z+1}^{0}f(y)} &\textrm{for } z<0.
\end{cases}$
\end{thm} 

\begin{remark}
It is clear that if $f$ and $s$ satisfy $(B3)$ then so do $\alpha f$ and $\alpha^{-1} s$ for any $\alpha>0$. This scaling simply shifts the value of $c$ in the blocking measure,
$$\mu_i^c(z)=\frac{1}{Z_i}\frac{\left(\frac{p_\textrm{asym}}{q_\textrm{asym}}\right)^{(i-c)z}}{\alpha^zf(z)!}=\frac{1}{Z_i}\frac{\left(\frac{p_\textrm{asym}}{q_\textrm{asym}}\right)^{(i-\Tilde{c})z}}{f(z)!}=\mu_i^{\Tilde{c}}(z).$$
\par \noindent So without loss of generality we can assume that $\prod\limits_{y=\omega^\textrm{min}+1}^{\omega^\textrm{max}}f(y)=1$ when $\omega^{\text{min}}$ and $\omega^{\text{max}}$ are finite.
\end{remark}

As expected, being a member of the blocking family imposes strict constraints on the jump rates. In particular for each $y,z \in I \setminus \{\omega^\textrm{min}\}$ we can use $(B3)$ to write
$$\frac{f(z)}{f(y)}=\frac{p_\textrm{asym}}{q_\textrm{asym}}\frac{q(y-1,z)}{p(y,z-1)}.$$
Setting $y=z$ and recalling that $p_{\text{asym}} = 1- q_{\text{asym}}$ gives
$$p_\textrm{asym}=\frac{p(y,y-1)}{q(y-1,y)+p(y,y-1)} \hspace{10mm} \textrm{and} \hspace{10mm} q_\textrm{asym}=\frac{q(y-1,y)}{q(y-1,y)+p(y,y-1)}.$$

\par \noindent Note that $p_{\text{asym}}>q_{\text{asym}}$ implies that $p(y,y-1)>q(y-1,y)$ for all $y \in I\setminus \{\omega^\textrm{min}\}$. This, along with $(B2)$ shows the condition $$(\text{a})\quad p(y,z)>q(z,y)\hspace{10mm} \text{for all } y\in I\setminus \{\omega^{\text{min}}\} \text{ and } z\in I\setminus\{\omega^{\text{max}}\}$$ (i.e.\ the process is asymmetric with right drift).

\par \noindent Also by assumption $\frac{1}{2} < p_{\text{asym}}$ is a constant (as is $q_{\text{asym}}<\frac{1}{2}$) and so we must have the condition that  
$$(\text{b})\quad\frac{p(y,y-1)}{q(y-1,y)}=\textrm{constant}>1 \hspace{10mm} \text{for all } y \in I\setminus \{\omega^\textrm{min}\}.$$

\par \noindent We then see that, for $y,z\in I\setminus \{\omega^{\text{min}}\}$
$$\frac{f(z)}{f(y)}=\frac{p(z,z-1)}{q(z-1,z)}\frac{q(y-1,z)}{p(y,z-1)}.$$
Since $\frac{f(y)}{f(z)}\cdot\frac{f(z)}{f(y)} = 1$ this gives the condition
$$(\text{c})\quad\frac{p(z,z-1)p(y,y-1)q(y-1,z)q(z-1,y)}{q(z-1,z)q(y-1,y)p(y,z-1)p(z,y-1)}=1 \hspace{10mm} \text{for all } y,z\in I\setminus \{\omega^{\text{min}}\}.$$
It is then clear that the function $s$ is uniquely determined as follows, for $y,z \in I\setminus \{\omega^\textrm{min}\}$
$$s(y,z)=\frac{p(y,z-1)\left(q(y-1,y)+p(y,y-1)\right)}{f(y)p(y,y-1)}=\frac{q(y-1,z)\left(q(y-1,y)+p(y,y-1)\right)}{f(z)q(y-1,y)}.$$

To summarise, conditions (a), (b) and (c) are necessary conditions on the rates that are implied by being a member of the blocking family. However in general a process with rates only satisfying $(B1)$, $(B2)$ and these three conditions is not expected to be a member of the blocking family (the quantities $p_{\text{asym}}$ and $q_{\text{asym}}$ as written above are well defined and satisfy $\frac{1}{2}<p_{\textrm{asym}}=1-q_{\textrm{asym}} \leq 1$, but it is not clear that the functions $f$ and $s$ should exist purely from these conditions).

\section{General $0$-$1$-$2$ systems on $\mathbb{Z}$ with blocking measure}\label{general section}
\par For the choices $I = \{0,1\} \textrm{ and }\Lambda={\mathbb{Z}}$ the only member of the blocking family is ASEP, handled in \cite{blocking}. In this section we will consider the case of $I= \{0,1,2\} \textrm{ and } \Lambda=\mathbb{Z}$. In this case the conditions $(B1)$, $(B2)$, $(B3)$ translate into the following conditions on the rates (using the assumptions and discussions in Section $2$):
\begin{enumerate}[label=(\conditions{B}{{\arabic*}})]
    \item  
    $$p(0, \cdot)=p(\cdot, 2)=q(2, \cdot)=q(\cdot, 0)=0,$$
    \item  \begin{align*} p(2,\cdot)&\geq p(1,\cdot) > 0\\
         p(\cdot,0)&\geq p(\cdot,1) > 0\\
         q(0,\cdot)&\geq q(1,\cdot) > 0\\
         q(\cdot,2) &\geq q(\cdot,1) > 0,
    \end{align*}
    \item For all $y,z\in\{0,1,2\}$ we have
    
    $$p(y,z)=p_{\textrm{asym}}\cdot s(y,z+1)\cdot f(y) \hspace{5mm}\textrm{ and } \hspace{5mm}q(y,z)=q_{\textrm{asym}}\cdot s(y+1,z)\cdot f(z)$$
    
\par \noindent    with $p_{\text{asym}},q_{\text{asym}}$ given by
    
     $$p_\textrm{asym}=\frac{p(1,0)}{q(0,1)+p(1,0)} >\frac{1}{2} \hspace{5mm} \hspace{5mm} q_\textrm{asym}=\frac{q(0,1)}{q(0,1)+p(1,0)}< \frac{1}{2}$$
     
     \par \noindent and $f,s$ given by
     
    \[
     f(z):=
    \begin{cases}
     0 \hspace{2mm} \textrm{if } z=0  \\
     \\
    t^{-1} \hspace{2mm}\textrm{if } z=1  \\
      \\
    t \hspace{2mm}\textrm{if } z=2
    \end{cases} 
    \qquad 
    s(y,z):=
    \begin{cases}
       t\left(q(0,1)+p(1,0)\right) &\textrm{if } y=z=1  \\
       \\
       \frac{tp(1,1)\left(q(0,1)+p(1,0)\right)}{p(1,0)}&\textrm{if } y=1 \textrm{ and } z=2 \\
       \\
       \frac{p(2,0)\left(q(0,1)+p(1,0)\right)}{p(1,0)t}  &\textrm{if } y=2 \textrm{ and } z=1 \\
       \\
       \frac{p(2,1)\left(q(0,1)+p(1,0)\right)}{p(1,0)t} &\textrm{if } y=z=2 \\
       \\
       0 &\textrm{if } y=3 \textrm{ or } z=3,
    \end{cases}
   \]
   for some $t\geq 1$ (recall that we can assume $f(1)f(2)=1$ without loss of generality).
   \end{enumerate}
   
   \begin{remark} Note that, since $\frac{f(2)}{f(1)} = t^2 = \frac{p(1,0)q(0,2)}{q(0,1)p(1,1)}$ we see that $t=\left(\frac{p(1,0)q(0,2)}{q(0,1)p(1,1)}\right)^{\frac{1}{2}}\geq 1$ is uniquely determined.\end{remark}
 
 \par \noindent As previously mentioned $(B1),(B2)$ and $(B3)$ imply the following necessary conditions for the rates:
    \begin{enumerate}[(a)]
        \item $p(y,z)>q(z,y)$ for all $y\in\{1,2\}$ and $z\in\{0,1\},$
        \item $\frac{p(1,0)}{q(0,1)}=\frac{p(2,1)}{q(1,2)},$
        \item $\frac{p(1,0)p(2,1)q(1,1)q(0,2)}{q(0,1)q(1,2)p(2,0)p(1,1)}=1.$
    \end{enumerate}
    
Unlike the general case these conditions are now also sufficient, in the sense that any process with $I = \{0,1,2\}$ and $\Lambda=\mathbb{Z}$ satisfying $(B1), (B2)$ and conditions $(a),(b)$ and $(c)$ will automatically satisfy $(B3)$ with $p_{\text{asym}},q_{\text{asym}},f$ and $s$ as given above (all well defined). Thus we have fully parametrised the family of such blocking processes using only conditions on the rates.

\par By Theorem \ref{stationary dist}, for each such process there is a one parameter family of product stationary blocking measures $\underline{\mu}^c$ given by the marginals

$$\mu_i^c(z)=\frac{t^{\mathbb{I}\{z=1\}}\Tilde{q}^{-(i-c)z}}{Z_i^c(\Tilde{q},t)} \hspace{10mm} \text{for } z\in\{0,1,2\}.$$
\par \noindent Here $0<\Tilde{q}=\frac{q_\textrm{asym}}{p_\textrm{asym}}=\frac{q(0,1)}{p(1,0)}<1$  and $Z_i^c(\Tilde{q},t):=1+t\Tilde{q}^{-(i-c)}+\Tilde{q}^{-2(i-c)}$ (the normalising factor). 
\par \noindent Explicitly: 
\[
\underline{\mu}^c(\underline{\eta})=\prod\limits_{i=-\infty}^\infty \mu_i^c(\eta_i)=\prod\limits_{i=-\infty}^0 \frac{ t^{\mathbb{I}\{\eta_i=1\}}\Tilde{q}^{-(i-c)\eta_i}}{Z_i^c(\Tilde{q},t)}\prod\limits_{i=1}^\infty \frac{ t^{\mathbb{I}\{\eta_i=1\}}\Tilde{q}^{(2-\eta_i)(i-c)}}{\Tilde{q}^{2(i-c)}Z_i^c(\Tilde{q},t)}.
\]

\begin{remark} Note that these measures only depend on the parameter $c$ and the quantities $\Tilde{q},t$ attached to the process. Roughly speaking this will be the reason for getting a three variable identity later; as we run through the whole family of such processes these parameters will be variables. Specialising to particular subfamilies of processes, for example fixing $t$ or letting $t$ and $\Tilde{q}$ be related, will give two variable identities ($c$ will still be a variable as will $\Tilde{q}$, since the rates will typically depend on an asymmetry parameter $0<q<1$).\end{remark}

\par By definition of the state space, each $\underline{\eta}\in\Omega$ has $\eta_i=0$ for all $i$ small enough and $\eta_i=2$ for all $i$ big enough. We refer to these events as having a left most particle (LMP) and right most hole (RMH). By asymmetry it then follows that the ground states of $\Omega$ (i.e.\ the most probable states) are all shifts of the following ``even" and ``odd" ground states:
\begin{align*}
    \eta^e_i=
    \begin{cases}
    2 &\textrm{ if } i\geq 1 \\
    0 &\textrm{otherwise}
    \end{cases}
    &\qquad
    \eta^o_i=
    \begin{cases}
    2 &\textrm{ if } i\geq 1 \\
    1 &\textrm{ if } i=0\\
    0 &\textrm{otherwise.}
    \end{cases}
\end{align*}
\vspace{5mm}
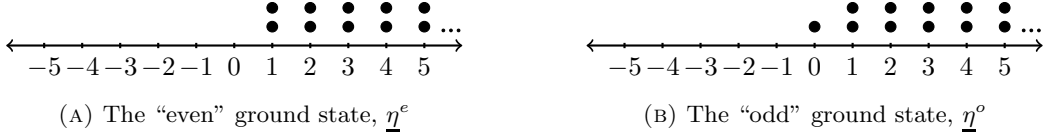
\begin{figure}[H]
    \centering 
    \begin{subfigure}[b]{0.4\textwidth}
    \centering
    \begin{tikzpicture}[scale=0.5]
\draw[thick, <->] (-6,0)--(6,0);
\foreach \x in {-5,-4,-3,-2,-1,0,1,2,3,4,5}
    \draw[thick, -](\x cm, 2pt)--(\x cm, -2pt)node[anchor=north]{$\x$};
    
\filldraw [black] (1,0.5) circle (4pt);
\filldraw [black] (1,1) circle (4pt);
\filldraw [black] (2,0.5) circle (4pt);
\filldraw [black] (2,1) circle (4pt);
\filldraw [black] (3,0.5) circle (4pt);
\filldraw [black] (3,1) circle (4pt);
\filldraw [black] (4,0.5) circle (4pt);
\filldraw [black] (4,1) circle (4pt);
\filldraw [black] (5,0.5) circle (4pt);
\filldraw [black] (5,1) circle (4pt);
\filldraw [black] (5.5,0.4) circle (1pt);
\filldraw [black] (5.7,0.4) circle (1pt);
\filldraw [black] (5.9,0.4) circle (1pt);
\end{tikzpicture}
\caption{The ``even" ground state, $\underline{\eta}^e$}
    \end{subfigure}
    \hspace{10mm}
    \begin{subfigure}[b]{0.4\textwidth}
    \centering
    \begin{tikzpicture}[scale=0.5]
\draw[thick, <->] (-6,0)--(6,0);
\foreach \x in {-5,-4,-3,-2,-1,0,1,2,3,4,5}
    \draw[thick, -](\x cm, 2pt)--(\x cm, -2pt) node[anchor=north]{$\x$};
    
\filldraw [black] (0,0.5) circle (4pt);
\filldraw [black] (1,0.5) circle (4pt);
\filldraw [black] (1,1) circle (4pt);
\filldraw [black] (2,0.5) circle (4pt);
\filldraw [black] (2,1) circle (4pt);
\filldraw [black] (3,0.5) circle (4pt);
\filldraw [black] (3,1) circle (4pt);
\filldraw [black] (4,0.5) circle (4pt);
\filldraw [black] (4,1) circle (4pt);
\filldraw [black] (5,0.5) circle (4pt);
\filldraw [black] (5,1) circle (4pt);
\filldraw [black] (5.5,0.4) circle (1pt);
\filldraw [black] (5.7,0.4) circle (1pt);
\filldraw [black] (5.9,0.4) circle (1pt);
\end{tikzpicture}
\caption{The ``odd" ground state, $\underline{\eta}^o$}
    \end{subfigure}
    \caption{The two ground states, $\underline{\eta}^e$ and $\underline{\eta}^o$, of $\Omega$.}
    \label{ground states}
\end{figure}
\par \noindent The reason for the terms ``odd" and ``even" will become clear in the next section.

\subsection{Ergodic decomposition of $\Omega$}\label{decomp} ~ 
\par For any $\underline{\eta}\in \Omega$ the quantity $N(\underline{\eta}):=\sum\limits_{i=1}^\infty (2-\eta_i)-\sum\limits_{i=-\infty}^0 \eta_i$ (as defined in \cite{blocking}) is finite and is conserved by the dynamics of the process. So we can decompose $\Omega=\bigcup_{n \in \mathbb{Z}}\Omega^n$, into irreducible components $\Omega^n:=\{\underline{\eta} \in \Omega: N(\underline{\eta})=n\}$. Note that the left shift operator $\tau$, defined by $(\tau\underline{\eta})_i=\eta_{i+1}$, gives a bijection $\Omega^n \xrightarrow[]{\tau} \Omega^{n-2}$ (i.e.\ if $\underline{\eta} \in \Omega^n$ then, $N(\tau \underline{\eta})=n-2$).
\begin{remark}
Since $N(\underline{\eta}^e)=0$, the shifts of $\underline{\eta}^e$ have even conserved quantity and give the ground states for the ``even" part $\bigcup_{n\in 2\mathbb{Z}}\Omega^n$ of $\Omega$. Similarly $N(\underline{\eta}^o)=-1$ and shifts provide the ground states for the ``odd" part $\bigcup_{n\in 2\mathbb{Z}+1}\Omega^n$. This explains the labels $\underline{\eta^e}$ and $\underline{\eta^o}$.
\end{remark}
\par We now calculate  $\underline{\nu}^{n,c}(\cdot):=\underline{\mu}^c(\cdot|N(\cdot)=n)$, the unique stationary distribution on $\Omega^n$. 
\begin{lem}\label{mu tau}
The following relation holds, 
$$\underline{\mu}^c(\tau \underline{\eta})=\Tilde{q}^{2c-N(\underline{\eta})}\underline{\mu}^c( \underline{\eta}).$$
\par \noindent This gives the recursion,
$$\underline{\mu}^c(\{N=n\})=\Tilde{q}^{n-2c}\underline{\mu}^c(\{N=n-2\}).$$
\end{lem}
\begin{proof}
(These are special cases of Lemma 6.1 and Corollary 6.2 in \cite{blocking}.)
\begin{align*}
     \underline{\mu}^c(\tau \underline{\eta})&=\prod\limits_{i=-\infty}^0\frac{ t^{\mathbb{I}\{\eta_{i+1}=1\}}\Tilde{q}^{-(i-c)\eta_{i+1}}}{Z_i^c(\Tilde{q},t)}\prod\limits_{i=1}^\infty \frac{ t^{\mathbb{I}\{\eta_{i+1}=1\}}\Tilde{q}^{(2-\eta_{i+1})(i-c)}}{\Tilde{q}^{2(i-c)}Z_i^c(\Tilde{q},t)}\\
     &=\Tilde{q}^{2c}\frac{\prod\limits_{j=-\infty}^0 t^{\mathbb{I}\{\eta_j=1\}}\Tilde{q}^{-(j-1-c)\eta_j}}{\prod\limits_{i=-\infty}^0Z_i^c(\Tilde{q},t)}\frac{\prod\limits_{j=1}^\infty  t^{\mathbb{I}\{\eta_j=1\}}\Tilde{q}^{(2-\eta_j)(j-1-c)}}{\prod\limits_{i=1}^\infty \Tilde{q}^{2(i-c)}Z_i^c(\Tilde{q},t)}\\
     &=\Tilde{q}^{2c+\sum\limits_{j=-\infty}^0\eta_j-\sum\limits_{j=1}^\infty(2-\eta_j)}\frac{\prod\limits_{j=-\infty}^0 t^{\mathbb{I}\{\eta_j=1\}}\Tilde{q}^{-(j-c)\eta_j}}{\prod\limits_{i=-\infty}^0Z_i^c(\Tilde{q},t)}\frac{\prod\limits_{j=1}^\infty t^{\mathbb{I}\{\eta_j=1\}}\Tilde{q}^{(2-\eta_j)(j-c)}}{\prod\limits_{i=1}^\infty \Tilde{q}^{2(i-c)}Z_i^c(\Tilde{q},t)}\\
     &= \Tilde{q}^{2c-N(\underline{\eta})}\underline{\mu}^c( \underline{\eta}).
\end{align*}
Since $N(\tau\underline{\eta})=N(\underline{\eta})-2$ we have that,
\begin{align*}
    \underline{\mu}^c(\{N=n-2\}) &= \sum_{\underline{\eta}:N(\underline{\eta})=n-2}\underline{\mu}^c(\underline{\eta}) \\
    &= \sum_{\hat{\underline{\eta}}:N(\tau\hat{\underline{\eta}})=n-2}\underline{\mu}^c(\tau \hat{\underline{\eta}})\\
    &=\sum_{\hat{\underline{\eta}}:N(\hat{\underline{\eta}})=n}\Tilde{q}^{2c -N(\hat{\underline{\eta}})}\underline{\mu}^c(\hat{\underline{\eta}})\\
    &= \Tilde{q}^{2c-n}\underline{\mu}^c(\{N=n\}). 
\end{align*}
\end{proof}
\par \noindent The general solution to the recursion above is
$$\underline{\mu}^c(\{N=n\})=\begin{cases}
\Tilde{q}^{\frac{n(n+2)}{4}-nc}\underline{\mu}^c(\{N=0\}) &\textrm{ if } n \in 2\mathbb{Z}\\
\Tilde{q}^{\frac{(n+1)^2}{4}-(n+1)c}\underline{\mu}^c(\{N=-1\}) &\textrm{ if } n \in 2\mathbb{Z}+1.
\end{cases}$$
\par \vspace{1mm}\noindent Since there is a dependence on parity we will need to calculate the probability $\underline{\mu}^c(\{N(\underline{\eta}) \in 2 \mathbb{Z}+1\})$ in order to finish our calculation of $\underline{\nu}^{n,c}$.
\begin{lem}\label{prob odd}
$$\underline{\mu}^c(\{N(\underline{\eta}) \in 2\mathbb{Z}+1\})=\frac{1-\prod\limits_{i=-\infty}^\infty\left(1-2\mu_i^c(1)\right)}{2}.$$
\end{lem}
\vspace{1mm}\begin{proof}
  (We adapt the proof of \cite{janos}, Proposition 1). 
  
\par \vspace{1mm} \noindent Define the partial conserved quantity
$$N_a(\underline{\eta})=\sum\limits_{i=1}^a(2-\eta_i)-\sum\limits_{i=-a}^0\eta_i \hspace{10mm} \text{for } a\geq 1 $$
\par \noindent and note that $N_a(\underline{\eta})\rightarrow N(\underline{\eta})$ as $a \rightarrow \infty$. For each $a\geq 1$ we define the random variables
$$Y_a:=(-1)^{\sum\limits_{i=-a}^a \eta_i}=(-1)^{N_a(\underline{\eta})}.$$
\par \noindent Since, $\eta_i  \rightarrow 0$ as $i\rightarrow -\infty$ and $\eta_i  \rightarrow 2$ as $i\rightarrow \infty$, we have that $Y_a\rightarrow Y$ as $a \rightarrow \infty$, where
$$Y:=(-1)^{N(\underline{\eta})}=\begin{cases}
1 \hspace{5mm} \textrm{if } N(\underline{\eta}) \textrm{ is even}\\
-1 \hspace{3mm} \textrm{if } N(\underline{\eta}) \textrm{ is odd.}
\end{cases}$$
\par \noindent Since $|Y_a|=1$ for all $a\geq 1$, dominated convergence applies and by the product structure of $\underline{\mu}^c$ we have that 
$$\mathbb{E}^c[Y]=\lim\limits_{a\rightarrow\infty}\mathbb{E}^c[Y_a]=\prod\limits_{i=-\infty}^\infty\big(1\cdot\mu_i^c(0,2)+(-1)\cdot \mu_i^c(1)\big)=\prod\limits_{i=-\infty}^\infty\big(1-2\mu_i^c(1)\big),$$
\par \noindent where $\mathbb{E}^c$ denotes the expectation taken w.r.t. $\underline{\mu}^c$. On the other hand we have
\[
    \mathbb{E}^c[Y]=1\cdot \underline{\mu}^c\left(\{N(\underline{\eta}) \in 2\mathbb{Z}\}\right)+(-1)\cdot \underline{\mu}^c\left(\{N(\underline{\eta}) \in 2\mathbb{Z}+1\}\right)\\
    =1-2\underline{\mu}^c\left(\{N(\underline{\eta}) \in 2\mathbb{Z}+1\}\right).
\]
\par \noindent Thus
$$\underline{\mu}^c(\{N(\underline{\eta}) \in 2\mathbb{Z}+1\})= \frac{1-\prod\limits_{i=-\infty}^\infty\big(1-2\mu_i^c(1)\big)}{2}$$
\par \noindent and also 
$$\underline{\mu}^c(\{N(\underline{\eta}) \in 2\mathbb{Z}\})=1-\underline{\mu}^c(\{N(\underline{\eta}) \in 2 \mathbb{Z}+1\})=\frac{1+\prod\limits_{i=-\infty}^\infty\big(1-2\mu_i^c(1)\big)}{2}.$$
\end{proof}
\par Now by combining Lemma \ref{mu tau} and Lemma \ref{prob odd} we get the unique stationary distribution $\underline{\nu}^{n,c}$ on $\Omega^n$.
\begin{prop}\label{nu n}
For $n \in 2\mathbb{Z}$ the unique stationary distribution on $\Omega^n$ is given by
$$\underline{\nu}^{n,c}(\underline{\eta})=\frac{2\sum\limits_{\ell=-\infty}^\infty \Tilde{q}^{\ell(\ell+1)-2\ell c}\underline{\mu}^c(\underline{\eta})\mathbb{I}\{N(\underline{\eta})=n\}}{\Tilde{q}^{\frac{n(n+2)}{4}-nc}\left(1+\prod\limits_{i=-\infty}^\infty\left(1-2\mu_i^c(1)\right)\right)}.$$
\par \noindent For $n \in 2\mathbb{Z}+1$ it is given by
$$\underline{\nu}^{n,c}(\underline{\eta})=\frac{2\sum\limits_{\ell=-\infty}^\infty \Tilde{q}^{(\ell+1)^2-2(\ell+1) c} \underline{\mu}^c(\underline{\eta})\mathbb{I}\{N(\underline{\eta})=n\}}{\Tilde{q}^{\frac{(n+1)^2}{4}-(n+1)c}\left(1-\prod\limits_{i=-\infty}^\infty\left(1-2\mu_i^c(1)\right)\right)}.$$
\end{prop}
\par Notice that, $(1-2\mu_i^c(1))=\frac{1-t\Tilde{q}^{-(i-c)}+\Tilde{q}^{-2(i-c)}}{Z_i^c(\Tilde{q},t)}$, and so if we let $W_i^c(\Tilde{q},t):=1-t\Tilde{q}^{-(i-c)}+\Tilde{q}^{-2(i-c)}$ we can write $\underline{\nu}^{n,c}$ as
$$\underline{\nu}^{n,c}(\underline{\eta})=\frac{2\sum\limits_{\ell=-\infty}^\infty \Tilde{q}^{\ell(\ell+1)-2\ell c} \prod\limits_{i=-\infty}^0t^{\mathbb{I}\{\eta_i=1\}}\Tilde{q}^{-\eta_i(i-c)}\prod\limits_{i=1}^\infty t^{\mathbb{I}\{\eta_i=1\}}\Tilde{q}^{(2-\eta_i)(i-c)}}{\Tilde{q}^{\frac{n(n+2)}{4}-nc}\left(\prod\limits_{i=1}^\infty\Tilde{q}^{2(i-c)}Z^c_i(\Tilde{q},t)Z^c_{-i+1}(\Tilde{q},t)+\prod\limits_{i=1}^\infty\Tilde{q}^{2(i-c)}W^c_i(\Tilde{q},t)W^c_{-i+1}(\Tilde{q},t)\right)}\mathbb{I}\{N(\underline{\eta})=n\}$$
\par \noindent when $n$ is even and
$$\underline{\nu}^{n,c}(\underline{\eta})=\frac{2\sum\limits_{\ell=-\infty}^\infty \Tilde{q}^{(\ell+1)^2-2(\ell+1) c} \prod\limits_{i=-\infty}^0t^{\mathbb{I}\{\eta_i=1\}}\Tilde{q}^{-\eta_i(i-c)}\prod\limits_{i=1}^\infty t^{\mathbb{I}\{\eta_i=1\}}\Tilde{q}^{(2-\eta_i)(i-c)}}{\Tilde{q}^{\frac{(n+1)^2}{4}-(n+1)c}\left(\prod\limits_{i=1}^\infty\Tilde{q}^{2(i-c)}Z^c_i(\Tilde{q},t)Z^c_{-i+1}(\Tilde{q},t)-\prod\limits_{i=1}^\infty\Tilde{q}^{2(i-c)}W^c_i(\Tilde{q},t)W^c_{-i+1}(\Tilde{q},t)\right)}\mathbb{I}\{N(\underline{\eta})=n\}$$ when $n$ is odd.

\begin{remark}
 These distributions are independent of $c$, as discussed in \cite{blocking}. However later we will need to stress the dependence of both the numerator and denominator on $c$ and so we will keep $c$ in our notation.
 \end{remark}

\subsection{Standing up/Laying Down}\label{stand up}~
\par In this section we transfer the dynamics on $\Omega^n$ to that of a restricted particle system on $\mathbb{Z}_{\geq 0}^{\mathbb{Z}_{< 0}}$, in direct analogy with the ``standing up/laying down" method of \cite{blocking}. By doing this we obtain an alternative characterisation of the stationary distributions given in Proposition \ref{nu n}.
\begin{define} \label{stand up transform}
Given $\underline{\eta} \in \Omega^n$, let $S_r$ be the site of the $r^\textrm{th}$ particle when reading left to right, bottom to top. The corresponding stood up state is then $T^n(\underline{\eta})=\underline{\omega}\in \mathbb{Z}_{\geq 0}^{\mathbb{Z}_{< 0}}$, with $\omega_{-r}=S_{r+1}-S_r$.
\end{define}
\vspace{5mm}
\begin{figure}[H]
    \centering

    \begin{tikzpicture}[scale=0.55]
    \centering
    \draw[thick, <->] (-6,0.8)--(6,0.8);
\foreach \x in {-5,-4,-3,-2,-1,0,1,2,3,4,5}
    \draw[thick, -](\x cm, 0.9)--(\x cm, 0.7) node[anchor=north]{};
 
 \filldraw [black] (-5.5,1.3) circle (1pt);
\filldraw [black] (-5.7,1.3) circle (1pt);
\filldraw [black] (-5.9,1.3) circle (1pt);   
\filldraw [black] (-3,1.5) circle (4pt);
\node (a) at (-3.35,1.15) [label=\tiny{1}]{};
\node (a) at (-3,0) [label=\tiny{LMP}]{};
\filldraw [black] (-2,1.5) circle (4pt);
\node (a) at (-2.35,1.15) [label=\tiny{2}]{};
\filldraw [black] (-2,2) circle (4pt);
\node (a) at (-2.35,1.65) [label=\tiny{3}]{};
\node (a) at (0,0) [label=\tiny{0}]{};
\filldraw [black] (1,1.5) circle (4pt);
\node (a) at (0.65,1.15) [label=\tiny{4}]{};
\filldraw [black] (1,2) circle (4pt);
\node (a) at (0.65,1.65) [label=\tiny{5}]{};
\filldraw [black] (2,1.5) circle (4pt);
\node (a) at (1.65,1.15) [label=\tiny{6}]{};
\filldraw [black] (3,1.5) circle (4pt);
\node (a) at (3,0) [label=\tiny{RMH}]{};
\node (a) at (2.65,1.15) [label=\tiny{7}]{};
\filldraw [black] (4,1.5) circle (4pt);
\node (a) at (3.65,1.15) [label=\tiny{8}]{};
\filldraw [black] (4,2) circle (4pt);
\node (a) at (3.65,1.65) [label=\tiny{9}]{};
\filldraw [black] (5,1.5) circle (4pt);
\node (a) at (4.65,1.15) [label=\tiny{10}]{};
\filldraw [black] (5,2) circle (4pt);
\node (a) at (4.65,1.65) [label=\tiny{11}]{};
 \filldraw [black] (5.5,1.3) circle (1pt);
\filldraw [black] (5.7,1.3) circle (1pt);
\filldraw [black] (5.9,1.3) circle (1pt); 
\node (a) at (0,-1.5) [label=\scriptsize{(A)} \normalsize A state $\underline{\eta} \in \Omega^{-1}$.]{};
    \draw[thick, <-] (8,0.8)--(20,0.8);
\foreach \x in {9,10,11,12,13,14,15,16,17,18,19}
    \draw[thick, -](\x cm, 0.9)--(\x cm, 0.7);
    \node (a) at (19,-0.3) [label={-1}]{};
    \node (a) at (18,-0.3) [label={-2}]{};
    \node (a) at (17,-0.3) [label={-3}]{};
    \node (a) at (16,-0.3) [label={-4}]{};
    \node (a) at (15,-0.3) [label={-5}]{};
    \node (a) at (14,-0.3) [label={-6}]{};
    \node (a) at (13,-0.3) [label={-7}]{};
    \node (a) at (12,-0.3) [label={-8}]{};
    \node (a) at (11,-0.3) [label={-9}]{};
    \node (a) at (10,-0.3) [label={-10}]{};
    \node (a) at (9,-0.3) [label={-11}]{};
    
\filldraw [black] (19,1.5) circle (4pt);
\filldraw [black] (17,1.5) circle (4pt);
\filldraw [black] (17,2) circle (4pt);
\filldraw [black] (17,2.5) circle (4pt);
\filldraw [black] (15,1.5) circle (4pt);
\filldraw [black] (14,1.5) circle (4pt);
\filldraw [black] (13,1.5) circle (4pt);
\filldraw [black] (11,1.5) circle (4pt);
\filldraw [black] (9,1.5) circle (4pt);
 \filldraw [black] (8.5,1.3) circle (1pt);
\filldraw [black] (8.3,1.3) circle (1pt);
\filldraw [black] (8.1,1.3) circle (1pt);
\node (a) at (14,-1.5) [label= \scriptsize{(B)} \normalsize The stood up state $\underline{\omega}$.]{};
    \end{tikzpicture}
    \caption{An example of standing up.}
    \label{standing up}
\end{figure}
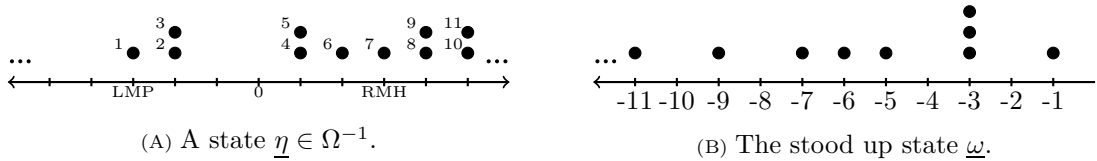
\begin{remark}
We can extend this method to stand up blocking processes with $I=\{0,1,2,..,k\}$ and $\Lambda=\mathbb{Z}$ for any $k \in \mathbb{N}$. However the corresponding particle systems are not always guaranteed to have product stationary blocking measures (this is dependent on the jump rates of the original system).  
\end{remark}
\newpage A priori the ``standing up" map $T^n$ is an injection into $\mathbb{Z}_{\geq 0}^{\mathbb{Z}_{<0}}$. However since $\eta_i\leq 2$ for all $i$ the image of $T^n$ lies in the restricted state space $$\mathcal{H}':=\{\underline{\omega} \in \mathbb{Z}_{\geq 0}^{\mathbb{Z}_{< 0}}: \hspace{2mm} \omega_{-i}=0 \Rightarrow \omega_{-i-1} \neq0, \forall i>0\}.$$

\par \noindent Since $\eta_i=2$ for $i$ large $\underline{\omega}$ must coincide far to the left with one of the following  ``even" or ``odd" states (dependent on the parity of $n$); $\underline{\omega}^e$ and $\underline{\omega}^o$ such that $\omega^e_{-i}=\mathbb{I}\{i \in 2\mathbb{Z}_{\geq 1}\}$ and $\omega^o_{-i}=\mathbb{I}\{i \in 2\mathbb{Z}_{\geq 0}+1\}$.
\vspace{5mm}
\begin{figure}[H]
    \centering 
    \begin{subfigure}[b]{0.4\textwidth}
    \centering
    \begin{tikzpicture}[scale=0.6]
\draw[thick, <-] (-11,0)--(-0.5,0);
\foreach \x in {-10,-9,-8,-7,-6,-5,-4,-3,-2,-1}
    \draw[thick, -](\x cm, 2pt)--(\x cm, -2pt) node[anchor=north]{$\x$};
    
\filldraw [black] (-2,0.5) circle (4pt);
\filldraw [black] (-4,0.5) circle (4pt);
\filldraw [black] (-6,0.5) circle (4pt);
\filldraw [black] (-8,0.5) circle (4pt);
\filldraw [black] (-10,0.5) circle (4pt);
\filldraw [black] (-10.5,0.4) circle (1pt);
\filldraw [black] (-10.7,0.4) circle (1pt);
\filldraw [black] (-10.9,0.4) circle (1pt);
\end{tikzpicture}
    \end{subfigure}
    \hspace{10mm}
    \begin{subfigure}[b]{0.4\textwidth}
    \centering
    \begin{tikzpicture}[scale=0.6]
\draw[thick, <-] (-11,0)--(-0.5,0);
\foreach \x in {-10,-9,-8,-7,-6,-5,-4,-3,-2,-1}
   \draw[thick, -](\x cm, 2pt)--(\x cm, -2pt) node[anchor=north]{$\x$};
    
\filldraw [black] (-1,0.5) circle (4pt);
\filldraw [black] (-3,0.5) circle (4pt);
\filldraw [black] (-5,0.5) circle (4pt);
\filldraw [black] (-7,0.5) circle (4pt);
\filldraw [black] (-9,0.5) circle (4pt);
\filldraw [black] (-10.5,0.4) circle (1pt);
\filldraw [black] (-10.7,0.4) circle (1pt);
\filldraw [black] (-10.9,0.4) circle (1pt);
\end{tikzpicture}
    \end{subfigure}
    \caption{The ``even" and ``odd" states $\underline{\omega}^e$ and $\underline{\omega}^o$.}
    \label{stood up ground states}
\end{figure}
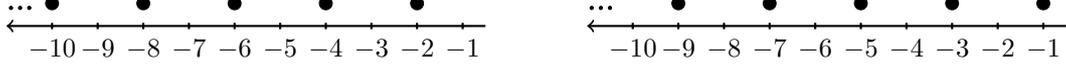
\begin{remark}
Note that all shifts of the ``even" ground state $\underline{\eta}^e\in \Omega^0$ stand up to give $\underline{\omega}^e$. Similarly all shifts of $\underline{\eta}^o \in \Omega^{-1}$ give $\underline{\omega}^o$. This is further justification for the notation.
\end{remark}
\par \noindent We now see that the image of $T^n$ lies in $\mathcal{H}:=\mathcal{H}^e \cup \mathcal{H}^o$, where the disjoint sets $\mathcal{H}^e$  and $\mathcal{H}^o$ are defined as,
\begin{align*}
    \mathcal{H}^e &=\{\underline{\omega} \in \mathcal{H}' : \exists N > 0 \hspace{2mm}\textrm{ s.t } \omega_{-i}=\omega_{-i}^e \hspace{2mm} \forall
    i \geq N\}, \\
    \mathcal{H}^o &=\{\underline{\omega} \in \mathcal{H}' : \exists N > 0 \hspace{2mm}\textrm{ s.t } \omega_{-i}=\omega_{-i}^o \hspace{2mm} \forall
    i \geq N\}.
\end{align*}

\par\noindent For $\underline{\omega}\in\mathcal{H}$ the minimum value of $N$ satisfying the above will be denoted by $E(\underline{\omega})$ or $O(\underline{\omega})$, dependent on whether $\underline{\omega}\in \mathcal{H}^e$ or $\underline{\omega}\in \mathcal{H}^o$.

\begin{lem}\label{im T}
$T^n(\Omega^n)=\begin{cases}
\mathcal{H}^e &\textrm{ if } n \in 2\mathbb{Z}\\
\mathcal{H}^o &\textrm{ if } n \in 2\mathbb{Z}+1.
\end{cases}$
\end{lem}
\begin{proof}~
\par It suffices to show surjectivity of $T^n$ for each $n$. 
\par If $n \in 2\mathbb{Z}$ and $\underline{\omega} \in \mathcal{H}^e$ then we construct the state $\underline{\eta}\in \Omega^n$ having leftmost particle at the site $S_1=\frac{n+E(\underline{\omega})+\mathbb{I}\{E(\underline{\omega})\in 2\mathbb{Z}+1\}}{2}-\sum\limits_{i = 1}^{E(\underline{\omega})-1}\omega_{-i}$ and $r^\textrm{th}$ particle at site $S_r=S_{r-1}+\omega_{1-r}$ for each $r\geq 2$.
\par Similarly for $n \in 2\mathbb{Z}+1$ and $\underline{\omega}\in \mathcal{H}^o$ construct the state $\underline{\eta}\in \Omega^n$ with left most particle at site $S_1=\frac{n+O(\underline{\omega})+\mathbb{I}\{O(\underline{\omega}) \in 2 \mathbb{Z}\}}{2}-\sum\limits_{i = 1}^{O(\underline{\omega})-1} \omega_{-i}$ and $r^\textrm{th}$ particle at site $S_r=S_{r-1}+\omega_{1-r}$ for each $r\geq 2$.
\par It is clear that in either case $T^n(\underline{\eta})=\underline{\omega}$ and hence $T^n$ is surjective.
\end{proof}
\par \noindent  The inverse maps described in the above proof are referred to as the ``laying down" maps.
\par Using the ``standing up" maps we define a particle system on $\mathcal{H}$ whose dynamics are inherited from those on $\Omega$. In particular right jumps in $\underline{\eta}$ correspond to right jumps in $\underline{\omega}$ and similarly for left jumps.  The explicit right/left jump rates are given in Table \ref{gen stood up rates} for $r \geq 2$.

\vspace{5mm}
\begin{table}[H]
\centering
\begin{tabular}{cc}
    \begin{tabular}{|c||c|c|}
         \hline
         & $\omega_{-r+1}=0$ & $\omega_{-r+1}\geq1$ \\
         \hline \hline
         $\omega_{-r}=0$ & $0$ & $0$\\
         \hline
        $\omega_{-r}=1$ & $p(2,1)\mathbb{I}\{\omega_{-r-1} \neq 0\}$ &$p(1,1)\mathbb{I}\{\omega_{-r-1}\neq 0\}$ \\
        \hline
        $\omega_{-r}\geq2$ & $p(2,0)$ & $p(1,0)$\\
        \hline
    \end{tabular}\\ \vspace{1mm}\\
    \begin{tabular}{|c||c|c|c|}
         \hline
        &$\omega_{-r+1}=0$ & $\omega_{-r+1}=1$ & $\omega_{-r+1}\geq2$\\
         \hline \hline
        $\omega_{-r}=0$ & $0$ & $q(1,2)\mathbb{I}\{\omega_{-r+2}\neq0\}$ & $q(0,2)$\\
        \hline
        $\omega_{-r}\geq1$ & $0$& $q(1,1)\mathbb{I}\{\omega_{-r+2}\neq0\}$  & $q(0,1)$ \\
        \hline
    \end{tabular}\\
\end{tabular}
\caption{The jump rates $p_\omega(\omega_{-r}, \omega_{-r+1})$ and $q_\omega(\omega_{-r},\omega_{-r+1})$ respectively of the stood up process.}
\label{gen stood up rates}
\end{table}
\par  \noindent Note that the jump rates $p_\omega(1,\cdot)$ and $q_\omega(\cdot,1)$ can be 0, due to the no consecutive zeroes condition, a jump that would cause $\omega_{-r}=\omega_{-r+1}=0$ for some $r \geq 2$ is blocked.  
\par Since the ``stood up " process is only defined on the negative half integer line we must consider what happens at the boundary site. We consider an open infinite type boundary, that is a reservoir of particles at ``site $0$" at which particles can enter or leave the system with the rates below.

\vspace{5mm}
\begin{table}[H]
    \centering
    \begin{tabular}{|c||c|c|}
         \hline
         &  Rate into the boundary &  Rate out of the boundary \\
         \hline \hline
          $\omega_{-1}=0$&$0$ &  $q(0,2)$\\
         \hline
       \multicolumn{1}{|c||}{ $\omega_{-1}=1$} & \multicolumn{1}{|c|}{$p(1,1)\mathbb{I}\{\omega_{-2}\neq 0\}$} & \multirow{2}{*}{$q(0,1)$} \\ \cline{1-2}
        \multicolumn{1}{|c||}{$\omega_{-1}\geq2$} & \multicolumn{1}{|c|}{$p(1,0)$} &\\
        \hline
    \end{tabular}
    \caption{Boundary jump rates for the stood up process.}
    \label{boundary rates}
\end{table}
\par \noindent We note that the dynamics at the boundary in $\underline{\omega}$ correspond exactly to that of the LMP in $\left(T^n\right)^{-1}(\underline{\omega})$.
\par To find the stationary distribution for the ``stood up" process we first consider the unrestricted process $\underline{\omega}^* \in \mathbb{Z}_{\geq 0}^{\mathbb{Z}_{<0}}$, i.e.\ the process described by the same jump rates as $\underline{\omega}$ but where the number of consecutive zeros is not restricted. It is clear to see that the unrestricted process is a member of the blocking family; $(B1)$ and $(B2)$ follow directly from the fact that the rates of the original process satisfy these conditions and similarly condition $(B3)$ is satisfied by the constants  
  $$p^*_{\textrm{asym}}=\frac{p(1,0)}{q(0,1)+p(1,0)} > \frac{1}{2} \hspace{5mm} \textrm{ and } \hspace{5mm} q^*_{\textrm{asym}}=\frac{q(0,1)}{q(0,1)+p(1,0)}< \frac{1}{2}$$
    \par \noindent and the functions
        \[
     f^*(z):=
    \begin{cases}
     0 \hspace{2mm} \textrm{if } z=0  \\
     \\
    t^{-1} \hspace{2mm}\textrm{if } z=1  \\
      \\
    t \hspace{2mm}\textrm{if } z\geq2
    \end{cases} 
    \qquad
    s^*(y,z):=
    \begin{cases}
       \frac{tp(2,1)\left(q(0,1)+p(1,0)\right)}{p(1,0)} &\textrm{if } y=z=1  \\
       \\
       \frac{tp(1,1)\left(q(0,1)+p(1,0)\right)}{p(1,0)}&\textrm{if } y=1 \textrm{ and } z\geq 2 \\
       \\
       \frac{p(2,0)\left(q(0,1)+p(1,0)\right)}{p(1,0)t}  &\textrm{if } y\geq2 \textrm{ and } z=1 \\
       \\
       \frac{q(0,1)+p(1,0)}{t} &\textrm{if } y,z\geq 2, 
    \end{cases}
   \]
 \par \noindent with $t=\left(\frac{q(0,1)p(2,0)}{p(1,0)q(1,1)}\right)^{\frac{1}{2}}>1$.
\par By Theorem \ref{stationary dist} we can find a one parameter family of product stationary blocking measures $\underline{\pi}^{*,\hat{c}}$ on $\mathbb{Z}_{\geq 0}^{\mathbb{Z}_{<0}}$ with marginals given by
$$\pi_{-i}^{*,\hat{c}}(z)=\frac{t^{(2-z)\mathbb{I}\{z \geq 1\}}\Tilde{q}^{(i+\hat{c})z}}{Z_{-i}^{*,\hat{c}}(\Tilde{q},t)} \hspace{10mm} \text{for } z\geq 0 \text{ and } i\geq 1$$
\par \noindent (here $Z_{-i}^{*,\hat{c}}(\Tilde{q},t)$ is the normalising factor). 

We fix the value of $\hat{c}$ by considering reversibility over the boundary edge $(-1,0)$. Suppose that $\underline{\pi}^{*,\hat{c}}$ satisfies detailed balance over this boundary edge:
$$\pi_{-1}^{*,\hat{c}}(y)\cdot \textrm{``rate into the boundary"}=\pi_{-1}^{*,\hat{c}}(y-1)\cdot \textrm{`` rate out of the boundary"} \hspace{5mm} \text{for all } y\geq 1.$$
\par \noindent We consider the following two cases, 
\begin{enumerate}
    \item If $y=1$ then this gives 
    \begin{align*}
    &\pi^{*,\hat{c}}_{-1}(1)p(1,1)=\pi^{*,\hat{c}}_{_1}(0)q(0,2)  \\
    \Leftrightarrow \textrm{ } &\frac{\pi^{*,\hat{c}}_{-1}(1)}{\pi_{-1}^{*,\hat{c}}(0)}=\frac{q(0,2)}{p(1,1)}=\frac{p(2,0)q(0,1)^2}{p(1,0)^2q(1,1)} \textrm{ by condition } (B3) (c)\\
   \Leftrightarrow \textrm{ } & t\Tilde{q}^{1+\hat{c}}=t^2\Tilde{q}\\
   \Leftrightarrow \textrm{ } & \Tilde{q}^{\hat{c}}=t
    \end{align*}
    \item If $y \geq 2$ then this gives
    \begin{align*}
    &\pi^{*,\hat{c}}_{-1}(y)p(1,0)=\pi^{*,\hat{c}}_{_1}(y-1)q(0,1)  \\
    \Leftrightarrow \textrm{ } &\frac{\pi^{*,\hat{c}}_{-1}(y)}{\pi_{-1}^{*,\hat{c}}(y-1)}=\frac{q(0,1)}{p(1,0)}=\Tilde{q}\\
   \Leftrightarrow \textrm{ } & t^{-1}\Tilde{q}^{1+\hat{c}}=\Tilde{q}\\
   \Leftrightarrow \textrm{ } & \Tilde{q}^{\hat{c}}=t.
    \end{align*}
\end{enumerate}
\par \noindent Thus we should choose $\hat{c}$ so that $\Tilde{q}^{\hat{c}}=t$. The marginals now become (dropping $\hat{c}$ from the notation)
$$
\pi_{-i}^*(z)=\frac{\Tilde{q}^{iz}t^{2\mathbb{I}\{z \geq 1\}}}{Z_{-i}^*(\Tilde{q},t)} \hspace{10mm} \text{for } z\geq 0 \text{ and } i\geq 1.
$$
\par Now that we have the stationary distribution for the unrestricted process we consider the restriction to $\mathcal{H}$ and find the stationary measure. Recall that $\mathcal{H}=\mathcal{H}^e\cup\mathcal{H}^o$, and note that $\mathcal{H}^e$ is the irreducible component of the ``even" ground state $\underline{\omega}^e$ and similarly $\mathcal{H}^o$ for the ``odd" ground state $\underline{\omega}^o$. We define stationary measures on these irreducible components in terms of $\underline{\pi}^*$, getting $\underline{\pi}^e$ on $\mathcal{H}^e$ and $\underline{\pi}^o$ on $\mathcal{H}^o$. It seems natural to define these measures as $\underline{\pi}^e(\cdot)=\underline{\pi}^*(\cdot|\cdot \in \mathcal{H}^e)$ and $\underline{\pi}^o(\cdot)=\underline{\pi}^*(\cdot|\cdot \in \mathcal{H}^o)$. However w.r.t $\underline{\pi}^*$ the probability of being in either irreducible component is zero and so these quantities are undefined. To rectify this we use the following formal reasoning, 
\begin{align*}
    \underline{\pi}^e(\underline{\omega})&=\underline{\pi}^*(\underline{\omega}|\underline{\omega} \in \mathcal{H}^e)\\
    &=\frac{\prod\limits_{i\geq 1}\pi_{-i}^*(\omega_{-i}) \mathbb{I}\{\underline{\omega}\in \mathcal{H}^e\}}{\sum\limits_{\underline{\omega}' \in \mathcal{H}^e}\prod\limits_{i\geq 1}\pi_{-i}^*(\omega_{-i}')} \\
    &=\frac{\prod\limits_{i\geq 1}\frac{\pi_{-i}^*(\omega_{-i})}{\pi_{-i}^*(\omega_{-i}^e)} \mathbb{I}\{\underline{\omega}\in \mathcal{H}^e\}}{\sum\limits_{\underline{\omega}' \in \mathcal{H}^e}\prod\limits_{i\geq 1} \frac{\pi_{-i}^*(\omega_{-i}')}{\pi_{-i}^*(\omega_{-i}^e)}}.
\end{align*}
\par \noindent This is now a well defined distribution since far to the left any configuration $\underline{\omega} \in \mathcal{H}^e$ agrees with $\underline{\omega}^e$, forcing the products to be finite and the denominator to no longer be 0. We apply a similar reasoning for $\underline{\pi}^o$. We then have
$$\underline{\pi}^e(\underline{\omega})=\frac{\prod\limits_{i\geq 1}\phi_{-i}^e(\omega_{-i})\mathbb{I}\{\underline{\omega}\in \mathcal{H}^e\}}{\sum\limits_{\underline{\omega}'\in \mathcal{H}^e}\prod\limits_{i\geq 1}\phi_{-i}^e(\omega'_{-i})} \hspace{5mm} \textrm{and} \hspace{5mm}\underline{\pi}^o(\underline{\omega})=\frac{\prod\limits_{i\geq 1}\phi_{-i}^o(\omega_{-i})\mathbb{I}\{\underline{\omega}\in \mathcal{H}^o\}}{\sum\limits_{\underline{\omega}' \in \mathcal{H}^o}\prod\limits_{i\geq 1}\phi_{-i}^o(\omega'_{-i})},$$
\par \noindent where $\phi_{-i}^e(\omega_{-i}) = \frac{\pi_{-i}^*(\omega_{-i})}{\pi_{-i}^*(\omega_{-i}^e)}$ and $\phi_{-i}^o(\omega_{-i}) = \frac{\pi_{-i}^*(\omega_{-i})}{\pi_{-i}^*(\omega_{-i}^o)}$, given explicitly as follows:
\begin{equation*}
    \phi_{-i}^e(\omega_{-i}):=
    \begin{cases}
    \Tilde{q}^{i\omega_{-i}}t^{2\mathbb{I}\{\omega_{-i}\geq 1\}} & i\in 2\mathbb{Z}+1 \\
    \\
    \Tilde{q}^{i(\omega_{-i}-1)}t^{-2\mathbb{I}\{\omega_{-i}=0\}}& i\in 2\mathbb{Z}
    \end{cases} 
 \qquad
    \phi_{-i}^o(\omega_{-i}):=
    \begin{cases}
   \Tilde{q}^{i(\omega_{-i}-1)}t^{-2\mathbb{I}\{\omega_{-i}=0\}}& i\in 2\mathbb{Z}+1\\  
  \\
  \Tilde{q}^{i\omega_{-i}}t^{2\mathbb{I}\{\omega_{-i}\geq 1\}} & i \in 2\mathbb{Z}
    \end{cases}
\end{equation*}

\par \noindent and so we get the following explicit formulae for $\pi^e(\omega)$ and $\pi^o(\omega)$.

\begin{prop}\label{2pi}
The unique stationary measures on $\mathcal{H}^e$ and $\mathcal{H}^o$ are given by
\begin{align*}
 \underline{\pi}^e(\underline{\omega})&=\frac{\Tilde{q}^{\sum\limits_{\textrm{ }i \text{ odd}}i\omega_{-i}+\sum\limits_{i \text{ even}}
  i(\omega_{-i}-1)}t^{2\left(\sum\limits_{i \text{ odd}}\mathbb{I}\{\omega_{-i}\geq 1\}-\sum\limits_{i \text{ even}}\mathbb{I}\{\omega _{-i}=0 \}\right)}}{S_{\text{even}}(\Tilde{q},t)}\\ 
   \underline{\pi}^o(\underline{\omega})&=\frac{\Tilde{q}^{\sum\limits_{\textrm{ }i \text{ even}}i\omega_{-i}+\sum\limits_{i \text{ odd}}
  i(\omega_{-i}-1)}t^{2\left(\sum\limits_{i \text{ even}}\mathbb{I}\{\omega_{-i}\geq 1\}-\sum\limits_{i \text{ odd}}\mathbb{I}\{\omega _{-i}=0 \}\right)}}{S_{\text{odd}}(\Tilde{q},t)}.\end{align*}
\end{prop}

\par \noindent Here $S_{\text{even}}(\Tilde{q},t)$ and $S_{\text{odd}}(\Tilde{q},t)$ are normalizing factors with respect to $\mathcal{H}^e$ and $\mathcal{H}^o$: 
\begin{align*}
S_{\text{even}}(\Tilde{q},t) &= \sum\limits_{\underline{\omega}'\in \mathcal{H}^e}\Tilde{q}^{\sum\limits_{\textrm{ }i \text{ odd}}i\omega_{-i}'+\sum\limits_{i \text{ even}}
  i(\omega_{-i}'-1)}t^{2\left(\sum\limits_{i \text{ odd}}\mathbb{I}\{\omega_{-i}'\geq 1\}-\sum\limits_{i \text{ even}}\mathbb{I}\{\omega _{-i}'=0 \}\right)} \\  S_{\text{odd}}(\Tilde{q},t) &= \sum\limits_{\underline{\omega}'\in \mathcal{H}^o}\Tilde{q}^{\sum\limits_{\textrm{ }i \text{ even}}i\omega_{-i}'+\sum\limits_{i \text{ odd}}
  i(\omega_{-i}'-1)}t^{2\left(\sum\limits_{i \text{ even}}\mathbb{I}\{\omega_{-i}'\geq 1\}-\sum\limits_{i \text{ odd}}\mathbb{I}\{\omega _{-i}'=0 \}\right)}.
  \end{align*}
\begin{proof} It is well known that the restriction of a reversible stationary measure on a continuous time Markov process is also reversible stationary (see Proposition 5.10 of \cite{liggett}). The result follows since $\underline{\pi}^o$ and $\underline{\pi}^e$ are restrictions of $\underline{\pi}^*$.
\end{proof}
It would be interesting to know whether the normalizing factors $S_{\textit{even}}(\Tilde{q},t)$ and $S_{\textit{odd}}(\Tilde{q},t)$ can be written as infinite products. We will see later that when specialising to certain processes (i.e.\ choosing certain values for $\Tilde{q}$ and $t$) the specialised normalizing factors appear to be products with combinatorial significance. We should compare the situation with that of \cite{blocking}, in which the authors stand up ASEP to get AZRP and get the partition function $\prod_{i\geq 1}(1-q^{2i})^{-1}$ as normalizing factor (one of the components of the product side of the Jacobi triple product). However the fact that they naturally get a product seems to be a direct consequence of the choice of ``standing up" map. Indeed it is possible to stand up ASEP in a different way, giving a normalizing factor that has a similar form to the two given above (i.e.\ not naturally given as a product) but is then recognised as the above partition function. Unfortunately we do not see an obvious way to recognise the normalizing factors $S_{\textit{even}}(\Tilde{q},t)$ and $S_{\textit{odd}}(\Tilde{q},t)$ as products.

\subsection{Identities}\label{general idenities section}~
\par By Lemma \ref{im T} the standing up transformation $T^n$ describes a bijection between $\Omega^n$ and one of the state spaces $\mathcal{H}^e$ or $\mathcal{H}^o$, depending on the parity of $n$. Since $T^n$ preserves the dynamics of the corresponding processes we get an equality of measures,  $\underline{\nu}^{n,c}(\underline{\eta}) =  \underline{\pi}^e(T^n(\underline{\eta}))$, when $n \in 2\mathbb{Z}$ and  $\underline{\nu}^{n,c}(\underline{\eta}) = \underline{\pi}^o(T^n(\underline{\eta}))$, when $n \in 2\mathbb{Z}+1$ (for all values of $c$). We will now see that evaluating these equalities at ground states leads to interesting combinatorial identities.

\par Recall that a process on $\Omega$ with blocking measure has two ground states up to shift, $\underline{\eta}^e\in\Omega^0$ and $\underline{\eta}^o\in\Omega^{-1}$, satisfying $T^0(\underline{\eta}^e) = \underline{\omega}^e$ and $T^{-1}(\underline{\eta}^o) = \underline{\omega}^o$. Thus $\underline{\nu}^{0,c}(\underline{\eta}^e) = \underline{\pi}^e(\underline{\omega}^e)$ and $\underline{\nu}^{-1,c}(\underline{\eta}^o) = \underline{\pi}^o(\underline{\omega}^o)$ and by Proposition \ref{nu n} and Proposition \ref{2pi} we get the following two identities (after rearrangement):

\begin{align*}
    2\sum\limits_{\ell=-\infty}^\infty S_{\textit{even}}(\Tilde{q},t)\,\Tilde{q}^{\ell(\ell+1)-2\ell c} &=
  \prod\limits_{i\geq 1} \Tilde{q}^{2(i-c)}Z^c_i(\Tilde{q},t)Z^c_{-i+1}(\Tilde{q},t) +\prod\limits_{i\geq 1} \Tilde{q}^{2(i-c)}W^c_i(\Tilde{q},t)W^c_{-i+1}(\Tilde{q},t) \\ 2t\sum\limits_{\ell=-\infty}^\infty S_{\textit{odd}}(\Tilde{q},t)\,\Tilde{q}^{(\ell+1)^2-(2\ell+1) c} &=
\prod\limits_{i\geq 1} \Tilde{q}^{2(i-c)}Z^c_i(\Tilde{q},t)Z^c_{-i+1}(\Tilde{q},t) -\prod\limits_{i\geq 1} \Tilde{q}^{2(i-c)}W^c_i(\Tilde{q},t)W^c_{-i+1}(\Tilde{q},t).
\end{align*}

Writing $Z_i^c(\Tilde{q},t)$ and $W_i^c(\Tilde{q},t)$ explicitly and letting $z = \Tilde{q}^{-c}$ proves the following identities.

\main*

We were unable find these three variable Jacobi-style identities explicitly written down in the literature and we believe that they are new. It is interesting to note that the above proof of these identities is purely probabilistic and makes no assumption of other classical identities, such as Jacobi triple product (although we will see later that specialising to the ASEP$(q,1)$ process gives the odd/even parts of Jacobi triple product).

We will now discuss the combinatorial nature of these identities. As is well known, the product $\prod_{i\geq 1}(1+\tilde{q}^i)$ is the generating function for partitions of $n$ with distinct parts. In a similar vein the product $\prod_{i\geq 1}(1+\tilde{q}^i + \Tilde{q}^{2i})$ is the generating function for partitions of $n$ with each part appearing at most twice. It is then clear that the two variable product $\prod_{i\geq 1}(1+t\Tilde{q}^i+\Tilde{q}^{2i}) = \sum_{n,m}a_{n,m}\Tilde{q}^nt^m$ is the generating function for such partitions of $n$ with exactly $m$ distinct parts. For example $a_{5,1}=3$ and $a_{5,2}=2$ count the partitions $[5,3+1+1,2+2+1]$ and $[4+1,3+2]$ respectively (the partitions $2+1+1+1$ and $1+1+1+1+1$ are not counted since $1$ appears more than twice). Going one step further the three variable product $\prod_{i\geq 1}(1+tz\Tilde{q}^i+z^2\Tilde{q}^{2i}) = \sum_{n,m,k}a_{n,m,k}\Tilde{q}^n t^m z^k$ is the generating function for such partitions of $n$ which have exactly $k$ parts in total. For example $a_{5,1,1}=1, a_{5,1,3}=2$ and $a_{5,2,2}=2$ count the partitions $[5], [3+1+1,2+2+1]$ and $[4+1,3+2]$ respectively. The meaning of the product $\prod_{i\geq 1}(1+tz\tilde{q}^i+z^2\Tilde{q}^{2i})(1+tz^{-1}\Tilde{q}^{i-1}+z^{-2}\Tilde{q}^{2(i-1)})$ is more subtle and relates to generalised Frobenius partitions.

A generalised Frobenius partition (GFP) of $n$ is a two row array of integers \vspace{2mm}\[\left(\begin{array}{cccc}a_1 & a_2 & ... & a_s\\ b_1 & b_2 & ... & b_s\end{array}\right)\vspace{2mm}\] such that $a_1\geq a_2\geq ... \geq a_s \geq 0$, $b_1\geq b_2\geq ... \geq b_s\geq 0$ and $s+\sum_{1\leq i\leq s}(a_i+b_i) = n$. Given an ordinary partition of $n$ we can produce a GFP of $n$ by letting $s$ be the length of the leading diagonal in the Young diagram, the $a_i$ be the lengths of rows to the right of the diagonal and the $b_i$ be the lengths of the columns under the diagonal (see Figure \ref{gfp example}). This map gives a bijection between ordinary partitions of $n$ and GFP's of $n$ with each row having distinct entries (note that GFP's in general allow repeats in the rows). Taking the conjugate of a partition becomes the natural operation of swapping rows in the corresponding GFP. This convenience was the classical motivation but GFP's are now studied as combinatorial objects in their own right.

\vspace{5mm}
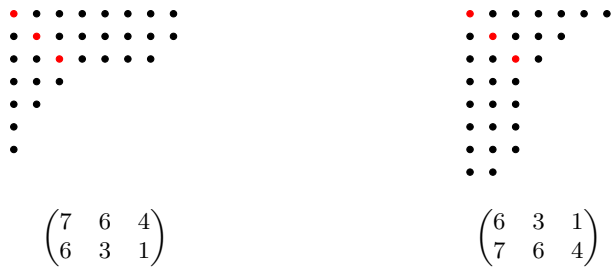
\begin{figure}[H]
    \centering
\begin{tikzpicture}[scale=0.6]
\filldraw [red] (0,3) circle (2pt);
\filldraw [black] (0.5,3) circle (2pt);
\filldraw [black] (1,3) circle (2pt);
\filldraw [black] (1.5,3) circle (2pt);
\filldraw [black] (2,3) circle (2pt);
\filldraw [black] (2.5,3) circle (2pt);
\filldraw [black] (3,3) circle (2pt);
\filldraw [black] (3.5,3) circle (2pt);
\filldraw [black] (0,2.5) circle (2pt);
\filldraw [red] (0.5,2.5) circle (2pt);
\filldraw [black] (1,2.5) circle (2pt);
\filldraw [black] (1.5,2.5) circle (2pt);
\filldraw [black] (2,2.5) circle (2pt);
\filldraw [black] (2.5,2.5) circle (2pt);
\filldraw [black] (3,2.5) circle (2pt);
\filldraw [black] (3.5,2.5) circle (2pt);
\filldraw [black] (0,2) circle (2pt);
\filldraw [black] (0.5,2) circle (2pt);
\filldraw [red] (1,2) circle (2pt);
\filldraw [black] (1.5,2) circle (2pt);
\filldraw [black] (2,2) circle (2pt);
\filldraw [black] (2.5,2) circle (2pt);
\filldraw [black] (3,2) circle (2pt);
\filldraw [black] (0,1.5) circle (2pt);
\filldraw [black] (0.5,1.5) circle (2pt);
\filldraw [black] (1,1.5) circle (2pt);
\filldraw [black] (0,1) circle (2pt);
\filldraw [black] (0.5,1) circle (2pt);
\filldraw [black] (0,0.5) circle (2pt);
\filldraw [black] (0,0) circle (2pt);
\node (a) at (2,-3) [label=\small{$\begin{pmatrix}
7&6&4\\
6&3&1
\end{pmatrix}$}]{};

\filldraw[red] (10,3) circle (2pt);
\filldraw [black] (10.5,3) circle (2pt);
\filldraw [black] (11,3) circle (2pt);
\filldraw [black] (11.5,3) circle (2pt);
\filldraw [black] (12,3) circle (2pt);
\filldraw [black] (12.5,3) circle (2pt);
\filldraw [black] (13,3) circle (2pt);
\filldraw [black] (10,2.5) circle (2pt);
\filldraw [red] (10.5,2.5) circle (2pt);
\filldraw [black] (11,2.5) circle (2pt);
\filldraw [black] (11.5,2.5) circle (2pt);
\filldraw [black] (12,2.5) circle (2pt);
\filldraw [black] (10,2) circle (2pt);
\filldraw [black] (10.5,2) circle (2pt);
\filldraw [red] (11,2) circle (2pt);
\filldraw [black] (11.5,2) circle (2pt);
\filldraw [black] (10,1.5) circle (2pt);
\filldraw [black] (10.5,1.5) circle (2pt);
\filldraw [black] (11,1.5) circle (2pt);
\filldraw [black] (10,1) circle (2pt);
\filldraw [black] (10.5,1) circle (2pt);
\filldraw [black] (11,1) circle (2pt);
\filldraw [black] (10,0.5) circle (2pt);
\filldraw [black] (10.5,0.5) circle (2pt);
\filldraw [black] (11,0.5) circle (2pt);
\filldraw [black] (10,0) circle (2pt);
\filldraw [black] (10.5,0) circle (2pt);
\filldraw [black] (11,0) circle (2pt);
\filldraw [black] (10,-0.5) circle (2pt);
\filldraw [black] (10.5,-0.5) circle (2pt);
\node (a) at (11.5,-3)[label=\small{$\begin{pmatrix}
6&3&1\\
7&6&4
\end{pmatrix}$}]{};
\end{tikzpicture}
    \caption{The conjugate partitions $30=8+8+7+3+2+1+1$ and $30=7+5+4+3+3+3+3+2$ with their corresponding GFP's.}
    \label{gfp example}
\end{figure}

Naturally, we wish to count certain families of GFP's. A ``General Principle" due to Andrews, given in \cite{frobenius}, provides a way to do this. Suppose $f_{A}(\Tilde{q},z) = \sum_{n,k}a_{n,k}\Tilde{q}^nz^k$ and $f_{B}(\Tilde{q},z) = \sum_{n,k}b_{n,k}\Tilde{q}^nz^k$ are generating functions for ordinary partitions of $n$ with $k$ parts and satisfying some conditions $A$ and $B$ respectively. Then the ``General Principle" states that the formal series $f_{A}(\Tilde{q},\Tilde{q}z)f_{B}(\Tilde{q},z^{-1}) = \sum_{k} f_{A,B,k}(\Tilde{q})z^k$ has constant term $f_{A,B,0}(\Tilde{q})$ equal to the generating function for GFP's of $n$ with first row satisfying condition $A$ and second row satisfying condition $B$. The set of such GFP's will be denoted $\text{GFP}_{A,B}(n)$. We can actually give a uniform interpretation of all of the formal series $f_{A,B,k}(\Tilde{q})$ if we generalise further to allow GFP's having rows of unequal length, i.e.\ two row arrays of integers

\vspace{3mm}
\[\left(\begin{array}{cccc}a_1 & a_2 & ... & a_{s_1}\\ b_1 & b_2 & ... & b_{s_2}\end{array}\right)\] 
\vspace{3mm}
\par \noindent such that $a_1 \geq a_2 \geq ... \geq a_{s_1} \geq 0, b_1\geq b_2\geq ... \geq b_{s_2} \geq 0$ and $s_1 + \sum_{1\leq i\leq s_1}a_i + \sum_{1\leq i \leq s_2}b_i = n$. We will refer to these as GFP's with offset $s_1-s_2$ (so that GFP's with offset $0$ are classical GFP's). The $|s_1-s_2|$ ``empty" entries in the shorter row will be labelled to the left with a dash (this will be important later). The full power of the ``General Principle" is then that the formal series $f_{A,B,k}(\Tilde{q})$ is the generating function for the sets $\text{GFP}_{A,B,k}(n)$, defined as above but for (possibly non-zero) offset $k$. A popular choice of condition on the rows is $A=B=D_r$, the condition that each part appears at most $r$ times.

One can use the ``General Principle" to give a combinatorial interpretation of the famous Jacobi Triple Product identity (written here in an equivalent form to the one given in the introduction): 
\vspace{3mm}
\[\prod_{i\geq 1}(1-\Tilde{q}^i)(1+\Tilde{q}^i z)(1+\Tilde{q}^{i-1}z^{-1}) = \sum_{k\in\mathbb{Z}}\Tilde{q}^{\frac{k(k+1)}{2}}z^k.\vspace{3mm}\] 
\par \noindent Indeed rearranging gives: 
\vspace{3mm}
\[\prod_{i\geq 1}(1+\Tilde{q}^i z)(1+\Tilde{q}^{i-1}z^{-1}) = \sum_{k\in\mathbb{Z}}\frac{1}{\prod_{i\geq 1}(1-\Tilde{q}^i)}\Tilde{q}^{\frac{k(k+1)}{2}}z^k\vspace{3mm}\] and so by applying the ``General Principle" to the LHS we see that this identity is equivalent to the equalities $f_{D_1,D_1,k}(\Tilde{q}) = \frac{1}{\prod_{i\geq 1}(1-\Tilde{q}^i)}\Tilde{q}^{\frac{k(k+1)}{2}}$ for all $k\in\mathbb{Z}$. This is true for the base case $k=0$ since we have already seen that $\text{GFP}_{D_1,D_1,0}(n)$ is in bijection with ordinary partitions of $n$. For other values of $k$ the corresponding equality follows from the equivalence of generating functions $f_{D_1,D_1,k}(\Tilde{q}) = f_{D_1,D_1,0}(\Tilde{q})\Tilde{q}^{\frac{k(k+1)}{2}}$, proved by the following bijection $\phi_k:\text{GFP}_{D_1,D_1,0}(n) \rightarrow \text{GFP}_{D_1,D_1,k}\left(n+\frac{k(k+1)}{2}\right)$ often attributed to Wright. Given an element of $\text{GFP}_{D_1,D_1,0}(n)$ we have a corresponding ordinary partition of $n$. Adjoin a right angled triangle of size $\frac{|k|(|k|+1)}{2}$ to either the left or top edge of its Young diagram, depending on whether $k\geq 0$ or $k<0$ respectively (see Figure \ref{gfps with offset example}). Then use the new leading diagonal implied by the triangle to read off an element of $\text{GFP}_{D_1,D_1,k}\left(n+\frac{k(k+1)}{2}\right)$ by letting $s_1$ be the length of the diagonal if $k\geq 0$ ($s_2$ if $k<0$), the $a_i$ be the sizes of rows to the right of the diagonal and the $b_i$ be the sizes of columns under the diagonal (the $|k|$ empty rows/columns coming from the triangle supply the required $|k|$ ``empty" entries in the corresponding GFP of offset $k$).

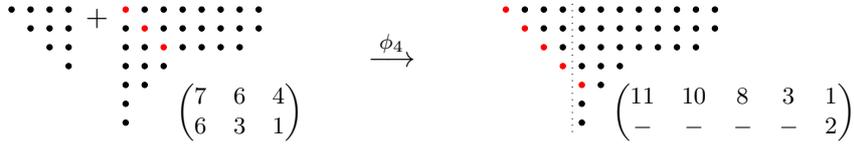
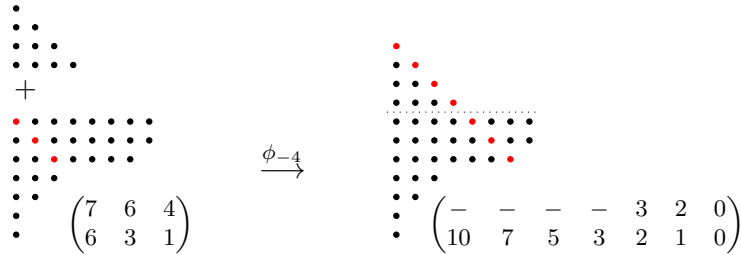
\begin{figure}[H]
    \centering
\begin{subfigure}[b]{0.75\textwidth}
    \centering
\begin{tikzpicture}[scale=0.5]
\filldraw[black](-1.5,3) circle (2pt);
\filldraw[black](-2,3) circle (2pt);
\filldraw[black](-2.5,3) circle (2pt);
\filldraw[black](-3,3) circle (2pt);
\filldraw[black](-1.5,2.5) circle (2pt);
\filldraw[black](-2,2.5) circle (2pt);
\filldraw[black](-2.5,2.5) circle (2pt);
\filldraw[black](-1.5,2) circle (2pt);
\filldraw[black](-2,2) circle (2pt);
\filldraw[black](-1.5,1.5) circle (2pt);
\node (a) at (-0.75,2)[label=\large{+}]{};
\filldraw [red] (0,3) circle (2pt);
\filldraw [black] (0.5,3) circle (2pt);
\filldraw [black] (1,3) circle (2pt);
\filldraw [black] (1.5,3) circle (2pt);
\filldraw [black] (2,3) circle (2pt);
\filldraw [black] (2.5,3) circle (2pt);
\filldraw [black] (3,3) circle (2pt);
\filldraw [black] (3.5,3) circle (2pt);
\filldraw [black] (0,2.5) circle (2pt);
\filldraw [red] (0.5,2.5) circle (2pt);
\filldraw [black] (1,2.5) circle (2pt);
\filldraw [black] (1.5,2.5) circle (2pt);
\filldraw [black] (2,2.5) circle (2pt);
\filldraw [black] (2.5,2.5) circle (2pt);
\filldraw [black] (3,2.5) circle (2pt);
\filldraw [black] (3.5,2.5) circle (2pt);
\filldraw [black] (0,2) circle (2pt);
\filldraw [black] (0.5,2) circle (2pt);
\filldraw [red] (1,2) circle (2pt);
\filldraw [black] (1.5,2) circle (2pt);
\filldraw [black] (2,2) circle (2pt);
\filldraw [black] (2.5,2) circle (2pt);
\filldraw [black] (3,2) circle (2pt);
\filldraw [black] (0,1.5) circle (2pt);
\filldraw [black] (0.5,1.5) circle (2pt);
\filldraw [black] (1,1.5) circle (2pt);
\filldraw [black] (0,1) circle (2pt);
\filldraw [black] (0.5,1) circle (2pt);
\filldraw [black] (0,0.5) circle (2pt);
\filldraw [black] (0,0) circle (2pt);
\node (a) at (3,-1)[label=\small{$\begin{pmatrix}
7&6&4\\
6&3&1
\end{pmatrix}$}]{};
\node (a) at (7,1)[label=\large{$\overset{\phi_{4}}{\longrightarrow}$}]{};
\draw [dotted] (11.75,-0.25) -- (11.75,3.25);
\filldraw[red] (10,3) circle (2pt);
\filldraw [black] (10.5,3) circle (2pt);
\filldraw [black] (11,3) circle (2pt);
\filldraw [black] (11.5,3) circle (2pt);
\filldraw [black] (12,3) circle (2pt);
\filldraw [black] (12.5,3) circle (2pt);
\filldraw [black] (13,3) circle (2pt);
\filldraw [black] (13.5,3) circle (2pt);
\filldraw [black] (14,3) circle (2pt);
\filldraw [black] (14.5,3) circle (2pt);
\filldraw [black] (15,3) circle (2pt);
\filldraw [black] (15.5,3) circle (2pt);
\filldraw [red] (10.5,2.5) circle (2pt);
\filldraw [black] (11,2.5) circle (2pt);
\filldraw [black] (11.5,2.5) circle (2pt);
\filldraw [black] (12,2.5) circle (2pt);
\filldraw [black] (12.5,2.5) circle (2pt);
\filldraw [black] (13,2.5) circle (2pt);
\filldraw [black] (13.5,2.5) circle (2pt);
\filldraw [black] (14,2.5) circle (2pt);
\filldraw [black] (14.5,2.5) circle (2pt);
\filldraw [black] (15,2.5) circle (2pt);
\filldraw [black] (15.5,2.5) circle (2pt);
\filldraw [red] (11,2) circle (2pt);
\filldraw [black] (11.5,2) circle (2pt);
\filldraw [black] (12,2) circle (2pt);
\filldraw [black] (12.5,2) circle (2pt);
\filldraw [black] (13,2) circle (2pt);
\filldraw [black] (13.5,2) circle (2pt);
\filldraw [black] (14,2) circle (2pt);
\filldraw [black] (14.5,2) circle (2pt);
\filldraw [black] (15,2) circle (2pt);
\filldraw [red] (11.5,1.5) circle (2pt);
\filldraw [black] (12,1.5) circle (2pt);
\filldraw [black] (12.5,1.5) circle (2pt);
\filldraw [black] (13,1.5) circle (2pt);
\filldraw [red] (12,1) circle (2pt);
\filldraw [black] (12.5,1) circle (2pt);
\filldraw [black] (12,0.5) circle (2pt);
\filldraw [black] (12,0) circle (2pt);
\node (a) at (16,-1)[label=\small{$\begin{pmatrix}
11&10&8&3&1\\
-&-&-&-&2
\end{pmatrix}$}]{};
\end{tikzpicture}
    \caption{An element of $\text{GFP}_{D_1,D_1,0}(30)$ and its image in $\text{GFP}_{D_1,D_1,4}(40)$ }
\end{subfigure}

\begin{subfigure}[b]{0.75\textwidth}
    \centering
    \vspace{4mm}
    \begin{tikzpicture}[scale=0.5]
\filldraw[black](0,6) circle (2pt);
\filldraw[black](0.5,5.5) circle (2pt);
\filldraw[black](0,5.5) circle (2pt);
\filldraw[black](1,5) circle (2pt);
\filldraw[black](0.5,5) circle (2pt);
\filldraw[black](0,5) circle (2pt);
\filldraw[black](1.5,4.5) circle (2pt);
\filldraw[black](1,4.5) circle (2pt);
\filldraw[black](0.5,4.5) circle (2pt);
\filldraw[black](0,4.5) circle (2pt);
\node (a) at (0.25,3.1)[label=\large{+}]{};
\filldraw [red] (0,3) circle (2pt);
\filldraw [black] (0.5,3) circle (2pt);
\filldraw [black] (1,3) circle (2pt);
\filldraw [black] (1.5,3) circle (2pt);
\filldraw [black] (2,3) circle (2pt);
\filldraw [black] (2.5,3) circle (2pt);
\filldraw [black] (3,3) circle (2pt);
\filldraw [black] (3.5,3) circle (2pt);
\filldraw [black] (0,2.5) circle (2pt);
\filldraw [red] (0.5,2.5) circle (2pt);
\filldraw [black] (1,2.5) circle (2pt);
\filldraw [black] (1.5,2.5) circle (2pt);
\filldraw [black] (2,2.5) circle (2pt);
\filldraw [black] (2.5,2.5) circle (2pt);
\filldraw [black] (3,2.5) circle (2pt);
\filldraw [black] (3.5,2.5) circle (2pt);
\filldraw [black] (0,2) circle (2pt);
\filldraw [black] (0.5,2) circle (2pt);
\filldraw [red] (1,2) circle (2pt);
\filldraw [black] (1.5,2) circle (2pt);
\filldraw [black] (2,2) circle (2pt);
\filldraw [black] (2.5,2) circle (2pt);
\filldraw [black] (3,2) circle (2pt);
\filldraw [black] (0,1.5) circle (2pt);
\filldraw [black] (0.5,1.5) circle (2pt);
\filldraw [black] (1,1.5) circle (2pt);
\filldraw [black] (0,1) circle (2pt);
\filldraw [black] (0.5,1) circle (2pt);
\filldraw [black] (0,0.5) circle (2pt);
\filldraw [black] (0,0) circle (2pt);
\node (a) at (3,-1)[label=\small{$\begin{pmatrix}
7&6&4\\
6&3&1
\end{pmatrix}$}]{};
\node (a) at (7,1)[label=\large{$\overset{\phi_{-4}}{\longrightarrow}$}]{};
\draw [dotted] (9.75,3.25) -- (13.75,3.25);
\filldraw[red](10,5) circle(2pt);
\filldraw[black](10,4.5) circle(2pt);
\filldraw[red](10.5,4.5) circle(2pt);
\filldraw[black](10,4) circle(2pt);
\filldraw[black](10.5,4) circle(2pt);
\filldraw[red](11,4) circle(2pt);
\filldraw[black](10,3.5) circle(2pt);
\filldraw[black](10.5,3.5) circle(2pt);
\filldraw[black](11,3.5) circle(2pt);
\filldraw[red](11.5,3.5) circle(2pt);

\filldraw[black] (10,3) circle (2pt);
\filldraw [black] (10.5,3) circle (2pt);
\filldraw [black] (11,3) circle (2pt);
\filldraw [black] (11.5,3) circle (2pt);
\filldraw [red] (12,3) circle (2pt);
\filldraw [black] (12.5,3) circle (2pt);
\filldraw [black] (13,3) circle (2pt);
\filldraw [black] (13.5,3) circle (2pt);
\filldraw [black] (10,2.5) circle (2pt);
\filldraw [black] (10.5,2.5) circle (2pt);
\filldraw [black] (11,2.5) circle (2pt);
\filldraw [black] (11.5,2.5) circle (2pt);
\filldraw [black] (12,2.5) circle (2pt);
\filldraw [red] (12.5,2.5) circle (2pt);
\filldraw [black] (13,2.5) circle (2pt);
\filldraw [black] (13.5,2.5) circle (2pt);
\filldraw [black] (10,2) circle (2pt);
\filldraw [black] (10.5,2) circle (2pt);
\filldraw [black] (11,2) circle (2pt);
\filldraw [black] (11.5,2) circle (2pt);
\filldraw [black] (12,2) circle (2pt);
\filldraw [black] (12.5,2) circle (2pt);
\filldraw [red] (13,2) circle (2pt);
\filldraw [black] (10,1.5) circle (2pt);
\filldraw [black] (10.5,1.5) circle (2pt);
\filldraw [black] (11,1.5) circle (2pt);
\filldraw [black] (10,1) circle (2pt);
\filldraw [black] (10.5,1) circle (2pt);
\filldraw [black] (10,0.5) circle (2pt);
\filldraw [black] (10,0) circle (2pt);
\node (a) at (15,-1)[label=\small{$\begin{pmatrix}
-&-&-&-&3&2&0\\
10&7&5&3&2&1&0
\end{pmatrix}$}]{};
\end{tikzpicture}
    \caption{An element of $\text{GFP}_{D_1,D_1,0}(30)$ and its image in $\text{GFP}_{D_1,D_1,-4}(36)$ }
\end{subfigure}
    \caption{The Wright bijections $\phi_4$ and $\phi_{-4}$ (cont).}
    \label{gfps with offset example}
\end{figure}
In comparison we can now give a combinatorial interpretation of the identities in Theorem \ref{main}. By a three variable adaptation of the ``General Principle" the formal series: \[\prod_{i\geq 1}(1+tz\Tilde{q}^i+z^2\Tilde{q}^{2i})(1+tz^{-1}\Tilde{q}^{i-1}+z^{-2}\Tilde{q}^{2(i-1)}) = \sum_{k\in\mathbb{Z}}f_{D_2,D_2,k}(\Tilde{q},t)z^k\] has coefficients $f_{D_2,D_2,k}(\Tilde{q},t) = \sum_{n,m}c_{n,m,k}\Tilde{q}^nt^m$ that are two variable generating functions for the sets $\text{GFP}_{D_2,D_2,k,m}(n)\subseteq\text{GFP}_{D_2,D_2,k}(n)$, consisting of such GFP's having a total of $m$ distinct parts (each row treated separately). The identities of Theorem \ref{main} are then equivalent to the equalities: \[f_{D_2,D_2,k}(\Tilde{q},t) = \begin{cases}S_{\text{even}}(\Tilde{q},t)\Tilde{q}^{\ell(\ell+1)} & \text{if } k=2\ell\\ tS_{\text{odd}}(\Tilde{q},t)\Tilde{q}^{(\ell+1)^2} & \text{if } k=2\ell+1 \end{cases}\] (The reason for the various sign changes in the two identities is merely to separate the cases of even and odd offset, since these behave differently). We have of course proved these equalities probabilistically but it is not explicitly clear that the normalising factors are related to counting GFP's with the $2$-repetition condition. However, we are able to give combinatorial proofs of these equalities, similar to the case of Jacobi triple product.

We start with the base cases $k=0$ and $-1$, i.e.\ that \begin{align*}S_{\text{even}}(\Tilde{q},t) &= f_{D_2,D_2,0}(\Tilde{q},t)\\ tS_{\text{odd}}(\Tilde{q},t) &= f_{D_2,D_2,-1}(\Tilde{q},t). \end{align*} Using MAGMA we were able to compute that the first few coefficients of the normalizing factors are given by: \begin{align*}S_{\text{even}}(\Tilde{q},t) &= 1 + \Tilde{q}t^2 + \Tilde{q}^2(1+2t^2) + 5\Tilde{q}^3 t^2 + \Tilde{q}^4(2+6t^2+t^4) + \Tilde{q}^5(12t^2+2t^4) \\ &+ \Tilde{q}^6(3+16t^2+5t^4) + \Tilde{q}^7(25t^2+10t^4) + \Tilde{q}^8(5+30t^2+20t^4) + ...\\ S_{\text{odd}}(\Tilde{q},t) &= 1 + 2\Tilde{q} + \Tilde{q}^2(2+t^2) + \Tilde{q}^3(4+2t^2) + \Tilde{q}^4(5+5t^2) + \Tilde{q}^5(6+10t^2)\\ &+\Tilde{q}^6(10+15t^2+t^4) + \Tilde{q}^7(12+26t^2+2t^4)+\Tilde{q}^8(15+40t^2+5t^4)+...\end{align*} (See the Appendix for the lists of GFP's of offset $0$ and $-1$ that are counted by these coefficients). Note that in both cases the exponents of $t$ are even. This is expected since $\text{GFP}_{D_2,D_2,k,m}(n)$ is empty when $m\not\equiv k\bmod 2$ (this justifies the extra $t$ in the odd case). These expansions will be useful later when specialising to particular processes. In order to prove these base cases we seek analogues of the Frobenius bijection between ordinary partitions of $n$ and elements of $\text{GFP}_{D_1,D_1,0}(n)$. However to describe these we must first generalise the notion of Young diagram to allow GFP's with the $2$-repetition condition. 

Elements of $\text{GFP}_{D_2,D_2,0}(n)$ do not correspond to ordinary partitions and so do not naturally give rise to Young diagrams. However they do naturally give rise to certain finite subsets of $C_e=\{(n_1,n_2)\in\mathbb{Z}^2\,:\, n_1+n_2 \equiv 0 \bmod 2\}$. The subset corresponding to such a GFP with $s$ columns consists of the $s$ leading diagonal points $(1,-1),...,(s,-s)$, the first $a_i$ points of $C_e$ to the right of $(i,-i)$ and the first $b_i$ points of $C_e$ under $(i,-i)$. Similarly, elements of $\text{GFP}_{D_2,D_2,-1}(n)$ will give well defined finite subsets of $C_o=\{(n_1,n_2)\in\mathbb{Z}^2\,:\, n_1+n_2 \equiv 1 \bmod 2\}$. The subset corresponding to such a GFP with $s_1$ entries on the top row contains the $s_1$ leading diagonal points $(2,-1),...,(s_1+1,-s_1)$, the first $a_i$ points of $C_o$ to the right of $(i+1,-i)$ and the first $b_i$ points of $C_o$ under $(i,-i+1)$ (the point $(1,0)$ is not included). See Figure \ref{GFP checkerboard rep} for an example of each kind of generalised Young diagram (points are labelled by black squares, the white squares are included only for aesthetics).

\vspace{5mm}
\begin{figure}[H]
\centering
\begin{tikzpicture}[scale=0.3]
\filldraw[red] (0,0) rectangle (1,1);
\draw[black] (1,0) rectangle (2,1);
\filldraw[black] (2,0) rectangle (3,1);
\draw[black] (3,0) rectangle (4,1);
\filldraw [black] (4,0) rectangle (5,1);
\draw[black] (5,0) rectangle (6,1);
\filldraw[black] (6,0) rectangle (7,1);
\draw[black] (7,0) rectangle (8,1);
\filldraw[black] (8,0) rectangle (9,1);
\draw[black](0,-1) rectangle (1,0);
\filldraw[red] (1,-1) rectangle (2,0);
\draw[black] (2,-1) rectangle (3,0);
\filldraw[black] (3,-1) rectangle (4,0);
\draw[black](4,-1) rectangle (5,0);
\filldraw[black] (5,-1) rectangle (6,0);
\draw[black] (6,-1) rectangle (7,0);
\filldraw[black] (0,-2) rectangle (1,-1);
\draw[black] (1,-2) rectangle (2,-1);
\filldraw[red] (2,-2) rectangle (3,-1);
\draw[black] (3,-2) rectangle (4,-1);
\filldraw[black] (4,-2) rectangle (5,-1);
\draw[black] (5,-2) rectangle (6,-1);
\filldraw[black] (6,-2) rectangle (7,-1);
\draw[black] (0,-3) rectangle (1,-2);
\filldraw[black](1,-3) rectangle (2,-2);
\filldraw[black] (0,-4) rectangle (1,-3);
\draw[black] (1,-4) rectangle (2,-3);
\draw[black] (0,-5) rectangle (1,-4);
\filldraw[black] (1,-5) rectangle (2,-4);
\node(a) at (5.75,-6.25)[label=\small{$\begin{pmatrix}
4&2&2\\
2&2&0
\end{pmatrix}$}]{};
\draw[red] (16,1) rectangle (17,2);
\draw[black] (16,0) rectangle (19,1);
\filldraw[red] (17,0) rectangle (18,1);
\draw[black] (18,0) rectangle (19,1);
\filldraw[black] (19,0) rectangle (20,1);
\draw[black] (20,0) rectangle (21,1);
\filldraw[black] (21,0) rectangle (22,1);
\draw[black](22,0)rectangle(23,1);
\filldraw[black](23,0)rectangle(24,1);
\draw[black](24,0)rectangle(25,1);
\filldraw[black](25,0)rectangle(26,1);
\filldraw[black](16,-1) rectangle (17,0);
\draw[black](17,-1) rectangle (18,0);
\filldraw[red] (18,-1) rectangle (19,0);
\draw[black] (19,-1) rectangle (20,0);
\filldraw[black] (20,-1) rectangle (21,0);
\draw[black] (21,-1) rectangle (22,0);
\filldraw[black] (22,-1) rectangle (23,0);
\draw[black](23,-1) rectangle (24,0);
\draw[black] (16,-2) rectangle (17,-1);
\filldraw[black] (17,-2) rectangle (18,-1);
\draw[black] (18,-2) rectangle (19,-1);
\filldraw[red] (19,-2) rectangle (20,-1);
\draw[black] (20,-2) rectangle (21,-1);
\filldraw[black] (21,-2) rectangle (22,-1);
\draw[black] (22,-2) rectangle (23,-1);
\filldraw[black] (23,-2) rectangle (24,-1);
\filldraw[black](16,-3) rectangle (17,-2);
\draw[black] (17,-3) rectangle (18,-2);
\draw[black] (16,-4) rectangle (17,-3);
\filldraw[black] (17,-4) rectangle (18,-3);
\node (a) at (23,-6) [label=\small{$\begin{pmatrix}
-&4&2&2\\
2&2&0&0
\end{pmatrix}$}]{};

\end{tikzpicture}
\caption{Generalised Young diagrams corresponding to two GFP's of $15$ with offset $0$ and $-1$.}
\label{GFP checkerboard rep}
\end{figure}
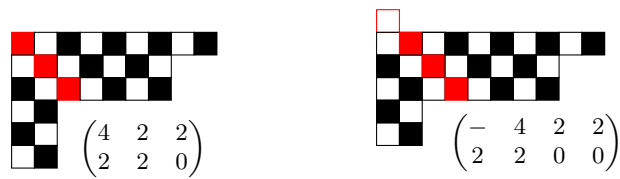
\newpage We can now return to proving the base cases. By comparing coefficients of $\Tilde{q}^nt^m$ on both sides it suffices to find bijections: \vspace{2mm}\begin{align*}\psi^e_{n,m}: \left\{\underline{\omega}\in\mathcal{H}^e\,:\,\begin{array}{c}\sum_{i \text{ odd}}i\omega_{-i} + \sum_{i \text{ even}}i(\omega_{-i}-1) = n\\ 2\left(\sum_{i \text{ odd}}\mathbb{I}\{\omega_{-i}\geq 1\}-\sum_{i \text{ even}}\mathbb{I}\{\omega_{-i}=0\}\right)=m\end{array}\right\} &\rightarrow \text{GFP}_{D_2,D_2,0,m}(n)\\ \psi^o_{n,m}: \left\{\underline{\omega}\in\mathcal{H}^o\,:\,\begin{array}{c}\sum_{i \text{ even}}i\omega_{-i} + \sum_{i \text{ odd}}i(\omega_{-i}-1) = n\\ 2\left(\sum_{i \text{ even}}\mathbb{I}\{\omega_{-i}\geq 1\}-\sum_{i \text{ odd}}\mathbb{I}\{\omega_{-i}=0\}\right)=m\end{array}\right\} &\rightarrow \text{GFP}_{D_2,D_2,-1,m+1}(n).\end{align*} \par \vspace{2mm}\noindent
These maps can be constructed as restrictions of bijections:
\vspace{2mm}\begin{align*}\psi^e_{n}: \left\{\underline{\omega}\in\mathcal{H}^e\,:\,\sum_{i \text{ odd}}i\omega_{-i} + \sum_{i \text{ even}}i(\omega_{-i}-1) = n\right\} &\rightarrow \text{GFP}_{D_2,D_2,0}(n)\\ \psi^o_{n}: \left\{\underline{\omega}\in\mathcal{H}^o\,:\,\sum_{i \text{ even}}i\omega_{-i} + \sum_{i \text{ odd}}i(\omega_{-i}-1) = n\right\} &\rightarrow \text{GFP}_{D_2,D_2,-1}(n).\end{align*}
\par \vspace{2mm} \noindent
In the following we use the notation $r_e$ to stand for a zigzag of length $r$ on $C_e$, a shift of the points of $C_e$ enclosed in the rectangle with opposite corners $(1,-1)$ and $(r,-2)$ and the notation $r_o$ for a zigzag of length $r$ on $C_o$, a shift of the points enclosed within the same rectangle on $C_o$.  Given $\underline{\omega}\in\mathcal{H}^e$ the map $\psi^e_{n}$ stacks $(\omega_{-i} - \mathbb{I}\{i \text{ even}\})$ copies of the zigzag $i_e$ (whenever this is non-negative) vertically in increasing order, and then removes a point from the bottom of each of the columns $1,2,...,i$, for each even $i$ such that $\omega_{-i}=0$ (giving the generalised Young diagram of an element of $\text{GFP}_{D_2,D_2,0}(n)$). Given $\underline{\omega}\in\mathcal{H}^o$ the map $\psi^o_{n}$ stacks $(\omega_{-i} - \mathbb{I}\{i \text{ odd}\})$ copies of the zigzag $i_o$ (whenever this is non-negative) vertically in increasing order, and then removes a point from the bottom of each of the columns $1,2,...,i$, for each odd $i$ such that $\omega_{-i}=0$ (giving the generalised Young diagram of an element of $\text{GFP}_{D_2,D_2,-1}(n)$). See Figure \ref{bijections from omega to gfp} for examples of these maps. It is possible, but tedious, to verify that these maps provide the necessary bijections. We leave this to the reader.

\begin{figure}[H]
 \begin{subfigure}[b]{0.75\textwidth}
    \centering
\begin{tikzpicture}[scale=0.55]
\draw[thick, <-] (-12,0)--(-0.5,0);
\foreach \x in {-11,-10,-9,-8,-7,-6,-5,-4,-3,-2,-1}
    \draw[thick, -](\x cm, 2pt)--(\x cm, -2pt) node[anchor=north]{$\x$};
 
\filldraw[black](-1,0.5) circle (4pt);   
\filldraw[black](-1,1) circle (4pt);  
\filldraw[black](-1,1.5) circle (4pt);
\filldraw [black] (-3,0.5) circle (4pt);
\filldraw [black] (-4,0.5) circle (4pt);
\filldraw[black](-4,1) circle (4pt);
\filldraw [black] (-5,0.5) circle (4pt);
\filldraw[black](-5,1) circle (4pt);
\filldraw [black] (-7,0.5) circle (4pt);
\filldraw [black] (-8,0.5) circle (4pt);
\filldraw[black](-10,0.5) circle (4pt);
\filldraw [black] (-11.5,0.4) circle (1pt);
\filldraw [black] (-11.7,0.4) circle (1pt);
\filldraw [black] (-11.9,0.4) circle (1pt);
\node (a) at (-6,-2.5) [label=\small{$\underline{\omega}=(3,0,1,2,2,0,1,1,0,1,0,...)$}]{};
\node(a) at(2,0)[label=\large{$\longrightarrow$}]{};
\node (a) at (-9,2)[label=\small{$n=19$}]{};
\node (a) at (-9,1.5)[label=\small{$m=4$}]{};

\draw [decorate,decoration={brace,amplitude=10pt},xshift=-4pt,yshift=0pt]
(8.75,4.25) -- (8.75,3.75) node [black,midway,xshift=-0.6cm] 
{}; 
\node(a) at (10.5,3.25) [label=\small{$ \omega_{-7}=1$}]{};
\filldraw[red] (5,4) rectangle (5.5,4.5);
\draw[black](5.5,4) rectangle (6,4.5);
\filldraw[black] (6,4) rectangle (6.5,4.5);
\draw[black] (6.5,4) rectangle (7,4.5);
\filldraw[black] (7,4) rectangle (7.5,4.5);
\draw[black] (7.5,4) rectangle (8,4.5);
\filldraw[black] (8,4) rectangle (8.5,4.5);
\draw[black] (5,3.5)rectangle(5.5,4);
\filldraw[red] (5.5,3.5) rectangle (6,4);
\draw[black](6,3.5) rectangle (6.5,4);
\filldraw[black] (6.5,3.5) rectangle (7,4);
\draw[black] (7,3.5) rectangle (7.5,4);
\filldraw[black] (7.5,3.5) rectangle (8,4);
\draw[black] (8,3.5) rectangle (8.5,4);
\draw [decorate,decoration={brace,amplitude=10pt},xshift=-4pt,yshift=0pt]
(7.75,3.25) -- (7.75,1.75) node [black,midway,xshift=-0.6cm] 
{}; 
\node(a) at (9.5,1.75) [label=\small{$ \omega_{-5}=2$}]{};
\filldraw[black] (5,3) rectangle (5.5,3.5);
\draw[black](5.5,3) rectangle (6,3.5);
\filldraw[red] (6,3) rectangle (6.5,3.5);
\draw[black] (6.5,3) rectangle (7,3.5);
\filldraw[black] (7,3) rectangle (7.5,3.5);
\draw[black] (5,2.5)rectangle(5.5,3);
\filldraw[black] (5.5,2.5) rectangle (6,3);
\draw[black](6,2.5) rectangle (6.5,3);
\filldraw[red] (6.5,2.5) rectangle (7,3);
\draw[black] (7,2.5) rectangle (7.5,3);
\filldraw[black] (5,2) rectangle (5.5,2.5);
\draw[black](5.5,2) rectangle (6,2.5);
\filldraw[black] (6,2) rectangle (6.5,2.5);
\draw[black] (6.5,2) rectangle (7,2.5);
\filldraw[red] (7,2) rectangle (7.5,2.5);
\draw[black] (5,1.5)rectangle(5.5,2);
\filldraw[black] (5.5,1.5) rectangle (6,2);
\draw[black](6,1.5) rectangle (6.5,2);
\filldraw[black] (6.5,1.5) rectangle (7,2);
\draw[black] (7,1.5) rectangle (7.5,2);
\draw [decorate,decoration={brace,amplitude=10pt},xshift=-4pt,yshift=0pt]
(7.25,1.25) -- (7.25,0.75) node [black,midway,xshift=-0.6cm] 
{}; 
\node(a) at (9,0.25) [label=\small{$ \omega_{-4}=2$}]{};
\filldraw[black] (5,1) rectangle (5.5,1.5);
\draw[black](5.5,1) rectangle (6,1.5);
\filldraw[black] (6,1) rectangle (6.5,1.5);
\draw[black] (6.5,1) rectangle (7,1.5);
\draw[black] (5,1.5)rectangle(5.5,2);
\draw[black](5,0.5) rectangle (5.5,1);
\filldraw[black] (5.5,0.5) rectangle (6,1);
\draw[black](6,0.5) rectangle (6.5,1);
\filldraw[black] (6.5,0.5) rectangle (7,1);
\draw [decorate,decoration={brace,amplitude=10pt},xshift=-4pt,yshift=0pt]
(6.75,0.25) -- (6.75,-0.25) node [black,midway,xshift=-0.6cm] 
{}; 
\node(a) at (8.5,-0.75) [label=\small{$ \omega_{-3}=1$}]{};
\filldraw[black] (5,0) rectangle (5.5,0.5);
\draw[black](5.5,0) rectangle (6,0.5);
\filldraw[black] (6,0) rectangle (6.5,0.5);
\draw[black](5,-0.5) rectangle (5.5,0);
\filldraw[black] (5.5,-0.5) rectangle (6,0);
\draw[black](6,-0.5) rectangle (6.5,0);
\draw [decorate,decoration={brace,amplitude=10pt},xshift=-4pt,yshift=0pt]
(6,-0.75) -- (6,-3.25) node [black,midway,xshift=-0.6cm] 
{}; 
\node(a) at (7.75,-2.75) [label=\small{$ \omega_{-1}=3$}]{};
\filldraw[black] (5,-1) rectangle (5.5,-0.5);
\draw[black](5,-1.5) rectangle (5.5,-1);
\filldraw[black] (5,-2) rectangle (5.5,-1.5);
\draw[black](5,-2.5) rectangle (5.5,-2);
\filldraw[black] (5,-3) rectangle (5.5,-2.5);
\draw[black](5,-3.5) rectangle (5.5,-3);
\node (a) at (8,-5.5) [label=\large{$\Big\downarrow$}]{};
\node(a) at (4.75,-4.75) [label=\footnotesize{Since $\omega_{-2}=0$, remove}]{};
\node(a) at (4.75,-5.25) [label=\footnotesize{a point from the bottom}]{};
\node(a) at (4.75,-5.75) [label=\footnotesize{of columns 1 and 2.}]{};

\filldraw[red] (5,-6.5) rectangle (5.5,-6);
\draw[black](5.5,-6.5) rectangle (6,-6);
\filldraw[black] (6,-6.5) rectangle (6.5,-6);
\draw[black] (6.5,-6.5) rectangle (7,-6);
\filldraw[black] (7,-6.5) rectangle (7.5,-6);
\draw[black] (7.5,-6.5) rectangle (8,-6);
\filldraw[black] (8,-6.5) rectangle (8.5,-6);
\draw[black] (5,-7)rectangle(5.5,-6.5);
\filldraw[red] (5.5,-7) rectangle (6,-6.5);
\draw[black](6,-7) rectangle (6.5,-6.5);
\filldraw[black] (6.5,-7) rectangle (7,-6.5);
\draw[black] (7,-7) rectangle (7.5,-6.5);
\filldraw[black] (7.5,-7) rectangle (8,-6.5);
\draw[black] (8,-7) rectangle (8.5,-6.5);
\filldraw[black] (5,-7.5) rectangle (5.5,-7);
\draw[black](5.5,-7.5) rectangle (6,-7);
\filldraw[red] (6,-7.5) rectangle (6.5,-7);
\draw[black] (6.5,-7.5) rectangle (7,-7);
\filldraw[black] (7,-7.5) rectangle (7.5,-7);
\draw[black] (5,-8)rectangle(5.5,-7.5);
\filldraw[black] (5.5,-8) rectangle (6,-7.5);
\draw[black](6,-8) rectangle (6.5,-7.5);
\filldraw[red] (6.5,-8) rectangle (7,-7.5);
\draw[black] (7,-8) rectangle (7.5,-7.5);
\filldraw[black] (5,-8.5) rectangle (5.5,-8);
\draw[black](5.5,-8.5) rectangle (6,-8);
\filldraw[black] (6,-8.5) rectangle (6.5,-8);
\draw[black] (6.5,-8.5) rectangle (7,-8);
\filldraw[red] (7,-8.5) rectangle (7.5,-8);
\draw[black] (5,-9)rectangle(5.5,-8.5);
\filldraw[black] (5.5,-9) rectangle (6,-8.5);
\draw[black](6,-9) rectangle (6.5,-8.5);
\filldraw[black] (6.5,-9) rectangle (7,-8.5);
\draw[black] (7,-9) rectangle (7.5,-8.5);
\filldraw[black] (5,-9.5) rectangle (5.5,-9);
\draw[black](5.5,-9.5) rectangle (6,-9);
\filldraw[black] (6,-9.5) rectangle (6.5,-9);
\draw[black] (6.5,-9.5) rectangle (7,-9);
\draw[black] (5,-10)rectangle(5.5,-9.5);
\draw[black](5,-10) rectangle (5.5,-9.5);
\filldraw[black] (5.5,-10) rectangle (6,-9.5);
\draw[black](6,-10) rectangle (6.5,-9.5);
\filldraw[black] (6.5,-10) rectangle (7,-9.5);
\filldraw[black] (5,-10.5) rectangle (5.5,-10);
\draw[black](5.5,-10.5) rectangle (6,-10);
\filldraw[black] (6,-10.5) rectangle (6.5,-10);
\draw[black](5,-11) rectangle (5.5,-10.5);
\filldraw[black] (5,-11.5) rectangle (5.5,-11);
\draw[black](5,-12) rectangle (5.5,-11.5);
\filldraw[black] (5,-12.5) rectangle (5.5,-12);
\draw[black](5,-13) rectangle (5.5,-12.5);
\node(a) at (1.25,-8.5) [label=\footnotesize{Since $\omega_{-6}=0$, remove}]{};
\node(a) at (1.25,-9) [label=\footnotesize{a point from the bottom}]{};
\node(a) at (1.25,-9.5) [label=\footnotesize{of columns 1 through 6.}]{};
\node(a) at(2,-10)[label=\large{$\longleftarrow$}]{};
\filldraw[red] (-8,-6.5) rectangle (-7.5,-6);
\draw[black](-7.5,-6.5) rectangle (-7,-6);
\filldraw[black] (-7,-6.5) rectangle (-6.5,-6);
\draw[black] (-6.5,-6.5) rectangle (-6,-6);
\filldraw[black] (-6,-6.5) rectangle (-5.5,-6);
\draw[black] (-5.5,-6.5) rectangle (-5,-6);
\filldraw[black] (-5,-6.5) rectangle (-4.5,-6);
\draw[black] (-8,-7)rectangle(-7.5,-6.5);
\filldraw[red] (-7.5,-7) rectangle (-7,-6.5);
\draw[black](-7,-7) rectangle (-6.5,-6.5);
\filldraw[black] (-6.5,-7) rectangle (-6,-6.5);
\draw[black] (-6,-7) rectangle (-5.5,-6.5);
\filldraw[black] (-8,-7.5) rectangle (-7.5,-7);
\draw[black](-7.5,-7.5) rectangle (-7,-7);
\filldraw[red] (-7,-7.5) rectangle (-6.5,-7);
\draw[black] (-6.5,-7.5) rectangle (-6,-7);
\filldraw[black] (-6,-7.5) rectangle (-5.5,-7);
\draw[black] (-8,-8)rectangle(-7.5,-7.5);
\filldraw[black] (-7.5,-8) rectangle (-7,-7.5);
\draw[black](-7,-8) rectangle (-6.5,-7.5);
\filldraw[red] (-6.5,-8) rectangle (-6,-7.5);
\filldraw[black] (-8,-8.5) rectangle (-7.5,-8);
\draw[black](-7.5,-8.5) rectangle (-7,-8);
\filldraw[black] (-7,-8.5) rectangle (-6.5,-8);
\draw[black] (-6.5,-8.5) rectangle (-6,-8);
\draw[black] (-8,-9)rectangle(-7.5,-8.5);
\filldraw[black] (-7.5,-9) rectangle (-7,-8.5);
\draw[black](-7,-9) rectangle (-6.5,-8.5);
\filldraw[black] (-6.5,-9) rectangle (-6,-8.5);
\filldraw[black] (-8,-9.5) rectangle (-7.5,-9);
\draw[black](-7.5,-9.5) rectangle (-7,-9);
\filldraw[black] (-7,-9.5) rectangle (-6.5,-9);
\draw[black] (-6.5,-9.5) rectangle (-6,-9);
\draw[black] (-8,-10)rectangle(-7.5,-9.5);
\filldraw[black] (-8,-10.5) rectangle (-7.5,-10);
\draw[black](-8,-11) rectangle (-7.5,-10.5);
\filldraw[black] (-8,-11.5) rectangle (-7.5,-11);
\draw[black](-8,-12) rectangle (-7.5,-11.5);
\node(a) at (-6,-5.5) [label=\large{$\Big\downarrow$}]{};
\node (a) at (-7,-5) [label=\small{$\psi^e_{19,4}$}]{};
\node(a) at (-4,-12) [label=\small{$\begin{pmatrix}
3&1&1&0\\
5&2&2&1
\end{pmatrix}$}]{};
\node (a) at (-10,-7)[label=\small{$n=19$}]{};
\node (a) at (-10,-7.5)[label=\small{$m=4$}]{};
\end{tikzpicture}
\caption{A state $\underline{\omega} \in \mathcal{H}^e$ and its image in $\text{GFP}_{D_2,D_2,0,4}(19)$ }
\end{subfigure}
    \caption{The bijections $\psi^e_{19,4}$ and $\psi_{19,4}^o$.}
\end{figure}
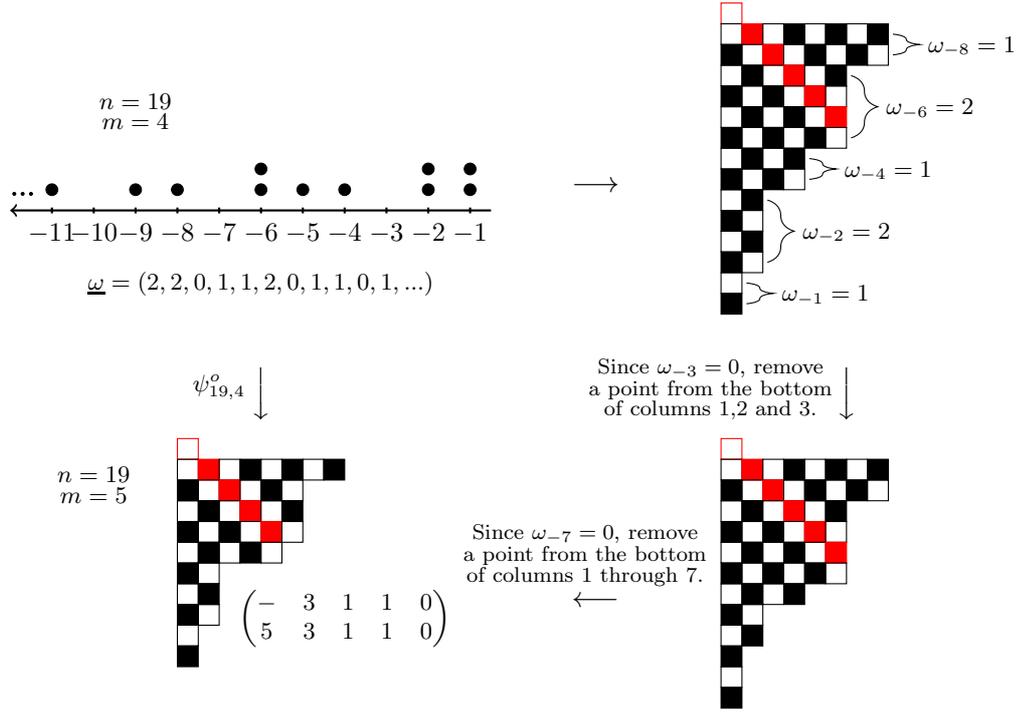
\begin{figure}[H]\ContinuedFloat
    \centering
\begin{subfigure}[b]{0.75\textwidth}
    \centering
    \begin{tikzpicture}[scale=0.55]
\draw[thick, <-] (-12,0)--(-0.5,0);
\foreach \x in {-11,-10,-9,-8,-7,-6,-5,-4,-3,-2,-1}
    \draw[thick, -](\x cm, 2pt)--(\x cm, -2pt) node[anchor=north]{$\x$};
 
\filldraw[black](-1,0.5) circle (4pt);   
\filldraw[black](-1,1) circle (4pt);  
\filldraw[black](-2,0.5) circle (4pt);
\filldraw[black](-2,1) circle (4pt);
\filldraw [black] (-4,0.5) circle (4pt);
\filldraw [black] (-5,0.5) circle (4pt);
\filldraw[black](-6,0.5) circle (4pt);
\filldraw[black](-6,1) circle (4pt);
\filldraw [black] (-8,0.5) circle (4pt);
\filldraw [black] (-9,0.5) circle (4pt);
\filldraw [black] (-11,0.5) circle (4pt);
\filldraw [black] (-11.5,0.4) circle (1pt);
\filldraw [black] (-11.7,0.4) circle (1pt);
\filldraw [black] (-11.9,0.4) circle (1pt);
\node (a) at (-6,-2.5) [label=\small{$\underline{\omega}=(2,2,0,1,1,2,0,1,1,0,1,...)$}]{};
\node(a) at(2,0)[label=\large{$\longrightarrow$}]{};
\node (a) at (-9,2)[label=\small{$n=19$}]{};
\node (a) at (-9,1.5)[label=\small{$m=4$}]{};

\draw [decorate,decoration={brace,amplitude=10pt},xshift=-4pt,yshift=0pt]
(9.25,4.25) -- (9.25,3.75) node [black,midway,xshift=-0.6cm] 
{}; 
\node(a) at (11,3.25) [label=\small{$ \omega_{-8}=1$}]{};
\draw[red] (5,4.5) rectangle (5.5,5);
\draw[black] (5,4) rectangle (5.5,4.5);
\filldraw[red](5.5,4) rectangle (6,4.5);
\draw[black] (6,4) rectangle (6.5,4.5);
\filldraw[black] (6.5,4) rectangle (7,4.5);
\draw[black] (7,4) rectangle (7.5,4.5);
\filldraw[black] (7.5,4) rectangle (8,4.5);
\draw[black] (8,4) rectangle (8.5,4.5);
\filldraw[black] (8.5,4) rectangle (9,4.5);
\filldraw[black] (5,3.5)rectangle(5.5,4);
\draw[black] (5.5,3.5) rectangle (6,4);
\filldraw[red](6,3.5) rectangle (6.5,4);
\draw[black] (6.5,3.5) rectangle (7,4);
\filldraw[black] (7,3.5) rectangle (7.5,4);
\draw[black] (7.5,3.5) rectangle (8,4);
\filldraw[black] (8,3.5) rectangle (8.5,4);
\draw[black] (8.5,3.5) rectangle (9,4);
\draw [decorate,decoration={brace,amplitude=10pt},xshift=-4pt,yshift=0pt]
(8.25,3.25) -- (8.25,1.75) node [black,midway,xshift=-0.6cm] 
{}; 
\node(a) at (10,1.75) [label=\small{$ \omega_{-6}=2$}]{};
\draw[black] (5,3) rectangle (5.5,3.5);
\filldraw[black](5.5,3) rectangle (6,3.5);
\draw[black] (6,3) rectangle (6.5,3.5);
\filldraw[red] (6.5,3) rectangle (7,3.5);
\draw[black] (7,3) rectangle (7.5,3.5);
\filldraw[black](7.5,3)rectangle(8,3.5);
\filldraw[black] (5,2.5)rectangle(5.5,3);
\draw[black] (5.5,2.5) rectangle (6,3);
\filldraw[black](6,2.5) rectangle (6.5,3);
\draw[black] (6.5,2.5) rectangle (7,3);
\filldraw[red] (7,2.5) rectangle (7.5,3);
\draw[black] (7.5,2.5) rectangle (8,3);
\draw[black] (5,2) rectangle (5.5,2.5);
\filldraw[black](5.5,2) rectangle (6,2.5);
\draw[black] (6,2) rectangle (6.5,2.5);
\filldraw[black] (6.5,2) rectangle (7,2.5);
\draw[black] (7,2) rectangle (7.5,2.5);
\filldraw[red](7.5,2) rectangle (8,2.5);
\filldraw[black] (5,1.5)rectangle(5.5,2);
\draw[black] (5.5,1.5) rectangle (6,2);
\filldraw[black](6,1.5) rectangle (6.5,2);
\draw[black] (6.5,1.5) rectangle (7,2);
\filldraw[black] (7,1.5) rectangle (7.5,2);
\draw[black] (7.5,1.5) rectangle (8,2);
\draw [decorate,decoration={brace,amplitude=10pt},xshift=-4pt,yshift=0pt]
(7.25,1.25) -- (7.25,0.75) node [black,midway,xshift=-0.6cm] 
{}; 
\node(a) at (9,0.25) [label=\small{$ \omega_{-4}=1$}]{};
\draw[black] (5,1) rectangle (5.5,1.5);
\filldraw[black](5.5,1) rectangle (6,1.5);
\draw[black] (6,1) rectangle (6.5,1.5);
\filldraw[black] (6.5,1) rectangle (7,1.5);
\filldraw[black](5,0.5) rectangle (5.5,1);
\draw[black] (5.5,0.5) rectangle (6,1);
\filldraw[black](6,0.5) rectangle (6.5,1);
\draw[black] (6.5,0.5) rectangle (7,1);
\draw [decorate,decoration={brace,amplitude=10pt},xshift=-4pt,yshift=0pt]
(6.25,0.25) -- (6.25,-1.25) node [black,midway,xshift=-0.6cm] 
{}; 
\node(a) at (8,-1.25) [label=\small{$ \omega_{-2}=2$}]{};
\draw[black] (5,0) rectangle (5.5,0.5);
\filldraw[black](5.5,0) rectangle (6,0.5);
\filldraw[black](5,-0.5) rectangle (5.5,0);
\draw[black] (5.5,-0.5) rectangle (6,0);
\draw[black] (5,-1) rectangle (5.5,-0.5);
\filldraw[black](5.5,-1)rectangle (6,-0.5);
\filldraw[black](5,-1.5) rectangle (5.5,-1);
\draw[black] (5.5,-1.5) rectangle (6,-1);
\draw [decorate,decoration={brace,amplitude=10pt},xshift=-4pt,yshift=0pt]
(5.75,-1.75) -- (5.75,-2.25) node [black,midway,xshift=-0.6cm] 
{}; 
\node(a) at (7.5,-2.75) [label=\small{$ \omega_{-1}=1$}]{};
\draw[black] (5,-2) rectangle (5.5,-1.5);
\filldraw[black](5,-2.5) rectangle (5.5,-2);
\node (a) at (8,-5.5) [label=\large{$\Big\downarrow$}]{};
\node(a) at (4.75,-4.5) [label=\footnotesize{Since $\omega_{-3}=0$, remove}]{};
\node(a) at (4.75,-5) [label=\footnotesize{a point from the bottom}]{};
\node(a) at (4.75,-5.5) [label=\footnotesize{of columns 1,2 and 3.}]{};

\draw[red] (5,-6) rectangle (5.5,-5.5);
\draw[black] (5,-6.5) rectangle (5.5,-6);
\filldraw[red](5.5,-6.5) rectangle (6,-6);
\draw[black] (6,-6.5) rectangle (6.5,-6);
\filldraw[black] (6.5,-6.5) rectangle (7,-6);
\draw[black] (7,-6.5) rectangle (7.5,-6);
\filldraw[black] (7.5,-6.5) rectangle (8,-6);
\draw[black] (8,-6.5) rectangle (8.5,-6);
\filldraw[black] (8.5,-6.5) rectangle (9,-6);
\filldraw[black] (5,-7)rectangle(5.5,-6.5);
\draw[black] (5.5,-7) rectangle (6,-6.5);
\filldraw[red](6,-7) rectangle (6.5,-6.5);
\draw[black] (6.5,-7) rectangle (7,-6.5);
\filldraw[black] (7,-7) rectangle (7.5,-6.5);
\draw[black] (7.5,-7) rectangle (8,-6.5);
\filldraw[black] (8,-7) rectangle (8.5,-6.5);
\draw[black] (8.5,-7) rectangle (9,-6.5);
\draw[black] (5,-7.5) rectangle (5.5,-7);
\filldraw[black](5.5,-7.5) rectangle (6,-7);
\draw[black] (6,-7.5) rectangle (6.5,-7);
\filldraw[red] (6.5,-7.5) rectangle (7,-7);
\draw[black] (7,-7.5) rectangle (7.5,-7);
\filldraw[black](7.5,-7.5)rectangle(8,-7);
\filldraw[black] (5,-8)rectangle(5.5,-7.5);
\draw[black] (5.5,-8) rectangle (6,-7.5);
\filldraw[black](6,-8) rectangle (6.5,-7.5);
\draw[black] (6.5,-8) rectangle (7,-7.5);
\filldraw[red] (7,-8) rectangle (7.5,-7.5);
\draw[black] (7.5,-8) rectangle (8,-7.5);
\draw[black] (5,-8.5) rectangle (5.5,-8);
\filldraw[black](5.5,-8.5) rectangle (6,-8);
\draw[black] (6,-8.5) rectangle (6.5,-8);
\filldraw[black] (6.5,-8.5) rectangle (7,-8);
\draw[black] (7,-8.5) rectangle (7.5,-8);
\filldraw[red](7.5,-8.5) rectangle (8,-8);
\filldraw[black] (5,-9)rectangle(5.5,-8.5);
\draw[black] (5.5,-9) rectangle (6,-8.5);
\filldraw[black](6,-9) rectangle (6.5,-8.5);
\draw[black] (6.5,-9) rectangle (7,-8.5);
\filldraw[black] (7,-9) rectangle (7.5,-8.5);
\draw[black] (7.5,-9) rectangle (8,-8.5);
\draw[black] (5,-9.5) rectangle (5.5,-9);
\filldraw[black](5.5,-9.5) rectangle (6,-9);
\draw[black] (6,-9.5) rectangle (6.5,-9);
\filldraw[black] (6.5,-9.5) rectangle (7,-9);
\filldraw[black](5,-10) rectangle (5.5,-9.5);
\draw[black] (5.5,-10) rectangle (6,-9.5);
\draw[black] (5,-10.5) rectangle (5.5,-10);
\filldraw[black](5.5,-10.5) rectangle (6,-10);
\filldraw[black](5,-11) rectangle (5.5,-10.5);
\draw[black] (5,-11.5) rectangle (5.5,-11);
\filldraw[black](5,-12) rectangle (5.5,-11.5);
\node(a) at (1.75,-8.5) [label=\footnotesize{Since $\omega_{-7}=0$, remove}]{};
\node(a) at (1.75,-9) [label=\footnotesize{a point from the bottom}]{};
\node(a) at (1.75,-9.5) [label=\footnotesize{of columns 1 through 7.}]{};
\node(a) at(2,-10)[label=\large{$\longleftarrow$}]{};
\draw[red](-8,-6) rectangle (-7.5,-5.5);
\draw[black] (-8,-6.5) rectangle (-7.5,-6);
\filldraw[red](-7.5,-6.5) rectangle (-7,-6);
\draw[black] (-7,-6.5) rectangle (-6.5,-6);
\filldraw[black] (-6.5,-6.5) rectangle (-6,-6);
\draw[black] (-6,-6.5) rectangle (-5.5,-6);
\filldraw[black] (-5.5,-6.5) rectangle (-5,-6);
\draw[black] (-5,-6.5) rectangle (-4.5,-6);
\filldraw[black] (-4.5,-6.5) rectangle (-4,-6);
\filldraw[black] (-8,-7)rectangle(-7.5,-6.5);
\draw[black] (-7.5,-7) rectangle (-7,-6.5);
\filldraw[red](-7,-7) rectangle (-6.5,-6.5);
\draw[black] (-6.5,-7) rectangle (-6,-6.5);
\filldraw[black] (-6,-7) rectangle (-5.5,-6.5);
\draw[black] (-5.5,-7) rectangle (-5,-6.5);
\draw[black] (-8,-7.5) rectangle (-7.5,-7);
\filldraw[black](-7.5,-7.5) rectangle (-7,-7);
\draw[black] (-7,-7.5) rectangle (-6.5,-7);
\filldraw[red] (-6.5,-7.5) rectangle (-6,-7);
\draw[black] (-6,-7.5) rectangle (-5.5,-7);
\filldraw[black] (-5.5,-7.5) rectangle (-5,-7);
\filldraw[black] (-8,-8)rectangle(-7.5,-7.5);
\draw[black] (-7.5,-8) rectangle (-7,-7.5);
\filldraw[black](-7,-8) rectangle (-6.5,-7.5);
\draw[black] (-6.5,-8) rectangle (-6,-7.5);
\filldraw[red](-6,-8) rectangle (-5.5,-7.5);
\draw[black] (-5.5,-8) rectangle (-5,-7.5);
\draw[black] (-8,-8.5) rectangle (-7.5,-8);
\filldraw[black](-7.5,-8.5) rectangle (-7,-8);
\draw[black] (-7,-8.5) rectangle (-6.5,-8);
\filldraw[black] (-6.5,-8.5) rectangle (-6,-8);
\draw[black] (-6,-8.5) rectangle (-5.5,-8);
\filldraw[black] (-8,-9)rectangle(-7.5,-8.5);
\draw[black] (-7.5,-9) rectangle (-7,-8.5);
\draw[black] (-8,-9.5) rectangle (-7.5,-9);
\filldraw[black](-7.5,-9.5) rectangle (-7,-9);
\filldraw[black] (-8,-10)rectangle(-7.5,-9.5);
\draw[black](-7.5,-10) rectangle (-7,-9.5);
\draw[black] (-8,-10.5) rectangle (-7.5,-10);
\filldraw[black](-8,-11) rectangle (-7.5,-10.5);
\node(a) at (-6,-5.5) [label=\large{$\Big\downarrow$}]{};
\node (a) at (-7,-5) [label=\small{$\psi^o_{19,4}$}]{};
\node(a) at (-4,-11) [label=\small{$\begin{pmatrix}
-&3&1&1&0\\
5&3&1&1&0
\end{pmatrix}$}]{};
\node (a) at (-10,-7)[label=\small{$n=19$}]{};
\node (a) at (-10,-7.5)[label=\small{$m=5$}]{};
\end{tikzpicture}
    \caption{A state $\underline{\omega} \in \mathcal{H}^o$ and its image in $\text{GFP}_{D_2,D_2,-1,5}(19)$ }
\end{subfigure}
    \caption{The bijections $\psi^e_{19,4}$ and $\psi_{19,4}^o$.}
    \label{bijections from omega to gfp}
\end{figure}

\par We now consider the other values of the offset $k$. To tackle these cases it suffices, as in the Jacobi triple product case, to prove the following equivalences of generating functions: \[f_{D_2,D_2,k}(\Tilde{q},t) = \begin{cases}f_{D_2,D_2,0}(\Tilde{q},t)\Tilde{q}^{\ell(\ell+1)} & \text{if } k=2\ell\\ f_{D_2,D_2,-1}(\Tilde{q},t)\Tilde{q}^{(\ell+1)^2} & \text{if } k=2\ell+1. \end{cases}\] These follow naturally from a generalisation of Wright's bijection $\phi_k$, i.e.\ bijections  \begin{align*}\phi^e_{\ell,n}: \text{GFP}_{D_2,D_2,0}(n) &\rightarrow \text{GFP}_{D_2,D_2,2\ell}(n+\ell(\ell+1))\\ \phi^o_{\ell,n}: \text{GFP}_{D_2,D_2,-1}(n) &\rightarrow \text{GFP}_{D_2,D_2,2\ell+1}(n+(\ell+1)^2)\end{align*} defined for each $n\geq 0$ and $\ell\in\mathbb{Z}$ as follows. Given an element of $\text{GFP}_{D_2,D_2,0}(n)$ the map $\phi^e_{\ell,n}$ adds the points inside the right angled triangle of size $\frac{|2\ell|(|2\ell|+1)}{2}$ to either the left or top edge of its generalised Young diagram, depending on whether $\ell\geq 0$ or $\ell<0$ respectively. It then uses the new leading diagonal implied by the triangle to read off an element of $\text{GFP}_{D_2,D_2,2\ell}(n+\ell(\ell+1))$. Given an element of $\text{GFP}_{D_2,D_2,-1}(n)$ the map $\phi^o_{\ell,n}$ adds the points inside the right angled triangle of size $\frac{|2\ell+1|(|2\ell+1|+1)}{2}$ to either the left or top edge of its generalised Young diagram, depending on whether $\ell\geq 0$ or $\ell<0$ respectively. It then uses the new leading diagonal implied by the triangle to read off an element of $\text{GFP}_{D_2,D_2,2\ell+1}(n+(\ell+1)^2)$. In both cases the offset is supplied by the empty rows/columns coming from the triangle (in a similar fashion to Wright's bijection). See Figure \ref{bijections from 0 offset to 4} for examples of $\phi^e_\ell$ and $\phi^o_\ell$. Note also that in both cases the number $m$ of distinct parts is clearly preserved and so the above maps restrict to give well defined bijections \begin{align*}\phi^e_{\ell,n,m}: \text{GFP}_{D_2,D_2,0,m}(n) &\rightarrow \text{GFP}_{D_2,D_2,2\ell,m}(n+\ell(\ell+1))\\ \phi^o_{\ell,n,m}: \text{GFP}_{D_2,D_2,-1,m}(n) &\rightarrow \text{GFP}_{D_2,D_2,2\ell+1,m}(n+(\ell+1)^2)\end{align*} for each $n,m\geq 0$ and $\ell\in\mathbb{Z}$. This proves the required equivalences of generating functions and hence completes the combinatorial justification of the identities in Theorem \ref{main}.
\begin{figure}[H]
    \centering
\begin{subfigure}[b]{0.75\textwidth}
    \centering
\begin{tikzpicture}[scale=0.275]
\draw[black](-4,0) rectangle (-3,1);
\filldraw[black] (-5,0) rectangle (-4,1);
\draw[black] (-6,0) rectangle (-5,1);
\filldraw[black](-7,0) rectangle (-6,1);
\filldraw[black](-4,-1)rectangle (-3,0);
\draw[black] (-5,-1) rectangle (-4,0);
\filldraw[black](-6,-1)rectangle(-5,0);
\draw[black](-4,-2) rectangle (-3,-1);
\filldraw[black] (-5,-2) rectangle (-4,-1);
\filldraw[black] (-4,-3) rectangle (-3,-2);
\node (a) at (-1.5,-2) [label=\large{+}]{};
\filldraw[red] (0,0) rectangle (1,1);
\draw[black] (1,0) rectangle (2,1);
\filldraw[black] (2,0) rectangle (3,1);
\draw[black] (3,0) rectangle (4,1);
\filldraw [black] (4,0) rectangle (5,1);
\draw[black] (5,0) rectangle (6,1);
\filldraw[black] (6,0) rectangle (7,1);
\draw[black] (7,0) rectangle (8,1);
\filldraw[black] (8,0) rectangle (9,1);
\draw[black](0,-1) rectangle (1,0);
\filldraw[red] (1,-1) rectangle (2,0);
\draw[black] (2,-1) rectangle (3,0);
\filldraw[black] (3,-1) rectangle (4,0);
\draw[black](4,-1) rectangle (5,0);
\filldraw[black] (5,-1) rectangle (6,0);
\draw[black] (6,-1) rectangle (7,0);
\filldraw[black] (0,-2) rectangle (1,-1);
\draw[black] (1,-2) rectangle (2,-1);
\filldraw[red] (2,-2) rectangle (3,-1);
\draw[black] (3,-2) rectangle (4,-1);
\filldraw[black] (4,-2) rectangle (5,-1);
\draw[black] (5,-2) rectangle (6,-1);
\filldraw[black] (6,-2) rectangle (7,-1);
\draw[black] (0,-3) rectangle (1,-2);
\filldraw[black](1,-3) rectangle (2,-2);
\filldraw[black] (0,-4) rectangle (1,-3);
\draw[black] (1,-4) rectangle (2,-3);
\draw[black] (0,-5) rectangle (1,-4);
\filldraw[black] (1,-5) rectangle (2,-4);
\node(a) at (5.75,-6.25)[label=\small{$\begin{pmatrix}
4&2&2\\
2&2&0
\end{pmatrix}$}]{};
\node (a) at (13,-2) [label=\large{$\overset{\phi_{2,15}^e}{\longrightarrow}$}]{};
\filldraw[red](18,0)rectangle(19,1);
\draw[black](19,0)rectangle (20,1);
\filldraw[black](20,0)rectangle(21,1);
\draw[black](21,0)rectangle(22,1);
\filldraw[black] (22,0) rectangle (23,1);
\draw[black] (23,0) rectangle (24,1);
\filldraw[black] (24,0) rectangle (25,1);
\draw[black] (25,0) rectangle (26,1);
\filldraw [black] (26,0) rectangle (27,1);
\draw[black] (27,0) rectangle (28,1);
\filldraw[black] (28,0) rectangle (29,1);
\draw[black] (29,0) rectangle (30,1);
\filldraw[black] (30,0) rectangle (31,1);
\filldraw[red](19,-1)rectangle(20,0);
\draw[black](20,-1)rectangle(21,0);
\filldraw[black](21,-1)rectangle (22,0);
\draw[black](22,-1) rectangle (23,0);
\filldraw[black] (23,-1) rectangle (24,0);
\draw[black] (24,-1) rectangle (25,0);
\filldraw[black] (25,-1) rectangle (26,0);
\draw[black](26,-1) rectangle (27,0);
\filldraw[black] (27,-1) rectangle (28,0);
\draw[black] (28,-1) rectangle (29,0);
\filldraw[red](20,-2) rectangle (21,-1);
\draw[black](21,-2)rectangle (22,-1);
\filldraw[black] (22,-2) rectangle (23,-1);
\draw[black] (23,-2) rectangle (24,-1);
\filldraw[black] (24,-2) rectangle (25,-1);
\draw[black] (25,-2) rectangle (26,-1);
\filldraw[black] (26,-2) rectangle (27,-1);
\draw[black] (27,-2) rectangle (28,-1);
\filldraw[black] (28,-2) rectangle (29,-1);
\filldraw[red](21,-3)rectangle(22,-2);
\draw[black] (22,-3) rectangle (23,-2);
\filldraw[black](23,-3) rectangle (24,-2);
\filldraw[red] (22,-4) rectangle (23,-3);
\draw[black] (23,-4) rectangle (24,-3);
\draw[black] (22,-5) rectangle (23,-4);
\filldraw[red] (23,-5) rectangle (24,-4);
\node (a) at (31.25,-6.25) [label=\small{$\begin{pmatrix}
6&4&4&1&0&0\\
-&-&-&-&0&0
\end{pmatrix}$}]{};
\end{tikzpicture}
    \caption{An element of $\text{GFP}_{D_2,D_2,0}(15)$ and its image in $\text{GFP}_{D_2,D_2,4}(21)$ }
\end{subfigure}

\begin{subfigure}[b]{0.75\textwidth}
    \centering
    \vspace{4mm}
    \begin{tikzpicture}[scale=0.275]
\filldraw[black](-4,0) rectangle (-3,1);
\draw[black] (-5,0) rectangle (-4,1);
\filldraw[black] (-6,0) rectangle (-5,1);
\draw[black](-4,-1)rectangle (-3,0);
\filldraw[black] (-5,-1) rectangle (-4,0);
\filldraw[black](-4,-2) rectangle (-3,-1);
\node (a) at (-1.5,-2) [label=\large{+}]{};
\draw[red] (0,1) rectangle (1,2);
\draw[black] (0,0) rectangle (1,1);
\filldraw[red] (1,0) rectangle (2,1);
\draw[black] (2,0) rectangle (3,1);
\filldraw[black] (3,0) rectangle (4,1);
\draw[black] (4,0) rectangle (5,1);
\filldraw[black] (5,0) rectangle (6,1);
\draw[black](6,0)rectangle(7,1);
\filldraw[black](7,0)rectangle(8,1);
\draw[black](8,0)rectangle(9,1);
\filldraw[black](9,0)rectangle(10,1);
\filldraw[black](0,-1) rectangle (1,0);
\draw[black](1,-1) rectangle (2,0);
\filldraw[red] (2,-1) rectangle (3,0);
\draw[black] (3,-1) rectangle (4,0);
\filldraw[black] (4,-1) rectangle (5,0);
\draw[black] (5,-1) rectangle (6,0);
\filldraw[black] (6,-1) rectangle (7,0);
\draw[black](7,-1) rectangle (8,0);
\draw[black] (0,-2) rectangle (1,-1);
\filldraw[black] (1,-2) rectangle (2,-1);
\draw[black] (2,-2) rectangle (3,-1);
\filldraw[red] (3,-2) rectangle (4,-1);
\draw[black] (4,-2) rectangle (5,-1);
\filldraw[black] (5,-2) rectangle (6,-1);
\draw[black] (6,-2) rectangle (7,-1);
\filldraw[black] (7,-2) rectangle (8,-1);
\filldraw[black](0,-3) rectangle (1,-2);
\draw[black] (1,-3) rectangle (2,-2);
\draw[black] (0,-4) rectangle (1,-3);
\filldraw[black] (1,-4) rectangle (2,-3);
\node (a) at (7,-6) [label=\small{$\begin{pmatrix}
-&4&2&2\\
2&2&0&0
\end{pmatrix}$}]{};
\node (a) at (13,-2) [label=\large{$\overset{\phi_{1,15}^o}{\longrightarrow}$}]{};
\filldraw[red](18,0)rectangle(19,1);
\draw[black](19,0)rectangle(20,1);
\filldraw[black](20,0)rectangle(21,1);
\draw[black] (21,0) rectangle (22,1);
\filldraw[black] (22,0) rectangle (23,1);
\draw[black] (23,0) rectangle (24,1);
\filldraw[black] (24,0) rectangle (25,1);
\draw[black] (25,0) rectangle (26,1);
\filldraw[black] (26,0) rectangle (27,1);
\draw[black](27,0)rectangle(28,1);
\filldraw[black](28,0)rectangle(29,1);
\draw[black](29,0)rectangle(30,1);
\filldraw[black](30,0)rectangle(31,1);
\filldraw[red](19,-1)rectangle(20,0);
\draw[black](20,-1)rectangle(21,0);
\filldraw[black](21,-1) rectangle (22,0);
\draw[black](22,-1) rectangle (23,0);
\filldraw[black] (23,-1) rectangle (24,0);
\draw[black] (24,-1) rectangle (25,0);
\filldraw[black] (25,-1) rectangle (26,0);
\draw[black] (26,-1) rectangle (27,0);
\filldraw[black] (27,-1) rectangle (28,0);
\draw[black](28,-1) rectangle (29,0);
\filldraw[red](20,-2)rectangle(21,-1);
\draw[black] (21,-2) rectangle (22,-1);
\filldraw[black] (22,-2) rectangle (23,-1);
\draw[black] (23,-2) rectangle (24,-1);
\filldraw[black] (24,-2) rectangle (25,-1);
\draw[black] (25,-2) rectangle (26,-1);
\filldraw[black] (26,-2) rectangle (27,-1);
\draw[black] (27,-2) rectangle (28,-1);
\filldraw[black] (28,-2) rectangle (29,-1);
\filldraw[red](21,-3) rectangle (22,-2);
\draw[black] (22,-3) rectangle (23,-2);
\draw[black] (21,-4) rectangle (22,-3);
\filldraw[red] (22,-4) rectangle (23,-3);
\node (a) at (29.5,-6) [label=\small{$\begin{pmatrix}
6&4&4&0&0\\
-&-&-&0&0
\end{pmatrix}$}]{};
    \end{tikzpicture}
    \caption{An element of $\text{GFP}_{D_2,D_2,-1}(15)$ and its image in $\text{GFP}_{D_2,D_2,3}(19)$ }
\end{subfigure}
    \caption{The bijections $\phi_{2,15}^e$ and $\phi_{1,15}^o$.}
    \label{bijections from 0 offset to 4}
\end{figure}

\section{Specialising to certain models} \label{specialising section}
In this section we specialise the work of Section \ref{general section} to the ASEP$(q,1)$, 2-exclusion and 3-state models. For each of these models the three variable identities in Theorem \ref{main} will specialise to give two variable identities related to other known Jacobi style identities of combinatorial significance.

\subsection{ASEP$(q,1)$}~
\par We will consider first the asymmetric simple exclusion process ASEP$(q,1)$ on $\Omega$, as described by Redig et al. in \cite{redig}. This has asymmetry parameter $0<q< 1$ and jump rates given by

$$p(\eta_i,\eta_{i+1})=q^{\eta_i-\eta_{i+1}-3}[\eta_i]_q[2-\eta_{i+1}]_q \hspace{5mm} \textrm{and} \hspace{5mm} q(\eta_i,\eta_{i+1})=q^{\eta_i-\eta_{i+1}+3}[2-\eta_i]_q[\eta_{i+1}]_q.$$ Here $[n]_q:=\frac{q^n-q^{-n}}{q-q^{-1}}$, a  $q$-deformation of $n$ (we have that $[n]_q\rightarrow n$ as $q\rightarrow 1$ ). 

The above rates might look strange at first sight but they are carefully constructed in order to equip the corresponding process with ``$\mathcal{U}_q(\mathfrak{gl}_2)$-symmetry" (here $\mathcal{U}_q(\mathfrak{gl}_2)$ is the quantum group associated to the Lie algebra $\mathfrak{gl}_2$). The original motivation for constructing processes with ``$\mathcal{U}_q(\mathfrak{g})$-symmetry" (for certain Lie algebras $\mathfrak{g}$) was the realisation that such processes have many explicit self-duality functions. The Markov generators for such processes are built by applying ground state transformations to quantum Hamiltonians associated to certain tensor products of representations of $\mathcal{U}_q(\mathfrak{g})$. Roughly speaking the dimension of the representation relates to the number of particles allowed at a site, the dimension of $\mathfrak{g}$ relates to the number of species of particle and tensors let us describe multi-site states. The Quantum Hamiltonian determines the jump rates. In the case of $\mathfrak{g}=\mathfrak{gl}_2$ the representations are the spin $2j$ representations for $j\in\frac{1}{2}\mathbb{Z}$ (up to twist) and so one gets a family ASEP$(q,j)$ of processes, each allowing up to $2j$ particles at a site. When $j=\frac{1}{2}$ we recover classical ASEP and when $j=1$ we get the process described above.

\par We show that ASEP$(q,1)$ is a member of the blocking family. Conditions $(B1)$ and $(B2)$ clearly hold and the jump rates satisfy: 
\begin{enumerate}[(a)]
    \item $p(y,z)>q(z,y)$ for all $y\in\{1,2\}$ and $z\in\{0,1\},$
    \item $\frac{p(1,0)}{q(0,1)}=q^{-4}=\frac{p(2,1)}{q(1,2)},$
    \item $\frac{p(1,0)p(2,1)q(1,1)q(0,2)}{q(0,1)q(1,2)p(2,0)p(1,1)}=q^{-8}\cdot\frac{q^4[2]_q^2}{ q^{-4}[2]_q^2}=1$.
\end{enumerate} 
\par \noindent Thus $(B3)$ is satisfied with the constants, 
  $$p_{\textrm{asym}}=\frac{q^{-2}[2]_q}{[2]_q\left(q^{-2}+q^2\right)}=\frac{q^{-4}}{1+q^{-4}} \hspace{5mm} \textrm{ and } \hspace{5mm} q_{\textrm{asym}}=\frac{q^2[2]_q}{[2]_q\left(q^{-2}+q^2\right)}=\frac{1}{1+q^{-4}}$$
\par \noindent and the functions
$$f(z)= \frac{[z]_q}{[3-z]_q} \hspace{10mm} s(y,z) = \frac{q^{y-z-2}(1+q^{-4})[3-y]_q[3-z]_q}{q^{-4}}.$$

\par In this case $\Tilde{q}=q^4,t=[2]_q$ and we have a one parameter family of product blocking measures of the form
$$\underline{\mu}^c(\underline{\eta})=\prod\limits_{i=-\infty}^0\frac{[2]_q^{\eta_i(2-\eta_i)}q^{-4\eta_i(i-c)}}{(1+[2]_q q^{-4(i-c)}+q^{-8(i-c)})}\prod\limits_{i=1}^\infty\frac{[2]_q^{\eta_i(2-\eta_i)}q^{4(2-\eta_i)(i-c)}}{(1+[2]_qq^{4(i-c)}+q^{8(i-c)})}.$$
\par Substituting the values of $\tilde{q}$ and $t$ into Theorem \ref{main} gives
\begin{align*}
  2\sum\limits_{\ell \in \mathbb{Z}} S_\textrm{even}(q^4,[2]_q)q^{4\ell(\ell+1)}z^{2\ell} &=
  \prod\limits_{i\geq1}(1+[2]_qq^{4i}z+q^{8i}z^2)(1+[2]_qq^{4(i-1)}z^{-1}+q^{8(i-1)}z^{-2})\\
  &+\prod\limits_{i\geq1}(1-[2]_qq^{4i}z+q^{8i}z^2)(1-[2]_qq^{4(i-1)}z^{-1}+q^{8(i-1)}z^{-2}))\\
  \\
    2[2]_q\sum\limits_{\ell\in \mathbb{Z}} S_\textrm{odd}(q^4,[2]_q)q^{4(\ell+1)^2}z^{2\ell +1}&=
 \prod\limits_{i\geq1}(1+[2]_qq^{4i}z+q^{8i}z^2)(1+[2]_qq^{4(i-1)}z^{-1}+q^{8(i-1)}z^{-2})\\
  &-\prod\limits_{i\geq1}(1-[2]_qq^{4i}z+q^{8i}z^2)(1-[2]_qq^{4(i-1)}z^{-1}+q^{8(i-1)}z^{-2})).
\end{align*}

Since $[2]_q = q+q^{-1}$ the quadratic terms on the RHS all factor and the products collapse as follows:

\begin{align*}
  &\prod\limits_{i\geq 1}(1+q^{4i-1}z)(1+q^{4i+1}z)(1+q^{4i-5}z^{-1})(1+q^{4i-3}z^{-1})\\&\pm\prod\limits_{i\geq 1}(1-q^{4i-1}z)(1-q^{4i+1}z)(1-q^{4i-5}z^{-1})(1-q^{4i-3}z^{-1})\\
  &=\frac{1+q^{-1}z^{-1}}{1+qz}\prod\limits_{i\geq1}(1+q^{2i-1}z)(1+q^{2i-1}z^{-1})\pm \frac{1-q^{-1}z^{-1}}{1-qz}\prod\limits_{i\geq 1}(1-q^{2i-1}z)(1-q^{2i-1}z^{-1})\\
  &=\frac{1}{qz}\left(\prod\limits_{i\geq1}(1+q^{2i-1}z)(1+q^{2i-1}z^{-1}) \mp \prod\limits_{i\geq1}(1-q^{2i-1}z)(1-q^{2i-1}z^{-1}) \right).
\end{align*}

We have therefore proved the following identities.
 
\begin{thm}\label{specialisation1}
\vspace{2mm}
\begin{align*}
   2\sum\limits_{\ell \in \mathbb{Z}}S_{\text{even}}(q^4,[2]_q)q^{(2\ell+1)^2}z^{2\ell+1}&=\prod\limits_{i\geq1}(1+q^{2i-1}z)(1+q^{2i-1}z^{-1})-\prod\limits_{i\geq1}(1-q^{2i-1}z)(1-q^{2i-1}z^{-1}) \\
    \\
      2(1+q^2)\sum\limits_{\ell \in \mathbb{Z}}S_{\text{odd}}(q^4,[2]_q)q^{(2\ell)^2}z^{2\ell}&=\prod\limits_{i\geq1}(1+q^{2i-1}z)(1+q^{2i-1}z^{-1})+\prod\limits_{i\geq1}(1-q^{2i-1}z)(1-q^{2i-1}z^{-1}).
\end{align*}
\end{thm}

The RHS of these identities should look familiar. The first term is part of the product side of the Jacobi triple product, as seen in the introduction (an alternative form is given in Section $3.3$). Indeed these two identities are isolating the odd/even terms respectively (as mentioned in Section $3.3$, this is the only reason for the sign changes on the RHS).

We can now proceed in two ways. If we assume the Jacobi triple product then Theorem \ref{specialisation1} proves the following closed forms for the specialised normalising factors:
\vspace{5mm} \begin{align*}S_{\text{even}}(q^4,[2]_q) &= \frac{1}{\prod_{j\geq 1}(1-q^{2j})} \quad(= f_{D_1,D_1,0}(q^2))\\
\\S_{\text{odd}}(q^4,[2]_q) &= \frac{1}{(1+q^2)\prod_{j\geq 1}(1-q^{2j})} \quad\left(= \frac{1}{(1+q^2)}f_{D_1,D_1,0}(q^2)\right)\end{align*}

\newpage On the other hand, if we were able to find probabilistic explanations for these two equalities then this would give a new probabilistic proof of the Jacobi triple product. However the authors were unable to find such an explanation. To elaborate further, ASEP$(q,1)$ is a particular $0$-$1$-$2$ system with blocking measure. However the above equalities suggest that this blocking measure (and the one for its stood up process) should be intimately related to that of the $0$-$1$ system ASEP$(q,\frac{1}{2})$ (i.e.\ classical ASEP) and its stood up process AZRP (which together prove the Jacobi triple product). Given the strange nature of the rates of ASEP$(q,1)$ it is not clear, at least to the authors, why this relationship should be expected.

We can give explicit combinatorial explanations for these equalities. Let us look at the even case first. We have already seen that $S_{\text{even}}(\Tilde{q},t) = f_{D_2,D_2,0}(\Tilde{q},t)$. Note that setting $\Tilde{q}=q^4$ and $t=q+q^{-1}$ sends $\Tilde{q}^nt^m$ to $q^{4n}(q+q^{-1})^m = q^{4n+m}+\binom{m}{1}q^{4n+(m-2)}+...+\binom{m}{m-1}q^{4n-(m-2)}+q^{4n-m}$ and so in order to prove the equality it suffices to show how elements of $\text{GFP}_{D_2,D_2,0,m}(n)$ can be used to construct $\binom{m}{i}$ unique ordinary partitions of $4n-(m-2i)$ into even parts for each $0\leq i\leq m$. Equivalently we must show how to obtain $\binom{m}{i}$ unique ordinary partitions of $2n-\frac{m-2i}{2}$. In order to do this we take the generalised Young diagram corresponding to such a GFP and for each square on the diagonal we colour in all intermediate white squares to the right and below. This almost creates a valid Young diagram but repeats in the rows of the GFP cause a problem, since the first entry in a repeat gives a row/column that is too short. To fix this we add an extra black square to the end of such rows/columns, creating a valid Young diagram for a partition of $2n-\frac{m}{2}$. The other required partitions are obtained by adding more dots to this Young diagram. The only rows/columns that we are guaranteed to add dots to and still get a valid Young diagram are those corresponding to distinct entries in the rows of the GFP. For each of the $\binom{m}{i}$ choices of $i$ such entries we can add a dot onto the corresponding rows/columns, giving unique partitions of $2n-\frac{m-2i}{2}$ (for each $0\leq i\leq m$). See Figure \ref{asep gfps even} for an example of the above construction (the colour grey is used to emphasise the squares that have been coloured black).

\begin{figure}[H]
    \centering
\begin{subfigure}[b]{0.75\textwidth}
    \centering
\begin{tikzpicture}[scale=0.35]
\filldraw[red] (0,0) rectangle (1,1);
\draw[black] (1,0) rectangle (2,1);
\filldraw[black] (2,0) rectangle (3,1);
\draw[black] (3,0) rectangle (4,1);
\filldraw [black] (4,0) rectangle (5,1);
\draw[black] (5,0) rectangle (6,1);
\filldraw[black] (6,0) rectangle (7,1);
\draw[black](0,-1) rectangle (1,0);
\filldraw[red] (1,-1) rectangle (2,0);
\draw[black] (2,-1) rectangle (3,0);
\filldraw[black] (3,-1) rectangle (4,0);
\draw[black](4,-1) rectangle (5,0);
\filldraw[black] (5,-1) rectangle (6,0);
\draw[black] (6,-1) rectangle (7,0);
\filldraw[black] (0,-2) rectangle (1,-1);
\draw[black] (1,-2) rectangle (2,-1);
\filldraw[red] (2,-2) rectangle (3,-1);
\draw[black] (3,-2) rectangle (4,-1);
\filldraw[black] (4,-2) rectangle (5,-1);
\draw[black] (5,-2) rectangle (6,-1);
\filldraw[black] (6,-2) rectangle (7,-1);
\draw[black] (0,-3) rectangle (1,-2);
\filldraw[black](1,-3) rectangle (2,-2);
\filldraw[black] (0,-4) rectangle (1,-3);
\draw[black] (1,-4) rectangle (2,-3);
\draw[black] (0,-5) rectangle (1,-4);
\filldraw[black] (1,-5) rectangle (2,-4);
\node(a) at (5,-6) [label=$\begin{pmatrix}
3&2&2\\
2&2&0
\end{pmatrix}$]{};
\node(a) at (9,-3) [label=$\longrightarrow$]{};

\filldraw[red] (13,0) rectangle (14,1);
\filldraw[lightgray] (14,0) rectangle (15,1);
\filldraw[black] (15,0) rectangle (16,1);
\filldraw[lightgray] (16,0) rectangle (17,1);
\filldraw [black] (17,0) rectangle (18,1);
\filldraw[lightgray] (18,0) rectangle (19,1);
\filldraw[black] (19,0) rectangle (20,1);
\filldraw[lightgray](13,-1) rectangle (14,0);
\filldraw[red] (14,-1) rectangle (15,0);
\filldraw[lightgray] (15,-1) rectangle (16,0);
\filldraw[black] (16,-1) rectangle (17,0);
\filldraw[lightgray](17,-1) rectangle (18,0);
\filldraw[black] (18,-1) rectangle (19,0);
\filldraw[lightgray] (19,-1) rectangle (20,0);
\filldraw[black] (13,-2) rectangle (14,-1);
\filldraw[lightgray] (14,-2) rectangle (15,-1);
\filldraw[red] (15,-2) rectangle (16,-1);
\filldraw[lightgray] (16,-2) rectangle (17,-1);
\filldraw[black] (17,-2) rectangle (18,-1);
\filldraw[lightgray] (18,-2) rectangle (19,-1);
\filldraw[black] (19,-2) rectangle (20,-1);
\filldraw[lightgray] (13,-3) rectangle (14,-2);
\filldraw[black](14,-3) rectangle (15,-2);
\filldraw[black] (13,-4) rectangle (14,-3);
\filldraw[lightgray] (14,-4) rectangle (15,-3);
\filldraw[lightgray](13,-5) rectangle (14,-4);
\filldraw[black] (14,-5) rectangle (15,-4);
\node(a) at (23,-3) [label=$\longrightarrow$]{};

\filldraw[red] (27,0.5) circle (5pt);
\filldraw[black] (28,0.5) circle (5pt);
\filldraw[black] (29,0.5) circle (5pt);
\filldraw[black] (30,0.5) circle (5pt);
\filldraw[black] (31,0.5) circle (5pt);
\filldraw[black] (32,0.5) circle (5pt);
\filldraw[black] (33,0.5) circle (5pt);
\filldraw[black] (27,-0.5) circle (5pt);
\filldraw[red] (28,-0.5) circle (5pt);
\filldraw[black] (29,-0.5) circle (5pt);
\filldraw[black] (30,-0.5) circle (5pt);
\filldraw[black] (31,-0.5) circle (5pt);
\filldraw[black] (32,-0.5) circle (5pt);
\filldraw[black] (33,-0.5) circle (5pt);
\filldraw[black] (27,-1.5) circle (5pt);
\filldraw[black] (28,-1.5) circle (5pt);
\filldraw[red] (29,-1.5) circle (5pt);
\filldraw[black] (30,-1.5) circle (5pt);
\filldraw[black] (31,-1.5) circle (5pt);
\filldraw[black] (32,-1.5) circle (5pt);
\filldraw[black] (33,-1.5) circle (5pt);
\filldraw[black] (27,-2.5) circle (5pt);
\filldraw[black] (28,-2.5) circle (5pt);
\filldraw[black] (27,-3.5) circle (5pt);
\filldraw[black] (28,-3.5) circle (5pt);
\filldraw[black] (27,-4.5) circle (5pt);
\filldraw[black] (28,-4.5) circle (5pt);
\end{tikzpicture}
\caption{An element of $\text{GFP}_{D_2,D_2,0,2}(14)$ and the associated partition of $2(14)-\frac{2}{2}=27$.}
\end{subfigure}
\vspace{5mm}
    
\begin{subfigure}[b]{0.75\textwidth}
    \centering
    \begin{tikzpicture}[scale=0.35]
    \filldraw[red] (0,0) circle (5pt);
    \filldraw[black] (1,0) circle (5pt);
    \filldraw[black] (2,0) circle (5pt);
    \filldraw[black] (3,0) circle (5pt);
    \filldraw[black] (4,0) circle (5pt);
    \filldraw[black] (5,0) circle (5pt);
    \filldraw[black] (6,0) circle (5pt);
    \filldraw[green] (7,0) circle (5pt);
    \filldraw[black] (0,-1) circle (5pt);
    \filldraw[red] (1,-1) circle (5pt);
    \filldraw[black] (2,-1) circle (5pt);
    \filldraw[black] (3,-1) circle (5pt);
    \filldraw[black] (4,-1) circle (5pt);
    \filldraw[black] (5,-1) circle (5pt);
    \filldraw[black] (6,-1) circle (5pt);
    \filldraw[black] (0,-2) circle (5pt);
    \filldraw[black] (1,-2) circle (5pt);
    \filldraw[red] (2,-2) circle (5pt);
    \filldraw[black] (3,-2) circle (5pt);
    \filldraw[black] (4,-2) circle (5pt);
    \filldraw[black] (5,-2) circle (5pt);
    \filldraw[black] (6,-2) circle (5pt);
     \filldraw[black] (0,-3) circle (5pt);
    \filldraw[black] (1,-3) circle (5pt);
     \filldraw[black] (0,-4) circle (5pt);
    \filldraw[black] (1,-4) circle (5pt);
     \filldraw[black] (0,-5) circle (5pt);
    \filldraw[black] (1,-5) circle (5pt);
    \node(a) at (4,-5) [label=$\{a_1\}$]{};
    
        \filldraw[red] (13,0) circle (5pt);
    \filldraw[black] (14,0) circle (5pt);
    \filldraw[black] (15,0) circle (5pt);
    \filldraw[black] (16,0) circle (5pt);
    \filldraw[black] (17,0) circle (5pt);
    \filldraw[black] (18,0) circle (5pt);
    \filldraw[black] (19,0) circle (5pt);
    \filldraw[green] (15,-3) circle (5pt);
    \filldraw[black] (13,-1) circle (5pt);
    \filldraw[red] (14,-1) circle (5pt);
    \filldraw[black] (15,-1) circle (5pt);
    \filldraw[black] (16,-1) circle (5pt);
    \filldraw[black] (17,-1) circle (5pt);
    \filldraw[black] (18,-1) circle (5pt);
    \filldraw[black] (19,-1) circle (5pt);
    \filldraw[black] (13,-2) circle (5pt);
    \filldraw[black] (14,-2) circle (5pt);
    \filldraw[red] (15,-2) circle (5pt);
    \filldraw[black] (16,-2) circle (5pt);
    \filldraw[black] (17,-2) circle (5pt);
    \filldraw[black] (18,-2) circle (5pt);
    \filldraw[black] (19,-2) circle (5pt);
     \filldraw[black] (13,-3) circle (5pt);
    \filldraw[black] (14,-3) circle (5pt);
     \filldraw[black] (13,-4) circle (5pt);
    \filldraw[black] (14,-4) circle (5pt);
     \filldraw[black] (13,-5) circle (5pt);
    \filldraw[black] (14,-5) circle (5pt);
    \node(a) at (17,-5) [label=$\{b_3\}$]{};
    
        \filldraw[red] (26,0) circle (5pt);
    \filldraw[black] (27,0) circle (5pt);
    \filldraw[black] (28,0) circle (5pt);
    \filldraw[black] (29,0) circle (5pt);
    \filldraw[black] (30,0) circle (5pt);
    \filldraw[black] (31,0) circle (5pt);
    \filldraw[black] (32,0) circle (5pt);
    \filldraw[green] (33,0) circle (5pt);
    \filldraw[black] (26,-1) circle (5pt);
    \filldraw[red] (27,-1) circle (5pt);
    \filldraw[black] (28,-1) circle (5pt);
    \filldraw[black] (29,-1) circle (5pt);
    \filldraw[black] (30,-1) circle (5pt);
    \filldraw[black] (31,-1) circle (5pt);
    \filldraw[black] (32,-1) circle (5pt);
    \filldraw[black] (26,-2) circle (5pt);
    \filldraw[black] (27,-2) circle (5pt);
    \filldraw[red] (28,-2) circle (5pt);
    \filldraw[black] (29,-2) circle (5pt);
    \filldraw[black] (30,-2) circle (5pt);
    \filldraw[black] (31,-2) circle (5pt);
    \filldraw[black] (32,-2) circle (5pt);
     \filldraw[black] (26,-3) circle (5pt);
    \filldraw[black] (27,-3) circle (5pt);
    \filldraw[green](28,-3) circle (5pt);
     \filldraw[black] (26,-4) circle (5pt);
    \filldraw[black] (27,-4) circle (5pt);
     \filldraw[black] (26,-5) circle (5pt);
    \filldraw[black] (27,-5) circle (5pt);
    \node(a) at (30,-5) [label=$\{a_1 \textrm{,} b_3\}$]{};
    \end{tikzpicture}
    \vspace{1mm}
    \caption{Partitions of 28 and 29 coming from subsets of distinct entries of the GFP.}
    \end{subfigure}
    \caption{Partitions of 27, 28 and 29 coming from an element of $\text{GFP}_{D_2,D_2,0,2}(14)$.}
    \label{asep gfps even}
\end{figure}
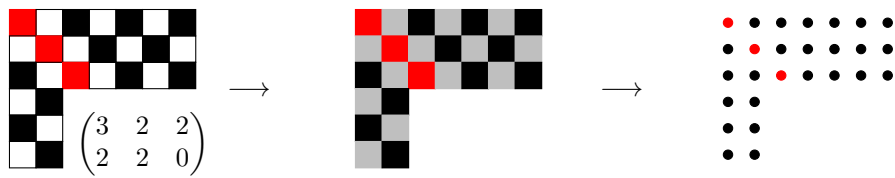
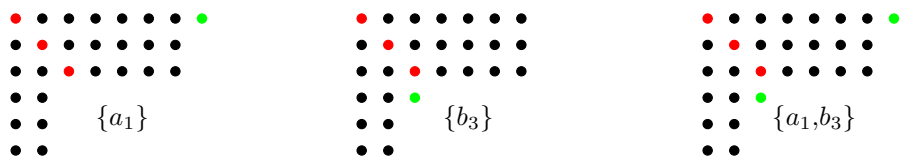
The odd offset case is similar. We have already seen that $tS_{\text{odd}}(\Tilde{q},t)=f_{D_2,D_2,-1}(\Tilde{q},t)$. Setting $\Tilde{q}=q^4$ and $t=q+q^{-1}$ sends $\Tilde{q}^nt^m$ to $q^{4n+m}+\binom{m}{1}q^{4n+(m-2)}+...+\binom{m}{m-1}q^{4n-(m-2)}+q^{4n-m}$ as before. Multiplying by an extra $q$ (to make the factor $(1+q^2)$ on the LHS) gives $q^{4n+(m+1)}+\binom{m}{1}q^{4n+(m-1)}+...+\binom{m}{m-1}q^{4n-(m-3)}+q^{4n-(m-1)}$. In order to prove the equality it suffices to show how elements of $\text{GFP}_{D_2,D_2,0,-1}(n)$ can be used to construct $\binom{m}{i}$ unique ordinary partitions of $4n-(m-1-2i)$ into even parts for each $0\leq i\leq m$. Equivalently we must show how to obtain $\binom{m}{i}$ unique ordinary partitions of $2n-\frac{m-1-2i}{2}$. We do this in the same way as the even case, but being careful to also colour the intermediate squares in the first column and add a dot to the first column of the Young diagram if $b_1$ is a distinct entry in row $2$ of the GFP (the first column doesn't correspond to a diagonal square). See Figure \ref{asep gfps odd} for an example.

\newpage \begin{figure}[H]
    \centering
\begin{subfigure}[b]{0.75\textwidth}
    \centering
\begin{tikzpicture}[scale=0.35]
\draw[red] (0,1) rectangle (1,2);
\draw[black] (0,0) rectangle (1,1);
\filldraw[red] (1,0) rectangle (2,1);
\draw[black] (2,0) rectangle (3,1);
\filldraw[black] (3,0) rectangle (4,1);
\draw[black] (4,0) rectangle (5,1);
\filldraw[black] (5,0) rectangle (6,1);
\draw[black](6,0)rectangle(7,1);
\filldraw[black](7,0)rectangle(8,1);
\draw[black](8,0)rectangle(9,1);
\filldraw[black](9,0)rectangle(10,1);
\filldraw[black](0,-1) rectangle (1,0);
\draw[black](1,-1) rectangle (2,0);
\filldraw[red] (2,-1) rectangle (3,0);
\draw[black] (3,-1) rectangle (4,0);
\filldraw[black] (4,-1) rectangle (5,0);
\draw[black] (5,-1) rectangle (6,0);
\filldraw[black] (6,-1) rectangle (7,0);
\draw[black](7,-1) rectangle (8,0);
\draw[black] (0,-2) rectangle (1,-1);
\filldraw[black] (1,-2) rectangle (2,-1);
\draw[black] (2,-2) rectangle (3,-1);
\filldraw[red] (3,-2) rectangle (4,-1);
\draw[black] (4,-2) rectangle (5,-1);
\filldraw[black] (5,-2) rectangle (6,-1);
\draw[black] (6,-2) rectangle (7,-1);
\filldraw[black] (7,-2) rectangle (8,-1);
\filldraw[black](0,-3) rectangle (1,-2);
\draw[black] (1,-3) rectangle (2,-2);
\draw[black] (0,-4) rectangle (1,-3);
\filldraw[black] (1,-4) rectangle (2,-3);
\filldraw[black](0,-5) rectangle (1,-4);
\draw[black](1,-5) rectangle (2,-4);
\node (a) at (6,-5.5) [label=\small{$\begin{pmatrix}
-&4&2&2\\
3&2&0&0
\end{pmatrix}$}]{};
\node(a) at (11,-3) [label=$\longrightarrow$]{};

\draw[red] (13,1) rectangle (14,2);
\filldraw[lightgray] (13,0) rectangle (14,1);
\filldraw[red] (14,0) rectangle (15,1);
\filldraw[lightgray] (15,0) rectangle (16,1);
\filldraw[black] (16,0) rectangle (17,1);
\filldraw[lightgray] (17,0) rectangle (18,1);
\filldraw[black] (18,0) rectangle (19,1);
\filldraw[lightgray](19,0)rectangle(20,1);
\filldraw[black](20,0)rectangle(21,1);
\filldraw[lightgray](21,0)rectangle(22,1);
\filldraw[black](22,0)rectangle(23,1);
\filldraw[black](13,-1) rectangle (14,0);
\filldraw[lightgray](14,-1) rectangle (15,0);
\filldraw[red] (15,-1) rectangle (16,0);
\filldraw[lightgray] (16,-1) rectangle (17,0);
\filldraw[black] (17,-1) rectangle (18,0);
\filldraw[lightgray] (18,-1) rectangle (19,0);
\filldraw[black] (19,-1) rectangle (20,0);
\filldraw[lightgray](20,-1) rectangle (21,0);
\filldraw[lightgray] (13,-2) rectangle (14,-1);
\filldraw[black] (14,-2) rectangle (15,-1);
\filldraw[lightgray] (15,-2) rectangle (16,-1);
\filldraw[red] (16,-2) rectangle (17,-1);
\filldraw[lightgray] (17,-2) rectangle (18,-1);
\filldraw[black] (18,-2) rectangle (19,-1);
\filldraw[lightgray] (19,-2) rectangle (20,-1);
\filldraw[black] (20,-2) rectangle (21,-1);
\filldraw[black](13,-3) rectangle (14,-2);
\filldraw[lightgray] (14,-3) rectangle (15,-2);
\filldraw[lightgray] (13,-4) rectangle (14,-3);
\filldraw[black] (14,-4) rectangle (15,-3);
\filldraw[black](13,-5) rectangle (14,-4);

\node(a) at (25,-3) [label=$\longrightarrow$]{};

\filldraw[black] (27,0.5) circle (5pt);
\filldraw[red] (28,0.5) circle (5pt);
\filldraw[black] (29,0.5) circle (5pt);
\filldraw[black] (30,0.5) circle (5pt);
\filldraw[black] (31,0.5) circle (5pt);
\filldraw[black] (32,0.5) circle (5pt);
\filldraw[black] (33,0.5) circle (5pt);
\filldraw[black] (34,0.5) circle (5pt);
\filldraw[black] (35,0.5) circle (5pt);
\filldraw[black] (36,0.5) circle (5pt);
\filldraw[black] (27,-0.5) circle (5pt);
\filldraw[black] (28,-0.5) circle (5pt);
\filldraw[red] (29,-0.5) circle (5pt);
\filldraw[black] (30,-0.5) circle (5pt);
\filldraw[black] (31,-0.5) circle (5pt);
\filldraw[black] (32,-0.5) circle (5pt);
\filldraw[black] (33,-0.5) circle (5pt);
\filldraw[black] (34,-0.5) circle (5pt);
\filldraw[black] (27,-1.5) circle (5pt);
\filldraw[black] (28,-1.5) circle (5pt);
\filldraw[black] (29,-1.5) circle (5pt);
\filldraw[red] (30,-1.5) circle (5pt);
\filldraw[black] (31,-1.5) circle (5pt);
\filldraw[black] (32,-1.5) circle (5pt);
\filldraw[black] (33,-1.5) circle (5pt);
\filldraw[black] (34,-1.5) circle (5pt);
\filldraw[black] (27,-2.5) circle (5pt);
\filldraw[black] (28,-2.5) circle (5pt);
\filldraw[black] (27,-3.5) circle (5pt);
\filldraw[black] (28,-3.5) circle (5pt);
\filldraw[black] (27,-4.5) circle (5pt);
\end{tikzpicture}
\caption{An element of $\text{GFP}_{D_2,D_2,-1,3}(16)$ and the associated partition of $2(16)-\frac{3-1}{2}=31$.}
\end{subfigure}
\vspace{5mm}
    
\begin{subfigure}[b]{0.75\textwidth}
    \centering
    \begin{tikzpicture}[scale=0.275]
    \filldraw[black] (-7,0) circle (5pt);
    \filldraw[red] (-6,0) circle (5pt);
    \filldraw[black] (-5,0) circle (5pt);
    \filldraw[black] (-4,0) circle (5pt);
    \filldraw[black] (-3,0) circle (5pt);
    \filldraw[black] (-2,0) circle (5pt);
    \filldraw[black] (-1,0) circle (5pt);
    \filldraw[black] (0,0) circle (5pt);
    \filldraw[black] (1,0) circle (5pt);
    \filldraw[black] (2,0) circle (5pt);
    \filldraw[green] (3,0) circle (5pt);
    \filldraw[black] (-7,-1) circle (5pt);
    \filldraw[black] (-6,-1) circle (5pt);
    \filldraw[red] (-5,-1) circle (5pt);
    \filldraw[black] (-4,-1) circle (5pt);
    \filldraw[black] (-3,-1) circle (5pt);
    \filldraw[black] (-2,-1) circle (5pt);
    \filldraw[black] (-1,-1) circle (5pt);
    \filldraw[black] (0,-1) circle (5pt);
    \filldraw[black] (-7,-2) circle (5pt);
    \filldraw[black] (-6,-2) circle (5pt);
    \filldraw[black] (-5,-2) circle (5pt);
    \filldraw[red] (-4,-2) circle (5pt);
    \filldraw[black] (-3,-2) circle (5pt);
    \filldraw[black] (-2,-2) circle (5pt);
    \filldraw[black] (-1,-2) circle (5pt);
    \filldraw[black] (0,-2) circle (5pt);
     \filldraw[black] (-7,-3) circle (5pt);
    \filldraw[black] (-6,-3) circle (5pt);
     \filldraw[black] (-7,-4) circle (5pt);
    \filldraw[black] (-6,-4) circle (5pt);
     \filldraw[black] (-7,-5) circle (5pt);
    \node(a) at (-3,-6) [label=$\{a_1\}$]{};
    
    \filldraw[black] (6,0) circle (5pt);
    \filldraw[red] (7,0) circle (5pt);
    \filldraw[black] (8,0) circle (5pt);
    \filldraw[black] (9,0) circle (5pt);
    \filldraw[black] (10,0) circle (5pt);
    \filldraw[black] (11,0) circle (5pt);
    \filldraw[black] (12,0) circle (5pt);
    \filldraw[black] (13,0) circle (5pt);
    \filldraw[black] (14,0) circle (5pt);
    \filldraw[black] (15,0) circle (5pt);
    \filldraw[black] (6,-1) circle (5pt);
    \filldraw[black] (7,-1) circle (5pt);
    \filldraw[red] (8,-1) circle (5pt);
    \filldraw[black] (9,-1) circle (5pt);
    \filldraw[black] (10,-1) circle (5pt);
    \filldraw[black] (11,-1) circle (5pt);
    \filldraw[black] (12,-1) circle (5pt);
    \filldraw[black] (13,-1) circle (5pt);
    \filldraw[black] (6,-2) circle (5pt);
    \filldraw[black] (7,-2) circle (5pt);
    \filldraw[black] (8,-2) circle (5pt);
    \filldraw[red] (9,-2) circle (5pt);
    \filldraw[black] (10,-2) circle (5pt);
    \filldraw[black] (11,-2) circle (5pt);
    \filldraw[black] (12,-2) circle (5pt);
    \filldraw[black] (13,-2) circle (5pt);
     \filldraw[black] (6,-3) circle (5pt);
    \filldraw[black] (7,-3) circle (5pt);
     \filldraw[black] (6,-4) circle (5pt);
    \filldraw[black] (7,-4) circle (5pt);
     \filldraw[black] (6,-5) circle (5pt);
     \filldraw[green] (6,-6) circle (5pt);
    \node(a) at (10,-6) [label=$\{b_1\}$]{};
    
    \filldraw[black] (18,0) circle (5pt);
    \filldraw[red] (19,0) circle (5pt);
    \filldraw[black] (20,0) circle (5pt);
    \filldraw[black] (21,0) circle (5pt);
    \filldraw[black] (22,0) circle (5pt);
    \filldraw[black] (23,0) circle (5pt);
    \filldraw[black] (24,0) circle (5pt);
    \filldraw[black] (25,0) circle (5pt);
    \filldraw[black] (26,0) circle (5pt);
    \filldraw[black] (27,0) circle (5pt);
    \filldraw[black] (18,-1) circle (5pt);
    \filldraw[black] (19,-1) circle (5pt);
    \filldraw[red] (20,-1) circle (5pt);
    \filldraw[black] (21,-1) circle (5pt);
    \filldraw[black] (22,-1) circle (5pt);
    \filldraw[black] (23,-1) circle (5pt);
    \filldraw[black] (24,-1) circle (5pt);
    \filldraw[black] (25,-1) circle (5pt);
    \filldraw[black] (18,-2) circle (5pt);
    \filldraw[black] (19,-2) circle (5pt);
    \filldraw[black] (20,-2) circle (5pt);
    \filldraw[red] (21,-2) circle (5pt);
    \filldraw[black] (22,-2) circle (5pt);
    \filldraw[black] (23,-2) circle (5pt);
    \filldraw[black] (24,-2) circle (5pt);
    \filldraw[black] (25,-2) circle (5pt);
     \filldraw[black] (18,-3) circle (5pt);
    \filldraw[black] (19,-3) circle (5pt);
     \filldraw[black] (18,-4) circle (5pt);
    \filldraw[black] (19,-4) circle (5pt);
     \filldraw[black] (18,-5) circle (5pt);
     \filldraw[green] (19,-5) circle (5pt);
    \node(a) at (22,-6) [label=$\{b_2\}$]{};
    
    \filldraw[black] (31,0) circle (5pt);
    \filldraw[red] (32,0) circle (5pt);
    \filldraw[black] (33,0) circle (5pt);
    \filldraw[black] (34,0) circle (5pt);
    \filldraw[black] (35,0) circle (5pt);
    \filldraw[black] (36,0) circle (5pt);
    \filldraw[black] (37,0) circle (5pt);
    \filldraw[black] (38,0) circle (5pt);
    \filldraw[black] (39,0) circle (5pt);
    \filldraw[black] (40,0) circle (5pt);
    \filldraw[green] (41,0) circle (5pt);
    \filldraw[black] (31,-1) circle (5pt);
    \filldraw[black] (32,-1) circle (5pt);
    \filldraw[red] (33,-1) circle (5pt);
    \filldraw[black] (34,-1) circle (5pt);
    \filldraw[black] (35,-1) circle (5pt);
    \filldraw[black] (36,-1) circle (5pt);
    \filldraw[black] (37,-1) circle (5pt);
    \filldraw[black] (38,-1) circle (5pt);
    \filldraw[black] (31,-2) circle (5pt);
    \filldraw[black] (32,-2) circle (5pt);
    \filldraw[black] (33,-2) circle (5pt);
    \filldraw[red] (34,-2) circle (5pt);
    \filldraw[black] (35,-2) circle (5pt);
    \filldraw[black] (36,-2) circle (5pt);
    \filldraw[black] (37,-2) circle (5pt);
    \filldraw[black] (38,-2) circle (5pt);
     \filldraw[black] (31,-3) circle (5pt);
    \filldraw[black] (32,-3) circle (5pt);
     \filldraw[black] (31,-4) circle (5pt);
    \filldraw[black] (32,-4) circle (5pt);
     \filldraw[black] (31,-5) circle (5pt);
     \filldraw[green] (31,-6) circle (5pt);
    \node(a) at (35,-6) [label=$\{a_1 \textrm{,} b_1\}$]{};
    
    \filldraw[black] (0,-10) circle (5pt);
    \filldraw[red] (1,-10) circle (5pt);
    \filldraw[black] (2,-10) circle (5pt);
    \filldraw[black] (3,-10) circle (5pt);
    \filldraw[black] (4,-10) circle (5pt);
    \filldraw[black] (5,-10) circle (5pt);
    \filldraw[black] (6,-10) circle (5pt);
    \filldraw[black] (7,-10) circle (5pt);
    \filldraw[black] (8,-10) circle (5pt);
    \filldraw[black] (9,-10) circle (5pt);
    \filldraw[green] (10,-10) circle (5pt);
    \filldraw[black] (0,-11) circle (5pt);
    \filldraw[black] (1,-11) circle (5pt);
    \filldraw[red] (2,-11) circle (5pt);
    \filldraw[black] (3,-11) circle (5pt);
    \filldraw[black] (4,-11) circle (5pt);
    \filldraw[black] (5,-11) circle (5pt);
    \filldraw[black] (6,-11) circle (5pt);
    \filldraw[black] (7,-11) circle (5pt);
    \filldraw[black] (0,-12) circle (5pt);
    \filldraw[black] (1,-12) circle (5pt);
    \filldraw[black] (2,-12) circle (5pt);
    \filldraw[red] (3,-12) circle (5pt);
    \filldraw[black] (4,-12) circle (5pt);
    \filldraw[black] (5,-12) circle (5pt);
    \filldraw[black] (6,-12) circle (5pt);
    \filldraw[black] (7,-12) circle (5pt);
     \filldraw[black] (0,-13) circle (5pt);
    \filldraw[black] (1,-13) circle (5pt);
     \filldraw[black] (0,-14) circle (5pt);
    \filldraw[black] (1,-14) circle (5pt);
     \filldraw[black] (0,-15) circle (5pt);
     \filldraw[green](1,-15) circle (5pt);
    \node(a) at (4,-16) [label=$\{a_1 \textrm{,} b_2\}$]{};
    
    \filldraw[black] (13,-10) circle (5pt);
    \filldraw[red] (14,-10) circle (5pt);
    \filldraw[black] (15,-10) circle (5pt);
    \filldraw[black] (16,-10) circle (5pt);
    \filldraw[black] (17,-10) circle (5pt);
    \filldraw[black] (18,-10) circle (5pt);
    \filldraw[black] (19,-10) circle (5pt);
    \filldraw[black] (20,-10) circle (5pt);
    \filldraw[black] (21,-10) circle (5pt);
    \filldraw[black] (22,-10) circle (5pt);
    \filldraw[black] (13,-11) circle (5pt);
    \filldraw[black] (14,-11) circle (5pt);
    \filldraw[red] (15,-11) circle (5pt);
    \filldraw[black] (16,-11) circle (5pt);
    \filldraw[black] (17,-11) circle (5pt);
    \filldraw[black] (18,-11) circle (5pt);
    \filldraw[black] (19,-11) circle (5pt);
    \filldraw[black] (20,-11) circle (5pt);
    \filldraw[black] (13,-12) circle (5pt);
    \filldraw[black] (14,-12) circle (5pt);
    \filldraw[black] (15,-12) circle (5pt);
    \filldraw[red] (16,-12) circle (5pt);
    \filldraw[black] (17,-12) circle (5pt);
    \filldraw[black] (18,-12) circle (5pt);
    \filldraw[black] (19,-12) circle (5pt);
    \filldraw[black] (20,-12) circle (5pt);
     \filldraw[black] (13,-13) circle (5pt);
    \filldraw[black] (14,-13) circle (5pt);
     \filldraw[black] (13,-14) circle (5pt);
    \filldraw[black] (14,-14) circle (5pt);
     \filldraw[black] (13,-15) circle (5pt);
     \filldraw[green] (14,-15) circle (5pt);
     \filldraw[green] (13,-16) circle (5pt);
    \node(a) at (17,-16) [label=$\{b_1 \textrm{,} b_2\}$]{};
    
    \filldraw[black] (26,-10) circle (5pt);
    \filldraw[red] (27,-10) circle (5pt);
    \filldraw[black] (28,-10) circle (5pt);
    \filldraw[black] (29,-10) circle (5pt);
    \filldraw[black] (30,-10) circle (5pt);
    \filldraw[black] (31,-10) circle (5pt);
    \filldraw[black] (32,-10) circle (5pt);
    \filldraw[black] (33,-10) circle (5pt);
    \filldraw[black] (34,-10) circle (5pt);
    \filldraw[black] (35,-10) circle (5pt);
    \filldraw[green] (36,-10) circle (5pt);
    \filldraw[black] (26,-11) circle (5pt);
    \filldraw[black] (27,-11) circle (5pt);
    \filldraw[red] (28,-11) circle (5pt);
    \filldraw[black] (29,-11) circle (5pt);
    \filldraw[black] (30,-11) circle (5pt);
    \filldraw[black] (31,-11) circle (5pt);
    \filldraw[black] (32,-11) circle (5pt);
    \filldraw[black] (33,-11) circle (5pt);
    \filldraw[black] (26,-12) circle (5pt);
    \filldraw[black] (27,-12) circle (5pt);
    \filldraw[black] (28,-12) circle (5pt);
    \filldraw[red] (29,-12) circle (5pt);
    \filldraw[black] (30,-12) circle (5pt);
    \filldraw[black] (31,-12) circle (5pt);
    \filldraw[black] (32,-12) circle (5pt);
    \filldraw[black] (33,-12) circle (5pt);
     \filldraw[black] (26,-13) circle (5pt);
    \filldraw[black] (27,-13) circle (5pt);
     \filldraw[black] (26,-14) circle (5pt);
    \filldraw[black] (27,-14) circle (5pt);
     \filldraw[black] (26,-15) circle (5pt);
     \filldraw[green] (27,-15) circle (5pt);
     \filldraw[green] (26,-16) circle (5pt);
    \node(a) at (31,-16) [label=$\{a_1 \textrm{,} b_1 \textrm{,} b_2\}$]{};
    
    \end{tikzpicture}
    \vspace{1mm}
    \caption{Partitions of 32, 33 and 34 coming from subsets of distinct entries of the GFP.}
    \end{subfigure}
    \caption{Partitions of 31, 32, 33 and 34 coming from an element of $\text{GFP}_{D_2,D_2,-1,3}(16)$.}
    \label{asep gfps odd}
\end{figure}
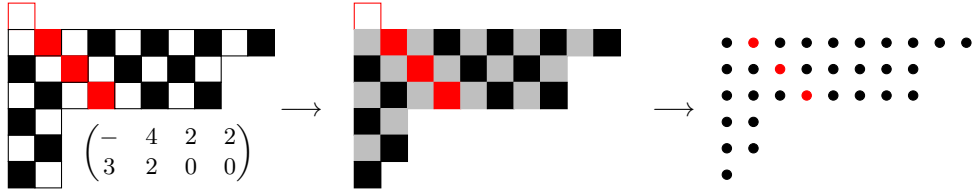
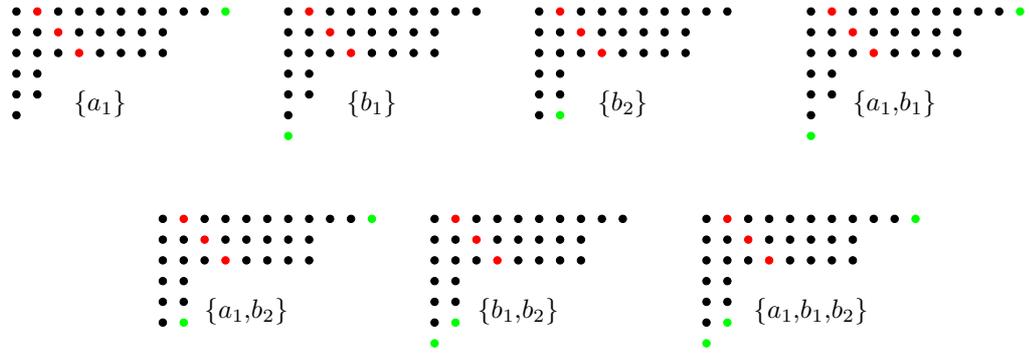
 
\subsection{Asymmetric particle-antiparticle exclusion process (3-state model)}~
\par We now consider the asymmetric particle-antiparticle process with $I=\{-1,0,1\}, \Lambda=\mathbb{Z}$, asymmetry parameter $0<q<1$, annihilation parameters $a,a'>0$, creation parameters $\gamma,\gamma' >0$ and jump rates as follows: 
\begin{itemize}
    \item Particles ($\eta_i=1$) jump left with rate $q$ or right with rate $1$ to empty sites,
    \item Antiparticles ($\eta_i=-1$) jump left with rate $1$ or right with rate $q$ to empty sites,
    \item A particle-antiparticle neighbouring pair can be annihilated with rate $a$ if $\eta_i=1$ and $\eta_{i+1}=-1$ or rate $a'q$ if $\eta_i=-1$ and $\eta_{i+1}=1$,
    \item A particle-antiparticle pair can be created at two neighbouring empty sites with rate $\gamma $ if $\eta_i^{(i,i+1)}=-1$ and $\eta_{i+1}^{(i,i+1)}=1$ or rate $\gamma'q$ if $\eta_i^{(i,i+1)}=1$ and $\eta_{i+1}^{(i,i+1)}=-1$.
\end{itemize}
\begin{remark}
This process has been studied in the literature (see \cite{3 state} for example) and it is known that the system can only have an i.i.d stationary distribution when the annihilation rate is twice the jump rate of a particle, i.e.\ when $a=a'=2$. We will consider this case from now on.
\end{remark}
\par \noindent We can view the $3$-state model as a process on $\Omega$, with annihilation being a jump from a two particle site to an empty site and creation a jump from a one particle site to another one particle site. When thought of in this way we have the following jump rates:

\begin{table}[H]
\centering
\begin{tabular}{cc}
\begin{tabular}{|c||c|c|c|}
         \hline
         & $\eta_{i+1}=0$ & $\eta_{i+1}=1$ & $\eta_{i+1}=2$\\
         \hline \hline
        $\eta_i=0$ & $0$ & $0$ &$0$ \\
        \hline
        $\eta_i=1$ & $1$ & $\gamma $ & $0$\\
        \hline
        $\eta_i=2$ & $2$ & $1$ & $0$ \\
        \hline
    \end{tabular}&
    \begin{tabular}{|c||c|c|c|}
         \hline
         & $\eta_{i+1}=0$ & $\eta_{i+1}=1$ & $\eta_{i+1}=2$\\
         \hline \hline
        $\eta_i=0$ & $0$ & $q$ &$2q$ \\
        \hline
        $\eta_i=1$ & $0$ & $\gamma'q$ & $q$\\
        \hline
        $\eta_i=2$ & $0$ & $0$ & $0$ \\
        \hline
    \end{tabular}\\
\end{tabular}
\caption{The jump rates $p(\eta_i,\eta_{i+1})$ and $q(\eta_i,\eta_{i+1})$ respectively of the 3- state model.}\label{3 state}
\end{table}
In order to be a member of the blocking family we see that $(B1)$ gives no constraints, $(B2)$ is satisfied if and only if $0< \gamma,\gamma' \leq 1$ and the following must be satisfied:
 
\begin{enumerate}[(a)]
\item $p(y,z)>q(z,y)$ for all $y\in\{1,2\}$ and $z\in\{0,1\}$,
\item $\frac{p(1,0)}{q(0,1)}=q^{-1}=\frac{p(2,1)}{q(1,2)}$,
\item $1=\frac{p(1,0)p(2,1)q(1,1)q(0,2)}{q(0,1)q(1,2)p(2,0)p(1,1)}=\frac{\gamma'}{\gamma},$ i.e.\ $0<\gamma = \gamma'\leq 1$.
\end{enumerate}
Under this assumption $(B3)$ is then satisfied with the constants,
   $$ p_\textrm{asym}=\frac{1}{1+q}\hspace{5mm} \textrm{and} \hspace{5mm}
       q_\textrm{asym}=\frac{q}{1+q}$$
\par \noindent and the functions,
$$f(z):=
    \begin{cases}
     0 \hspace{2mm} &\textrm{if } z=0  \\
     \\
    \sqrt{\frac{\gamma}{2}} &\textrm{if } z=1  \\
      \\
    \sqrt{\frac{2}{\gamma}} &\textrm{if } z=2
    \end{cases} 
    \qquad
    s(y,z):=
    \begin{cases}
     \sqrt{\frac{2}{\gamma}}(1+q) &\textrm{if } y=z=1  \\
       \\
      \sqrt{2\gamma}(1+q) &\textrm{if } y=1 \textrm{ and } z=2 \\
       \\
       \sqrt{2\gamma}(1+q)  &\textrm{if } y=2 \textrm{ and } z=1 \\
       \\
       \sqrt{\frac{\gamma}{2}}(1+q) &\textrm{if } y=z=2 \\
       \\
       0 &\textrm{if } y=3 \textrm{ or } z=3.
    \end{cases}$$

\par \noindent In this case $\Tilde{q}=q$ and $t=\sqrt{\frac{2}{\gamma}}$ and we have a one-parameter family of blocking measures of the form
$$\underline{\mu}^c(\underline{\eta})=\prod\limits_{i=-\infty}^0 \frac{\left(\frac{2}{\gamma}\right)^{\frac{1}{2}\mathbb{I}\{\eta_i=1\}}q^{-(i-c)\eta_i}}{1+\sqrt{\frac{2}{\gamma}}q^{-(i-c)}+q^{-2(i-c)}}\prod\limits_{i=1}^\infty \frac{\left(\frac{2}{\gamma}\right)^{\frac{1}{2}\mathbb{I}\{\eta_i=1\}}q^{(2-\eta_i)(i-c)}}{1+\sqrt{\frac{2}{\gamma}}q^{(i-c)}+q^{2(i-c)}}.$$

Since $\gamma$ can be almost freely chosen (there is only the constraint that $0< \gamma\leq 1$) applying Theorem \ref{main} to this subfamily will not tell us anything new (we will get the same three variable identity but with domain restricted to $t\geq \sqrt{2}$). However certain members of this subfamily can give interesting two variable identities. As an example we will choose the creation parameter to be $\gamma=\frac{1}{2}$ (the inverse of the annihilation parameter $a=2$). For this process we then have $\Tilde{q}=q$ and $t=2$.

Substituting these values of $\Tilde{q}$ and $t$ into Theorem \ref{main} gives

\begin{align*}
  2\sum\limits_{\ell \in \mathbb{Z}} S_\textrm{even}(q,2)q^{\ell(\ell+1)}z^{2\ell} &=
  \prod\limits_{i\geq1}(1+2q^{i}z+q^{2i}z^2)(1+2q^{i-1}z^{-1}+q^{2(i-1)}z^{-2})\\
  &+\prod\limits_{i\geq1}(1-2q^{i}z+q^{2i}z^2)(1-2q^{i-1}z^{-1}+q^{2(i-1)}z^{-2}))\\
  \\
    4\sum\limits_{\ell\in \mathbb{Z}} S_\textrm{odd}(q,2)q^{(\ell+1)^2}z^{2\ell +1}&=
 \prod\limits_{i\geq1}(1+2q^{i}z+q^{2i}z^2)(1+2q^{i-1}z^{-1}+q^{2(i-1)}z^{-2})\\
  &-\prod\limits_{i\geq1}(1-2q^{i}z+q^{2i}z^2)(1-2q^{i-1}z^{-1}+q^{2(i-1)}z^{-2})).
\end{align*}

Once again the quadratic terms on the RHS all factor and so we have proved the following identities

\begin{thm}\label{specialisation2}
\begin{align*}2\sum_{\ell\in\mathbb{Z}}S_{\text{even}}(q,2)q^{\ell(\ell+1)}z^{2\ell}&=\prod_{i\geq 1}((1+q^{i}z)(1+q^{i-1}z^{-1}))^2\\ &+\prod_{i\geq 1}((1-q^{i}z)(1-q^{i-1}z^{-1}))^2\\ 4\sum_{\ell\in\mathbb{Z}}S_{\text{odd}}(q,2)q^{(\ell+1)^2}z^{2\ell+1}&=\prod_{i\geq 1}((1+q^{i}z)(1+q^{i-1}z^{-1}))^2\\ &-\prod_{i\geq 1}((1-q^{i}z)(1-q^{i-1}z^{-1}))^2.\end{align*}\end{thm}

Again the first term on the RHS looks familiar and is part of the square of the product side of the Jacobi triple product. Indeed if we assume Jacobi triple product then we find that \begin{align*}\prod_{i\geq 1}((1+q^{i}z)(1+q^{i-1}z^{-1}))^2 &= \left(\sum_{k\in\mathbb{Z}}\frac{1}{\prod_{i\geq 1}(1-q^{i})}q^{\frac{k(k+1)}{2}}z^k\right)^2\\ &=\sum_{k\in\mathbb{Z}}\left(\frac{1}{\prod_{i\geq 1}(1-q^{i})^2}\sum_{k'\in\mathbb{Z}}q^{\frac{k'(k'+1)}{2} + \frac{(k-k')(k-k'+1)}{2}}\right)z^k.\end{align*}
The term corresponding to $k=2\ell$ can be written as (using Jacobi triple product in the second equality, written as in the introduction but setting $z=1$) \begin{align*}\left(\frac{1}{\prod_{i\geq 1}(1-q^{i})^2}\sum_{k'\in\mathbb{Z}}q^{\frac{k'(k'+1)}{2} + \frac{(2\ell-k')(2\ell-k'+1)}{2}}\right)z^{2\ell} &=\left(\frac{1}{\prod_{i\geq 1}(1-q^{i})^2}\sum_{k'\in\mathbb{Z}}q^{(k'-\ell)^2}\right)q^{\ell(\ell+1)}z^{2\ell}\\ &= \left(\prod_{i\geq 1}\frac{(1+q^{2i-1})^2(1-q^{2i})}{(1-q^{i})^2}\right)q^{\ell(\ell+1)}z^{2\ell}\\ &= \left(\prod_{i\geq 1}\frac{(1+q^{2i-1})^2(1+q^{i})}{(1-q^{i})}\right)q^{\ell(\ell+1)}z^{2\ell}.\end{align*}
Similarly the term corresponding to $k=2\ell+1$ can be written as (using Jacobi triple product in the second equality, written as in Section $3.3$ but setting $z=1$) \begin{align*}\left(\frac{1}{\prod_{i\geq 1}(1-q^{i})^2}\sum_{k'\in\mathbb{Z}}q^{\frac{k'(k'+1)}{2} + \frac{(2\ell-k'+1)(2\ell-k'+2)}{2}}\right)z^{2\ell+1} &=\left(\frac{1}{\prod_{i\geq 1}(1-q^{i})^2}\sum_{k'\in\mathbb{Z}}q^{(k'-\ell)(k'-\ell-1)}\right)q^{(\ell+1)^2}z^{2\ell+1}\\ &= 2\left(\prod_{i\geq 1}\frac{(1+q^{2i})^2(1-q^{2i})}{(1-q^{i})^2}\right)q^{(\ell+1)^2}z^{2\ell+1}\\ &= 2\left(\prod_{i\geq 1}\frac{(1+q^{2i})^2(1+q^{i})}{(1-q^{i})}\right)q^{(\ell+1)^2}z^{2\ell+1}.\end{align*}
Thus we see that, assuming Jacobi triple product, the identities are equivalent to the equalities \begin{align*}S_{\text{even}}(q,2) &= \prod_{i\geq 1}\frac{(1+q^{2i-1})^2(1+q^{i})}{(1-q^{i})}\\ S_{\text{odd}}(q,2) &= \prod_{i\geq 1}\frac{(1+q^{2i})^2(1+q^{i})}{(1-q^{i})}. \end{align*}

\par \noindent As in the previous section it would be interesting to find purely probabilistic proofs of these product forms.  The more complicated nature of the products suggests that this is non-trivial.

Alternatively one could have not assumed Jacobi triple product in the above, but instead assumed Theorem \ref{specialisation1}. The identities would then give (non-trivial) relations between the four specialised normalising factors $S_{\text{even}}(q,2), S_{\text{odd}}(q,2), S_{\text{even}}(q^4,[2]_q)$ and $S_{\text{odd}}(q^4,[2]_q)$ that we would have proved probabilistically without additional assumptions.

While the identities in Theorem \ref{specialisation2} might seem unnatural they do have a combinatorial interpretation in terms of $2$-coloured GFP's. The product $\prod_{i\geq 1}(1+q^i)^2$ is the generating function for coloured partitions of $n$ into red/blue parts where each red/blue part appears at most once (for example we allow $5={\color{red}2}+{\color{blue}2}+{\color{blue}1}$ but not $5={\color{red}2}+{\color{red}2}+{\color{blue}1}$). The pair of $2$-coloured partitions ${\color{red}2}+{\color{blue}2}+{\color{blue}1}$ and ${\color{blue}2}+{\color{red}2}+{\color{blue}1}$ are counted as the same in the above product and so to avoid overcounting we favour a particular colour when listing repeats. 

By the ``General Principle" of Andrews we have that: \[\prod_{i\geq 1}((1+q^{i}z)(1+q^{i-1}z^{-1}))^2 = \sum_{k\in\mathbb{Z}}f_{C_2,C_2,k}(q)z^k,\] where $f_{C_2,C_2,k}(q)$ is the generating function for the sets $\text{GFP}_{C_2,C_2,k}(n)$ of GFP's of $n$ with offset $k$ and each row being a $2$-coloured partition.

The content of the above identities is then that \[f_{C_2,C_2,k}(q) = \begin{cases}S_{\text{even}}(q,2)q^{\ell(\ell+1)} & \text{if } k=2\ell\\ 2S_{\text{odd}}(q,2)q^{(\ell+1)^2} &\text{if } k=2\ell+1.\end{cases}\] Thus the two specialised normalising factors satisfy $f_{C_2,C_2,-1}(q) = 2S_{\text{odd}}(q,2)$ and $f_{C_2,C_2,0}(q) = S_{\text{even}}(q,2)$ and so have a natural combinatorial interpretation. These two base cases are both explicitly clear since we know from Section $3.3$ that $S_{\text{even}}(\Tilde{q},t) = f_{D_2,D_2,0}(\Tilde{q},t)$ and $tS_{\text{odd}}(\Tilde{q},t) = f_{D_2,D_2,-1}(\Tilde{q},t)$, and each element of $\text{GFP}_{D_2,D_2,0,m}(n)$ and $\text{GFP}_{D_2,D_2,-1,m}(n)$ can be $2$-coloured in $2^m$ ways (exactly what is counted when setting $t=2$). The other offset cases can be proved in the usual fashion, by the equalities of generating functions: \[f_{C_2,C_2,k}(q) = \begin{cases}f_{C_2,C_2,0}(q)q^{\ell(\ell+1)} & \text{if } k=2\ell,\\ f_{C_2,C_2,-1}(q)q^{(\ell+1)^2} &\text{if } k=2\ell+1.\end{cases}\] (The proof follows from the maps $\phi^e_{\ell,n}, \phi^o_{\ell,n}$ of Section $3.3$, but by using coloured generalised Young diagrams).

The function $f_{C_2,C_2,0}(q)$ is the function $C\Phi_2(q)$ defined by Andrews on p.$7$ of \cite{frobenius}. Indeed the specialised normalising factor is: \[S_{\text{even}}(q,2) = 1+4q+9q^2+20q^3+42q^4+80q^5+147q^6+260q^7+445q^8+...\] which agrees with the expansion of $C\Phi_2(q)$ found on p.$8$ of the same book. In Corollary $5.2$ of this book Andrews uses Jacobi triple product to prove a product formula for $C\Phi_2(q)$, which is equivalent to the product formula we found above for $S_{\text{even}}(q,2)$. Andrews' book only considers GFP's with offset $0$ and so does not give a similar analysis of the function $f_{C_2,C_2,-1}(q)$. However, using MAGMA we were able to compute that \[f_{C_2,C_2,-1}(q) = 2+4q+12q^2+24q^3+50q^4+92q^5+172q^6+296q^7+510q^8+...\] which agrees with $2S_{\text{odd}}(q,2)$, as expected, and appears to have the same coefficients as twice OEIS sequence A$137829$ \cite{OEIS1}, implying the product form we derived above.

The function $C\Phi_2(q)$ is one of a family of functions $C\Phi_k(q)$, counting GFP's of offset $0$ with rows having a similar condition to the above but with $k$ colours. The whole family of functions is studied in Andrews' book. In general they are not given by products but can be shown to be sums of products. It would be interesting to know whether there exists a $k$-state model for each $k$, whose stood up process and stationary blocking measures provide normalising factors relating to these functions (and their corresponding non-zero offset analogues).

\subsection{Asymmetric 2-exclusion}~
\par Now we consider the asymmetric 2-exclusion process on $\Omega$ with asymmetry parameter $0<q<1$. The non-zero left/right jump rates are $$p(\eta_i,\eta_{i+1})=\mathbb{I}\{\eta_i\neq 0\}\mathbb{I}\{\eta_{i+1}\neq 2\} \hspace{5mm} \textrm{and} \hspace{5mm} q(\eta_i,\eta_{i+1})=q\mathbb{I}\{\eta_{i+1}\neq 0\}\mathbb{I}\{\eta_{i}\neq 2\}.$$

We show that this is a member of the blocking family. Conditions $(B1)$ and $(B2)$ clearly hold and the jump rates satisfy:
\begin{enumerate}[(a)]
\item $p(y,z)>q(z,y)$ for all $y\in\{1,2\}$ and $z\in\{0,1\},$
\item $\frac{p(1,0)}{q(0,1)}=q^{-1}=\frac{p(2,1)}{q(1,2)},$
\item $\frac{p(1,0)p(2,1)q(1,1)q(0,2)}{q(0,1)q(1,2)p(2,0)p(1,1)}=1.$
\end{enumerate}
\par \noindent Thus $(B3)$ is satisfied with the constants
$$p_\textrm{asym}=\frac{1}{1+q} \hspace{5mm} \textrm{and} \hspace{5mm} q_\textrm{asym}=\frac{q}{1+q},$$
\par \noindent and the functions,
$$f(z)=\mathbb{I}\{z \neq 0\} \hspace{10mm}  s(y,z)= \begin{cases}
       (1+q) &\textrm{for } y,z \in \{1,2\} \\
       0 &\textrm{if } y=3 \textrm{ or } z=3.
       \end{cases}$$
\par \noindent In this case $\Tilde{q}=q$, $t=1$ and we have a one-parameter family of product stationary blocking measures of the form
$$\underline{\mu}^c(\underline{\eta})=\prod\limits_{i \leq0}\frac{q^{-\eta_i(i-c)}}{(1+q^{-(i-c)}+q^{-2(i-c)})}\prod\limits_{i \geq 1}\frac{q^{(2-\eta_i)(i-c)}}{(1+q^{(i-c)}+q^{2(i-c)})}.$$
\par Substituting the values of $\Tilde{q}$ and $t$ into Theorem \ref{main} gives

\begin{thm}\label{2-exc explicit}
\begin{align*}
    2\sum\limits_{\ell \in \mathbb{Z}} S_\textrm{even}(q,1)q^{\ell(\ell+1)}z^{2\ell} &=
  \prod\limits_{i\geq1} \left(1+q^{i}z+q^{2i}z^2\right)\left(1+q^{i-1}z^{-1}+q^{2(i-1)}z^{-2}\right)\\ &+\prod\limits_{i\geq1} \left(1-q^{i}z+q^{2i}z^2\right)\left(1-q^{i-1}z^{-1}+q^{2(i-1)}z^{-2}\right)\\
\\
    2 \sum\limits_{\ell\in \mathbb{Z}} S_\textrm{odd}(q,1)q^{(\ell+1)^2}z^{2\ell+1}&= 
    \prod\limits_{i=1}^\infty \left(1+q^{i}z+q^{2i}z^2\right)\left(1+q^{i-1}z^{-1}+q^{2(i-1)}z^{-2}\right)\\&-\prod\limits_{i=1}^\infty \left(1-q^{i}z+q^{2i}z^2\right)\left(1-q^{i-1}z^{-1}+q^{2(i-1)}z^{-2}\right).
\end{align*}
\end{thm}

It is clear that this specialisation has a natural combinatorial meaning. Recall that in Section $3.3$ we used the ``General Principle" to expand \[\prod_{i\geq 1}(1+tz\Tilde{q}^i+z^2\Tilde{q}^{2i})(1+tz^{-1}\Tilde{q}^{i-1}+z^{-2}\Tilde{q}^{2(i-1)}) = \sum_{k\in\mathbb{Z}}f_{D_2,D_2,k}(\Tilde{q},t)z^k\] with $f_{D_2,D_2,k}(\Tilde{q},t)$ being the two variable generating function for the sets $\text{GFP}_{D_2,D_2,k,m}(n)$ defined earlier. Setting $t=1$ is naturally interpreted as not distinguishing GFP's by their number of distinct parts per row, i.e.\ $f_{D_2,D_2,k}(\Tilde{q},1) = f_{D_2,D_2,k}(\Tilde{q})$, the one variable generating function for the sets GFP$_{D_2,D_2,k}(n)$. The content of the above identities is then that 

\[f_{D_2,D_2,k}(q) = \begin{cases}S_{\text{even}}(q,1)q^{\ell(\ell+1)} & \text{if } k=2\ell\\ S_{\text{odd}}(q,1)q^{(\ell+1)^2} & \text{if } k=2\ell+1. \end{cases}\]

\par \noindent The maps $\psi^e_{n}, \psi^o_{n},\phi^e_{\ell,n}$ and $\phi^o_{\ell,n}$ of Section $3.3$ give explicit proofs for all of these equalities, as expected.

Let's consider the two base cases in more detail. The function $f_{D_2,D_2,0}(q)$ is the function $\Phi_2(q)$ defined by Andrews on p.$6$ of \cite{frobenius}. The specialised even normalising factor is: \[S_{\text{even}}(q,1) = 1+q+3q^2+5q^3+9q^4+14q^5+24q^6+35q^7+55q^8+...\] which agrees with the expansion of $\Phi_2(q)$ on p.$7$ of the same book, as expected. An interesting result in this book is Corollary $5.1$, which uses the Jacobi triple product and other results to prove that \[\Phi_2(q) = \frac{1}{\prod_{i\geq 1}(1-q^i)(1-q^{12i-10})(1-q^{12i-9})(1-q^{12i-3})(1-q^{12i-2})}.\] So the specialised even normalising factor can be expressed as the above product.

Andrews' book only considers GFP's with offset $0$ and so does not give a similar analysis of the function $f_{D_2,D_2,-1}(q)$. However, using MAGMA we were able to compute that \[f_{D_2,D_2,-1}(q) = 1+2q+3q^2+6q^3+10q^4+16q^5+26q^6+40q^7+60q^8+...\] which agrees with $S_\textrm{odd}(q,1)$ and appears to have the same coefficients as OEIS sequence A$201077$ \cite{OEIS2}, implying that \[S_\textrm{odd}(q,1)=f_{D_2,D_2,-1}(q) = \frac{1}{\prod_{i\geq 1}(1-q^{2i-1})^2(1-q^{12i-8})(1-q^{12i-6})(1-q^{12i-4})(1-q^{12i})}.\] 

This discussion raises two interesting questions. Firstly, is there a probabilistic explanation for the above product forms? These products are quite complicated and it is not clear a priori that we should expect such a factorisation. It could be possible that there is an alternative way to stand up the $2$-exclusion process, giving a more natural normalising factor. Secondly, in this case, the ASEP$(q,1)$ case and the $3$-state model case the specialised normalising factors appear to be products. Could it be that the unspecialised normalising factors $S_{\text{even}}(\Tilde{q},t)$ and $S_{\text{odd}}(\Tilde{q},t)$ are products?

 \section{Asymmetric $k$-Exclusion}\label{k-exc section}
 It is natural to ask whether we can generalise the results of this paper to higher order particle systems with blocking measure. Unfortunately, as remarked in Section \ref{stand up} the ``stood up" process in general is not guaranteed to have a product stationary blocking measure. However we will now see that the asymmetric $k$-exclusion processes for $k\geq 1$ are sufficiently well behaved and lead to an interesting family of combinatorial identities, generalising the ones found in the $2$-exclusion section.
 
 \par The asymmetric $k$-exclusion process is the particle system with $I=\{0,1,...,k\}, \Lambda=\mathbb{Z}$, asymmetry parameter $0<q<1$ and jump rates 
 $$p(\eta_i,\eta_{i+1})=\mathbb{I}\{\eta_i\neq 0\}\mathbb{I}\{\eta_{i+1}\neq k\} \hspace{5mm} \textrm{and} \hspace{5mm} q(\eta_i,\eta_{i+1})=q\mathbb{I}\{\eta_{i+1}\neq 0\}\mathbb{I}\{\eta_{i}\neq k\}.$$

 \par  We now show that this is member of the blocking family. Conditions $(B1)$ and $(B2)$ clearly hold so it suffices to check $(B3)$. We see that the jump rates are described by the constants
   $$   p_\textrm{asym}=\frac{1}{1+q}\hspace{10mm}
       q_\textrm{asym}=\frac{q}{1+q}$$
\par \noindent and functions
$$ f(z)=\mathbb{I}\{z\neq 0\} \hspace{5mm} \textrm{and} \hspace{5mm} s(y,z)= \begin{cases}
       (1+q) \hspace{5mm} \textrm{for } y,z \in \{1,2,..,k\} \\
       0 \hspace{17mm} \textrm{if } y=k+1 \textrm{ or } z=k+1.
       \end{cases} $$

\par \noindent Thus $k$-exclusion is a member of the blocking family and by Theorem \ref{stationary dist} has one parameter family of product stationary blocking measures
$$\underline{\mu}^c(\underline{\eta})=\prod\limits_{i=-\infty}^0\frac{q^{-(i-c)\eta_i}}{Z_i^c(q)}\prod\limits_{i=1}^\infty \frac{q^{(k-\eta_i)(i-c)}}{q^{k(i-c)}Z_i^c(q)}$$
\par \noindent where $Z_i^c(q)=\sum\limits_{y=0}^k q^{-(i-c)y}$ is the normalising factor.

By definition of the state space, every $\underline{\eta}\in\Omega$ has a left most particle and a right most hole, i.e.\ $\eta_i=0$ for small enough $i$ and $\eta_i=k$ for big enough $i$. By asymmetry it then follows that the ground states of $\Omega$ are all shifts of the following $k$ ground states $\underline{\eta}^{-m}$, for $m \in \{0,1,...,k-1\}$ with
$$ \eta_i^{-m}=
    \begin{cases}
      0 &\textrm{  if } i<0  \\
      m &\textrm{  if } i=0  \\
      k &\textrm{ if } i>0.
    \end{cases} $$
      \vspace{5mm}
\begin{figure}[H]
    \centering
           \hspace{5mm}
       \begin{subfigure}[b]{0.35\textwidth}
       \centering
    \begin{tikzpicture}[scale=0.5]
    \draw[thick, <->] (-6,-1)--(6,-1);
\foreach \x in {-5,-4,-3,-2,-1,0,1,2,3,4,5}
    \draw[thick, -](\x cm, -1.1)--(\x cm, -0.9) node[anchor=north]{$\x$};
\filldraw [black] (-5.5,-0.7) circle (1pt);
\filldraw [black] (-5.7,-0.7) circle (1pt);
\filldraw [black] (-5.9,-0.7) circle (1pt); 
\filldraw [black] (1,-0.5) circle (4pt);
\filldraw [black] (2,-0.5) circle (4pt);
\filldraw [black] (3,-0.5) circle (4pt);
\filldraw [black] (4,-0.5) circle (4pt);
\filldraw [black] (5,-0.5) circle (4pt);   
\filldraw [black] (1,0) circle (4pt);
\filldraw [black] (2,0) circle (4pt);
\filldraw [black] (3,0) circle (4pt);
\filldraw [black] (4,0) circle (4pt);
\filldraw [black] (5,0) circle (4pt);
\filldraw [black] (1,0.5) circle (4pt);
\filldraw [black] (2,0.5) circle (4pt);
\filldraw [black] (3,0.5) circle (4pt);
\filldraw [black] (4,0.5) circle (4pt);
\filldraw [black] (5,0.5) circle (4pt); 
\filldraw [black] (1,1) circle (4pt);
\filldraw [black] (2,1) circle (4pt);
\filldraw [black] (3,1) circle (4pt);
\filldraw [black] (4,1) circle (4pt);
\filldraw [black] (5,1) circle (4pt);
\filldraw [black] (5.5,-0.7) circle (1pt);
\filldraw [black] (5.7,-0.7) circle (1pt);
\filldraw [black] (5.9,-0.7) circle (1pt); 
    \end{tikzpicture}
    \caption{Ground state $\underline{\eta}^0$}
    \end{subfigure}
    \hspace{10mm}
       \begin{subfigure}[b]{0.35\textwidth}
       \centering
    \begin{tikzpicture}[scale=0.5]
    \draw[thick, <->] (-6,-1)--(6,-1);
\foreach \x in {-5,-4,-3,-2,-1,0,1,2,3,4,5}
    \draw[thick, -](\x cm, -1.1)--(\x cm, -0.9) node[anchor=north]{$\x$};
\filldraw [black] (-5.5,-0.7) circle (1pt);
\filldraw [black] (-5.7,-0.7) circle (1pt);
\filldraw [black] (-5.9,-0.7) circle (1pt); 
\filldraw [black] (0,-0.5) circle (4pt);
\filldraw [black] (1,-0.5) circle (4pt);
\filldraw [black] (2,-0.5) circle (4pt);
\filldraw [black] (3,-0.5) circle (4pt);
\filldraw [black] (4,-0.5) circle (4pt);
\filldraw [black] (5,-0.5) circle (4pt);   
\filldraw [black] (1,0) circle (4pt);
\filldraw [black] (2,0) circle (4pt);
\filldraw [black] (3,0) circle (4pt);
\filldraw [black] (4,0) circle (4pt);
\filldraw [black] (5,0) circle (4pt);
\filldraw [black] (1,0.5) circle (4pt);
\filldraw [black] (2,0.5) circle (4pt);
\filldraw [black] (3,0.5) circle (4pt);
\filldraw [black] (4,0.5) circle (4pt);
\filldraw [black] (5,0.5) circle (4pt); 
\filldraw [black] (1,1) circle (4pt);
\filldraw [black] (2,1) circle (4pt);
\filldraw [black] (3,1) circle (4pt);
\filldraw [black] (4,1) circle (4pt);
\filldraw [black] (5,1) circle (4pt);
\filldraw [black] (5.5,-0.7) circle (1pt);
\filldraw [black] (5.7,-0.7) circle (1pt);
\filldraw [black] (5.9,-0.7) circle (1pt); 
    \end{tikzpicture}
    \caption{Ground state $\underline{\eta}^{-1}$}
    \end{subfigure}
    \hspace{5mm}
        \par\bigskip
        \hspace{5mm}
            \begin{subfigure}[b]{0.35\textwidth}
       \centering
    \begin{tikzpicture}[scale=0.5]
    \draw[thick, <->] (-6,-1)--(6,-1);
\foreach \x in {-5,-4,-3,-2,-1,0,1,2,3,4,5}
    \draw[thick, -](\x cm, -1.1)--(\x cm, -0.9) node[anchor=north]{$\x$};
\filldraw [black] (-5.5,-0.7) circle (1pt);
\filldraw [black] (-5.7,-0.7) circle (1pt);
\filldraw [black] (-5.9,-0.7) circle (1pt); 
\filldraw [black] (0,-0.5) circle (4pt);
\filldraw [black] (1,-0.5) circle (4pt);
\filldraw [black] (2,-0.5) circle (4pt);
\filldraw [black] (3,-0.5) circle (4pt);
\filldraw [black] (4,-0.5) circle (4pt);
\filldraw [black] (5,-0.5) circle (4pt);  
\filldraw [black] (0,0) circle (4pt);
\filldraw [black] (1,0) circle (4pt);
\filldraw [black] (2,0) circle (4pt);
\filldraw [black] (3,0) circle (4pt);
\filldraw [black] (4,0) circle (4pt);
\filldraw [black] (5,0) circle (4pt);
\filldraw [black] (1,0.5) circle (4pt);
\filldraw [black] (2,0.5) circle (4pt);
\filldraw [black] (3,0.5) circle (4pt);
\filldraw [black] (4,0.5) circle (4pt);
\filldraw [black] (5,0.5) circle (4pt); 
\filldraw [black] (1,1) circle (4pt);
\filldraw [black] (2,1) circle (4pt);
\filldraw [black] (3,1) circle (4pt);
\filldraw [black] (4,1) circle (4pt);
\filldraw [black] (5,1) circle (4pt);
\filldraw [black] (5.5,-0.7) circle (1pt);
\filldraw [black] (5.7,-0.7) circle (1pt);
\filldraw [black] (5.9,-0.7) circle (1pt); 
    \end{tikzpicture}
    \caption{Ground state $\underline{\eta}^{-2}$}
    \end{subfigure}
    \hspace{10mm}
           \begin{subfigure}[b]{0.35\textwidth}
       \centering
    \begin{tikzpicture}[scale=0.5]
    \draw[thick, <->] (-6,-1)--(6,-1);
\foreach \x in {-5,-4,-3,-2,-1,0,1,2,3,4,5}
    \draw[thick, -](\x cm, -1.1)--(\x cm, -0.9) node[anchor=north]{$\x$};
\filldraw [black] (-5.5,-0.7) circle (1pt);
\filldraw [black] (-5.7,-0.7) circle (1pt);
\filldraw [black] (-5.9,-0.7) circle (1pt); 
\filldraw [black] (0,-0.5) circle (4pt);
\filldraw [black] (1,-0.5) circle (4pt);
\filldraw [black] (2,-0.5) circle (4pt);
\filldraw [black] (3,-0.5) circle (4pt);
\filldraw [black] (4,-0.5) circle (4pt);
\filldraw [black] (5,-0.5) circle (4pt); 
\filldraw [black] (0,0) circle (4pt);
\filldraw [black] (1,0) circle (4pt);
\filldraw [black] (2,0) circle (4pt);
\filldraw [black] (3,0) circle (4pt);
\filldraw [black] (4,0) circle (4pt);
\filldraw [black] (5,0) circle (4pt);
\filldraw [black] (0,0.5) circle (4pt);
\filldraw [black] (1,0.5) circle (4pt);
\filldraw [black] (2,0.5) circle (4pt);
\filldraw [black] (3,0.5) circle (4pt);
\filldraw [black] (4,0.5) circle (4pt);
\filldraw [black] (5,0.5) circle (4pt); 
\filldraw [black] (1,1) circle (4pt);
\filldraw [black] (2,1) circle (4pt);
\filldraw [black] (3,1) circle (4pt);
\filldraw [black] (4,1) circle (4pt);
\filldraw [black] (5,1) circle (4pt);
\filldraw [black] (5.5,-0.7) circle (1pt);
\filldraw [black] (5.7,-0.7) circle (1pt);
\filldraw [black] (5.9,-0.7) circle (1pt); 
    \end{tikzpicture}
    \caption{Ground state $\underline{\eta}^{-3}$}
    \end{subfigure}
\hspace{5mm}
    \caption{The four ground states of 4-exclusion.}
    \label{k-exclus ground states}
\end{figure}
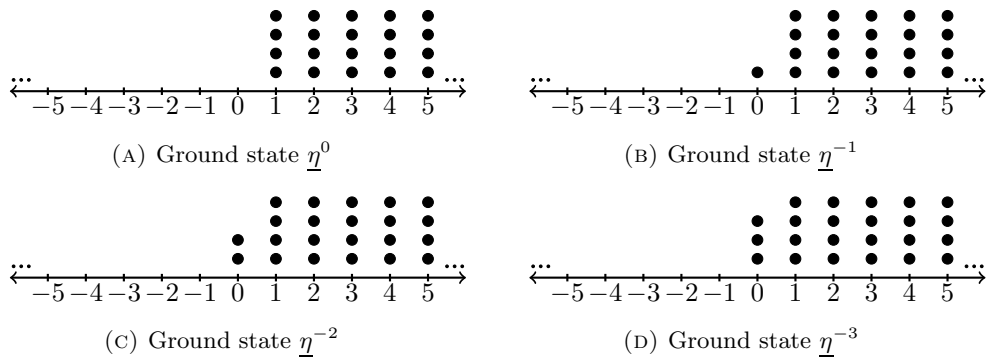
\subsection{Ergodic decomposition of $\Omega$} ~ 
\par The quantity $N(\underline{\eta}):=\sum\limits_{i=1}^\infty (k-\eta_i)-\sum\limits_{i=-\infty}^0 \eta_i$ is finite and is conserved by the dynamics of the system. Thus just as in Section \ref{decomp} we can decompose $\Omega=\bigcup_{n \in \mathbb{Z}}\Omega^n$, into irreducible components $\Omega^n:=\{\underline{\eta} \in \Omega: N(\underline{\eta})=n\}$. Note now that the left shift operator $\tau$ gives a bijection $\Omega^n \xrightarrow[]{\tau} \Omega^{n-k}$ (i.e.\ if $\underline{\eta} \in \Omega^n$ then, $N(\tau \underline{\eta})=n-k$).
\begin{remark}
Since $N(\underline{\eta}^{-m})=-m$ for each $m \in \{0,1,..,k-1\}$ the shifts of $\underline{\eta}^{-m}$ have conserved quantity in $k\mathbb{Z}-m $ and give the ground states for the $(-m \bmod k)$ part $\bigcup_{n\in k\mathbb{Z}-m}\Omega^n$ of $\Omega$. 
\end{remark}
\par We now calculate  $\underline{\nu}^{n,c}(\cdot):=\underline{\mu}^c(\cdot|N(\cdot)=n)$, the unique stationary distribution on $\Omega^n$.
\begin{lem}\label{mu tau k-exc}
The following relation holds, 
$$\underline{\mu}^c(\tau \underline{\eta})=q^{kc-N(\underline{\eta})}\underline{\mu}^c( \underline{\eta}).$$
\par \noindent This gives the recursion,
$$\underline{\mu}^c(\{N=n\})=q^{n-kc}\underline{\mu}^c(\{N=n-k\}).$$
\end{lem}
\par \noindent The proof is similar to that of Lemma \ref{mu tau} (both claims are special cases of Lemma 6.1 and Corollary 6.2 in \cite{blocking}).

The general solution of this recursion is
$$\underline{\mu}^c(\{N=n\})= q^{\frac{(n+m)(n+k-m)}{2k}-(n+m)c}\underline{\mu}^c(\{N=-m\}) \quad\textrm{ if } n \in k\mathbb{Z}-m \textrm{ with } m \in \{0,1,..,k-1\}.$$
\par \vspace{1mm}\noindent Since there is a dependence on the class of $n$ modulo $k$ we will need to calculate the probabilities $\underline{\mu}^c(\{N(\underline{\eta}) \equiv -m \bmod k\})$ for $m \in \{0,1,...,k-1\}$ in order to finish our calculation of $\underline{\nu}^{n,c}$.
\begin{lem}\label{N mod k}
$$\underline{\mu}^c(\{N\equiv -m \bmod k\})=\frac{1}{k}\left(1+\sum_{r=1}^{k-1}\zeta_k^{-rm}\left(\prod_{i=-\infty}^{\infty}\left(1+\sum_{j=1}^{k-1}(\zeta_{k}^{rj}-1)\mu_i^c(j)\right)\right)\right),$$
\par \noindent where $\zeta_k$ is a primitive $k^\textrm{th}$ root of unity.
\end{lem}
\begin{proof}~

\par \noindent Define the partial conserved quantity
$$N_a(\underline{\eta})=\sum\limits_{i=1}^a(k-\eta_i)-\sum\limits_{i=-a}^0\eta_i \hspace{10mm} \text{for } a\geq 1$$

\par \noindent and note that $N_a(\underline{\eta})\rightarrow N(\underline{\eta})$ as $a \rightarrow \infty$. 
\par \vspace{2mm} \noindent Consider the character group
$$\widehat{\mathbb{Z}/k\mathbb{Z}}:=\textrm{Hom}\big(\mathbb{Z}/k\mathbb{Z}, \mathbb{C}^\times\big)=\{\chi_0,\chi_1,..,\chi_{k-1}\}$$
\par \noindent where $\chi_i(1)=\zeta_k^i$ for $i \in \{0,1,..,k-1\}$ and $\zeta_k$ a primitive $k$-th root of unity. Note that $\chi_i= \chi_1^i$.
\par \vspace{2mm} \noindent For each $a\geq 1$ we define the random variables
$$Y_a:=\chi_1(-N_a(\underline{\eta}))=\chi_1\big(\sum\limits_{i=-a}^a \eta_i\big)=\prod\limits_{i=-a}^a\chi_1(\eta_i).$$
\par \noindent Since $\eta_i  \rightarrow 0$ as $i\rightarrow -\infty$ and $\eta_i \rightarrow k$ as $i\rightarrow \infty$ we have that $Y_a\rightarrow Y$ as $a \rightarrow \infty$, where 
$$Y:=\chi_1(-N(\underline{\eta}))=\chi_1(\sum\limits_{i=-\infty}^\infty \eta_i)=\prod\limits_{i=-\infty}^\infty \chi_1(\eta_i).$$
\par \noindent We now compute the moments $\mathbb{E}^c[Y^r]$ for $r \in \{0,1,..,k-1\}$ (the expectation being with respect to $\underline{\mu}^c$).  Since $|Y_a|=1$ for all $a\geq 1$, dominated convergence applies and by the product structure of $\underline{\mu}^c$ we have that
$$\mathbb{E}^c[Y^r]=\lim\limits_{a\rightarrow \infty}\mathbb{E}^c[Y_a^r]=\prod\limits_{i=-\infty}^\infty\Big(\sum\limits_{j=0}^k\chi_1(j)^r\mu_i^c(j)\Big)=\prod\limits_{i=-\infty}^\infty\Big(\sum\limits_{j=0}^k\chi_r(j)\mu_i^c(j)\Big).$$
\par \noindent On the other hand we have that
\begin{align*}
    \mathbb{E}^c[Y^r]&=\sum\limits_{m=0}^{k-1}\chi_1(m)^r  \underline{\mu}^c\big(-N(\underline{\eta}) \equiv m \bmod k\big)\\
    &=\sum\limits_{m=0}^{k-1}\chi_r(m)\underline{\mu}^c\big(N(\underline{\eta}) \equiv -m \bmod k\big).
\end{align*}
\par \noindent Hence we have the linear system of equations 
$$\prod\limits_{i=-\infty}^\infty\Big(\sum\limits_{j=0}^k\chi_r(j)\mu_i^c(j)\Big)=\sum\limits_{m=0}^{k-1}\chi_r(m)\underline{\mu}^c\big(N(\underline{\eta}) \equiv -m \bmod k\big).$$
\par \noindent By orthogonality of characters we get 
\begin{align*}
\underline{\mu}^c(N(\underline{\eta})\equiv - m \bmod k)&=\frac{1}{k}\sum\limits_{r=0}^{k-1}\overline{\chi_r(m)}\Bigg(\prod\limits_{i=-\infty}^\infty\Big(\sum\limits_{j=0}^k\chi_r(j)\mu_i^c(j)\Big)\Bigg)\\
&=\frac{1}{k}\sum\limits_{r=0}^{k-1}\chi_r(-m)\Bigg(\prod\limits_{i=-\infty}^\infty\Big(\sum\limits_{j=0}^k\chi_r(j)\mu_i^c(j)\Big)\Bigg)\\ &= \frac{1}{k}\sum_{r=0}^{k-1}\zeta_k^{-rm}\left(\prod_{i=-\infty}^{\infty}\left(\sum_{j=0}^{k}\zeta_k^{rj}\mu_i^c(j)\right)\right)\\ &=\frac{1}{k}\left(1+\sum_{r=1}^{k-1}\zeta_k^{-rm}\left(\prod_{i=-\infty}^{\infty}\left(1+\sum_{j=1}^{k-1}(\zeta_{k}^{rj}-1)\mu_i^c(j)\right)\right)\right).
\end{align*}
\end{proof}
\begin{remark}
A priori this appears to be complex-valued. However it is in fact real-valued since it is fixed by complex conjugation. When $k=2$ this is clearly real valued and agrees with Lemma \ref{prob odd} (recalling that $t=1$ for $2$-exclusion).
\end{remark}

\par If we combine Lemma \ref{mu tau k-exc} and Lemma \ref{N mod k} we find the form of the unique stationary distribution $\underline{\nu}^{n,c}$ on $\Omega^n$.
\begin{prop}\label{nu n k-exc}
For $n \in k\mathbb{Z}-m$ with $m \in \{0,1,...,k-1\}$ the stationary distribution on $\Omega^n$ is given by $$\underline{\nu}^{n,c}(\underline{\eta})=\frac{k\sum\limits_{\ell\in \mathbb{Z}} q^{\frac{k\ell(\ell+1)}{2}-m\ell-k\ell c} \prod\limits_{i=-\infty}^0\frac{q^{-(i-c)\eta_i}}{Z_i^c(q)}\prod\limits_{i=1}^\infty \frac{q^{(k-\eta_i)(i-c)}}{q^{k(i-c)}Z_i^c(q)} \mathbb{I}\{N(\underline{\eta})=n\}}{q^{\frac{(n+m)(n+k-m)}{2k}-(n+m)c}\left(1+\sum_{r=1}^{k-1}\zeta_k^{-rm}\left(\prod_{i=-\infty}^{\infty}\left(1+\sum_{j=1}^{k-1}(\zeta_{k}^{rj}-1)\mu_i^c(j)\right)\right)\right)}.$$
\end{prop}
\begin{remark}
As in Section $3.1$ these distributions are independent of $c$ but we will need to stress the dependence of both the numerator and denominator on $c$.
\end{remark}

Notice that $1+\sum_{j=1}^{k-1}(\zeta_k^{rj}-1)\mu_i^c(j) = \frac{\sum_{j=0}^{k-1}\zeta_k^{rj}q^{-j(i-c)}}{Z_i^c(q)}$ for each $r\in\{1,...,k-1\}$. Defining $W_{i,r}^c(q) := \sum_{j=0}^{k-1}\zeta_k^{rj}q^{-j(i-c)}$ for each such $r$ we can write $\underline{\nu}^{n,c}$ as 

\Small{\[\underline{\nu}^{n,c}(\underline{\eta})=\frac{k\sum\limits_{\ell \in \mathbb{Z}} q^{\frac{k\ell(\ell+1)}{2}-m\ell-k\ell c} \prod\limits_{i=-\infty}^0 q^{-(i-c)\eta_i}\prod\limits_{i=1}^\infty q^{(k-\eta_i)(i-c)}\mathbb{I}\{N(\underline{\eta})=n\}}{q^{\frac{(n+m)(n+k-m)}{2k}-(n+m)c}\left(\prod_{i\geq 1}q^{k(i-c)}Z_{-i+1}^c(q)Z_i^c(q)+\sum_{r=1}^{k-1}\zeta_k^{-rm}\left(\prod_{i\geq 1}q^{k(i-c)}W_{-i+1,r}^c(q)W_{i,r}^c(q)\right)\right)}.\]}\normalsize

\par \noindent when $n\in k\mathbb{Z}-m$ with $m\in\{0,1,...,k-1\}$.

\begin{remark}
When $k=2$ there is only the possibility $r=1$ and $W_{i,1}^c(q) = W_i^c(q,1)$, as expected (here $W_i^c(q,t)$ is the function defined in Section $3.1$). The corresponding stationary distributions then agree with the ones for $2$-exclusion given by Proposition \ref{nu n} (after setting $t=1$).
\end{remark}

\subsection{Standing up / Laying Down}~
\par In this section we transfer the dynamics on $\Omega^n$ to that of a restricted particle system on $\mathbb{Z}_{\geq 0}^{\mathbb{Z}_{< 0}}$ by using the same ``standing up'' method as described in Section \ref{stand up}. By doing this we obtain an alternative characterisation of the stationary distributions given in Proposition \ref{nu n k-exc}.

We recall the standing up map. Given $\underline{\eta} \in \Omega^n$ let $S_r$ being the site of the $r^\textrm{th}$ particle when reading left to right, bottom to top. The corresponding stood up state is then  $T^n(\underline{\eta})=\underline{\omega}\in \mathbb{Z}_{\geq 0}^{\mathbb{Z}_{< 0}}$, with $\omega_{-r}=S_{r+1}-S_r$. See Figure \ref{standing up K} for an example with $k=4$.

\vspace{5mm}
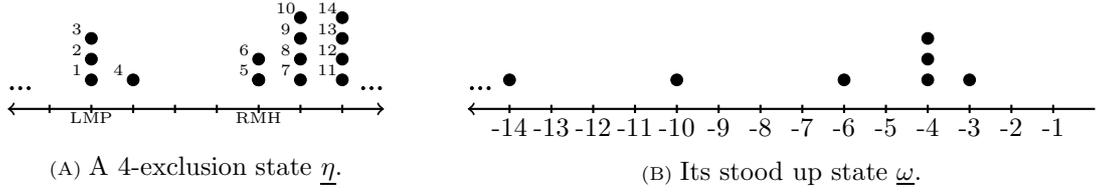
\begin{figure}[H]
\centering
    \begin{tikzpicture}[scale=0.55]
    \draw[thick, <->] (-4,0.8)--(5,0.8);
\foreach \x in {-3,-2,-1,0,1,2,3,4}
    \draw[thick, -](\x cm, 0.9)--(\x cm, 0.7) node[anchor=north]{};
 
 \filldraw [black] (-3.5,1.3) circle (1pt);
\filldraw [black] (-3.7,1.3) circle (1pt);
\filldraw [black] (-3.9,1.3) circle (1pt);   
\filldraw [black] (-2,1.5) circle (4pt);
\node (a) at (-2,0) [label=\tiny{LMP}]{};
\node (a) at (-2.35,1.15) [label=\tiny{1}]{};
\filldraw [black] (-2,2) circle (4pt);
\filldraw [black] (-2,2.5) circle (4pt);
\filldraw [black] (-1,1.5) circle (4pt);
\node (a) at (-2.35,1.65) [label=\tiny{2}]{};
\filldraw [black] (2,1.5) circle (4pt);
\node (a) at (-2.35,2.15) [label=\tiny{3}]{};
\filldraw [black] (2,2) circle (4pt);
\node (a) at (-1.35,1.15) [label=\tiny{4}]{};
\filldraw [black] (2,1.5) circle (4pt);
\node (a) at (1.65,1.15) [label=\tiny{5}]{};
\filldraw [black] (3,1.5) circle (4pt);
\node (a) at (1.65,1.65) [label=\tiny{6}]{};
\node (a) at (2,0) [label=\tiny{RMH}]{};
\filldraw [black] (2,1.5) circle (4pt);
\node (a) at (2.65,1.15) [label=\tiny{7}]{};
\filldraw [black] (3,2) circle (4pt);
\node (a) at (2.65,1.65) [label=\tiny{8}]{};
\filldraw [black] (3,2.5) circle (4pt);
\node (a) at (2.65,2.15) [label=\tiny{9}]{};
\filldraw [black] (3,3) circle (4pt);
\node (a) at (2.65,2.65) [label=\tiny{10}]{};
\filldraw [black] (4,2) circle (4pt);
\node (a) at (3.65,1.65) [label=\tiny{12}]{};
\filldraw [black] (4,1.5) circle (4pt);
\filldraw [black] (4,2.5) circle (4pt);
\node (a) at (3.65,1.15) [label=\tiny{11}]{};
\node (a) at (3.65,2.15) [label=\tiny{13}]{};
\filldraw [black] (4,3) circle (4pt);
\node (a) at (3.65,2.65) [label=\tiny{14}]{};
 \filldraw [black] (4.5,1.3) circle (1pt);
\filldraw [black] (4.7,1.3) circle (1pt);
\filldraw [black] (4.9,1.3) circle (1pt); 
\node (a) at (0.5,-1.5) [label=\scriptsize{(A)} \normalsize A $4$-exclusion state $\underline{\eta}$.]{};
    \draw[thick, <-] (7,0.8)--(22,0.8);
\foreach \x in {8,9,10,11,12,13,14,15,16,17,18,19,20,21}
    \draw[thick, -](\x cm, 0.9)--(\x cm, 0.7);
        \node (a) at (21,-0.3) [label={-1}]{};
    \node (a) at (20,-0.3) [label={-2}]{};
    \node (a) at (19,-0.3) [label={-3}]{};
    \node (a) at (18,-0.3) [label={-4}]{};
    \node (a) at (17,-0.3) [label={-5}]{};
    \node (a) at (16,-0.3) [label={-6}]{};
    \node (a) at (15,-0.3) [label={-7}]{};
    \node (a) at (14,-0.3) [label={-8}]{};
    \node (a) at (13,-0.3) [label={-9}]{};
    \node (a) at (12,-0.3) [label={-10}]{};
    \node (a) at (11,-0.3) [label={-11}]{};
    \node (a) at (10,-0.3) [label={-12}]{};
    \node (a) at (9,-0.3) [label={-13}]{};
    \node (a) at (8,-0.3) [label={-14}]{};
    
\filldraw [black] (19,1.5) circle (4pt);
\filldraw [black] (18,2.5) circle (4pt);
\filldraw [black] (18,2) circle (4pt);
\filldraw [black] (18,1.5) circle (4pt);
\filldraw [black] (16,1.5) circle (4pt);
\filldraw [black] (12,1.5) circle (4pt);
\filldraw [black] (8,1.5) circle (4pt);
\filldraw [black] (7.5,1.3) circle (1pt);
\filldraw [black] (7.3,1.3) circle (1pt);
\filldraw [black] (7.1,1.3) circle (1pt); 
\node (a) at (14.5,-1.5) [label= \scriptsize{(B)} \normalsize Its stood up state $\underline{\omega}$.]{};

    \end{tikzpicture}
    \caption{The standing up map on $4$-exclusion.}
    \label{standing up K}
\end{figure}
\par A priori the ``standing up" map $T^n$ is an injection into $\mathbb{Z}_{\geq 0}^{\mathbb{Z}_{<0}}$. However since $\eta_i\leq k$ for all $i$ the image of $T^n$ lies in the restricted state space
$$\mathcal{H}':=\{\underline{\omega} \in \mathbb{Z}_{\geq 0}^{\mathbb{Z}_{< 0}}: \hspace{2mm} \omega_{-i}=\omega_{-i-1}=...=\omega_{-i-(k-2)}=0 \Rightarrow \omega_{-i-(k-1)} \neq0, \forall i>0\}.$$
\par \noindent Since $\eta_i=k$ for $i$ large $\underline{\omega}$ must coincide far to the left with one of the following states $\underline{\omega}^{-m}$, for $m \in \{0,1,..,k-1\}$, defined by $\omega^{-m}_{-i}=\mathbb{I}\{i \in k\mathbb{Z}+m\}$. More precisely $m$ is uniquely determined by $n\in k\mathbb{Z}-m$.
\vspace{5mm}
\begin{figure}[H]
    \centering
       \begin{subfigure}[b]{0.4\textwidth}
       \centering
    \begin{tikzpicture}[scale=0.55]
    \draw[thick, <-] (-12,-1)--(0,-1);
\foreach \x in {-11,-10,-9,-8,-7,-6,-5,-4,-3,-2,-1}
    \draw[thick, -](\x cm, -1.1)--(\x cm, -0.9) node[anchor=north]{$\x$};
\filldraw [black] (-11.5,-0.7) circle (1pt);
\filldraw [black] (-11.7,-0.7) circle (1pt);
\filldraw [black] (-11.9,-0.7) circle (1pt); 
\filldraw [black] (-8,-0.5) circle (4pt);
\filldraw [black] (-4,-0.5) circle (4pt);

    \end{tikzpicture}
    \caption{ $\underline{\omega}^0$}
    \end{subfigure}
    \hspace{10mm}
       \begin{subfigure}[b]{0.4\textwidth}
       \centering
    \begin{tikzpicture}[scale=0.55]
    \draw[thick, <-] (-12,-1)--(0,-1);
\foreach \x in {-11,-10,-9,-8,-7,-6,-5,-4,-3,-2,-1}
    \draw[thick, -](\x cm, -1.1)--(\x cm, -0.9) node[anchor=north]{$\x$};
\filldraw [black] (-11.5,-0.7) circle (1pt);
\filldraw [black] (-11.7,-0.7) circle (1pt);
\filldraw [black] (-11.9,-0.7) circle (1pt); 
\filldraw [black] (-9,-0.5) circle (4pt);
\filldraw [black] (-5,-0.5) circle (4pt);
\filldraw [black] (-1,-0.5) circle (4pt);

    \end{tikzpicture}
    \caption{ $\underline{\omega}^{-1}$}
    \end{subfigure}
        \par\bigskip \vspace{5mm}
            \begin{subfigure}[b]{0.4\textwidth}
       \centering
    \begin{tikzpicture}[scale=0.55]
    \draw[thick, <-] (-12,-1)--(0,-1);
\foreach \x in {-11,-10,-9,-8,-7,-6,-5,-4,-3,-2,-1}
    \draw[thick, -](\x cm, -1.1)--(\x cm, -0.9) node[anchor=north]{$\x$};
\filldraw [black] (-11.5,-0.7) circle (1pt);
\filldraw [black] (-11.7,-0.7) circle (1pt);
\filldraw [black] (-11.9,-0.7) circle (1pt); 
\filldraw [black] (-10,-0.5) circle (4pt);
\filldraw [black] (-6,-0.5) circle (4pt);
\filldraw [black] (-2,-0.5) circle (4pt);
    \end{tikzpicture}
    \caption{$\underline{\omega}^{-2}$}
    \end{subfigure}
    \hspace{10mm}
           \begin{subfigure}[b]{0.4\textwidth}
       \centering
    \begin{tikzpicture}[scale=0.55]
    \draw[thick, <-] (-12,-1)--(0,-1);
\foreach \x in {-11,-10,-9,-8,-7,-6,-5,-4,-3,-2,-1}
    \draw[thick, -](\x cm, -1.1)--(\x cm, -0.9) node[anchor=north]{$\x$};
\filldraw [black] (-11.5,-0.7) circle (1pt);
\filldraw [black] (-11.7,-0.7) circle (1pt);
\filldraw [black] (-11.9,-0.7) circle (1pt); 
\filldraw [black] (-11,-0.5) circle (4pt);
\filldraw [black] (-7,-0.5) circle (4pt);
\filldraw [black] (-3,-0.5) circle (4pt);
    \end{tikzpicture}
    \caption{$\underline{\omega}^{-3}$}
    \end{subfigure}
    
    \caption{The four ground states of stood up 4-exclusion.}
    \label{k-exclus stood up ground states}
\end{figure}
\begin{remark}
Note that all shifts of the ground state $\underline{\eta}^{-m}\in \Omega^{-m}$ stand up to give $\underline{\omega}^{-m}$ for each $m \in \{0,1,..,k-1\}$. This is in direct analogy with the $2$-exclusion case.
\end{remark}
\par \noindent We now see that the image of $T^n$ lies in $\mathcal{H}:=\bigcup\limits_{m =0}^{k-1}\mathcal{H}^{-m}$, where the disjoint sets $\mathcal{H}^{-m}$ are defined as,
$$\mathcal{H}^{-m} := \{\underline{\omega} \in \mathcal{H}' : \exists N > 0 \hspace{2mm}\textrm{ s.t } \omega_{-i}=\omega_{-i}^{-m} \hspace{2mm} \forall i \geq N\}.$$
\par \noindent Given $\underline{\omega} \in \mathcal{H}^{-m}$ we let $D_{-m}(\underline{\omega})$ be the minimum such $N$ in the above. 
\begin{lem}\label{bijec k-exc}
For $n\in k\mathbb{Z}-m$ with $m\in\{0,1,...,k-1\}$ we have $T^n(\Omega^n) = \mathcal{H}^{-m}$.
\end{lem}
\begin{proof} It suffices to show surjectivity of $T^n$ for each $n$.
\par \noindent  Take $n\in k\mathbb{Z}-m$ and $\omega\in\mathcal{H}^{-m}$, then construct the state $\eta\in\Omega^n$ having LMP at site $$S_1=\frac{n+D_{-m}(\underline{\omega})-\mathbb{I}\{D_{-m}(\underline{\omega})\notin k\mathbb{Z}-m\}}{k}+1-\sum\limits_{i=1}^{D_{-m}(\underline{\omega})}\omega_{-i}$$ 
\par \noindent and $r^\textrm{th}$ particle at site $S_r=S_{r-1}+\omega_{1-r}$. It is clear that $T^n(\underline{\eta})=\underline{\omega}$ and hence $T^n$ is surjective.
\end{proof}
\par \noindent Just as before we call these inverse maps the ``laying down`` maps.
\par Using the ``standing up" maps we define a particle system on $\mathcal{H}$ whose dynamics are inherited from those on $\Omega$. In particular right jumps in $\underline{\eta}$ correspond to right jumps in $\underline{\omega}$ and similarly for left jumps.  The explicit right/left jump rates are given in Table \ref{stood up rates} for $r \geq 2$.
\begin{table}[H]
\centering
\begin{tabular}{cc}
 \begin{tabular}{|c||c|}
    \hline& $\omega_{-r+1} \geq 0$\\
    \hline\hline
        $\omega_{-r}=0$ & $0$  \\
        \hline
        $\omega_{-r}=1$ & $1-\prod\limits_{j=1}^{k-1}\mathbb{I}\{\omega_{-r-j}=0\}$ \\
        \hline
        $\omega_{-r} \geq 2$ & $1$ \\
        \hline
    \end{tabular}\\ \vspace{1mm}\\
\begin{tabular}{|c||c|c|c|}
    \hline
         &$\omega_{r+1}=0$ &$\omega_{-r+1}=1$ &$\omega_{-r+1} \geq 2$  \\
         \hline \hline
         $\omega_{-r}\geq0$&$0$& $q(1-\prod\limits_{j=1}^{k-1}\mathbb{I}\{\omega_{-r+1+j}=0\})$ & $q$\\
         \hline
    \end{tabular}\\
\end{tabular}
\caption{The jump rates $p_\omega(\omega_{-r}, \omega_{-r+1})$ and $q_\omega(\omega_{-r},\omega_{-r+1})$ respectively.}
\label{stood up rates}
\end{table}
\par Since the ``stood up " process is only defined on the negative half integer line we must consider what happens at the boundary site. As in Section $3.2$ we will consider an open infinite type boundary.
\begin{table}[H]
    \centering
    \begin{tabular}{|c||c|c|}
         \hline
         &  Rate into the boundary &  Rate out of the boundary \\
         \hline \hline
          $\omega_{-1}=0$&$0$ &  \multirow{3}{*}{$q$}\\
         \cline{1-2}
       \multicolumn{1}{|c||}{ $\omega_{-1}=1$} & \multicolumn{1}{|c|}{$1-\prod\limits_{j=1}^{k-1}\mathbb{I}\{\omega_{-r-j}=0\}$} &  \\ \cline{1-2}
        \multicolumn{1}{|c||}{$\omega_{-1}\geq2$} & \multicolumn{1}{|c|}{$1$} &\\
        \hline
    \end{tabular}
    \caption{Boundary jump rates for the stood up process, $\underline{\omega}$.}
    \label{boundary rates k}
\end{table}
\par To find the stationary distribution for the ``stood up" process we first consider the unrestricted process $\underline{\omega}^* \in \mathbb{Z}_{\geq 0}^{\mathbb{Z}_{<0}}$, i.e.\ the process described by the same jump rates as $\underline{\omega}$ but where the number of consecutive zeros is not restricted. It is clear that the unrestricted process is simply the zero-range process which is a member of the blocking family \cite{blocking}, with one parameter family of blocking measures given by the marginals
$$\pi_{-i}^{*,\hat{c}}(z)=q^{(i+\hat{c})z}(1-q^{(i+\hat{c})}).$$
\par \noindent  By considering dynamics at the boundary we can fix the value of $\hat{c}$ and so have a single product blocking measure. We suppose that $\underline{\pi}^{*,\hat{c}}$ satisfies detailed balance over this boundary edge, i.e.\
$$\pi_{-1}^{*,\hat{c}}(y)\cdot \textrm{``rate into the boundary"}=\pi_{-1}^{*,\hat{c}}(y-1)\cdot \textrm{`` rate out of the boundary"} \hspace{5mm} \text{for all } y\geq 1.$$
\par \noindent Thus for all $y \geq 1$ we have $$q^{(1+\hat{c})y}(1-q^{(i-c)})=q^{(1+\hat{c})(y-1)}(1-q^{(i-c)})q$$ and hence $\hat{c} = 0$, giving the stationary blocking measure 
$$\underline{\pi}^*(\underline{\omega}^*)=\prod\limits_{i \geq 1}q^{i\omega^*_{-i}}(1-q^{i}).$$
\par Now that we have the stationary distribution for the unrestricted process we consider the restriction to $\mathcal{H}$ and find the stationary measure. Recall that $\mathcal{H}=\bigcup\limits_{m=0}^{k-1}\mathcal{H}^{-m}$, and note that for each $\mathcal{H}^{-m}$ is the irreducible component of the ground state $\underline{\omega}^{-m}$. We define stationary measures $\underline{\pi}^{-m}$ on these irreducible components in terms of $\underline{\pi}^*$. It seems natural to define these measures as $\underline{\pi}^{-m}(\cdot)=\underline{\pi}^*(\cdot|\cdot \in \mathcal{H}^{-m})$. However w.r.t $\underline{\pi}^*$ the probability of being in any irreducible component is zero and so these quantities are undefined. To rectify this we use a similar formal reasoning as we did in Section \ref{stand up} in order to give the stationary distributions in the following form: 
$$\underline{\pi}^{-m}(\underline{\omega})=\frac{\prod\limits_{i\geq 1}\phi_{-i}^{-m}(\omega_{-i})\mathbb{I}\{\underline{\omega}\in \mathcal{H}^{-m}\}}{\sum\limits_{\underline{\omega}'\in \mathcal{H}^{-m}}\prod\limits_{i\geq 1}\phi_{-i}^{-m}(\omega'_{-i})},$$
\par \noindent where
$$\phi_{-i}^{-m}(\omega_{-i}) = \frac{\pi_{-i}^*(\omega_{-i})}{\pi_{-i}^*(\omega_{-i}^{-m})}= \begin{cases}
q^{i(\omega_{-i}-1)} &\textrm{if } i \in k\mathbb{Z}+m \\
q^{i\omega_{-i}} &\textrm{otherwise}.
\end{cases}$$ These are given explicitly in the following Proposition.
\begin{prop}\label{pi ell}
For each $m\in\{0,1,...,k-1\}$ the unique stationary measure on $\mathcal{H}^{-m}$ is
$$\underline{\pi}^{-m}(\underline{\omega})=\frac{q^{\sum\limits_{i \notin k \mathbb{Z}+m}i\omega_{-i}+\sum\limits_{i \in k \mathbb{Z}+m}i(\omega_{-i}-1)}}{S_{-m}^{(k)}(q)},$$
\par \noindent where $S_{-m}^{(k)}(q)=\sum\limits_{\underline{\omega}'\in\mathcal{H}^{-m}}q^{\sum\limits_{i \notin k \mathbb{Z}+m}i\omega_{-i}'+\sum\limits_{i \in k \mathbb{Z}+m}i(\omega_{-i}'-1)}$ is the normalising factor with respect to $\mathcal{H}^{-m}$.
\end{prop}
\begin{proof}
The result once again follows from Proposition 5.10 of \cite{liggett} since each $\underline{\pi}^{-m}$ is a restriction of $\underline{\pi}^*$.
\end{proof}

\subsection{Identities}~
\par By Lemma \ref{bijec k-exc} the standing up transformation $T^n$ describes a bijection between $\Omega^n$ and $\mathcal{H}^{-m}$, for the unique value $m\in\{0,1,...,k-1\}$ such that $n\in k\mathbb{Z}-m$. Since $T^n$ preserves the dynamics of the corresponding processes we get an equality of measures,  $\underline{\nu}^{n,c}(\underline{\eta}) =  \underline{\pi}^{-m}(T^n(\underline{\eta}))$ for this value of $m$.
\par Recall that $k$-exclusion has $k$ ground states up to shift; $\underline{\eta}^{-m}\in\Omega^{-m}$ for $m \in \{0,1,..,k-1\}$, satisfying $T^{-m}(\underline{\eta}^{-m}) = \underline{\omega}^{-m}$. Thus $\underline{\nu}^{-m,c}(\underline{\eta}^{-m}) = \underline{\pi}^{-m}(\underline{\omega}^{-m})$ and so by Proposition \ref{nu n k-exc} and Proposition \ref{pi ell} we have the following identities for $m\in \{0,1,..,k-1\}$ (after rearrangement):

    \begin{align*}k\sum\limits_{\ell \in \mathbb{Z}}S_{-m}^{(k)}(q) q^{\frac{k\ell(\ell+1)}{2}-m\ell-(k\ell-m)c}&=\prod_{i\geq1}q^{k(i-c)}Z_{-i+1}^c(q)Z_i^c(q)\\ &+ \sum_{r=1}^{k-1}\zeta_k^{-rm}\left(\prod_{i\geq 1}q^{k(i-c)}W_{-i+1,r}^c(q)W_{i,r}^c(q)\right).\end{align*}
Writing $Z_{i}^c(q)$ and $W_{i,r}^c(q)$ explicitly and letting $z=q^{-c}$ proves the following identities. 

\kexc*
\vspace{5mm}
\begin{remark}
When $k=2$ these agree with the two identities coming from $2$-exclusion in Theorem \ref{2-exc explicit}.
\end{remark}

We now discuss the combinatorial significance of these identities.  Note first that $\prod_{i\geq 1}(1+q^i+q^{2i}+...q^{k})$ is the generating function for ordinary partitions of $n$ with each part appearing at most $k$ times. By the ``General Principle" of Andrews we can then interpret the first term on the RHS as:
\vspace{5mm}
$$\prod_{i\geq 1}(1+q^{i}z+q^{2i}z^2+...+q^{ki}z^k)(1+q^{i-1}z^{-1}+q^{2(i-1)}z^{-2}+...+q^{k(i-1)}z^{-k}) = \sum_{k'\in\mathbb{Z}}f_{D_k,D_k,k'}(q)z^{k'}$$ 
\vspace{5mm} \par \noindent where $f_{D_k,D_k,k'}(q)$ is the generating function for $\text{GFP}_{D_k,D_k,k'}(n)$. The function $f_{D_k,D_k,0}(q)$ is the function $\Phi_k(q)$ defined by Andrews on p.$6$ of \cite{frobenius} (the special case $k=2$ appeared earlier when considering $2$-exclusion). The case of non-zero offset is not considered in the book.

\newpage The other terms on the RHS might look unwieldy at first glance but are in fact the above but with change of variables $z \mapsto \zeta_k^{-r}z$. We saw this in all previous examples; $\zeta_2 = -1$ and this was providing the sign changes in the identities, the purpose of which were to isolate odd/even terms of another identity. Here we have the same behaviour but now we are using $k$th roots of unity to isolate the $z^{k'}$ terms where $k'$ lies in a fixed class mod $k$. To justify this we expand the whole RHS; for $m\in\{0,1,...,k-1\}$ we have
\vspace{5mm}
\begin{align*}
\sum_{r=0}^{k-1}\zeta_k^{-rm}\left(\prod_{i\geq 1}\left(\sum_{\alpha=0}^k \zeta_k^{-\alpha r}q^{\alpha i}z^{\alpha}\right)\left(\sum_{\alpha=0}^k \zeta_k^{\alpha r}q^{\alpha (i-1)}z^{-\alpha}\right)\right) &= \sum_{r=0}^{k-1}\zeta_k^{-rm}\left(\sum_{k'\in\mathbb{Z}}\zeta_k^{-rk'}f_{D_k,D_k,k'}(q)z^{k'}\right)\\ &= \sum_{k'\in\mathbb{Z}}\left(\sum_{r=0}^{k-1}\zeta_k^{-r(m+k')}\right)f_{D_k,D_k,k'}(q)z^{k'}\\ &= k\sum_{k'\equiv -m \bmod k}f_{D_k,D_k,k'}(q)z^{k'}.\end{align*}\vspace{5mm}

\par\noindent Given the above we now see that our identities are equivalent to the equalities \[f_{D_k,D_k,k'}(q) =S_{-m}^{(k)}(q)q^{\frac{k\ell(\ell+1)}{2}-m\ell}\quad\text{ if } k'=k\ell-m \text{ with } m\in\{0,1,...k-1\}.\]

\vspace{5mm}\par \noindent Once again, we have proved this probabilistically but are also interested in explicit combinatorial explanations. In a similar vein to $2$-exclusion we find that there are $k$ base cases, $f_{D_k,D_k,-m}(q) = S_{-m}^{(k)}(q)$ for $m\in\{0,1,...,k-1\}$. These can be proved by adapting the maps $\psi^e_{n}, \psi^o_{n}$ of Section $3.3$ to give maps: \vspace{5mm}
\[\psi^{(m)}_{n} : \left\{\underline{\omega}\in\mathcal{H}^{-m}\,:\,\sum_{i\notin k\mathbb{Z}+m}i\omega_{-i} + \sum_{i \in k\mathbb{Z}+m}i(\omega_{-i}-1) = n\right\} \rightarrow \text{GFP}_{D_k,D_k,-m}(n),\] 
\vspace{5mm}\par\noindent for $n\geq 0$ and $m\in\{0,1,...,k-1\}$. 

We must first describe the process of assigning generalised Young diagrams to GFP's with the $k$-repetition condition. An element of $\text{GFP}_{D_k,D_k,-m}(n)$ with $m\in\{0,1,...,k-1\}$ can be assigned a generalised Young diagram on the set of points $C_{m}=\{(n_1,n_2)\in\mathbb{Z}^2\,|\, n_1+n_2\equiv m \bmod k\}$ as follows. The subset corresponding to such an element with $s_1$ entries on the top row consists of the $s_1$ leading diagonal points $(m+1,-1),(m+2,-2),...,(m+s_1,-s_1)$, the first $a_i$ points of $C_{m}$ to the right of $(m+i,-i)$ and the first $b_i$ points of $C_{m}$ under $(i,m-i)$. Note that the points $(i,m-i)$ for $1\leq i \leq m$ are not included.

The maps $\psi^{(m)}_n$ are then defined in a similar fashion to $\psi^{e}_n, \psi^{o}_n$. In the following we use the notation $r_{(m)}$ to stand for a wave of length $r$ on $C_m$, a shift of the points of $C_m$ enclosed in the rectangle with opposite corners $(1,-1),(r,-k)$. Given $\underline{\omega}\in\mathcal{H}^{-m}$ the map $\psi^{(m)}_n$ stacks $(\omega_{-i}-\mathbb{I}\{i\equiv m \bmod k\})$ copies of the wave $i_{(m)}$ (whenever this is non-negative) vertically in increasing order, and then removes a point from the bottom of each of the columns $1,2,...,i$ for each $i\equiv m \bmod k$ such that $\omega_{-i} = 0$ (giving the generalised Young diagram of an element of $\text{GFP}_{D_k,D_k,-m}(n)$). See Figure \ref{k-exc bijections from omega to gfp} for an example when $k=3$ (as earlier, points are labelled as black squares). It is possible to check that these maps are bijections but as before this is tedious and is left to the reader.
\newpage
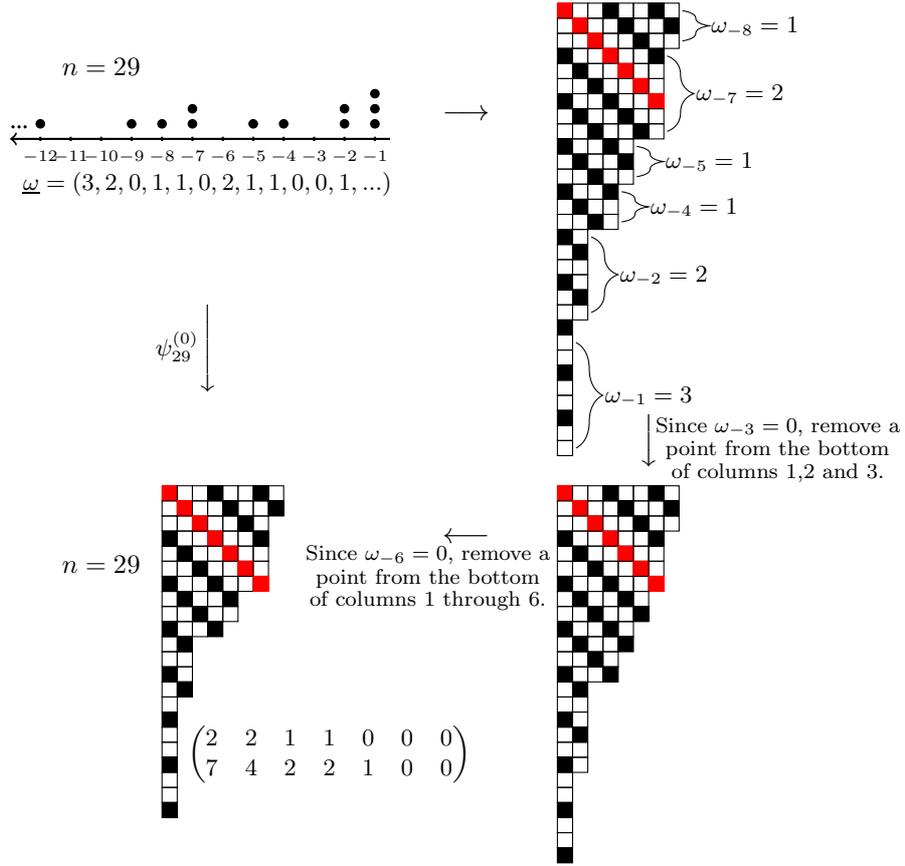
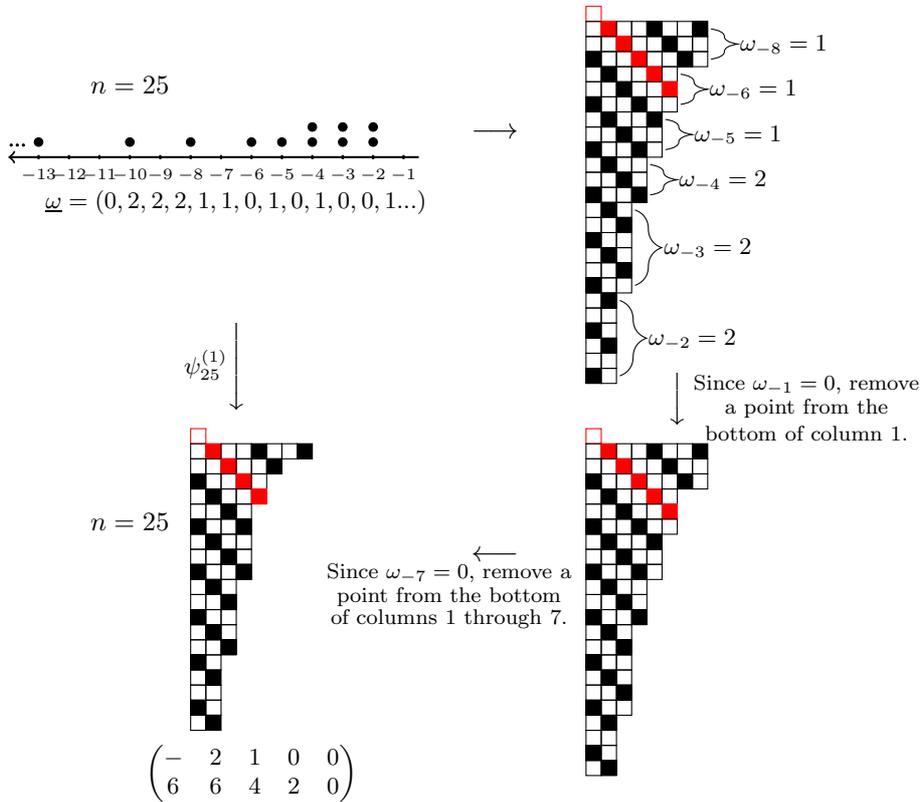
\begin{figure}[H]
    \centering
\begin{subfigure}[b]{0.75\textwidth}
    \centering
\begin{tikzpicture}[scale=0.4]
\draw[thick, <-] (-13,0)--(-0.5,0);
\foreach \x in {-12,-11,-10,-9,-8,-7,-6,-5,-4,-3,-2,-1}
    \draw[thick, -](\x cm, 2pt)--(\x cm, -2pt) node[anchor=north]{\tiny{$\x$}};
 
\filldraw[black](-1,0.5) circle (4pt);   
\filldraw[black](-1,1) circle (4pt);  
\filldraw[black](-1,1.5) circle (4pt);
\filldraw[black](-2,0.5) circle (4pt);   
\filldraw[black](-2,1) circle (4pt);  
\filldraw [black] (-4,0.5) circle (4pt);
\filldraw [black] (-5,0.5) circle (4pt);
\filldraw [black] (-7,0.5) circle (4pt);
\filldraw[black](-7,1) circle (4pt);  
\filldraw [black] (-8,0.5) circle (4pt);
\filldraw[black](-9,0.5) circle (4pt);
\filldraw[black](-12,0.5) circle (4pt);   
\filldraw [black] (-12.5,0.4) circle (1pt);
\filldraw [black] (-12.7,0.4) circle (1pt);
\filldraw [black] (-12.9,0.4) circle (1pt);
\node (a) at (-6.5,-2.5) [label=\small{$\underline{\omega}=(3,2,0,1,1,0,2,1,1,0,0,1,...)$}]{};
\node(a) at(2,0)[label=\large{$\longrightarrow$}]{};
\node(a) at(-10,1.5)[label={$n=29$}]{};

\filldraw[red] (5,4) rectangle (5.5,4.5);
\draw[black](5.5,4) rectangle (6,4.5);
\draw[black] (6,4) rectangle (6.5,4.5);
\filldraw[black] (6.5,4) rectangle (7,4.5);
\draw[black] (7,4) rectangle (7.5,4.5);
\draw[black] (7.5,4) rectangle (8,4.5);
\filldraw[black] (8,4) rectangle (8.5,4.5);
\draw[black] (8.5,4) rectangle (9,4.5);
\draw[black] (5,3.5)rectangle(5.5,4);
\filldraw[red] (5.5,3.5) rectangle (6,4);
\draw[black](6,3.5) rectangle (6.5,4);
\draw[black] (6.5,3.5) rectangle (7,4);
\filldraw[black] (7,3.5) rectangle (7.5,4);
\draw[black] (7.5,3.5) rectangle (8,4);
\draw[black] (8,3.5) rectangle (8.5,4);
\filldraw[black] (8.5,3.5) rectangle (9,4);
\draw[black] (5,3) rectangle (5.5,3.5);
\draw[black](5.5,3) rectangle (6,3.5);
\filldraw[red] (6,3) rectangle (6.5,3.5);
\draw[black] (6.5,3) rectangle (7,3.5);
\draw[black] (7,3) rectangle (7.5,3.5);
\filldraw[black] (7.5,3) rectangle (8,3.5);
\draw[black] (8,3) rectangle (8.5,3.5);
\draw[black] (8.5,3) rectangle (9,3.5);
\draw [decorate,decoration={brace,amplitude=10pt},xshift=-4pt,yshift=0pt]
(9.25,4.25) -- (9.25,3.25) node [black,midway,xshift=-0.6cm] 
{}; 
\node(a) at (11.5,2.75) [label=\small{$ \omega_{-8}=1$}]{};
\filldraw[black] (5,2.5)rectangle(5.5,3);
\draw[black] (5.5,2.5) rectangle (6,3);
\draw[black](6,2.5) rectangle (6.5,3);
\filldraw[red] (6.5,2.5) rectangle (7,3);
\draw[black] (7,2.5) rectangle (7.5,3);
\draw[black] (7.5,2.5) rectangle (8,3);
\filldraw[black] (8,2.5) rectangle (8.5,3);
\draw[black] (5,2) rectangle (5.5,2.5);
\filldraw[black](5.5,2) rectangle (6,2.5);
\draw[black] (6,2) rectangle (6.5,2.5);
\draw[black] (6.5,2) rectangle (7,2.5);
\filldraw[red] (7,2) rectangle (7.5,2.5);
\draw[black] (7.5,2) rectangle (8,2.5);
\draw[black] (8,2) rectangle (8.5,2.5);
\draw[black] (5,1.5)rectangle(5.5,2);
\draw[black] (5.5,1.5) rectangle (6,2);
\filldraw[black](6,1.5) rectangle (6.5,2);
\draw[black] (6.5,1.5) rectangle (7,2);
\draw[black] (7,1.5) rectangle (7.5,2);
\filldraw[red](7.5,1.5) rectangle (8,2);
\draw[black] (8,1.5) rectangle (8.5,2);
\filldraw[black] (5,1)rectangle(5.5,1.5);
\draw[black] (5.5,1) rectangle (6,1.5);
\draw[black](6,1) rectangle (6.5,1.5);
\filldraw[black] (6.5,1) rectangle (7,1.5);
\draw[black] (7,1) rectangle (7.5,1.5);
\draw[black] (7.5,1) rectangle (8,1.5);
\filldraw[red] (8,1) rectangle (8.5,1.5);
\draw[black] (5,0.5) rectangle (5.5,1);
\filldraw[black](5.5,0.5) rectangle (6,1);
\draw[black] (6,0.5) rectangle (6.5,1);
\draw[black] (6.5,0.5) rectangle (7,1);
\filldraw[black] (7,0.5) rectangle (7.5,1);
\draw[black] (7.5,0.5) rectangle (8,1);
\draw[black] (8,0.5) rectangle (8.5,1);
\draw[black] (5,0)rectangle(5.5,0.5);
\draw[black] (5.5,0) rectangle (6,0.5);
\filldraw[black](6,0) rectangle (6.5,0.5);
\draw[black] (6.5,0) rectangle (7,0.5);
\draw[black] (7,0) rectangle (7.5,0.5);
\filldraw[black](7.5,0) rectangle (8,0.5);
\draw[black] (8,0) rectangle (8.5,0.5);
\draw [decorate,decoration={brace,amplitude=10pt},xshift=-4pt,yshift=0pt]
(8.75,2.75) -- (8.75,0.25) node [black,midway,xshift=-0.6cm] 
{}; 
\node(a) at (11,0.5) [label=\small{$ \omega_{-7}=2$}]{};
\filldraw[black](5,-0.5) rectangle (5.5,0);
\draw[black] (5.5,-0.5) rectangle (6,0);
\draw[black](6,-0.5) rectangle (6.5,0);
\filldraw[black](6.5,-0.5) rectangle (7,0);
\draw[black] (7,-0.5) rectangle (7.5,0);
\draw[black] (5,-1) rectangle (5.5,-0.5);
\filldraw[black](5.5,-1) rectangle (6,-0.5);
\draw[black] (6,-1) rectangle (6.5,-0.5);
\draw[black] (6.5,-1) rectangle (7,-0.5);
\filldraw[black](7,-1) rectangle (7.5,-0.5);
\draw[black](5,-1.5) rectangle (5.5,-1);
\draw[black](5.5,-1.5) rectangle (6,-1);
\filldraw[black](6,-1.5) rectangle (6.5,-1);
\draw[black](6.5,-1.5) rectangle (7,-1);
\draw[black](7,-1.5) rectangle (7.5,-1);
\draw [decorate,decoration={brace,amplitude=10pt},xshift=-4pt,yshift=0pt]
(7.75,-0.25) -- (7.75,-1.25) node [black,midway,xshift=-0.6cm] 
{}; 
\node(a) at (10,-1.75) [label=\small{$ \omega_{-5}=1$}]{};
\filldraw[black] (5,-2) rectangle (5.5,-1.5);
\draw[black] (5.5,-2) rectangle (6,-1.5);
\draw[black] (6,-2) rectangle (6.5,-1.5);
\filldraw[black] (6.5,-2) rectangle (7,-1.5);
\draw[black] (5,-2.5) rectangle (5.5,-2);
\filldraw[black] (5.5,-2.5) rectangle (6,-2);
\draw[black] (6,-2.5) rectangle (6.5,-2);
\draw[black] (6.5,-2.5) rectangle (7,-2);
\draw[black] (5,-3) rectangle (5.5,-2.5);
\draw[black] (5.5,-3) rectangle (6,-2.5);
\filldraw[black] (6,-3) rectangle (6.5,-2.5);
\draw[black] (6.5,-3) rectangle (7,-2.5);
\draw [decorate,decoration={brace,amplitude=10pt},xshift=-4pt,yshift=0pt]
(7.25,-1.75) -- (7.25,-2.75) node [black,midway,xshift=-0.6cm] 
{}; 
\node(a) at (9.5,-3.25) [label=\small{$ \omega_{-4}=1$}]{};
\filldraw[black] (5,-3.5) rectangle (5.5,-3);
\draw[black] (5.5,-3.5) rectangle (6,-3);
\draw[black] (5,-4) rectangle (5.5,-3.5);
\filldraw[black] (5.5,-4) rectangle (6,-3.5);
\draw[black] (5,-4.5) rectangle (5.5,-4);
\draw[black] (5.5,-4.5) rectangle (6,-4);
\filldraw[black] (5,-5) rectangle (5.5,-4.5);
\draw[black] (5.5,-5) rectangle (6,-4.5);
\draw[black] (5,-5.5) rectangle (5.5,-5);
\filldraw[black] (5.5,-5.5) rectangle (6,-5);
\draw[black] (5,-6) rectangle (5.5,-5.5);
\draw[black] (5.5,-6) rectangle (6,-5.5);
\draw [decorate,decoration={brace,amplitude=10pt},xshift=-4pt,yshift=0pt]
(6.25,-3.25) -- (6.25,-5.75) node [black,midway,xshift=-0.6cm] 
{}; 
\node(a) at (8.5,-5.5) [label=\small{$ \omega_{-2}=2$}]{};
\filldraw[black](5,-6.5)rectangle (5.5,-6);
\draw[black](5,-7)rectangle(5.5,-6.5);
\draw[black] (5,-7.5) rectangle (5.5,-7);
\filldraw[black](5,-8)rectangle (5.5,-7.5);
\draw[black](5,-8.5)rectangle(5.5,-8);
\draw[black] (5,-9) rectangle (5.5,-8.5);
\filldraw[black](5,-9.5)rectangle (5.5,-9);
\draw[black](5,-10)rectangle(5.5,-9.5);
\draw[black] (5,-10.5) rectangle (5.5,-10);
\draw [decorate,decoration={brace,amplitude=10pt},xshift=-4pt,yshift=0pt]
(5.75,-6.75) -- (5.75,-10.25) node [black,midway,xshift=-0.6cm] 
{}; 
\node(a) at (8,-9.5) [label=\small{$ \omega_{-1}=3$}]{};

\node(a) at (12.25,-10.5) [label=\footnotesize{Since $\omega_{-3}=0$, remove a}]{};
\node(a) at (12.25,-11.25) [label=\footnotesize{point from the bottom}]{};
\node(a) at (12.25,-12) [label=\footnotesize{of columns 1,2 and 3.}]{};
\node (a) at (8,-11.5) [label=\large{$\Big\downarrow$}]{};
\filldraw[red] (5,-12) rectangle (5.5,-11.5);
\draw[black](5.5,-12) rectangle (6,-11.5);
\draw[black] (6,-12) rectangle (6.5,-11.5);
\filldraw[black] (6.5,-12) rectangle (7,-11.5);
\draw[black] (7,-12) rectangle (7.5,-11.5);
\draw[black] (7.5,-12) rectangle (8,-11.5);
\filldraw[black] (8,-12) rectangle (8.5,-11.5);
\draw[black] (8.5,-12) rectangle (9,-11.5);
\draw[black] (5,-12.5)rectangle(5.5,-12);
\filldraw[red] (5.5,-12.5) rectangle (6,-12);
\draw[black](6,-12.5) rectangle (6.5,-12);
\draw[black] (6.5,-12.5) rectangle (7,-12);
\filldraw[black] (7,-12.5) rectangle (7.5,-12);
\draw[black] (7.5,-12.5) rectangle (8,-12);
\draw[black] (8,-12.5) rectangle (8.5,-12);
\filldraw[black] (8.5,-12.5) rectangle (9,-12);
\draw[black] (5,-13) rectangle (5.5,-12.5);
\draw[black](5.5,-13) rectangle (6,-12.5);
\filldraw[red] (6,-13) rectangle (6.5,-12.5);
\draw[black] (6.5,-13) rectangle (7,-12.5);
\draw[black] (7,-13) rectangle (7.5,-12.5);
\filldraw[black] (7.5,-13) rectangle (8,-12.5);
\draw[black] (8,-13) rectangle (8.5,-12.5);
\draw[black] (8.5,-13) rectangle (9,-12.5);
\filldraw[black] (5,-13.5)rectangle(5.5,-13);
\draw[black] (5.5,-13.5) rectangle (6,-13);
\draw[black](6,-13.5) rectangle (6.5,-13);
\filldraw[red] (6.5,-13.5) rectangle (7,-13);
\draw[black] (7,-13.5) rectangle (7.5,-13);
\draw[black] (7.5,-13.5) rectangle (8,-13);
\filldraw[black] (8,-13.5) rectangle (8.5,-13);
\draw[black] (5,-14) rectangle (5.5,-13.5);
\filldraw[black](5.5,-14) rectangle (6,-13.5);
\draw[black] (6,-14) rectangle (6.5,-13.5);
\draw[black] (6.5,-14) rectangle (7,-13.5);
\filldraw[red] (7,-14) rectangle (7.5,-13.5);
\draw[black] (7.5,-14) rectangle (8,-13.5);
\draw[black] (8,-14) rectangle (8.5,-13.5);
\draw[black] (5,-14.5)rectangle(5.5,-14);
\draw[black] (5.5,-14.5) rectangle (6,-14);
\filldraw[black](6,-14.5) rectangle (6.5,-14);
\draw[black] (6.5,-14.5) rectangle (7,-14);
\draw[black] (7,-14.5) rectangle (7.5,-14);
\filldraw[red](7.5,-14.5) rectangle (8,-14);
\draw[black] (8,-14.5) rectangle (8.5,-14);
\filldraw[black] (5,-15)rectangle(5.5,-14.5);
\draw[black] (5.5,-15) rectangle (6,-14.5);
\draw[black](6,-15) rectangle (6.5,-14.5);
\filldraw[black] (6.5,-15) rectangle (7,-14.5);
\draw[black] (7,-15) rectangle (7.5,-14.5);
\draw[black] (7.5,-15) rectangle (8,-14.5);
\filldraw[red] (8,-15) rectangle (8.5,-14.5);
\draw[black] (5,-15.5) rectangle (5.5,-15);
\filldraw[black](5.5,-15.5) rectangle (6,-15);
\draw[black] (6,-15.5) rectangle (6.5,-15);
\draw[black] (6.5,-15.5) rectangle (7,-15);
\filldraw[black] (7,-15.5) rectangle (7.5,-15);
\draw[black] (7.5,-15.5) rectangle (8,-15);
\draw[black] (5,-16)rectangle(5.5,-15.5);
\draw[black] (5.5,-16) rectangle (6,-15.5);
\filldraw[black](6,-16) rectangle (6.5,-15.5);
\draw[black] (6.5,-16) rectangle (7,-15.5);
\draw[black] (7,-16) rectangle (7.5,-15.5);
\filldraw[black](7.5,-16) rectangle (8,-15.5);
\filldraw[black](5,-16.5) rectangle (5.5,-16);
\draw[black] (5.5,-16.5) rectangle (6,-16);
\draw[black](6,-16.5) rectangle (6.5,-16);
\filldraw[black](6.5,-16.5) rectangle (7,-16);
\draw[black] (7,-16.5) rectangle (7.5,-16);
\draw[black] (5,-17) rectangle (5.5,-16.5);
\filldraw[black](5.5,-17) rectangle (6,-16.5);
\draw[black] (6,-17) rectangle (6.5,-16.5);
\draw[black] (6.5,-17) rectangle (7,-16.5);
\filldraw[black](7,-17) rectangle (7.5,-16.5);
\draw[black](5,-17.5) rectangle (5.5,-17);
\draw[black](5.5,-17.5) rectangle (6,-17);
\filldraw[black](6,-17.5) rectangle (6.5,-17);
\draw[black](6.5,-17.5) rectangle (7,-17);
\filldraw[black] (5,-18) rectangle (5.5,-17.5);
\draw[black] (5.5,-18) rectangle (6,-17.5);
\draw[black] (6,-18) rectangle (6.5,-17.5);
\filldraw[black] (6.5,-18) rectangle (7,-17.5);
\draw[black] (5,-18.5) rectangle (5.5,-18);
\filldraw[black] (5.5,-18.5) rectangle (6,-18);
\draw[black] (5,-19) rectangle (5.5,-18.5);
\draw[black] (5.5,-19) rectangle (6,-18.5);
\filldraw[black] (5,-19.5) rectangle (5.5,-19);
\draw[black] (5.5,-19.5) rectangle (6,-19);
\draw[black] (5,-20) rectangle (5.5,-19.5);
\filldraw[black] (5.5,-20) rectangle (6,-19.5);
\draw[black] (5,-20.5) rectangle (5.5,-20);
\draw[black] (5.5,-20.5) rectangle (6,-20);
\filldraw[black] (5,-21) rectangle (5.5,-20.5);
\draw[black] (5.5,-21) rectangle (6,-20.5);
\draw[black](5,-21.5) rectangle (5.5,-21);
\draw[black](5,-22) rectangle (5.5,-21.5);
\filldraw[black](5,-22.5) rectangle (5.5,-22);
\draw[black](5,-23) rectangle (5.5,-22.5);
\draw[black](5,-23.5) rectangle (5.5,-23);
\filldraw[black](5,-24) rectangle (5.5,-23.5);
\node(a) at (0.75,-14.75) [label=\footnotesize{Since $\omega_{-6}=0$, remove a}]{};
\node(a) at (0.75,-15.5) [label=\footnotesize{point from the bottom}]{};
\node(a) at (0.75,-16.25) [label=\footnotesize{of columns 1 through 6.}]{};

\node(a) at(2,-14)[label=\large{$\longleftarrow$}]{};
\filldraw[red] (-8,-12) rectangle (-7.5,-11.5);
\draw[black](-7.5,-12) rectangle (-7,-11.5);
\draw[black] (-7,-12) rectangle (-6.5,-11.5);
\filldraw[black] (-6.5,-12) rectangle (-6,-11.5);
\draw[black] (-6,-12) rectangle (-5.5,-11.5);
\draw[black] (-5.5,-12) rectangle (-5,-11.5);
\filldraw[black] (-5,-12) rectangle (-4.5,-11.5);
\draw[black] (-4.5,-12) rectangle (-4,-11.5);
\draw[black] (-8,-12.5)rectangle(-7.5,-12);
\filldraw[red] (-7.5,-12.5) rectangle (-7,-12);
\draw[black](-7,-12.5) rectangle (-6.5,-12);
\draw[black] (-6.5,-12.5) rectangle (-6,-12);3
\filldraw[black] (-6,-12.5) rectangle (-5.5,-12);
\draw[black] (-5.5,-12.5) rectangle (-5,-12);
\draw[black] (-5,-12.5) rectangle (-4.5,-12);
\filldraw[black] (-4.5,-12.5) rectangle (-4,-12);
\draw[black] (-8,-13) rectangle (-7.5,-12.5);
\draw[black](-7.5,-13) rectangle (-7,-12.5);
\filldraw[red] (-7,-13) rectangle (-6.5,-12.5);
\draw[black] (-6.5,-13) rectangle (-6,-12.5);
\draw[black] (-6,-13) rectangle (-5.5,-12.5);
\filldraw[black] (-5.5,-13) rectangle (-5,-12.5);
\draw[black] (-5,-13) rectangle (-4.5,-12.5);
\filldraw[black] (-8,-13.5)rectangle(-7.5,-13);
\draw[black] (-7.5,-13.5) rectangle (-7,-13);
\draw[black](-7,-13.5) rectangle (-6.5,-13);
\filldraw[red] (-6.5,-13.5) rectangle (-6,-13);
\draw[black] (-6,-13.5) rectangle (-5.5,-13);
\draw[black] (-5.5,-13.5) rectangle (-5,-13);
\filldraw[black] (-5,-13.5) rectangle (-4.5,-13);
\draw[black] (-8,-14) rectangle (-7.5,-13.5);
\filldraw[black](-7.5,-14) rectangle (-7,-13.5);
\draw[black] (-7,-14) rectangle (-6.5,-13.5);
\draw[black] (-6.5,-14) rectangle (-6,-13.5);
\filldraw[red] (-6,-14) rectangle (-5.5,-13.5);
\draw[black] (-5.5,-14) rectangle (-5,-13.5);
\draw[black] (-5,-14) rectangle (-4.5,-13.5);
\draw[black] (-8,-14.5)rectangle(-7.5,-14);
\draw[black] (-7.5,-14.5) rectangle (-7,-14);
\filldraw[black](-7,-14.5) rectangle (-6.5,-14);
\draw[black] (-6.5,-14.5) rectangle (-6,-14);
\draw[black] (-6,-14.5) rectangle (-5.5,-14);
\filldraw[red](-5.5,-14.5) rectangle (-5,-14);
\draw[black] (-5,-14.5) rectangle (-4.5,-14);
\filldraw[black] (-8,-15)rectangle(-7.5,-14.5);
\draw[black] (-7.5,-15) rectangle (-7,-14.5);
\draw[black](-7,-15) rectangle (-6.5,-14.5);
\filldraw[black] (-6.5,-15) rectangle (-6,-14.5);
\draw[black] (-6,-15) rectangle (-5.5,-14.5);
\draw[black] (-5.5,-15) rectangle (-5,-14.5);
\filldraw[red] (-5,-15) rectangle (-4.5,-14.5);
\draw[black] (-8,-15.5) rectangle (-7.5,-15);
\filldraw[black](-7.5,-15.5) rectangle (-7,-15);
\draw[black] (-7,-15.5) rectangle (-6.5,-15);
\draw[black] (-6.5,-15.5) rectangle (-6,-15);
\filldraw[black] (-6,-15.5) rectangle (-5.5,-15);
\draw[black] (-8,-16)rectangle(-7.5,-15.5);
\draw[black] (-7.5,-16) rectangle (-7,-15.5);
\filldraw[black](-7,-16) rectangle (-6.5,-15.5);
\draw[black] (-6.5,-16) rectangle (-6,-15.5);
\draw[black] (-6,-16) rectangle (-5.5,-15.5);
\filldraw[black](-8,-16.5) rectangle (-7.5,-16);
\draw[black] (-7.5,-16.5) rectangle (-7,-16);
\draw[black](-7,-16.5) rectangle (-6.5,-16);
\filldraw[black](-6.5,-16.5) rectangle (-6,-16);
\draw[black] (-8,-17) rectangle (-7.5,-16.5);
\filldraw[black](-7.5,-17) rectangle (-7,-16.5);
\draw[black](-8,-17.5) rectangle (-7.5,-17);
\draw[black](-7.5,-17.5) rectangle (-7,-17);
\filldraw[black] (-8,-18) rectangle (-7.5,-17.5);
\draw[black] (-7.5,-18) rectangle (-7,-17.5);
\draw[black] (-8,-18.5) rectangle (-7.5,-18);
\filldraw[black] (-7.5,-18.5) rectangle (-7,-18);
\draw[black] (-8,-19) rectangle (-7.5,-18.5);
\filldraw[black] (-8,-19.5) rectangle (-7.5,-19);
\draw[black](-8,-20) rectangle (-7.5,-19.5);
\draw[black](-8,-20.5) rectangle (-7.5,-20);
\filldraw[black](-8,-21) rectangle (-7.5,-20.5);
\draw[black](-8,-21.5) rectangle (-7.5,-21);
\draw[black](-8,-22) rectangle (-7.5,-21.5);
\filldraw[black](-8,-22.5) rectangle (-7.5,-22);
 \node (a) at (-6.5,-9)[label=\large{$\Bigg\downarrow$}]{};
\node (a) at (-7.5,-8) [label=\small{$\psi^{(0)}_{29}$}]{};
\node(a) at (-2.5,-22) [label=\small{$\begin{pmatrix}
2&2&1&1&0&0&0\\
7&4&2&2&1&0&0
\end{pmatrix}$}]{};
\node(a) at(-10,-15)[label={$n=29$}]{};
\end{tikzpicture}
\caption{A state $\underline{\omega} \in \mathcal{H}^0$ and its image in $\text{GFP}_{D_3,D_3,0}(29)$ }
\end{subfigure}
\vspace{5mm}

\begin{subfigure}[b]{0.75\textwidth}
    \centering
    \vspace{8mm}
    \begin{tikzpicture}[scale=0.4]
\draw[thick, <-] (-14,0)--(-0.5,0);
\foreach \x in {-13,-12,-11,-10,-9,-8,-7,-6,-5,-4,-3,-2,-1}
    \draw[thick, -](\x cm, 2pt)--(\x cm, -2pt) node[anchor=north]{\tiny{$\x$}};
 
\filldraw[black](-2,0.5) circle (4pt);   
\filldraw[black](-2,1) circle (4pt);  
\filldraw[black](-3,0.5) circle (4pt);   
\filldraw[black](-3,1) circle (4pt);  
\filldraw [black] (-4,0.5) circle (4pt);
\filldraw [black] (-4,1) circle (4pt);
\filldraw [black] (-5,0.5) circle (4pt);
\filldraw [black] (-6,0.5) circle (4pt);
\filldraw [black] (-8,0.5) circle (4pt);
\filldraw[black](-10,0.5) circle (4pt);
\filldraw[black](-13,0.5) circle (4pt);
\filldraw [black] (-13.5,0.4) circle (1pt);
\filldraw [black] (-13.7,0.4) circle (1pt);
\filldraw [black] (-13.9,0.4) circle (1pt);
\node (a) at (-6.5,-2.5) [label=\small{$\underline{\omega}=(0,2,2,2,1,1,0,1,0,1,0,0,1...)$}]{};
\node(a) at(2,0)[label=\large{$\longrightarrow$}]{};
\node(a) at(-10,1.5)[label={$n=25$}]{};
\draw [decorate,decoration={brace,amplitude=10pt},xshift=-4pt,yshift=0pt]
(9.25,4.25) -- (9.25,3.25) node [black,midway,xshift=-0.6cm] 
{}; 
\node(a) at (11.5,2.75) [label=\small{$ \omega_{-8}=1$}]{};
\draw[red](5,4.5) rectangle (5.5,5);
\draw[black] (5,4) rectangle (5.5,4.5);
\filldraw[red](5.5,4) rectangle (6,4.5);
\draw[black] (6,4) rectangle (6.5,4.5);
\draw[black] (6.5,4) rectangle (7,4.5);
\filldraw[black] (7,4) rectangle (7.5,4.5);
\draw[black] (7.5,4) rectangle (8,4.5);
\draw[black] (8,4) rectangle (8.5,4.5);
\filldraw[black] (8.5,4) rectangle (9,4.5);
\draw[black] (5,3.5)rectangle(5.5,4);
\draw[black] (5.5,3.5) rectangle (6,4);
\filldraw[red](6,3.5) rectangle (6.5,4);
\draw[black] (6.5,3.5) rectangle (7,4);
\draw[black] (7,3.5) rectangle (7.5,4);
\filldraw[black] (7.5,3.5) rectangle (8,4);
\draw[black] (8,3.5) rectangle (8.5,4);
\draw[black] (8.5,3.5) rectangle (9,4);
\filldraw[black] (5,3) rectangle (5.5,3.5);
\draw[black](5.5,3) rectangle (6,3.5);
\draw[black] (6,3) rectangle (6.5,3.5);
\filldraw[red] (6.5,3) rectangle (7,3.5);
\draw[black] (7,3) rectangle (7.5,3.5);
\draw[black] (7.5,3) rectangle (8,3.5);
\filldraw[black] (8,3) rectangle (8.5,3.5);
\draw[black] (8.5,3) rectangle (9,3.5);
\draw [decorate,decoration={brace,amplitude=10pt},xshift=-4pt,yshift=0pt]
(8.25,2.75) -- (8.25,1.75) node [black,midway,xshift=-0.6cm] 
{}; 
\node(a) at (10.5,1.25) [label=\small{$ \omega_{-6}=1$}]{};
\draw[black] (5,2.5)rectangle(5.5,3);
\filldraw[black] (5.5,2.5) rectangle (6,3);
\draw[black](6,2.5) rectangle (6.5,3);
\draw[black] (6.5,2.5) rectangle (7,3);
\filldraw[red] (7,2.5) rectangle (7.5,3);
\draw[black] (7.5,2.5) rectangle (8,3);
\draw[black] (5,2) rectangle (5.5,2.5);
\draw[black](5.5,2) rectangle (6,2.5);
\filldraw[black] (6,2) rectangle (6.5,2.5);
\draw[black] (6.5,2) rectangle (7,2.5);
\draw[black] (7,2) rectangle (7.5,2.5);
\filldraw[red] (7.5,2) rectangle (8,2.5);
\filldraw[black] (5,1.5)rectangle(5.5,2);
\draw[black] (5.5,1.5) rectangle (6,2);
\draw[black](6,1.5) rectangle (6.5,2);
\filldraw[black] (6.5,1.5) rectangle (7,2);
\draw[black] (7,1.5) rectangle (7.5,2);
\draw[](7.5,1.5) rectangle (8,2);
\draw [decorate,decoration={brace,amplitude=10pt},xshift=-4pt,yshift=0pt]
(7.75,1.25) -- (7.75,0.25) node [black,midway,xshift=-0.6cm] 
{}; 
\node(a) at (10,-0.25) [label=\small{$ \omega_{-5}=1$}]{};
\draw[black] (5,1)rectangle(5.5,1.5);
\filldraw[black] (5.5,1) rectangle (6,1.5);
\draw[black](6,1) rectangle (6.5,1.5);
\draw[black] (6.5,1) rectangle (7,1.5);
\filldraw[black] (7,1) rectangle (7.5,1.5);
\draw[black] (5,0.5) rectangle (5.5,1);
\draw[black](5.5,0.5) rectangle (6,1);
\filldraw[black] (6,0.5) rectangle (6.5,1);
\draw[black] (6.5,0.5) rectangle (7,1);
\draw[black] (7,0.5) rectangle (7.5,1);
\filldraw[black] (5,0)rectangle(5.5,0.5);
\draw[black] (5.5,0) rectangle (6,0.5);
\draw[black](6,0) rectangle (6.5,0.5);
\filldraw[black] (6.5,0) rectangle (7,0.5);
\draw[black] (7,0) rectangle (7.5,0.5);
\draw [decorate,decoration={brace,amplitude=10pt},xshift=-4pt,yshift=0pt]
(7.25,-0.25) -- (7.25,-1.25) node [black,midway,xshift=-0.6cm] 
{}; 
\node(a) at (9.5,-1.75) [label=\small{$ \omega_{-4}=2$}]{};
\draw[black](5,-0.5) rectangle (5.5,0);
\filldraw[black] (5.5,-0.5) rectangle (6,0);
\draw[black](6,-0.5) rectangle (6.5,0);
\draw[black](6.5,-0.5) rectangle (7,0);
\draw[black] (5,-1) rectangle (5.5,-0.5);
\draw[black](5.5,-1) rectangle (6,-0.5);
\filldraw[black] (6,-1) rectangle (6.5,-0.5);
\draw[black] (6.5,-1) rectangle (7,-0.5);
\filldraw[black](5,-1.5) rectangle (5.5,-1);
\draw[black](5.5,-1.5) rectangle (6,-1);
\draw[black](6,-1.5) rectangle (6.5,-1);
\filldraw[black](6.5,-1.5) rectangle (7,-1);
\draw [decorate,decoration={brace,amplitude=10pt},xshift=-4pt,yshift=0pt]
(6.75,-1.75) -- (6.75,-4.25) node [black,midway,xshift=-0.6cm] 
{}; 
\node(a) at (9,-4) [label=\small{$ \omega_{-3}=2$}]{};
\draw[black] (5,-2) rectangle (5.5,-1.5);
\filldraw[black] (5.5,-2) rectangle (6,-1.5);
\draw[black] (6,-2) rectangle (6.5,-1.5);
\draw[black] (5,-2.5) rectangle (5.5,-2);
\draw[black] (5.5,-2.5) rectangle (6,-2);
\filldraw[black] (6,-2.5) rectangle (6.5,-2);
\filldraw[black] (5,-3) rectangle (5.5,-2.5);
\draw[black] (5.5,-3) rectangle (6,-2.5);
\draw[black] (6,-3) rectangle (6.5,-2.5);
\draw[black] (5,-3.5) rectangle (5.5,-3);
\filldraw[black] (5.5,-3.5) rectangle (6,-3);
\draw[black] (6,-3.5) rectangle (6.5,-3);
\draw[black] (5,-4) rectangle (5.5,-3.5);
\draw[black] (5.5,-4) rectangle (6,-3.5);
\filldraw[black] (6,-4) rectangle (6.5,-3.5);
\filldraw[black] (5,-4.5) rectangle (5.5,-4);
\draw[black] (5.5,-4.5) rectangle (6,-4);
\draw[black] (6,-4.5) rectangle (6.5,-4);

\draw[black] (5,-5) rectangle (5.5,-4.5);
\filldraw[black] (5.5,-5) rectangle (6,-4.5);
\draw[black] (5,-5.5) rectangle (5.5,-5);
\draw[black] (5.5,-5.5) rectangle (6,-5);
\filldraw[black] (5,-6) rectangle (5.5,-5.5);
\draw[black] (5.5,-6) rectangle (6,-5.5);
\draw[black](5,-6.5)rectangle (5.5,-6);
\filldraw[black](5.5,-6.5)rectangle (6,-6);
\draw[black](5,-7)rectangle(5.5,-6.5);
\draw[black](5.5,-7)rectangle(6,-6.5);
\filldraw[black] (5,-7.5) rectangle (5.5,-7);
\draw[black] (5.5,-7.5) rectangle (6,-7);
\draw [decorate,decoration={brace,amplitude=10pt},xshift=-4pt,yshift=0pt]
(6.25,-4.75) -- (6.25,-7.25) node [black,midway,xshift=-0.6cm] 
{}; 
\node(a) at (8.5,-7) [label=\small{$ \omega_{-2}=2$}]{};

\node (a) at (8,-9.5) [label=\large{$\Big\downarrow$}]{};
\node(a) at (12.25,-8.5) [label=\footnotesize{Since $\omega_{-1}=0$, remove}]{};
\node(a) at (12.25,-9.25) [label=\footnotesize{a point from the }]{};
\node(a) at (12.25,-10) [label=\footnotesize{bottom of column 1.}]{};
\draw[red](5,-9.5) rectangle (5.5,-9);
\draw[black] (5,-10) rectangle (5.5,-9.5);
\filldraw[red](5.5,-10) rectangle (6,-9.5);
\draw[black] (6,-10) rectangle (6.5,-9.5);
\draw[black] (6.5,-10) rectangle (7,-9.5);
\filldraw[black] (7,-10) rectangle (7.5,-9.5);
\draw[black] (7.5,-10) rectangle (8,-9.5);
\draw[black] (8,-10) rectangle (8.5,-9.5);
\filldraw[black] (8.5,-10) rectangle (9,-9.5);
\draw[black] (5,-10.5)rectangle(5.5,-10);
\draw[black] (5.5,-10.5) rectangle (6,-10);
\filldraw[red](6,-10.5) rectangle (6.5,-10);
\draw[black] (6.5,-10.5) rectangle (7,-10);
\draw[black] (7,-10.5) rectangle (7.5,-10);
\filldraw[black] (7.5,-10.5) rectangle (8,-10);
\draw[black] (8,-10.5) rectangle (8.5,-10);
\draw[black] (8.5,-10.5) rectangle (9,-10);
\filldraw[black] (5,-11) rectangle (5.5,-10.5);
\draw[black](5.5,-11) rectangle (6,-10.5);
\draw[black] (6,-11) rectangle (6.5,-10.5);
\filldraw[red] (6.5,-11) rectangle (7,-10.5);
\draw[black] (7,-11) rectangle (7.5,-10.5);
\draw[black] (7.5,-11) rectangle (8,-10.5);
\filldraw[black] (8,-11) rectangle (8.5,-10.5);
\draw[black] (8.5,-11) rectangle (9,-10.5);
\draw[black] (5,-11.5)rectangle(5.5,-11);
\filldraw[black] (5.5,-11.5) rectangle (6,-11);
\draw[black](6,-11.5) rectangle (6.5,-11);
\draw[black] (6.5,-11.5) rectangle (7,-11);
\filldraw[red] (7,-11.5) rectangle (7.5,-11);
\draw[black] (7.5,-11.5) rectangle (8,-11);
\draw[black] (5,-12) rectangle (5.5,-11.5);
\draw[black](5.5,-12) rectangle (6,-11.5);
\filldraw[black] (6,-12) rectangle (6.5,-11.5);
\draw[black] (6.5,-12) rectangle (7,-11.5);
\draw[black] (7,-12) rectangle (7.5,-11.5);
\filldraw[red] (7.5,-12) rectangle (8,-11.5);
\filldraw[black] (5,-12.5)rectangle(5.5,-12);
\draw[black] (5.5,-12.5) rectangle (6,-12);
\draw[black](6,-12.5) rectangle (6.5,-12);
\filldraw[black] (6.5,-12.5) rectangle (7,-12);
\draw[black] (7,-12.5) rectangle (7.5,-12);
\draw[](7.5,-12.5) rectangle (8,-12);
\draw[black] (5,-13)rectangle(5.5,-12.5);
\filldraw[black] (5.5,-13) rectangle (6,-12.5);
\draw[black](6,-13) rectangle (6.5,-12.5);
\draw[black] (6.5,-13) rectangle (7,-12.5);
\filldraw[black] (7,-13) rectangle (7.5,-12.5);
\draw[black] (5,-13.5) rectangle (5.5,-13);
\draw[black](5.5,-13.5) rectangle (6,-13);
\filldraw[black] (6,-13.5) rectangle (6.5,-13);
\draw[black] (6.5,-13.5) rectangle (7,-13);
\draw[black] (7,-13.5) rectangle (7.5,-13);
\filldraw[black] (5,-14)rectangle(5.5,-13.5);
\draw[black] (5.5,-14) rectangle (6,-13.5);
\draw[black](6,-14) rectangle (6.5,-13.5);
\filldraw[black] (6.5,-14) rectangle (7,-13.5);
\draw[black] (7,-14) rectangle (7.5,-13.5);
\draw[black](5,-14.5) rectangle (5.5,-14);
\filldraw[black] (5.5,-14.5) rectangle (6,-14);
\draw[black](6,-14.5) rectangle (6.5,-14);
\draw[black](6.5,-14.5) rectangle (7,-14);
\draw[black] (5,-15) rectangle (5.5,-14.5);
\draw[black](5.5,-15) rectangle (6,-14.5);
\filldraw[black] (6,-15) rectangle (6.5,-14.5);
\draw[black] (6.5,-15) rectangle (7,-14.5);
\filldraw[black](5,-15.5) rectangle (5.5,-15);
\draw[black](5.5,-15.5) rectangle (6,-15);
\draw[black](6,-15.5) rectangle (6.5,-15);
\filldraw[black](6.5,-15.5) rectangle (7,-15);
\draw[black] (5,-16) rectangle (5.5,-15.5);
\filldraw[black] (5.5,-16) rectangle (6,-15.5);
\draw[black] (6,-16) rectangle (6.5,-15.5);
\draw[black] (5,-16.5) rectangle (5.5,-16);
\draw[black] (5.5,-16.5) rectangle (6,-16);
\filldraw[black] (6,-16.5) rectangle (6.5,-16);
\filldraw[black] (5,-17) rectangle (5.5,-16.5);
\draw[black] (5.5,-17) rectangle (6,-16.5);
\draw[black] (6,-17) rectangle (6.5,-16.5);
\draw[black] (5,-17.5) rectangle (5.5,-17);
\filldraw[black] (5.5,-17.5) rectangle (6,-17);
\draw[black] (6,-17.5) rectangle (6.5,-17);
\draw[black] (5,-18) rectangle (5.5,-17.5);
\draw[black] (5.5,-18) rectangle (6,-17.5);
\filldraw[black] (6,-18) rectangle (6.5,-17.5);
\filldraw[black] (5,-18.5) rectangle (5.5,-18);
\draw[black] (5.5,-18.5) rectangle (6,-18);
\draw[black] (6,-18.5) rectangle (6.5,-18);
\draw[black] (5,-19) rectangle (5.5,-18.5);
\filldraw[black] (5.5,-19) rectangle (6,-18.5);
\draw[black] (5,-19.5) rectangle (5.5,-19);
\draw[black] (5.5,-19.5) rectangle (6,-19);
\filldraw[black] (5,-20) rectangle (5.5,-19.5);
\draw[black] (5.5,-20) rectangle (6,-19.5);
\draw[black](5,-20.5)rectangle (5.5,-20);
\filldraw[black](5.5,-20.5)rectangle (6,-20);
\node(a) at(2,-14)[label=\large{$\longleftarrow$}]{};
\node(a) at (0.5,-14.75) [label=\footnotesize{Since $\omega_{-7}=0$, remove a}]{};
\node(a) at (0.5,-15.5) [label=\footnotesize{point from the bottom}]{};
\node(a) at (0.5,-16.25) [label=\footnotesize{of columns 1 through 7.}]{};
\draw[red](-8,-9.5) rectangle (-7.5,-9);
\draw[black] (-8,-10) rectangle (-7.5,-9.5);
\filldraw[red](-7.5,-10) rectangle (-7,-9.5);
\draw[black] (-7,-10) rectangle (-6.5,-9.5);
\draw[black] (-6.5,-10) rectangle (-6,-9.5);
\filldraw[black] (-6,-10) rectangle (-5.5,-9.5);
\draw[black] (-5.5,-10) rectangle (-5,-9.5);
\draw[black] (-5,-10) rectangle (-4.5,-9.5);
\filldraw[black] (-4.5,-10) rectangle (-4,-9.5);
\draw[black] (-8,-10.5)rectangle(-7.5,-10);
\draw[black] (-7.5,-10.5) rectangle (-7,-10);
\filldraw[red](-7,-10.5) rectangle (-6.5,-10);
\draw[black] (-6.5,-10.5) rectangle (-6,-10);
\draw[black] (-6,-10.5) rectangle (-5.5,-10);
\filldraw[black] (-5.5,-10.5) rectangle (-5,-10);
\filldraw[black] (-8,-11) rectangle (-7.5,-10.5);
\draw[black](-7.5,-11) rectangle (-7,-10.5);
\draw[black] (-7,-11) rectangle (-6.5,-10.5);
\filldraw[red] (-6.5,-11) rectangle (-6,-10.5);
\draw[black] (-6,-11) rectangle (-5.5,-10.5);
\draw[black] (-8,-11.5)rectangle(-7.5,-11);
\filldraw[black] (-7.5,-11.5) rectangle (-7,-11);
\draw[black](-7,-11.5) rectangle (-6.5,-11);
\draw[black] (-6.5,-11.5) rectangle (-6,-11);
\filldraw[red] (-6,-11.5) rectangle (-5.5,-11);
\draw[black] (-8,-12) rectangle (-7.5,-11.5);
\draw[black](-7.5,-12) rectangle (-7,-11.5);
\filldraw[black] (-7,-12) rectangle (-6.5,-11.5);
\draw[black] (-6.5,-12) rectangle (-6,-11.5);
\filldraw[black] (-8,-12.5)rectangle(-7.5,-12);
\draw[black] (-7.5,-12.5) rectangle (-7,-12);
\draw[black](-7,-12.5) rectangle (-6.5,-12);
\filldraw[black] (-6.5,-12.5) rectangle (-6,-12);
\draw[black] (-8,-13)rectangle(-7.5,-12.5);
\filldraw[black] (-7.5,-13) rectangle (-7,-12.5);
\draw[black](-7,-13) rectangle (-6.5,-12.5);
\draw[black] (-6.5,-13) rectangle (-6,-12.5);
\draw[black] (-8,-13.5) rectangle (-7.5,-13);
\draw[black](-7.5,-13.5) rectangle (-7,-13);
\filldraw[black] (-7,-13.5) rectangle (-6.5,-13);
\draw[black] (-6.5,-13.5) rectangle (-6,-13);
\filldraw[black] (-8,-14)rectangle(-7.5,-13.5);
\draw[black] (-7.5,-14) rectangle (-7,-13.5);
\draw[black](-7,-14) rectangle (-6.5,-13.5);
\filldraw[black] (-6.5,-14) rectangle (-6,-13.5);
\draw[black](-8,-14.5) rectangle (-7.5,-14);
\filldraw[black] (-7.5,-14.5) rectangle (-7,-14);
\draw[black](-7,-14.5) rectangle (-6.5,-14);
\draw[black] (-8,-15) rectangle (-7.5,-14.5);
\draw[black](-7.5,-15) rectangle (-7,-14.5);
\filldraw[black] (-7,-15) rectangle (-6.5,-14.5);
\filldraw[black](-8,-15.5) rectangle (-7.5,-15);
\draw[black](-7.5,-15.5) rectangle (-7,-15);
\draw[black](-7,-15.5) rectangle (-6.5,-15);
\draw[black] (-8,-16) rectangle (-7.5,-15.5);
\filldraw[black] (-7.5,-16) rectangle (-7,-15.5);
\draw[black] (-7,-16) rectangle (-6.5,-15.5);
\draw[black] (-8,-16.5) rectangle (-7.5,-16);
\draw[black] (-7.5,-16.5) rectangle (-7,-16);
\filldraw[black] (-7,-16.5) rectangle (-6.5,-16);
\filldraw[black] (-8,-17) rectangle (-7.5,-16.5);
\draw[black] (-7.5,-17) rectangle (-7,-16.5);
\draw[black] (-8,-17.5) rectangle (-7.5,-17);
\filldraw[black] (-7.5,-17.5) rectangle (-7,-17);
\draw[black] (-8,-18) rectangle (-7.5,-17.5);
\draw[black] (-7.5,-18) rectangle (-7,-17.5);
\filldraw[black] (-8,-18.5) rectangle (-7.5,-18);
\draw[black] (-7.5,-18.5) rectangle (-7,-18);
\draw[black] (-8,-19) rectangle (-7.5,-18.5);
\filldraw[black] (-7.5,-19) rectangle (-7,-18.5);

 \node (a) at (-6.5,-9)[label=\large{$\Bigg\downarrow$}]{};
\node (a) at (-7.5,-8) [label=\small{$\psi^{(1)}_{25}$}]{};
\node(a) at (-6,-22) [label=\small{$\begin{pmatrix}
-&2&1&0&0\\
6&6&4&2&0
\end{pmatrix}$}]{};
\node(a) at(-10,-13)[label={$n=25$}]{};
\end{tikzpicture}
    \caption{A state $\underline{\omega} \in \mathcal{H}^{-1}$ and its image in $\text{GFP}_{D_3,D_3,-1}(25)$ }
\end{subfigure}
\caption{The bijections $\psi^{(0)}_{29}$, $\psi^{(1)}_{25}$ and $\psi_{19}^{(2)}$.}
    \label{k-exc bijections from omega to gfp}
\end{figure}

\newpage
\begin{figure}[H]\ContinuedFloat
\centering
\vspace{5mm}
\begin{subfigure}[b]{0.75\textwidth}
    \centering
   \begin{tikzpicture}[scale=0.4]
\draw[thick, <-] (-13,0)--(-0.5,0);
\foreach \x in {-12,-11,-10,-9,-8,-7,-6,-5,-4,-3,-2,-1}
    \draw[thick, -](\x cm, 2pt)--(\x cm, -2pt) node[anchor=north]{\tiny{$\x$}};
 
\filldraw[black](-1,0.5) circle (4pt);   
\filldraw[black](-1,1) circle (4pt);  
\filldraw[black](-1,1.5) circle (4pt);
\filldraw[black](-3,0.5) circle (4pt);   
\filldraw[black](-3,1) circle (4pt);  
\filldraw [black] (-4,0.5) circle (4pt);
\filldraw [black] (-6,0.5) circle (4pt);
\filldraw [black] (-7,0.5) circle (4pt);
\filldraw [black] (-8,0.5) circle (4pt);
\filldraw[black](-11,0.5) circle (4pt);
\filldraw [black] (-12.5,0.4) circle (1pt);
\filldraw [black] (-12.7,0.4) circle (1pt);
\filldraw [black] (-12.9,0.4) circle (1pt);
\node (a) at (-6.5,-2.5) [label=\small{$\underline{\omega}=(3,0,2,1,0,1,1,1,0,0,1,0,...)$}]{};
\node(a) at(2,0)[label=\large{$\longrightarrow$}]{};
\node(a) at(-10,1.5)[label={$n=19$}]{};
\draw [decorate,decoration={brace,amplitude=10pt},xshift=-4pt,yshift=0pt]
(8.75,4.25) -- (8.75,3.25) node [black,midway,xshift=-0.6cm] 
{}; 
\node(a) at (11,2.75) [label=\small{$ \omega_{-7}=1$}]{};
\draw[red](5,5)rectangle(5.5,5.5);
\draw[black](5,4.5)rectangle(5.5,5);
\draw[red](5.5,4.5)rectangle (6,5);
\draw[black] (5,4) rectangle (5.5,4.5);
\draw[black](5.5,4) rectangle (6,4.5);
\filldraw[red] (6,4) rectangle (6.5,4.5);
\draw[black] (6.5,4) rectangle (7,4.5);
\draw[black] (7,4) rectangle (7.5,4.5);
\filldraw[black] (7.5,4) rectangle (8,4.5);
\draw[black] (8,4) rectangle (8.5,4.5);
\filldraw[black] (5,3.5)rectangle(5.5,4);
\draw[black] (5.5,3.5) rectangle (6,4);
\draw[black](6,3.5) rectangle (6.5,4);
\filldraw[red] (6.5,3.5) rectangle (7,4);
\draw[black] (7,3.5) rectangle (7.5,4);
\draw[black] (7.5,3.5) rectangle (8,4);
\filldraw[black] (8,3.5) rectangle (8.5,4);
\draw[black] (5,3) rectangle (5.5,3.5);
\filldraw[black](5.5,3) rectangle (6,3.5);
\draw[black] (6,3) rectangle (6.5,3.5);
\draw[black] (6.5,3) rectangle (7,3.5);
\filldraw[red] (7,3) rectangle (7.5,3.5);
\draw[black] (7.5,3) rectangle (8,3.5);
\draw[black] (8,3) rectangle (8.5,3.5);
\draw [decorate,decoration={brace,amplitude=10pt},xshift=-4pt,yshift=0pt]
(8.25,2.75) -- (8.25,1.75) node [black,midway,xshift=-0.6cm] 
{}; 
\node(a) at (10.5,1.25) [label=\small{$ \omega_{-6}=1$}]{};
\draw[black] (5,2.5)rectangle(5.5,3);
\draw[black] (5.5,2.5) rectangle (6,3);
\filldraw[black](6,2.5) rectangle (6.5,3);
\draw[black] (6.5,2.5) rectangle (7,3);
\draw[black] (7,2.5) rectangle (7.5,3);
\filldraw[red] (7.5,2.5) rectangle (8,3);
\filldraw[black] (5,2) rectangle (5.5,2.5);
\draw[black](5.5,2) rectangle (6,2.5);
\draw[black] (6,2) rectangle (6.5,2.5);
\filldraw[black] (6.5,2) rectangle (7,2.5);
\draw[black] (7,2) rectangle (7.5,2.5);
\draw[black] (7.5,2) rectangle (8,2.5);
\draw[black] (5,1.5)rectangle(5.5,2);
\filldraw[black] (5.5,1.5) rectangle (6,2);
\draw[black](6,1.5) rectangle (6.5,2);
\draw[black] (6.5,1.5) rectangle (7,2);
\filldraw[black] (7,1.5) rectangle (7.5,2);
\draw[black](7.5,1.5) rectangle (8,2);
\draw [decorate,decoration={brace,amplitude=10pt},xshift=-4pt,yshift=0pt]
(7.25,1.25) -- (7.25,0.25) node [black,midway,xshift=-0.6cm] 
{}; 
\node(a) at (9.5,-0.25) [label=\small{$ \omega_{-4}=1$}]{};
\draw[black] (5,1)rectangle(5.5,1.5);
\draw[black] (5.5,1) rectangle (6,1.5);
\filldraw[black](6,1) rectangle (6.5,1.5);
\draw[black] (6.5,1) rectangle (7,1.5);
\filldraw[black] (5,0.5) rectangle (5.5,1);
\draw[black](5.5,0.5) rectangle (6,1);
\draw[black] (6,0.5) rectangle (6.5,1);
\filldraw[black] (6.5,0.5) rectangle (7,1);
\draw[black] (5,0)rectangle(5.5,0.5);
\filldraw[black] (5.5,0) rectangle (6,0.5);
\draw[black](6,0) rectangle (6.5,0.5);
\draw[black] (6.5,0) rectangle (7,0.5);
\draw [decorate,decoration={brace,amplitude=10pt},xshift=-4pt,yshift=0pt]
(6.75,-0.25) -- (6.75,-2.75) node [black,midway,xshift=-0.6cm] 
{}; 
\node(a) at (9,-2.5) [label=\small{$ \omega_{-3}=2$}]{};
\draw[black](5,-0.5) rectangle (5.5,0);
\draw[black] (5.5,-0.5) rectangle (6,0);
\filldraw[black](6,-0.5) rectangle (6.5,0);
\filldraw[black] (5,-1) rectangle (5.5,-0.5);
\draw[black](5.5,-1) rectangle (6,-0.5);
\draw[black] (6,-1) rectangle (6.5,-0.5);
\draw[black](5,-1.5) rectangle (5.5,-1);
\filldraw[black](5.5,-1.5) rectangle (6,-1);
\draw[black](6,-1.5) rectangle (6.5,-1);
\draw[black] (5,-2) rectangle (5.5,-1.5);
\draw[black] (5.5,-2) rectangle (6,-1.5);
\filldraw[black] (6,-2) rectangle (6.5,-1.5);
\filldraw[black] (5,-2.5) rectangle (5.5,-2);
\draw[black] (5.5,-2.5) rectangle (6,-2);
\draw[black] (6,-2.5) rectangle (6.5,-2);
\draw[black] (5,-3) rectangle (5.5,-2.5);
\filldraw[black] (5.5,-3) rectangle (6,-2.5);
\draw[black] (6,-3) rectangle (6.5,-2.5);
\draw [decorate,decoration={brace,amplitude=10pt},xshift=-4pt,yshift=0pt]
(5.75,-3.25) -- (5.75,-7.25) node [black,midway,xshift=-0.6cm] 
{}; 
\node(a) at (8,-6.25) [label=\small{$ \omega_{-1}=3$}]{};
\draw[black] (5,-3.5) rectangle (5.5,-3);
\filldraw[black] (5,-4) rectangle (5.5,-3.5);
\draw[black] (5,-4.5) rectangle (5.5,-4);
\draw[black] (5,-5) rectangle (5.5,-4.5);
\filldraw[black] (5,-5.5) rectangle (5.5,-5);
\draw[black] (5,-6) rectangle (5.5,-5.5);
\draw[black](5,-6.5)rectangle (5.5,-6);
\filldraw[black](5,-7)rectangle(5.5,-6.5);
\draw[black] (5,-7.5) rectangle (5.5,-7);
\node (a) at (8,-9.5) [label=\large{$\Big\downarrow$}]{};
\node(a) at (12.25,-8.5) [label=\footnotesize{Since $\omega_{-2}=0$, remove a}]{};
\node(a) at (12.25,-9.25) [label=\footnotesize{point from the bottom}]{};
\node(a) at (12.25,-10) [label=\footnotesize{of columns 1 and 2.}]{};
\draw[red](5,-11)rectangle(5.5,-10.5);
\draw[black](5,-11.5)rectangle(5.5,-11);
\draw[red](5.5,-11.5)rectangle (6,-11);
\draw[black] (5,-12) rectangle (5.5,-11.5);
\draw[black](5.5,-12) rectangle (6,-11.5);
\filldraw[red] (6,-12) rectangle (6.5,-11.5);
\draw[black] (6.5,-12) rectangle (7,-11.5);
\draw[black] (7,-12) rectangle (7.5,-11.5);
\filldraw[black] (7.5,-12) rectangle (8,-11.5);
\draw[black] (8,-12) rectangle (8.5,-11.5);
\filldraw[black] (5,-12.5)rectangle(5.5,-12);
\draw[black] (5.5,-12.5) rectangle (6,-12);
\draw[black](6,-12.5) rectangle (6.5,-12);
\filldraw[red] (6.5,-12.5) rectangle (7,-12);
\draw[black] (7,-12.5) rectangle (7.5,-12);
\draw[black] (7.5,-12.5) rectangle (8,-12);
\filldraw[black] (8,-12.5) rectangle (8.5,-12);
\draw[black] (5,-13) rectangle (5.5,-12.5);
\filldraw[black](5.5,-13) rectangle (6,-12.5);
\draw[black] (6,-13) rectangle (6.5,-12.5);
\draw[black] (6.5,-13) rectangle (7,-12.5);
\filldraw[red] (7,-13) rectangle (7.5,-12.5);
\draw[black] (7.5,-13) rectangle (8,-12.5);
\draw[black] (8,-13) rectangle (8.5,-12.5);
\draw[black] (5,-13.5)rectangle(5.5,-13);
\draw[black] (5.5,-13.5) rectangle (6,-13);
\filldraw[black](6,-13.5) rectangle (6.5,-13);
\draw[black] (6.5,-13.5) rectangle (7,-13);
\draw[black] (7,-13.5) rectangle (7.5,-13);
\filldraw[red] (7.5,-13.5) rectangle (8,-13);
\filldraw[black] (5,-14) rectangle (5.5,-13.5);
\draw[black](5.5,-14) rectangle (6,-13.5);
\draw[black] (6,-14) rectangle (6.5,-13.5);
\filldraw[black] (6.5,-14) rectangle (7,-13.5);
\draw[black] (7,-14) rectangle (7.5,-13.5);
\draw[black] (7.5,-14) rectangle (8,-13.5);
\draw[black] (5,-14.5)rectangle(5.5,-14);
\filldraw[black] (5.5,-14.5) rectangle (6,-14);
\draw[black](6,-14.5) rectangle (6.5,-14);
\draw[black] (6.5,-14.5) rectangle (7,-14);
\filldraw[black] (7,-14.5) rectangle (7.5,-14);
\draw[black](7.5,-14.5) rectangle (8,-14);
\draw[black] (5,-15)rectangle(5.5,-14.5);
\draw[black] (5.5,-15) rectangle (6,-14.5);
\filldraw[black](6,-15) rectangle (6.5,-14.5);
\draw[black] (6.5,-15) rectangle (7,-14.5);
\filldraw[black] (5,-15.5) rectangle (5.5,-15);
\draw[black](5.5,-15.5) rectangle (6,-15);
\draw[black] (6,-15.5) rectangle (6.5,-15);
\filldraw[black] (6.5,-15.5) rectangle (7,-15);
\draw[black] (5,-16)rectangle(5.5,-15.5);
\filldraw[black] (5.5,-16) rectangle (6,-15.5);
\draw[black](6,-16) rectangle (6.5,-15.5);
\draw[black] (6.5,-16) rectangle (7,-15.5);
\draw[black](5,-16.5) rectangle (5.5,-16);
\draw[black] (5.5,-16.5) rectangle (6,-16);
\filldraw[black](6,-16.5) rectangle (6.5,-16);
\filldraw[black] (5,-17) rectangle (5.5,-16.5);
\draw[black](5.5,-17) rectangle (6,-16.5);
\draw[black] (6,-17) rectangle (6.5,-16.5);
\draw[black](5,-17.5) rectangle (5.5,-17);
\filldraw[black](5.5,-17.5) rectangle (6,-17);
\draw[black](6,-17.5) rectangle (6.5,-17);
\draw[black] (5,-18) rectangle (5.5,-17.5);
\draw[black] (5.5,-18) rectangle (6,-17.5);
\filldraw[black] (6,-18) rectangle (6.5,-17.5);
\filldraw[black] (5,-18.5) rectangle (5.5,-18);
\draw[black] (5,-19) rectangle (5.5,-18.5);
\draw[black] (5,-19.5) rectangle (5.5,-19);
\filldraw[black] (5,-20) rectangle (5.5,-19.5);
\draw[black] (5,-20.5) rectangle (5.5,-20);
\draw[black] (5,-21) rectangle (5.5,-20.5);
\filldraw[black] (5,-21.5) rectangle (5.5,-21);

\node (a) at (8,-20.5) [label=\small{$ \omega_{-5}=0$}]{};
\node(a) at (0.5,-14.75) [label=\footnotesize{Since $\omega_{-5}=0$, remove a}]{};
\node(a) at (0.5,-15.5) [label=\footnotesize{point from the bottom}]{};
\node(a) at (0.5,-16.25) [label=\footnotesize{of columns 1 through 5.}]{};
\node(a) at(2,-14)[label=\large{$\longleftarrow$}]{};
\draw[red](-8,-11)rectangle(-7.5,-10.5);
\draw[black](-8,-11.5)rectangle(-7.5,-11);
\draw[red](-7.5,-11.5)rectangle (-7,-11);
\draw[black] (-8,-12) rectangle (-7.5,-11.5);
\draw[black](-7.5,-12) rectangle (-7,-11.5);
\filldraw[red] (-7,-12) rectangle (-6.5,-11.5);
\draw[black] (-6.5,-12) rectangle (-6,-11.5);
\draw[black] (-6,-12) rectangle (-5.5,-11.5);
\filldraw[black] (-5.5,-12) rectangle (-5,-11.5);
\draw[black] (-5,-12) rectangle (-4.5,-11.5);
\filldraw[black] (-8,-12.5)rectangle(-7.5,-12);
\draw[black] (-7.5,-12.5) rectangle (-7,-12);
\draw[black](-7,-12.5) rectangle (-6.5,-12);
\filldraw[red] (-6.5,-12.5) rectangle (-6,-12);
\draw[black] (-6,-12.5) rectangle (-5.5,-12);
\draw[black] (-5.5,-12.5) rectangle (-5,-12);
\filldraw[black] (-5,-12.5) rectangle (-4.5,-12);
\draw[black] (-8,-13) rectangle (-7.5,-12.5);
\filldraw[black](-7.5,-13) rectangle (-7,-12.5);
\draw[black] (-7,-13) rectangle (-6.5,-12.5);
\draw[black] (-6.5,-13) rectangle (-6,-12.5);
\filldraw[red] (-6,-13) rectangle (-5.5,-12.5);
\draw[black] (-5.5,-13) rectangle (-5,-12.5);
\draw[black] (-8,-13.5)rectangle(-7.5,-13);
\draw[black] (-7.5,-13.5) rectangle (-7,-13);
\filldraw[black](-7,-13.5) rectangle (-6.5,-13);
\draw[black] (-6.5,-13.5) rectangle (-6,-13);
\draw[black] (-6,-13.5) rectangle (-5.5,-13);
\filldraw[red] (-5.5,-13.5) rectangle (-5,-13);
\filldraw[black] (-8,-14) rectangle (-7.5,-13.5);
\draw[black](-7.5,-14) rectangle (-7,-13.5);
\draw[black] (-7,-14) rectangle (-6.5,-13.5);
\filldraw[black] (-6.5,-14) rectangle (-6,-13.5);
\draw[black] (-8,-14.5)rectangle(-7.5,-14);
\filldraw[black] (-7.5,-14.5) rectangle (-7,-14);
\draw[black](-7,-14.5) rectangle (-6.5,-14);
\draw[black] (-8,-15)rectangle(-7.5,-14.5);
\draw[black] (-7.5,-15) rectangle (-7,-14.5);
\filldraw[black](-7,-15) rectangle (-6.5,-14.5);
\filldraw[black] (-8,-15.5) rectangle (-7.5,-15);
\draw[black](-7.5,-15.5) rectangle (-7,-15);
\draw[black] (-7,-15.5) rectangle (-6.5,-15);
\draw[black] (-8,-16)rectangle(-7.5,-15.5);
\filldraw[black] (-7.5,-16) rectangle (-7,-15.5);
\draw[black](-7,-16) rectangle (-6.5,-15.5);
\draw[black](-8,-16.5) rectangle (-7.5,-16);
\draw[black] (-7.5,-16.5) rectangle (-7,-16);
\filldraw[black](-7,-16.5) rectangle (-6.5,-16);
\filldraw[black] (-8,-17) rectangle (-7.5,-16.5);
\draw[black](-8,-17.5) rectangle (-7.5,-17);
\draw[black] (-8,-18) rectangle (-7.5,-17.5);
\filldraw[black] (-8,-18.5) rectangle (-7.5,-18);
\draw[black] (-8,-19) rectangle (-7.5,-18.5);
\draw[black] (-8,-19.5) rectangle (-7.5,-19);
\filldraw[black] (-8,-20) rectangle (-7.5,-19.5);

 \node (a) at (-6.5,-9)[label=\large{$\Bigg\downarrow$}]{};
\node (a) at (-7.5,-8) [label=\small{$\psi^{(2)}_{19}$}]{};
\node(a) at (-6,-23) [label=\small{$\begin{pmatrix}
-&-&1&1&0&0\\
6&3&3&1&0&0
\end{pmatrix}$}]{};
\node(a) at(-10,-14)[label={$n=19$}]{};
\end{tikzpicture}
    \caption{A state $\underline{\omega} \in \mathcal{H}^{-2}$ and its image in $\text{GFP}_{D_3,D_3,-2}(19)$ }
\end{subfigure}
    \caption{The bijections $\psi^{(0)}_{29}$, $\psi^{(1)}_{25}$ and $\psi_{19}^{(2)}$ (cont.).}
    \label{k-exc bijections from omega to gfp2}
\end{figure}
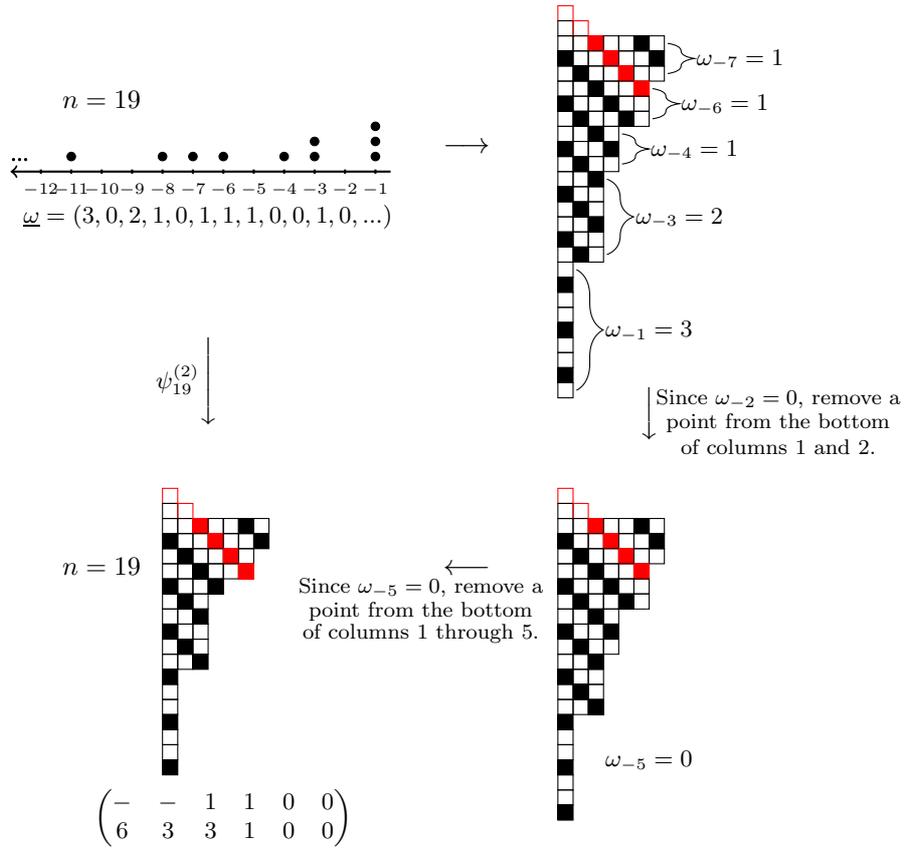
\vspace{5mm}\par\noindent The cases for other offset can be proved in the usual fashion, by proving the equalities of generating functions: \[f_{D_k,D_k,k'}(q) = f_{D_k,D_k,-m}(q)q^{\frac{k\ell(\ell+1)}{2}-m\ell}\quad\text{ if } k'=k\ell-m \text{ with } m\in\{0,1,...k-1\}.\]
\vspace{2mm} \par \noindent These follow naturally from a generalisation of Wright's bijection and the bijections $\phi^e_{\ell,n}, \phi^o_{\ell,n}$ of Section $3.3$, i.e.\ bijections 
\vspace{2mm}
\[\phi_{\ell,n}^{(m)}:\text{GFP}_{D_k,D_k,-m}(n)\rightarrow \text{GFP}_{D_k,D_k,k\ell-m}\left(n+\frac{k\ell(\ell+1)}{2}-m\ell\right)\]
\vspace{2mm} \par \noindent for each $\ell\in\mathbb{Z}$ and $m\in\{1,2,...,k-1\}$ as follows. Given an element $\text{GFP}_{D_k,D_k,-m}(n)$ the map $\phi^{(m)}_{\ell,n}$ adds the points inside the right angled triangle of size $\frac{|k\ell-m|(|k\ell-m|+1)}{2}$ attached to either the left or top edge of its generalised Young diagram, depending on whether $\ell\geq 0$ or $\ell<0$ respectively. It then uses the new leading diagonal implied by the triangle to read off an element of $\text{GFP}_{D_k,D_k,k\ell-m}\left(n+\frac{k\ell(\ell+1)}{2}-m\ell\right)$. See Figure \ref{k-exc bijections offset} for examples when $k=3$.
\newpage \begin{figure}[H]
    \centering
\begin{subfigure}[b]{0.75\textwidth}
    \centering
\begin{tikzpicture}[scale=0.25]
\draw[black](-4,0) rectangle (-3,1);
\draw[black] (-5,0) rectangle (-4,1);
\filldraw[black] (-6,0) rectangle (-5,1);
\draw[black](-7,0) rectangle (-6,1);
\draw[black](-8,0)rectangle(-7,1);
\filldraw[black](-9,0)rectangle (-8,1);
\draw[black](-4,-1)rectangle (-3,0);
\filldraw[black] (-5,-1) rectangle (-4,0);
\draw[black](-6,-1)rectangle(-5,0);
\draw[black](-7,-1)rectangle(-6,0);
\filldraw[black] (-8,-1) rectangle (-7,0);
\filldraw[black](-4,-2) rectangle (-3,-1);
\draw[black] (-5,-2) rectangle (-4,-1);
\draw[black](-6,-2)rectangle(-5,-1);
\filldraw[black](-7,-2)rectangle(-6,-1);
\draw[black] (-4,-3) rectangle (-3,-2);
\draw[black](-5,-3)rectangle(-4,-2);
\filldraw[black](-6,-3)rectangle(-5,-2);
\draw[black](-4,-4) rectangle (-3,-3);
\filldraw[black](-5,-4)rectangle(-4,-3);
\filldraw[black](-4,-5)rectangle(-3,-4);
\node (a) at (-1.5,-2) [label=\large{+}]{};
\filldraw[red] (0,0) rectangle (1,1);
\draw[black] (1,0) rectangle (2,1);
\draw[black] (2,0) rectangle (3,1);
\filldraw[black] (3,0) rectangle (4,1);
\draw [black] (4,0) rectangle (5,1);
\draw[black] (5,0) rectangle (6,1);
\filldraw[black] (6,0) rectangle (7,1);
\draw[black] (7,0) rectangle (8,1);
\draw[black] (8,0) rectangle (9,1);
\filldraw[black](9,0)rectangle(10,1);
\draw[black] (10,0) rectangle (11,1);
\draw[black] (11,0) rectangle (12,1);
\filldraw[black](12,0)rectangle(13,1);
\draw[black](0,-1) rectangle (1,0);
\filldraw[red] (1,-1) rectangle (2,0);
\draw[black] (2,-1) rectangle (3,0);
\draw[black] (3,-1) rectangle (4,0);
\filldraw[black](4,-1) rectangle (5,0);
\draw[black] (5,-1) rectangle (6,0);
\draw[black] (6,-1) rectangle (7,0);
\filldraw[black](7,-1)rectangle (8,0);
\draw[black] (8,-1) rectangle (9,0);
\draw[black] (9,-1) rectangle (10,0);
\draw[black] (0,-2) rectangle (1,-1);
\draw[black] (1,-2) rectangle (2,-1);
\filldraw[red] (2,-2) rectangle (3,-1);
\draw[black] (3,-2) rectangle (4,-1);
\draw[black] (4,-2) rectangle (5,-1);
\filldraw[black] (5,-2) rectangle (6,-1);
\draw[black] (6,-2) rectangle (7,-1);
\draw[black] (7,-2) rectangle (8,-1);
\filldraw[black] (8,-2) rectangle (9,-1);
\draw[black] (9,-2) rectangle (10,-1);
\filldraw[black] (0,-3) rectangle (1,-2);
\draw[black](1,-3) rectangle (2,-2);
\draw[black](2,-3) rectangle (3,-2);
\filldraw[red](3,-3) rectangle (4,-2);
\draw[black](4,-3) rectangle (5,-2);
\draw[black](5,-3) rectangle (6,-2);
\filldraw[black](6,-3) rectangle (7,-2);
\draw[black](7,-3) rectangle (8,-2);
\draw[black](8,-3) rectangle (9,-2);
\filldraw[black](9,-3) rectangle (10,-2);
\draw[black] (0,-4) rectangle (1,-3);
\filldraw[black] (1,-4) rectangle (2,-3);
\draw[black](2,-4)rectangle(3,-3);
\draw[black](3,-4)rectangle(4,-3);
\filldraw[red](4,-4)rectangle(5,-3);
\draw[black](5,-4)rectangle(6,-3);
\draw[black](6,-4)rectangle(7,-3);
\filldraw[black](7,-4)rectangle(8,-3);
\draw[black] (0,-5) rectangle (1,-4);
\draw[black] (1,-5) rectangle (2,-4);
\filldraw[black] (2,-5) rectangle (3,-4);
\filldraw[black] (0,-6) rectangle (1,-5);
\draw[black] (1,-6) rectangle (2,-5);
\draw[black] (0,-7) rectangle (1,-6);
\filldraw[black] (1,-7) rectangle (2,-6);
\draw[black] (0,-8) rectangle (1,-7);
\draw[black] (1,-8) rectangle (2,-7);
\filldraw[black] (0,-9) rectangle (1,-8);
\draw[black] (1,-9) rectangle (2,-8);
\draw[black] (0,-10) rectangle (1,-9);
\filldraw[black] (1,-10) rectangle (2,-9);

\node(a) at (5,-15)[label=\small{$\begin{pmatrix}
4&2&2&2&1\\
3&3&1&0&0
\end{pmatrix}$}]{};
\node (a) at (15,-4) [label=\large{$\overset{\phi_{2,23}^{(0)}}{\longrightarrow}$}]{};
\filldraw[red](18,0)rectangle(19,1);
\draw[black](19,0)rectangle (20,1);
\draw[black](20,0)rectangle(21,1);
\filldraw[black](21,0)rectangle(22,1);
\draw[black] (22,0) rectangle (23,1);
\draw[black] (23,0) rectangle (24,1);
\filldraw[black] (24,0) rectangle (25,1);
\draw[black] (25,0) rectangle (26,1);
\draw [black] (26,0) rectangle (27,1);
\filldraw[black] (27,0) rectangle (28,1);
\draw[black] (28,0) rectangle (29,1);
\draw[black] (29,0) rectangle (30,1);
\filldraw[black] (30,0) rectangle (31,1);
\draw[black] (31,0) rectangle (32,1);
\draw[black] (32,0) rectangle (33,1);
\filldraw[black] (33,0) rectangle (34,1);
\draw[black] (34,0) rectangle (35,1);
\draw[black] (35,0) rectangle (36,1);
\filldraw[black] (36,0) rectangle (37,1);
\filldraw[red](19,-1)rectangle(20,0);
\draw[black](20,-1)rectangle(21,0);
\draw[black](21,-1)rectangle (22,0);
\filldraw[black](22,-1) rectangle (23,0);
\draw[black] (23,-1) rectangle (24,0);
\draw[black] (24,-1) rectangle (25,0);
\filldraw[black] (25,-1) rectangle (26,0);
\draw[black](26,-1) rectangle (27,0);
\draw[black] (27,-1) rectangle (28,0);
\filldraw[black] (28,-1) rectangle (29,0);
\draw[black](29,-1) rectangle (30,0);
\draw[black] (30,-1) rectangle (31,0);
\filldraw[black] (31,-1) rectangle (32,0);
\draw[black](32,-1) rectangle (33,0);
\draw[black] (33,-1) rectangle (34,0);
\filldraw[red](20,-2) rectangle (21,-1);
\draw[black](21,-2)rectangle (22,-1);
\draw[black] (22,-2) rectangle (23,-1);
\filldraw[black] (23,-2) rectangle (24,-1);
\draw[black] (24,-2) rectangle (25,-1);
\draw[black] (25,-2) rectangle (26,-1);
\filldraw[black] (26,-2) rectangle (27,-1);
\draw[black] (27,-2) rectangle (28,-1);
\draw[black] (28,-2) rectangle (29,-1);
\filldraw[black] (29,-2) rectangle (30,-1);
\draw[black] (30,-2) rectangle (31,-1);
\draw[black] (31,-2) rectangle (32,-1);
\filldraw[black] (32,-2) rectangle (33,-1);
\draw[black] (33,-2) rectangle (34,-1);
\filldraw[red](21,-3)rectangle(22,-2);
\draw[black] (22,-3) rectangle (23,-2);
\draw[black](23,-3) rectangle (24,-2);
\filldraw[black](24,-3)rectangle(25,-2);
\draw[black] (25,-3) rectangle (26,-2);
\draw[black](26,-3) rectangle (27,-2);
\filldraw[black](27,-3)rectangle(28,-2);
\draw[black] (28,-3) rectangle (29,-2);
\draw[black](29,-3) rectangle (30,-2);
\filldraw[black](30,-3)rectangle(31,-2);
\draw[black] (31,-3) rectangle (32,-2);
\draw[black](32,-3) rectangle (33,-2);
\filldraw[black](33,-3)rectangle(34,-2);
\filldraw[red] (22,-4) rectangle (23,-3);
\draw[black] (23,-4) rectangle (24,-3);
\draw[black] (24,-4) rectangle (25,-3);
\filldraw[black] (25,-4) rectangle (26,-3);
\draw[black] (26,-4) rectangle (27,-3);
\draw[black] (27,-4) rectangle (28,-3);
\filldraw[black] (28,-4) rectangle (29,-3);
\draw[black] (29,-4) rectangle (30,-3);
\draw[black] (30,-4) rectangle (31,-3);
\filldraw[black] (31,-4) rectangle (32,-3);
\filldraw[red] (23,-5) rectangle (24,-4);
\draw[black](24,-5)rectangle(25,-4);
\draw[black](25,-5)rectangle(26,-4);
\filldraw[black](26,-5)rectangle(27,-4);
\draw[black](25,-6)rectangle(26,-5);
\filldraw[red](24,-6)rectangle(25,-5);
\filldraw[red](25,-7)rectangle(26,-6);
\draw[black](24,-7)rectangle(25,-6);
\draw[black](25,-8)rectangle(26,-7);
\draw[black](24,-8)rectangle(25,-7);
\draw[black](25,-9)rectangle(26,-8);
\draw[black](24,-10)rectangle(25,-9);
\filldraw[black](25,-10) rectangle (26,-9);
\filldraw[black](24,-9)rectangle(25,-8);

\node (a) at (26,-15) [label=\small{$\begin{pmatrix}
6&4&4&4&3&1&0&0\\
-&-&-&-&-&-&1&1
\end{pmatrix}$}]{};
\end{tikzpicture}
    \caption{An element of $\text{GFP}_{D_3,D_3,0}(23)$ and its image in $\text{GFP}_{D_3,D_3,6}(32)$ }
\end{subfigure}

\begin{subfigure}[b]{0.75\textwidth}
    \centering
    \vspace{4mm}
    \begin{tikzpicture}[scale=0.25]
\draw[black](-4,0) rectangle (-3,1);
\filldraw[black] (-5,0) rectangle (-4,1);
\draw[black] (-6,0) rectangle (-5,1);
\draw[black](-7,0) rectangle (-6,1);
\filldraw[black] (-8,0) rectangle (-7,1);
\filldraw[black](-4,-1)rectangle (-3,0);
\draw[black] (-5,-1) rectangle (-4,0);
\draw[black] (-6,-1) rectangle (-5,0);
\filldraw[black] (-7,-1) rectangle (-6,0);
\draw[black](-4,-2) rectangle (-3,-1);
\draw[black](-5,-2) rectangle (-4,-1);
\filldraw[black](-6,-2) rectangle (-5,-1);
\draw[black](-4,-3) rectangle (-3,-2);
\filldraw[black](-5,-3) rectangle (-4,-2);
\filldraw[black](-4,-4) rectangle (-3,-3);
\node (a) at (-1.5,-2) [label=\large{+}]{};
\draw[red] (0,1) rectangle (1,2);
\draw[black] (0,0) rectangle (1,1);
\filldraw[red] (1,0) rectangle (2,1);
\draw[black] (2,0) rectangle (3,1);
\draw[black] (3,0) rectangle (4,1);
\filldraw[black] (4,0) rectangle (5,1);
\draw[black] (5,0) rectangle (6,1);
\draw[black](6,0)rectangle(7,1);
\filldraw[black](7,0)rectangle(8,1);
\draw[black](8,0)rectangle(9,1);
\draw[black](9,0)rectangle(10,1);
\filldraw[black](10,0)rectangle(11,1);
\draw[black](11,0)rectangle(12,1);
\draw[black](12,0)rectangle(13,1);
\filldraw[black](13,0)rectangle(14,1);
\draw[black](14,0)rectangle(15,1);
\draw[black](15,0)rectangle(16,1);
\filldraw[black](16,0)rectangle(17,1);
\draw[black](0,-1) rectangle (1,0);
\draw[black](1,-1) rectangle (2,0);
\filldraw[red] (2,-1) rectangle (3,0);
\draw[black] (3,-1) rectangle (4,0);
\draw[black] (4,-1) rectangle (5,0);
\filldraw[black] (5,-1) rectangle (6,0);
\draw[black] (6,-1) rectangle (7,0);
\draw[black](7,-1) rectangle (8,0);
\filldraw[black] (8,-1) rectangle (9,0);
\draw[black] (9,-1) rectangle (10,0);
\draw[black](10,-1) rectangle (11,0);
\filldraw[black] (11,-1) rectangle (12,0);
\draw[black] (12,-1) rectangle (13,0);
\draw[black](13,-1) rectangle (14,0);
\filldraw[black] (14,-1) rectangle (15,0);
\filldraw[black] (0,-2) rectangle (1,-1);
\draw[black] (1,-2) rectangle (2,-1);
\draw[black] (2,-2) rectangle (3,-1);
\filldraw[red] (3,-2) rectangle (4,-1);
\draw[black] (4,-2) rectangle (5,-1);
\draw[black] (5,-2) rectangle (6,-1);
\filldraw[black] (6,-2) rectangle (7,-1);
\draw[black] (7,-2) rectangle (8,-1);
\draw[black] (8,-2) rectangle (9,-1);
\filldraw[black] (9,-2) rectangle (10,-1);
\draw[black](0,-3) rectangle (1,-2);
\filldraw[black] (1,-3) rectangle (2,-2);
\draw[black](2,-3) rectangle (3,-2);
\draw[black](3,-3) rectangle (4,-2);
\filldraw[red](4,-3) rectangle (5,-2);
\draw[black](5,-3)rectangle(6,-2);
\draw[black](6,-3) rectangle (7,-2);
\draw[black] (0,-4) rectangle (1,-3);
\draw[black] (1,-4) rectangle (2,-3);
\filldraw[black] (2,-4) rectangle (3,-3);
\draw[black] (3,-4) rectangle (4,-3);
\draw[black] (4,-4) rectangle (5,-3);
\filldraw[red] (5,-4) rectangle (6,-3);
\draw[black](6,-4) rectangle (7,-3);
\filldraw[black](0,-5)rectangle(1,-4);
\draw[black](1,-5)rectangle(2,-4);
\draw[black](2,-5)rectangle(3,-4);
\filldraw[black](3,-5)rectangle(4,-4);
\draw[black](4,-5)rectangle(5,-4);
\draw[black](5,-5)rectangle(6,-4);
\filldraw[red](6,-5)rectangle(7,-4);
\draw[black](0,-6)rectangle (1,-5);
\filldraw[black](1,-6)rectangle (2,-5);

\node (a) at (7,-11) [label=\small{$\begin{pmatrix}
-&5&4&2&0&0&0\\
2&2&1&1&0&0&0
\end{pmatrix}$}]{};
\node (a) at (20,-2) [label=\large{$\overset{\phi_{1,23}^{(1)}}{\longrightarrow}$}]{};
\filldraw[red](22,0)rectangle(23,1);
\draw[black](23,0)rectangle(24,1);
\draw[black](24,0)rectangle(25,1);
\filldraw[black] (25,0) rectangle (26,1);
\draw[black] (26,0) rectangle (27,1);
\draw[black] (27,0) rectangle (28,1);
\filldraw[black] (28,0) rectangle (29,1);
\draw[black] (29,0) rectangle (30,1);
\draw[black] (30,0) rectangle (31,1);
\filldraw[black](31,0)rectangle(32,1);
\draw[black](32,0)rectangle(33,1);
\draw[black](33,0)rectangle(34,1);
\filldraw[black](34,0)rectangle(35,1);
\draw[black](35,0)rectangle(36,1);
\draw[black](36,0)rectangle(37,1);
\filldraw[black](37,0)rectangle(38,1);
\draw[black](38,0)rectangle(39,1);
\draw[black](39,0)rectangle(40,1);
\filldraw[black](40,0)rectangle(41,1);
\draw[black](41,0)rectangle(42,1);
\draw[black](42,0)rectangle(43,1);
\filldraw[black](43,0)rectangle(44,1);
\filldraw[red](23,-1)rectangle(24,0);
\draw[black](24,-1)rectangle(25,0);
\draw[black](25,-1) rectangle (26,0);
\filldraw[black](26,-1) rectangle (27,0);
\draw[black] (27,-1) rectangle (28,0);
\draw[black] (28,-1) rectangle (29,0);
\filldraw[black] (29,-1) rectangle (30,0);
\draw[black] (30,-1) rectangle (31,0);
\draw[black] (31,-1) rectangle (32,0);
\filldraw[black](32,-1) rectangle (33,0);
\draw[black](33,-1) rectangle (34,0);
\draw[black] (34,-1) rectangle (35,0);
\filldraw[black] (35,-1) rectangle (36,0);
\draw[black] (36,-1) rectangle (37,0);
\draw[black] (37,-1) rectangle (38,0);
\filldraw[black] (38,-1) rectangle (39,0);
\draw[black] (39,-1) rectangle (40,0);
\draw[black] (40,-1) rectangle (41,0);
\filldraw[black] (41,-1) rectangle (42,0);
\filldraw[red](24,-2)rectangle(25,-1);
\draw[black] (25,-2) rectangle (26,-1);
\draw[black] (26,-2) rectangle (27,-1);
\filldraw[black] (27,-2) rectangle (28,-1);
\draw[black] (28,-2) rectangle (29,-1);
\draw[black] (29,-2) rectangle (30,-1);
\filldraw[black] (30,-2) rectangle (31,-1);
\draw[black] (31,-2) rectangle (32,-1);
\draw[black] (32,-2) rectangle (33,-1);
\filldraw[black] (33,-2) rectangle (34,-1);
\draw[black] (34,-2) rectangle (35,-1);
\draw[black] (35,-2) rectangle (36,-1);
\filldraw[black] (36,-2) rectangle (37,-1);
\filldraw[red](25,-3) rectangle (26,-2);
\draw[black] (26,-3) rectangle (27,-2);
\draw[black] (27,-3) rectangle (28,-2);
\filldraw[black] (28,-3) rectangle (29,-2);
\draw[black] (29,-3) rectangle (30,-2);
\draw[black] (30,-3) rectangle (31,-2);
\filldraw[black] (31,-3) rectangle (32,-2);
\draw[black] (32,-3)rectangle(33,-2);
\draw[black] (33,-3) rectangle (34,-2);
\filldraw[red] (26,-4) rectangle (27,-3);
\draw[black] (27,-4) rectangle (28,-3);
\draw[black] (28,-4) rectangle (29,-3);
\filldraw[black] (29,-4) rectangle (30,-3);
\draw[black] (30,-4) rectangle (31,-3);
\draw[black] (31,-4) rectangle (32,-3);
\filldraw[black] (32,-4) rectangle (33,-3);
\draw[black](33,-4) rectangle (34,-3);
\filldraw[red] (27,-5) rectangle (28,-4);
\draw[black] (28,-5) rectangle (29,-4);
\draw[black] (29,-5) rectangle (30,-4);
\filldraw[black] (30,-5) rectangle (31,-4);
\draw[black] (31,-5) rectangle (32,-4);
\draw[black] (32,-5) rectangle (33,-4);
\filldraw[black] (33,-5) rectangle (34,-4);
\draw[black](27,-6)rectangle(28,-5);
\filldraw[red](28,-6)rectangle(29,-5);

\node (a) at (31.5,-11) [label=\small{$\begin{pmatrix}
7&6&4&2&2&2&0\\
-&-&-&-&-&0&0
\end{pmatrix}$}]{};
    \end{tikzpicture}
    \caption{An element of $\text{GFP}_{D_3,D_3,-1}(23)$ and its image in $\text{GFP}_{D_3,D_3,5}(30)$ }
\end{subfigure}
\begin{subfigure}[b]{0.75\textwidth}
    \centering
    \vspace{4mm}
        \begin{tikzpicture}[scale=0.25]
\filldraw[black](-4,0) rectangle (-3,1);
\draw[black] (-5,0) rectangle (-4,1);
\draw[black] (-6,0) rectangle (-5,1);
\filldraw[black](-7,0)rectangle(-6,1);
\draw[black](-4,-1)rectangle (-3,0);
\draw[black] (-5,-1) rectangle (-4,0);
\filldraw[black](-6,-1)rectangle(-5,0);
\draw[black](-4,-2) rectangle (-3,-1);
\filldraw[black](-5,-2)rectangle(-4,-1);
\filldraw[black](-4,-3)rectangle(-3,-2);
\node (a) at (-1.5,-2) [label=\large{+}]{};

\draw[red] (0,2) rectangle (1,3);
\draw[black] (0,1) rectangle (1,2);
\draw[red] (1,1) rectangle (2,2);
\draw[black] (0,0) rectangle (1,1);
\draw[black] (1,0) rectangle (2,1);
\filldraw[red] (2,0) rectangle (3,1);
\draw[black] (3,0) rectangle (4,1);
\draw[black] (4,0) rectangle (5,1);
\filldraw[black] (5,0) rectangle (6,1);
\draw[black](6,0)rectangle(7,1);
\draw[black](7,0)rectangle(8,1);
\filldraw[black](8,0)rectangle(9,1);
\draw[black](9,0)rectangle(10,1);
\draw[black](10,0)rectangle(11,1);
\filldraw[black](11,0)rectangle(12,1);
\draw[black](12,0)rectangle(13,1);
\draw[black](13,0)rectangle(14,1);
\filldraw[black](14,0)rectangle(15,1);
\draw[black](15,0)rectangle(16,1);
\draw[black](16,0)rectangle(17,1);
\filldraw[black](17,0)rectangle(18,1);

\filldraw[black](0,-1) rectangle (1,0);
\draw[black](1,-1) rectangle (2,0);
\draw[black] (2,-1) rectangle (3,0);
\filldraw[red] (3,-1) rectangle (4,0);
\draw[black] (4,-1) rectangle (5,0);
\draw[black] (5,-1) rectangle (6,0);
\filldraw[black] (6,-1) rectangle (7,0);
\draw[black](7,-1) rectangle (8,0);
\draw[black] (8,-1) rectangle (9,0);
\filldraw[black] (9,-1) rectangle (10,0);
\draw[black](10,-1) rectangle (11,0);
\draw[black] (11,-1) rectangle (12,0);
\filldraw[black] (12,-1) rectangle (13,0);
\draw[black] (0,-2) rectangle (1,-1);
\filldraw[black] (1,-2) rectangle (2,-1);
\draw[black] (2,-2) rectangle (3,-1);
\draw[black] (3,-2) rectangle (4,-1);
\filldraw[red] (4,-2) rectangle (5,-1);
\draw[black] (5,-2) rectangle (6,-1);
\draw[black] (6,-2) rectangle (7,-1);
\filldraw[black] (7,-2) rectangle (8,-1);
\draw[black](8,-2) rectangle (9,-1);
\draw[black](0,-3) rectangle (1,-2);
\draw[black] (1,-3) rectangle (2,-2);
\filldraw[black] (2,-3) rectangle (3,-2);
\draw[black](3,-3) rectangle (4,-2);
\draw[black] (4,-3) rectangle (5,-2);
\filldraw[red] (5,-3) rectangle (6,-2);
\draw[black](6,-3) rectangle (7,-2);
\draw[black] (7,-3) rectangle (8,-2);
\filldraw[black] (8,-3) rectangle (9,-2);

\filldraw[black] (0,-4) rectangle (1,-3);
\draw[black] (1,-4) rectangle (2,-3);
\draw[black] (2,-4) rectangle (3,-3);
\filldraw[black] (3,-4) rectangle (4,-3);
\draw[black] (4,-4) rectangle (5,-3);
\draw[black] (5,-4) rectangle (6,-3);
\filldraw[red] (6,-4) rectangle (7,-3);

\draw[black](0,-5)rectangle(1,-4);
\filldraw[black](1,-5)rectangle(2,-4);
\draw[black](2,-5)rectangle(3,-4);
\draw[black](3,-5)rectangle(4,-4);
\filldraw[black](4,-5)rectangle(5,-4);

\draw[black](0,-6)rectangle(1,-5);
\draw[black](1,-6)rectangle(2,-5);
\filldraw[black](2,-6)rectangle(3,-5);

\node (a) at (7,-11) [label=\small{$\begin{pmatrix}
-&-&5&3&1&1&0\\
2&2&2&1&1&0&0
\end{pmatrix}$}]{};
\node (a) at (20,-2) [label=\large{$\overset{\phi_{1,23}^{(2)}}{\longrightarrow}$}]{};
\filldraw[red](22,0)rectangle(23,1);
\draw[black](23,0)rectangle(24,1);
\draw[black](24,0)rectangle(25,1);
\filldraw[black] (25,0) rectangle (26,1);
\draw[black] (26,0) rectangle (27,1);
\draw[black] (27,0) rectangle (28,1);
\filldraw[black] (28,0) rectangle (29,1);
\draw[black] (29,0) rectangle (30,1);
\draw[black] (30,0) rectangle (31,1);
\filldraw[black](31,0)rectangle(32,1);
\draw[black](32,0)rectangle(33,1);
\draw[black](33,0)rectangle(34,1);
\filldraw[black](34,0)rectangle(35,1);
\draw[black](35,0)rectangle(36,1);
\draw[black](36,0)rectangle(37,1);
\filldraw[black](37,0)rectangle(38,1);
\draw[black](38,0)rectangle(39,1);
\draw[black](39,0)rectangle(40,1);
\filldraw[black](40,0)rectangle(41,1);
\draw[black](41,0)rectangle(42,1);
\draw[black](42,0)rectangle(43,1);
\filldraw[black](43,0)rectangle(44,1);
\filldraw[red](23,-1)rectangle(24,0);
\draw[black](24,-1)rectangle(25,0);
\draw[black](25,-1) rectangle (26,0);
\filldraw[black](26,-1) rectangle (27,0);
\draw[black] (27,-1) rectangle (28,0);
\draw[black] (28,-1) rectangle (29,0);
\filldraw[black] (29,-1) rectangle (30,0);
\draw[black] (30,-1) rectangle (31,0);
\draw[black] (31,-1) rectangle (32,0);
\filldraw[black](32,-1) rectangle (33,0);
\draw[black] (33,-1) rectangle (34,0);
\draw[black] (34,-1) rectangle (35,0);
\filldraw[black](35,-1) rectangle (36,0);
\draw[black] (36,-1) rectangle (37,0);
\draw[black] (37,-1) rectangle (38,0);
\filldraw[black](38,-1) rectangle (39,0);

\filldraw[red](24,-2)rectangle(25,-1);
\draw[black] (25,-2) rectangle (26,-1);
\draw[black] (26,-2) rectangle (27,-1);
\filldraw[black] (27,-2) rectangle (28,-1);
\draw[black] (28,-2) rectangle (29,-1);
\draw[black] (29,-2) rectangle (30,-1);
\filldraw[black] (30,-2) rectangle (31,-1);
\draw[black] (31,-2) rectangle (32,-1);
\draw[black] (32,-2) rectangle (33,-1);
\filldraw[black] (33,-2) rectangle (34,-1);
\draw[black] (34,-2) rectangle (35,-1);

\filldraw[red](25,-3) rectangle (26,-2);
\draw[black] (26,-3) rectangle (27,-2);
\draw[black] (27,-3) rectangle (28,-2);
\filldraw[black] (28,-3) rectangle (29,-2);
\draw[black] (29,-3) rectangle (30,-2);
\draw[black] (30,-3) rectangle (31,-2);
\filldraw[black] (31,-3) rectangle (32,-2);
\draw[black] (32,-3) rectangle (33,-2);
\draw[black] (33,-3) rectangle (34,-2);
\filldraw[black] (34,-3) rectangle (35,-2);

\filldraw[red] (26,-4) rectangle (27,-3);
\draw[black] (27,-4) rectangle (28,-3);
\draw[black] (28,-4) rectangle (29,-3);
\filldraw[black] (29,-4) rectangle (30,-3);
\draw[black] (30,-4) rectangle (31,-3);
\draw[black] (31,-4) rectangle (32,-3);
\filldraw[black] (32,-4) rectangle (33,-3);

\draw[black](26,-5) rectangle (27,-4);
\filldraw[red](27,-5) rectangle (28,-4);
\draw[black](28,-5) rectangle (29,-4);
\draw[black](29,-5) rectangle (30,-4);
\filldraw[black](30,-5) rectangle (31,-4);

\draw[black](26,-6) rectangle (27,-5);
\draw[black](27,-6) rectangle (28,-5);
\filldraw[red] (28,-6) rectangle (29,-5);

\node (a) at (31.5,-11) [label=\small{$\begin{pmatrix}
7&5&3&3&2&1&0\\
-&-&-&-&0&0&0
\end{pmatrix}$}]{};

    \end{tikzpicture}
        \caption{An element of $\text{GFP}_{D_3,D_3,-2}(23)$ and its image in $\text{GFP}_{D_3,D_3,4}(28)$ }
\end{subfigure}
    \caption{The bijections $\phi_{2,23}^{(0)}$, $\phi_{1,23}^{(1)}$ and $\phi_{1,23}^{(2)}$.}
    \label{k-exc bijections offset}
\end{figure}
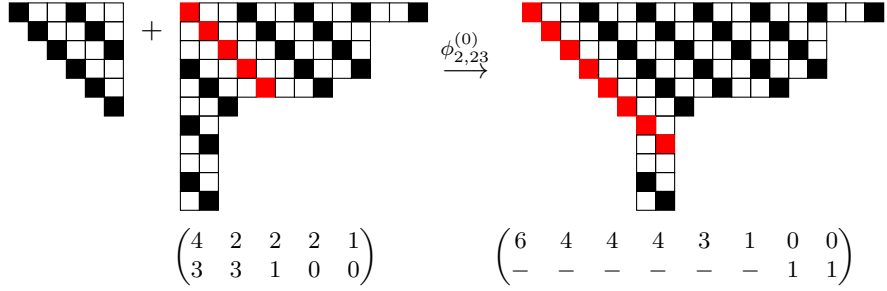
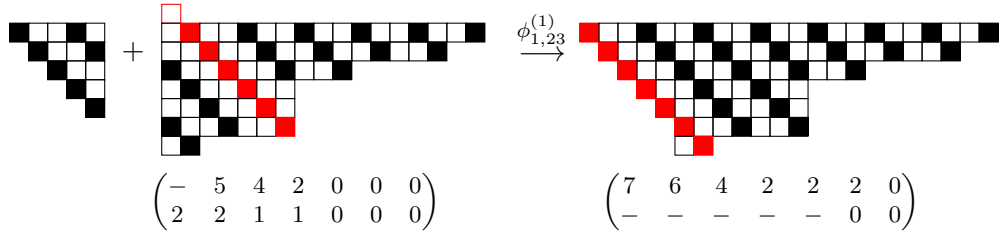
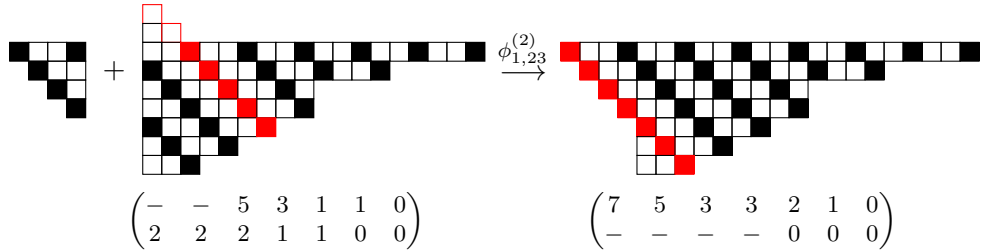
In general not much is known about the functions $f_{D_k,D_k,k'}(q)$ for arbitrary $k$ and $k'$. As mentioned earlier the functions $\Phi_k(q) (=f_{D_k,D_k,0}(q))$ are studied in Andrews' book \cite{frobenius}, and we saw that $\Phi_1(q)$ and $\Phi_2(q)$ have product expansions (which gave us product expansions for the specialised normalizing factors). It turns out that $\Phi_3(q)$ also has a product expansion (proved in Corollary $5.1$ of \cite{frobenius}, assuming Jacobi triple product) and so $S_0^{(3)}(q)$ can be written as: \begin{align*}S_0^{(3)}(q)=\Phi_3(q) &= 1+q+3q^2+6q^3+11q^4+18q^5+31q^6+49q^7+78q^8+...\\&=\prod_{i\geq 1}\frac{(1-q^{12i-6})}{(1-q^{6i-5})(1-q^{6i-4})^2(1-q^{6i-3})^3(1-q^{6i-2})^2(1-q^{6i-1})(1-q^{12i})}\end{align*}

As in the $2$-exclusion case it would be interesting to know if this could be proved using purely probabilistic methods. It would also be interesting to know whether the other functions $f_{D_3,D_3,-1}(q)$ and $f_{D_3,D_3,-2}(q)$ are products, as this would imply that the normalising factors $S_{-1}^{(3)}(q)$ and $S_{-2}^{(3)}(q)$ are products.

For $k\geq 4$ the functions $\Phi_k(q)$ are not expected to be products, but can be shown to be sums of products. However the formulae are very tedious to write down and probably not too illuminating.

We finish with another natural question. We have seen that the family of (non-degenerate) nearest neighbour interacting $0$-$1$-$2$ systems on $\mathbb{Z}$ satisfying the blocking measure axioms can be stood up in a uniform way, leading to three variable Jacobi style identities. By specialising to the case of $2$-exclusion we then get two variable identities that are a special case of Theorem \ref{k-exc main} (setting $k=2$). Is there a family of $0$-$1$-...-$k$ processes for each $k$ that is nicely behaved (i.e.\ has blocking measure and can be uniformly stood up) giving a multivariable identity that explains all cases of Theorem \ref{k-exc main}? In general we cannot expect the whole family of $0$-$1$...-$k$ systems on $\mathbb{Z}$ to be this well behaved (the stood up processes no longer have nearest neighbour interactions). However, there could be a low dimensional subfamily of processes for each $k$ that works.

\section*{Appendix}

The following are the non-empty sets $\text{GFP}_{D_2,D_2,0,m}(n)$ for $0\leq n\leq 8$, whose sizes agree with the coefficients in the expansion of $S_{\text{even}}(\Tilde{q},t)$ given in Section $3.3$.

\[\text{GFP}_{D_2,D_2,0,0}(0) = \left\{\binom{-}{-}\right\},\]
\[\text{GFP}_{D_2,D_2,0,2}(1) = \left\{\binom{0}{0}\right\},\]
\[\text{GFP}_{D_2,D_2,0,0}(2) = \left\{\binom{0\,0}{0\,0}\right\},\] 
\[\text{GFP}_{D_2,D_2,0,2}(2) = \left\{\binom{1}{0}, \binom{0}{1}\right\},\]
\[\text{GFP}_{D_2,D_2,0,2}(3) = \left\{\binom{2}{0}, \binom{1}{1}, \binom{0}{2}, \binom{1\,0}{0\,0}, \binom{0\,0}{1\,0}\right\},\]
\[\text{GFP}_{D_2,D_2,0,0}(4) = \left\{\binom{1\,1}{0\,0}, \binom{0\,0}{1\,1}\right\},\] 
\[\text{GFP}_{D_2,D_2,0,2}(4) = \left\{\binom{3}{0}, \binom{2}{1}, \binom{1}{2}, \binom{0}{3},  \binom{2\,0}{0\,0}, \binom{0\,0}{2\,0}\right\},\] 
\[\text{GFP}_{D_2,D_2,0,4}(4) = \left\{\binom{1\,0}{1\,0}\right\},\]
\[\text{GFP}_{D_2,D_2,0,2}(5) = \left\{\binom{4}{0}, \binom{3}{1}, \binom{2}{2}, \binom{1}{3}, \binom{0}{4},  \binom{3\,0}{0\,0}, \binom{2\,1}{0\,0}, \binom{1\,1}{1\,0}, \binom{1\,0}{1\,1},  \binom{0\,0}{2\,1}, \binom{0\,0}{3\,0}, \binom{1\,0\,0}{1\,0\,0} \right\},\] 
\[\text{GFP}_{D_2,D_2,0,4}(5) = \left\{\binom{2\,0}{1\,0}, \binom{1\,0}{2\,0}\right\},\]
\[\text{GFP}_{D_2,D_2,0,0}(6) = \left\{\binom{2\,2}{0\,0}, \binom{1\,1}{1\,1}, \binom{0\,0}{2\,2}\right\},\] 
\begin{align*}\text{GFP}_{D_2,D_2,0,2}(6) = \bigg\{&\binom{5}{0}, \binom{4}{1}, \binom{3}{2}, \binom{2}{3}, \binom{1}{4}, \binom{0}{5},  \binom{4\,0}{0\,0}, \binom{3\,1}{0\,0}, \binom{2\,0}{1\,1},\\ &\binom{1\,1}{2\,0}, \binom{0\,0}{3\,1}, \binom{0\,0}{4\,0}, \binom{2\,0\,0}{1\,0\,0}, \binom{1\,1\,0}{1\,0\,0}, \binom{1\,0\,0}{1\,1\,0}, \binom{1\,0\,0}{2\,0\,0}\bigg\},\end{align*}
\[\text{GFP}_{D_2,D_2,0,4}(6) = \left\{\binom{3\,0}{1\,0},\binom{2\,1}{1\,0}, \binom{2\,0}{2\,0}, \binom{1\,0}{2\,1}, \binom{1\,0}{3\,0}\right\},\]
\begin{align*}\text{GFP}_{D_2,D_2,0,2}(7) = \bigg\{&\binom{6}{0},\binom{5}{1}, \binom{4}{2}, \binom{3}{3}, \binom{2}{4}, \binom{1}{5}, \binom{0}{6}, \binom{5\,0}{0\,0}, \binom{4\,1}{0\,0}, \binom{3\,2}{0\,0},\\&\binom{2\,2}{1\,0}, \binom{3\,0}{1\,1}, \binom{2\,1}{1\,1},\binom{1\,1}{2\,1},\binom{1\,1}{3\,0}, \binom{1\,0}{2\,2}, \binom{0\,0}{3\,2}, \binom{0\,0}{4\,1}, \\&\binom{0\,0}{5\,0},
\binom{3\,0\,0}{1\,0\,0},\binom{2\,0\,0}{2\,0\,0}, \binom{2\,0\,0}{1\,1\,0}, \binom{1\,1\,0}{2\,0\,0}, \binom{1\,1\,0}{1\,1\,0}, \binom{1\,0\,0}{3\,0\,0} \bigg\},\end{align*}
\[\text{GFP}_{D_2,D_2,0,4}(7) = \left\{\binom{4\,0}{1\,0}, \binom{3\,1}{1\,0}, \binom{3\,0}{2\,0}, \binom{2\,1}{2\,0}, \binom{2\,0}{2\,1}, \binom{2\,0}{3\,0}, \binom{1\,0}{3\,1}, \binom{1\,0}{4\,0}, \binom{2\,1\,0}{1\,0\,0}, \binom{1\,0\,0}{2\,1\,0}\right\},\]
\[\text{GFP}_{D_2,D_2,0,0}(8) = \left\{\binom{3\,3}{0\,0},\binom{2\,2}{1\,1}, \binom{1\,1}{2\,2}, \binom{0\,0}{3\,3}, \binom{1\,1\,0\,0}{1\,1\,0\,0}\right\},\]
\begin{align*}\text{GFP}_{D_2,D_2,0,2}(8) = \bigg\{&\binom{7}{0}, \binom{6}{1}, \binom{5}{2},\binom{4}{3}, \binom{3}{4}, \binom{2}{5}, \binom{1}{6}, \binom{0}{7}, \binom{6\,0}{0\,0}, \binom{5\,1}{0\,0}, \binom{4\,2}{0\,0}, \binom{4\,0}{1\,1},\\ &\binom{3\,1}{1\,1}, \binom{2\,2}{2\,0}, \binom{2\,0}{2\,2}, \binom{1\,1}{3\,1}, \binom{1\,1}{4\,0}, \binom{0\,0}{4\,2}, \binom{0\,0}{5\,1}, \binom{0\,0}{6\,0}, \binom{1\,0\,0}{4\,0\,0}, \binom{2\,2\,0}{1\,0\,0},\\ &\binom{2\,1\,1}{1\,0\,0}, \binom{1\,0\,0}{2\,1\,1}, \binom{3\,0\,0}{2\,0\,0}, \binom{2\,0\,0}{3\,0\,0}, \binom{3\,0\,0}{1\,1\,0}, \binom{1\,1\,0}{3\,0\,0}, \binom{1\,0\,0}{2\,2\,0}, \binom{4\,0\,0}{1\,0\,0} \bigg\},\end{align*}
\begin{align*}\text{GFP}_{D_2,D_2,0,4}(8) = \bigg\{&\binom{5\,0}{1\,0}, \binom{4\,1}{1\,0}, \binom{3\,2}{1\,0}, \binom{4\,0}{2\,0}, \binom{3\,1}{2\,0}, \binom{3\,0}{3\,0}, \binom{2\,1}{3\,0}, \binom{2\,1}{2\,1}, \binom{3\,0}{2\,1} \binom{2\,0}{3\,1}, \binom{2\,0}{4\,0},\\ &\binom{1\,0}{3\,2}, \binom{1\,0}{4\,1}, \binom{1\,0}{5\,0}, \binom{3\,1\,0}{1\,0\,0}, \binom{2\,1\,0}{2\,0\,0}, \binom{2\,1\,0}{1\,1\,0}, \binom{1\,1\,0}{2\,1\,0}, \binom{2\,0\,0}{2\,1\,0}, \binom{1\,0\,0}{3\,1\,0}\bigg\}.\end{align*}

The following are the non-empty sets $\text{GFP}_{D_2,D_2,-1,m}(n)$ for $0\leq n\leq 8$, whose sizes agree with the coefficients in the expansion of $S_{\text{odd}}(\Tilde{q},t)$ given in Section $3.3$.

\[\text{GFP}_{D_2,D_2,-1,1}(0) = \left\{\binom{-}{0}\right\},\]
\[\text{GFP}_{D_2,D_2,-1,1}(1)  = \left\{\binom{-}{1},\binom{-\,0}{0\,\,0}\right\},\]
\[\text{GFP}_{D_2,D_2,-1,1}(2)  = \left\{\binom{-}{2}, \binom{-\,1}{0\,\,0}\right\},\] 
\[\text{GFP}_{D_2,D_2,-1,3}(2)  = \left\{\binom{-\,0}{1\,\,0}\right\},\]
\[\text{GFP}_{D_2,D_2,-1,1}(3)  = \left\{\binom{-}{3}, \binom{-\,2}{0\,\,0}, \binom{-\,0}{1\,\,1}, \binom{-\,0\,0}{1\,\,0\,0}\right\},\]
\[\text{GFP}_{D_2,D_2,-1,3}(3)  = \left\{\binom{-\,1}{1\,\,0}, \binom{-\,0}{2\,\,0}\right\},\] 
\[\text{GFP}_{D_2,D_2,-1,1}(4)  = \left\{\binom{-}{4}, \binom{-\,3}{0\,\,0}, \binom{-\,1}{1\,\,1}, \binom{-\,0\,0}{2\,\,0\,0},  \binom{-\,0\,0}{1\,\,1\,0}\right\},\] 
\[\text{GFP}_{D_2,D_2,-1,3}(4)  = \left\{\binom{-\,2}{1\,\,0},\binom{-\,1}{2\,\,0},\binom{-\,0}{2\,\,1},\binom{-\,0}{3\,\,0},\binom{-\,1\,0}{1\,\,0\,0}\right\},\]
\[\text{GFP}_{D_2,D_2,-1,1}(5)  = \left\{\binom{-}{5}, \binom{-\,4}{0\,\,0}, \binom{-\,2}{1\,\,1}, \binom{-\,0}{2\,\,2}, \binom{-\,1\,1}{1\,\,0\,0},  \binom{-\,0\,0}{3\,\,0\,0} \right\},\] 
\begin{align*}\text{GFP}_{D_2,D_2,-1,3}(5)  = \bigg\{&\binom{-\,3}{1\,\,0}, \binom{-\,2}{2\,\,0},\binom{-\,1}{2\,\,1},\binom{-\,1}{3\,\,0},\binom{-\,0}{3\,\,1},\binom{-\,0}{4\,\,0},\\&\binom{-\,2\,0}{1\,\,0\,0},\binom{-\,1\,0}{1\,\,1\,0},\binom{-\,1\,0}{2\,\,0\,0},\binom{-\,0\,0}{2\,\,1\,0}\bigg\},\end{align*}
\begin{align*}\text{GFP}_{D_2,D_2,-1,1}(6)  = \bigg\{&\binom{-}{6}, \binom{-\,5}{0\,\,0}, \binom{-\,3}{1\,\,1},\binom{-\,1}{2\,\,2},\binom{-\,1\,1}{1\,\,1\,0},\binom{-\,1\,1}{2\,\,0\,0},\\&\binom{-\,0\,0}{2\,\,1\,1},\binom{-\,0\,0}{2\,\,2\,0},\binom{-\,0\,0}{4\,\,0\,0},\binom{-\,1\,0\,0}{1\,\,1\,0\,0}\bigg\},\end{align*} 
\begin{align*}\text{GFP}_{D_2,D_2,-1,3}(6)  = \bigg\{&\binom{-\,4}{1\,\,0}, \binom{-\,3}{2\,\,0}, \binom{-\,2}{2\,\,1}, \binom{-\,2}{3\,\,0}, \binom{-\,1}{3\,\,1}, \binom{-\,1}{4\,\,0},  \binom{-\,0}{3\,\,2}, \binom{-\,0}{4\,\,1}, \\&\binom{-\,0}{5\,\,0},\binom{-\,3\,0}{1\,\,0\,0}, \binom{-\,2\,1}{1\,\,0\,0}, \binom{-\,2\,0}{1\,\,1\,0}, \binom{-\,2\,0}{2\,\,0\,0}, \binom{-\,1\,0}{3\,\,0\,0}, \binom{-\,0\,0}{3\,\,1\,0}\bigg\},\end{align*}
\[\text{GFP}_{D_2,D_2,-1,5}(6)  = \left\{\binom{-\,1\,0}{2\,\,1\,0}\right\},\]
\begin{align*}\text{GFP}_{D_2,D_2,-1,1}(7)  = \bigg\{&\binom{-}{7},\binom{-\,6}{0\,\,0}, \binom{-\,4}{1\,\,1}, \binom{-\,2}{2\,\,2}, \binom{-\,0}{3\,\,3}, \binom{-\,2\,2}{1\,\,0\,0}, \binom{-\,1\,1}{3\,\,0\,0},\\& \binom{-\,0\,0}{2\,\,2\,1}, \binom{-\,0\,0}{3\,\,1\,1}, \binom{-\,0\,0}{5\,\,0\,0},\binom{-\,2\,0\,0}{1\,\,1\,0\,0}, \binom{-\,1\,1\,0}{1\,\,1\,0\,0} \bigg\},\end{align*}
 \begin{align*}\text{GFP}_{D_2,D_2,-1,3}(7)  =\bigg\{&\binom{-\,5}{1\,\,0}, \binom{-\,4}{2\,\,0}, \binom{-\,3}{2\,\,1}, \binom{-\,3}{3\,\,0}, \binom{-\,2}{3\,\,1}, \binom{-\,2}{4\,\,0}, \binom{-\,1}{3\,\,2}, \binom{-\,1}{4\,\,1}, \binom{-\,1}{5\,\,0}, \binom{-\,0}{4\,\,2},\\&\binom{-\,0}{5\,\,1},\binom{-\,0}{6\,\,0},\binom{-\,4\,0}{1\,\,0\,0},\binom{-\,3\,1}{1\,\,0\,0},\binom{-\,3\,0}{1\,\,1\,0},\binom{-\,3\,0}{2\,\,0\,0},\binom{-\,2\,1}{1\,\,1\,0},\binom{-\,2\,1}{2\,\,0\,0},\\&\binom{-\,2\,0}{3\,\,0\,0},\binom{-\,1\,1}{2\,\,1\,0},\binom{-\,1\,0}{2\,\,1\,1},\binom{-\,1\,0}{2\,\,2\,0},\binom{-\,1\,0}{4\,\,0\,0},\binom{-\,0\,0}{3\,\,2\,0},\binom{-\,0\,0}{4\,\,1\,0},\binom{-\,1\,0\,0}{2\,\,1\,0\,0} \bigg\},\end{align*}
\[\text{GFP}_{D_2,D_2,-1,5}(7)  = \left\{\binom{-\,1\,0}{3\,\,1\,0},\binom{-\,2\,0}{2\,\,1\,0}, \right\},\]
\begin{align*}\text{GFP}_{D_2,D_2,-1,1}(8)  = \bigg\{&\binom{-}{8}, \binom{-\,7}{0\,\,0}, \binom{-\,5}{1\,\,1},\binom{-\,3}{2\,\,2}, \binom{-\,1}{3\,\,3}, \binom{-\,2\,2}{1\,\,1\,0}, \binom{-\,2\,2}{2\,\,0\,0}, \binom{-\,1\,1}{2\,\,1\,1}, \\ &\binom{-\,1\,1}{2\,\,2\,0}, \binom{-\,1\,1}{4\,\,0\,0}, \binom{-\,0\,0}{3\,\,3\,0}, \binom{-\,0\,0}{4\,\,1\,1
},\binom{-\,0\,0}{6\,\,0\,0},\binom{-\,3\,0\,0}{1\,\,1\,0\,0}, \binom{-\,1\,0\,0}{2\,\,2\,0\,0} \bigg\},\end{align*}
\begin{align*}\text{GFP}_{D_2,D_2,-1,3}(8)  = \bigg\{&\binom{-\,6}{1\,\,0}, \binom{-\,5}{2\,\,0}, \binom{-\,4}{2\,\,1}, \binom{-\,4}{3\,\,0}, \binom{-\,3}{3\,\,1}, \binom{-\,3}{4\,\,0}, \binom{-\,2}{3\,\,2}, \binom{-\,2}{4\,\,1}, \binom{-\,2}{5\,\,0} \binom{-\,1}{4\,\,2}, \\&\binom{-\,1}{5\,\,1},\binom{-\,1}{6\,\,0},\binom{-\,0}{4\,\,3}, \binom{-\,0}{5\,\,2}, \binom{-\,0}{6\,\,1}, \binom{-\,0}{7\,\,0}, \binom{-\,5\,0}{1\,0\,0}, \binom{-\,4\,1}{1\,\,0\,0}, \binom{-\,3\,2}{1\,\,0\,0}, \\&\binom{-\,4\,0}{1\,\,1\,0},\binom{-\,4\,0}{2\,\,0\,0}, \binom{-\,3\,1}{1\,\,1\,0}, \binom{-\,3\,1}{2\,\,0\,0},  \binom{-\,3\,0}{3\,\,0\,0},\binom{-\,2\,1}{3\,\,0\,0},\binom{-\,2\,0}{2\,\,1\,1}, \binom{-\,2\,0}{2\,\,2\,0},\\&\binom{-\,2\,0}{4\,\,0\,0},\binom{-\,1\,1}{3\,\,1\,0}, \binom{-\,1\,0}{2\,\,2\,1}, \binom{-\,1\,0}{3\,\,1\,1},\binom{-\,1\,0}{5\,\,0\,0},\binom{-\,0\,0}{3\,\,2\,1},\binom{-\,0\,0}{4\,\,2\,0},\binom{-\,0\,0}{5\,\,1\,0},\\&\binom{-\,2\,1\,0}{1\,\,1\,0\,0},\binom{-\,2\,0\,0}{2\,\,1\,0\,0},\binom{-\,1\,1\,0}{2\,\,1\,0\,0},\binom{-\,1\,0\,0}{2\,\,1\,1\,0},\binom{-\,1\,0\,0}{3\,\,1\,0\,0}\bigg\}.\end{align*}
\[\text{GFP}_{D_2,D_2,-1,5}(8)  = \left\{\binom{-\,3\,0}{2\,\,1\,0},\binom{-\,2\,1}{2\,\,1\,0},\binom{-\,2\,0}{3\,\,1\,0},\binom{-\,1\,0}{3\,\,2\,0},\binom{-\,1\,0}{4\,\,1\,0} \right\}.\]


\begin{thebibliography}{9}
\bibitem{frobenius} G. Andrews. \textit{``Generalized Frobenius partitions"}, American Mathematical Society, Volume 49, Number 301, (1984).
\bibitem{3 state}
M. Bal\'azs.
\textit{``Growth fluctuations in a class of deposition models"}, Ann. Inst.H. Poincar\'e  Probab. Statist., Volume 39, Issue 4 (2003) 639-685.
\bibitem{blocking}
M. Bal\'azs and R. Bowen.
\textit{``Product blocking measures and a particle system proof of the Jacobi triple product "}, Ann. Inst.H. Poincar\'e  Probab. Statist., Volume 54, Number 1 (2018) 514-528.


\bibitem{redig} G.Carinci, C. Giardin\'a, F. Redig and T. Sasamoto.
\textit{``A generalized asymmetric exclusion process with $U_q(\mathfrak{sl}_2)$ stochastic duality "}, Probab. Theory Relat. Fields, 166: 887-933, (2016).

\bibitem{janos} J. Engl\"ander and S. Volkov. \textit{``Turning a coin over instead of tossing it"}, Journal of Theoretical Probability, 31(2):1097-1118, (2018). 

\bibitem{liggett} T. Liggett \textit{``Interacting Particle Systems"}, Classics in Mathematics, Sprigner, (2005).

\bibitem{OEIS1} Online Encyclopedia of Integer Sequences (A$137829$), \url{https://oeis.org/A137829}.

\bibitem{OEIS2} Online Encyclopedia of Integer Sequences (A$201077$), \url{https://oeis.org/A201077}.
\end{thebibliography}
\end{document}